\documentclass[a4paper]{article}
\usepackage[utf8]{inputenc}
\usepackage[T1]{fontenc}
\usepackage{graphicx}
\usepackage{lmodern}
\usepackage{geometry}
\usepackage{amsmath}
\usepackage{amssymb}
\usepackage{amsfonts}
\usepackage{amsthm}
\allowdisplaybreaks
\usepackage[capitalize]{cleveref}
\usepackage{nameref}
\usepackage{thmtools, thm-restate}
\usepackage{array}
\usepackage{ifthen}
\newboolean{extended}
\setboolean{extended}{true}

\usepackage[numbers]{natbib}
\newcommand{\textcite}{\citet}

\newcommand{\R}{\mathbb{R}}
\newcommand{\N}{\mathbb{N}}
\newcommand{\Z}{\mathbb{Z}}
\newcommand{\argmin}{\mathop{\mathrm{argmin}}}
\newcommand{\argmax}{\mathop{\mathrm{argmax}}}

\newcommand{\pen}{\mathop{\mathrm{pen}}}
\renewcommand{\d}{d}
\newcommand{\E}{\mathbb{E}}
\newcommand{\Prob}{\mathbb{P}}
\newcommand{\widebar}{\overline}
\newcommand{\ud}{\mathrm{d}}

\newcommand{\Log}{\ln}
\newcommand{\Span}{\mathop{\mathrm{span}}}
\newcommand{\Characonly}{\mathbf{1}}
\newcommand{\Charac}[1]{\Characonly_{\left\{ #1 \right\}}}
\renewcommand{\L}{L}

\newcommand{\SpaceX}{\mathcal{X}}
\newcommand{\SpaceY}{\mathcal{Y}}

\newcommand{\Gauss}{\Phi}
\newcommand{\dimSp}{p}
\newcommand{\Part}{\mathcal{P}}
\newcommand{\PartB}{\mathcal{Q}}
\newcommand{\PartX}{\Part}
\newcommand{\PartY}{\PartB}
\newcommand{\PartXY}{\PartB^\Part}

\newcommand{\Leaf}{\mathcal{R}}
\newcommand{\LeafX}{\Leaf}
\newcommand{\LeafY}{\Leaf^{\scriptscriptstyle\SpaceY}}
\newcommand{\LeafXY}{\Leaf^{\times}}
\newcommand{\LeafXl}{\LeafX_l}
\newcommand{\LeafYlk}{\LeafY_{l,k}}

\newcommand{\LeafXYlkp}{\LeafXY_{l',k'}}
\newcommand{\NbPartX}{\|\PartX\|}
\newcommand{\NbPartY}{\|\PartY\|}
\newcommand{\PartYLeafXl}{\PartY_{l}}

\newcommand{\NbPartYLeafXl}{\|\PartYLeafXl\|}

\newcommand{\NbPartXY}{\|\PartXY\|}
\newcommand{\LeafXYlk}{\LeafXY_{l,k}}

\newcommand{\dimX}{d_X}
\newcommand{\dimY}{d_Y}
\newcommand{\dimP}{D}
\newcommand{\BdimP}{\mathbf{D}}

\newcommand{\dimPY}{\dimP}
\newcommand{\dimPYi}[1]{\dimP_{#1}}
\newcommand{\dimPmax}{\BdimP}
\newcommand{\dimPYmax}{\BdimP}
\newcommand{\dimPYmaxi}[1]{\BdimP_{#1}}

\newcommand{\dimPYfamily}{\mathcal{D}^{M}}
\newcommand{\dimPYmaxPartXY}{\dimPYmax}

\newcommand{\UDP}{\ensuremath{\mathrm{UDP}}}
\newcommand{\RDP}{\ensuremath{\mathrm{RDP}}}
\newcommand{\RDSP}{\ensuremath{\mathrm{RDSP}}}
\newcommand{\RSP}{\ensuremath{\mathrm{RSP}}}
\newcommand{\HRP}{\ensuremath{\mathrm{HRP}}}
\newcommand{\UDPX}{\UDP({\SpaceX})}
\newcommand{\RDPX}{\RDP({\SpaceX})}
\newcommand{\RDSPX}{\RDSP({\SpaceX})}
\newcommand{\RSPX}{\RSP({\SpaceX})}
\newcommand{\HRPX}{\HRP({\SpaceX})}
\newcommand{\starX}{\star({\SpaceX})}
\newcommand{\UDPY}{\UDP({\SpaceY})}
\newcommand{\RDPY}{\RDP({\SpaceY})}
\newcommand{\RDSPY}{\RDSP({\SpaceY})}
\newcommand{\RSPY}{\RSP({\SpaceY})}
\newcommand{\HRPY}{\HRP({\SpaceY})}
\newcommand{\starY}{\star({\SpaceY})}

\newcommand{\CollPart}{\Models_{\Part}}

\newcommand{\meas}{\lambda}
\newcommand{\tens}{\otimes_n}

\newcommand{\Bat}{\mathcal{D}}
\newcommand{\dtens}{\d^{\otimes_n}}
\newcommand{\dtwotens}{\d^{2\otimes_n}}
\newcommand{\dtwosup}{\d^{2\sup}}
\newcommand{\Battwotens}{\Bat^{2\otimes_n}}
\newcommand{\dsup}{\d^{\sup}}
\newcommand{\dmax}{\d^{\max}}
\newcommand{\dtwomax}{\d^{2\max}}
\newcommand{\KL}{\ensuremath{K\!L}}
\newcommand{\KLtens}{\KL^{\otimes_n}}
\newcommand{\kl}{\ensuremath{k\mspace{-1mu}l}}
\newcommand{\JKL}{\ensuremath{J\!K\!L}}
\newcommand{\JKLtens}{\JKL^{\otimes_n}}
\newcommand{\jkl}{\ensuremath{j\mspace{-1mu}k\mspace{-1mu}l}}
\newcommand{\propJKL}{\rho}
\newcommand{\ConstHellKL}{C}

\newcommand{\indmset}{\mathcal{M}}
\newcommand{\indm}{m}
\newcommand{\setmodel}{S}
\newcommand{\model}{S}
\newcommand{\modeldisc}{\mathfrak{S}}
\newcommand{\sdisc}{\mathfrak{s}}
\newcommand{\Models}{\mathcal{S}}

\newcommand{\Pemp}{P^{\otimes_n}_{n}}
\newcommand{\Ptens}{P^{\otimes_n}}
\newcommand{\Prec}{\nu^{\otimes_n}_{n}}

\newcommand{\paramphi}{\beta_{\phi}}

\newcommand{\constcore}{\theta}
\newcommand{\constcoreepsipen}{\theta_{\pen}}
\newcommand{\constpenmin}{\kappa_0}
\newcommand{\constpen}{\kappa}
\newcommand{\constpensimple}{\widetilde{\kappa}}

\newcommand{\constoraone}{C_1}
\newcommand{\constoratwo}{C_2}
\newcommand{\constlemmaone}{\kappa'_1}
\newcommand{\constlemmatwo}{\kappa'_2}
\newcommand{\constlemmaonebis}{\kappa''_1}
\newcommand{\constlemmatwobis}{\kappa''_2}

\newcommand{\paramrhoKL}{\delta_{\KL}}
\newcommand{\rhoapp}{\eta}
\newcommand{\rhomin}{\eta'}
\newcommand{\paramepsid}{\epsilon_{d}}
\newcommand{\paramB}{\kappa'}
\newcommand{\paramBd}{{\kappa'_{d}}}
\newcommand{\constlemmathree}{\kappa'_0}
\newcommand{\paramepsipen}{\epsilon_{\pen}}

\newcommand{\paramepsiconstlemmathree}{\epsilon}
\newcommand{\kappalemma}{\kappa}
\newcommand{\klbarsindm}{\kl(\widebar{s}_{\indm})}
\newcommand{\klhatsindm}{\kl(\widehat{s}_{\indm})}
\newcommand{\klhatsindmp}{\kl(\widehat{s}_{\indm'})}
\newcommand{\jklhatsindm}{\jkl(\widehat{s}_{\indm})}
\newcommand{\jklsindm}{\jkl(s_{\indm})}
\newcommand{\jkls}{\jkl(s)}
\newcommand{\jklhatsindmp}{\jkl(\widehat{s}_{\indm'})}
\newcommand{\jkltildes}{\jkl(\widetilde{s})}
\newcommand{\jkltildesindm}{\jkl(\widetilde{s}_{\indm})}

\newcommand{\ConstH}{\mathcal{V}}
\newcommand{\DimH}{\mathcal{D}}
\newcommand{\DIMH}{\mathfrak{D}}
\newcommand{\ConstMultH}{\mathcal{C}}

\newcommand{\constUentropy}{C}
\newcommand{\Set}{\mathcal{G}}
\newcommand{\DSet}{\mathcal{F}}
\newcommand{\DDSet}{\mathcal{F}}
\newcommand{\Simplex}{\mathcal{S}}
\newcommand{\Space}{E}
\newcommand{\dimSpace}{\dimSp_{\Space}}
\newcommand{\dimSpaceperp}{\dimSp_{\Space^\perp}}

\newcommand{\constspatgausslog}{C_\star}
\newcommand{\constspatgausslogindm}{C_{\star,\indm}}
\newcommand{\constspatgausstwo}{c_\star}
\newcommand{\CstSpace}{\mathcal{D}_\Space}
\newcommand{\DimHSimpKm}{(K-1)}
\newcommand{\Diag}{\mathcal{A}}
\newcommand{\Grid}{\mathcal{G}}
\newcommand{\ConstSzarek}{c_S}

\newcommand{\mumod}{\mathrm{\mu}}
\newcommand{\Lmod}{\mathrm{L}}
\newcommand{\Dmod}{\mathrm{D}}
\newcommand{\Amod}{\mathrm{A}}
\newcommand{\modany}{\star}
\newcommand{\modknown}{0}
\newcommand{\modfree}{K}
\newcommand{\mumodany}{\mumod_{\modany}}
\newcommand{\Lmodany}{\Lmod_{\modany}}
\newcommand{\Dmodany}{\Dmod_{\modany}}
\newcommand{\Amodany}{\Amod_{\modany}}
\newcommand{\mumodknown}{\mumod_{\modknown}}
\newcommand{\Lmodknown}{\Lmod_{\modknown}}
\newcommand{\Dmodknown}{\Dmod_{\modknown}}
\newcommand{\Amodknown}{\Amod_{\modknown}}
\newcommand{\mumodfree}{\mumod_{\modfree}}
\newcommand{\Lmodfree}{\Lmod_{\modfree}}
\newcommand{\Dmodfree}{\Dmod_{\modfree}}
\newcommand{\Amodfree}{\Amod_{\modfree}}
\newcommand{\mumodsame}{\mumod}
\newcommand{\Lmodsame}{\Lmod}
\newcommand{\Dmodsame}{\Dmod}
\newcommand{\Amodsame}{\Amod}
\newcommand{\starK}{\star}
\newcommand{\starmuLDA}{\star}
\newcommand{\deltaVar}{\delta_{\Sigma}}
\newcommand{\lambdam}{\lambda_{-}}
\newcommand{\lambdaM}{\lambda_{+}}
\newcommand{\Lm}{L_{-}}
\newcommand{\LM}{L_{+}}

\ifthenelse{\boolean{extended}}{
\newcommand{\kappag}{\kappa}

\newcommand{\gammakappa}{\gamma_\kappa}
\newcommand{\constcosh}{\beta_{\kappa}}
}{
\newcommand{\kappag}{}

\newcommand{\gammakappa}{\gamma}
\newcommand{\constcosh}{\beta}
}

\newcommand{\constpolyparttwo}{C_\star}
\newcommand{\constpolyparttwobis}{D_\star}
\newcommand{\constpolypartthree}{c_\star}
\newcommand{\dimPart}{d}

\newcommand{\PartCstA}{A_0}
\newcommand{\PartCstB}{B_0}
\newcommand{\PartCstC}{C_0}
\newcommand{\PartCstc}{c_0}
\newcommand{\PartCstSigma}{\Sigma_0}

\newcommand{\Hindm}{(\ensuremath{\text{H}_{\indm}})}
\newcommand{\Sepindm}{(\ensuremath{\text{Sep}_{\indm}})}

\declaretheorem{theorem}
\declaretheorem{lemma}
\declaretheorem{proposition}

\declaretheorem[numbered=no,title=Assumption \Hindm]{assumptionHm}

\declaretheorem[numbered=no,title=Assumption (K)]{assumptionK}

\declaretheorem[numbered=no,title=Assumption \Sepindm]{assumptionSm}

\newenvironment{alignI}{\begin{math}\displaystyle}{\end{math}}

\begin{document}

\title{Conditional Density Estimation by Penalized Likelihood Model
  Selection and Applications}
\author{S. X. Cohen (IPANEMA / Synchrotron Soleil)\\ and\\ E. Le Pennec
  (SELECT / Inria
  Saclay - Île de France and Université Paris Sud)}

\maketitle

\begin{abstract}
In this technical report, we consider conditional density estimation
with a maximum likelihood approach. Under weak assumptions, we
obtain a theoretical bound for a Kullback-Leibler type loss for a
single model maximum likelihood estimate. We use a penalized model
selection technique to select a best model within a collection. We
give a general condition on penalty choice that leads to oracle type inequality 
for the resulting estimate. This construction is applied to two
examples of partition-based conditional density models, models
in which the conditional density depends only in a piecewise manner
from the covariate. The first example relies on
classical piecewise polynomial densities while the second uses
Gaussian mixtures with varying mixing proportion but same mixture
components. We show how this last case is related to an
unsupervised segmentation application that has been the source of our
motivation to this study.
\end{abstract}


\section{Introduction}

Assume we observe $n$ pairs $\left((X_i,Y_i)\right)_{1 \leq i \leq n}$ of random variables, we are
interested in estimating the law of the second variable $Y_i \in \SpaceY$
conditionally to the first one $X_i \in \SpaceX$.
In this paper, we assume that
the pairs $(X_i,Y_i)$ are independent while $Y_i$ depends on $X_i$
through its law. More precisely, we assume that the covariates $X_i$
are independent but not necessarily identically distributed. 
Assumptions on the $Y_i$s are stronger: we assume that, conditionally to
the $X_i$s, they are
independents and  each variable $Y_i$ follows a law with density
$s_0(\cdot|X_i)$ with respect to a common known measure $\ud\meas$.
Our goal is to estimate this two-variables conditional density
function $s_0(\cdot|\cdot)$ from the observations.

This problem has been introduced by \textcite{rosenblatt69:_condit} in
the late 60's. He considered a stationary framework in which
$s_0(y|x)$ is linked to the supposed existing 
densities $s_{0'}(x)$ and $s_{0''}(x,y)$  of respectively $X_i$
and $(X_i,Y_i)$ by
\begin{align*}
  s_0(y|x)=\frac{s_{0''}(x,y)}{s_{0'}(x)},
\end{align*}
and proposed a plugin estimate based on kernel estimation of both
$s_{0'}(x)$ and $s_{0''}(x,y)$. Few other references on this subject seem to
exist before the mid 90's with a study of a spline tensor based
maximum likelihood estimator
proposed by \textcite{stone94:_use_polyn_splin_their_tensor}
and a bias correction of
\citeauthor{rosenblatt69:_condit}'s estimator due to
\textcite{hyndman96:_estim}. 

Kernel based method have been much studied since. For instance,
\textcite{fan96:_estim} and \textcite{gooijer03} 
consider local polynomial estimator, 
\textcite{hall99:_method} study a locally logistic estimator
that is later extended by \textcite{hyndman02:_nonpar}. In this
setting, pointwise convergence properties are considered, and extensions to
dependent data are often obtained.
The results depend however
on a critical bandwidth that should be chosen according to the
regularity of the unknown conditional density. Its
practical choice is rarely discussed with the notable exceptions of
\textcite{bashtannyk01:_bandw}, \textcite{fan04:_cross_method_estim_condit_densit} and \textcite{hall04:_cross_valid_estim_condit_probab_densit}. Extensions to censored cases have also
been discussed for instance by \textcite{keilegom02:_densit}. See for
instance \textcite{li07:_nonpar_econom} for a comprehensive review of
this topic.

In the approach of \textcite{stone94:_use_polyn_splin_their_tensor},
the conditional density is estimated through a parametrized
modelization. This idea has been reused since by \textcite{gyorfi07:_nonpar}
with a histogram based approach, by
\textcite{efromovich07:_condit,efromovich10:_oracl}  with a Fourier
basis, and by
\textcite{brunel07:_adapt_estim_condit_densit_presen_censor} and \textcite{akakpo11:_inhom}
with piecewise polynomial representation. Those
authors are able to control an integrated estimation error: with
an integrated total variation loss for the first one and a quadratic distance loss for the others.
Furthermore, in the quadratic framework, they manage to construct
adaptive estimators, estimators that do not require the knowledge of
the regularity to be minimax optimal (up to a logarithmic factor),
using respectively a blockwise attenuation principle and a model
selection by penalization approach. Note that
\textcite{brunel07:_adapt_estim_condit_densit_presen_censor} extend their
result to censored cases while \textcite{akakpo11:_inhom} are able to
consider weakly dependent data. 

In this paper, we consider a direct estimation of the conditional
density function through a maximum likelihood approach. Although
natural, this approach has been considered so far only by \textcite{stone94:_use_polyn_splin_their_tensor}
as mentioned before and by \textcite{blanchard97:_optim} in a classification
setting with histogram type estimators. Assume we have a set $\model_{\indm}$ of
candidate conditional densities, our estimate $\widehat{s}_{\indm}$ will be simply the
maximum likelihood estimate
 \begin{align*}
   \widehat{s}_{\indm} = \argmin_{s_{\indm}\in\model_{\indm}} \left( - \sum_{i=1}^n
   \Log s_{\indm}(Y_i|X_i) \right).
 \end{align*}
Although this estimator may look like a maximum likelihood estimator of the
joint density of $(X_i,Y_i)$, it does not generally coincide, even when the $X_i$s are assumed to be
i.i.d., with such
an estimator as every function of $\model_{\indm}$
is assumed to be a conditional density and not a density. 
 The only exceptions are 
when the $X_i$s are assumed in the model to be i.i.d. uniform or non
random and equal.
Our aim is then to analyze the finite sample performance of such an
estimator in term of Kullback-Leibler type loss. As often, a trade-off
between a bias term measuring the closeness of $s_0$ to the set
$\model_{\indm}$ and a variance term depending on the complexity of the set
$\model_{\indm}$ and on the sample size appears. A good set $\model_{\indm}$ will be
thus one for which this trade-off leads to a small risk bound. Using a
penalized model
selection approach, we propose then a way to select the best model
$\model_{\widehat{\indm}}$ among a collection
$\Models=(\model_\indm)_{\indm\in\indmset}$. For
a given family of penalties
$\pen(\indm)$, we define
the \emph{best} model $\model_{\widehat{\indm}}$ as the one
that minimized
\begin{align*}
  \widehat{\indm} = \argmin_{\indm\in\indmset} \left(- \sum_{i=1}^n
   \Log \widehat{s}_\indm(Y_i|X_i)\right) + \pen(\indm).
\end{align*}
The main result of this paper is a sufficient condition on the penalty
$\pen(\indm)$ such that for any density function $s_0$ and any sample
size $n$ the adaptive estimate $\widehat{s}_{\widehat{\indm}}$
performs almost as well as the best one in the family $\{
\widehat{s}_{\indm}\}_{\indm \in \indmset}$.

The very frequent use of conditional density estimation in
econometrics, see \textcite{li07:_nonpar_econom} for instance,
could have provided a sufficient motivation for this study. However it turns
out that this work stems from a completely different subject: unsupervised
hyperspectral image segmentation. Using the synchrotron beam of
Soleil, the IPANEMA platform\cite{Bertrand:kv5096}, for which one of
the author works,  is able to acquire high quality
hyperspectral images, high resolution images for which a spectrum is
measured at each pixel location. This provides rapidly a huge amount of data
for which an automatic processing is almost necessary. One of this
processing is the segmentation of these images into homogeneous zones,
so that the spectral analysis can be performed on fewer places and the
geometrical structures can be exhibited. The most classical unsupervised
classification method relies on the density estimation of Gaussian
mixture by a maximum likelihood principle. The component of the
estimated mixtures will correspond to classes. In the spirit of
\textcite{kolaczyk05:_multisimage} and \textcite{antoniadis08}, we have extended this
method by taking into account the localization of the pixel in the
mixture weight, going thus from density estimation to conditional
density estimation. As stressed by
\textcite{maugis09:_gauss}, understanding finely the density
estimator is crucial to be able to select the right
number of classes. This theoretical work has been motivated by a similar
issue for the conditional density estimation case.

\Cref{sec:single-model-maximum-1} is devoted to the analysis of the
maximum likelihood estimation in a single model. It starts by
\cref{sec:sett-maxim-likel} in which the setting and some notations are
given. The risk of the maximum likelihood in the classical case of \emph{misspecified} parametric model is
recalled in
\cref{sec:asympt-analys-param}. \Cref{sec:jens-kullb-leibl} provides
some tools required for the extension of this analysis to more general
setting presented in \cref{sec:single-model-maximum}.
We focus then in \ref{sec:model-select-penal} to the multiple model
case. The penalty used is described in \cref{sec:framework} while the
main theorem is given in \cref{sec:gener-theor-penal}.
\Cref{sec:emphp-const-cond} introduces partition-based
conditional density estimator: we use model in which the conditional
density depends from the covariate only in a piecewise constant
manner. We study in details two instances of such model: one in which,
conditionally to the covariate, the densities are piecewise polynomial for the $Y$ variable and the other,
which corresponds to our hyperspectral image segmentation motivation,
in which, again conditionally to the covariate, the densities are
Gaussian mixtures with the same mixture components but different
mixture weights.


\ifthenelse{\boolean{extended}}{}{
\section{Introduction}

Assume we observe $n$ pairs $\left((X_i,Y_i)\right)_{1 \leq i \leq n}$ of random variables, we are
interested in estimating the law of the second one $Y_i \in \SpaceY$,
called variable,
conditionally to the first one $X_i \in \SpaceX$, called covariate.
In this paper, we assume that
the pairs $(X_i,Y_i)$ are independent while $Y_i$ depends on $X_i$
through its law. More precisely, we assume that the covariates $X_i$'s
are independent but not necessarily identically distributed. The
assumptions on the $Y_i$s are stronger: we assume that, conditionally to
the $X_i$'s, they are
independents and  each variable $Y_i$ follows a law with density
$s_0(\cdot|X_i)$ with respect to a common known measure $\ud\meas$.
Our goal is to estimate this two-variable conditional density
function $s_0(\cdot|\cdot)$ from the observations.

This problem has been introduced by \textcite{rosenblatt69:_condit} in
the late 60's. He considered a stationary framework in which
$s_0(y|x)$ is linked to the supposed existing 
densities $s_{0'}(x)$ and $s_{0''}(x,y)$  of respectively $X_i$
and $(X_i,Y_i)$ by
\begin{align*}
  s_0(y|x)=\frac{s_{0''}(x,y)}{s_{0'}(x)},
\end{align*}
and proposed a plugin estimate based on kernel estimation of both
$s_{0'}(x)$ and $s_{0''}(x,y)$. Few other references on this subject seem to
exist before the mid 90's with a study of a spline tensor based
maximum likelihood estimator
proposed by \textcite{stone94:_use_polyn_splin_their_tensor}
and a bias correction of
\citeauthor{rosenblatt69:_condit}'s estimator due to
\textcite{hyndman96:_estim}. 

Kernel based method have been much studied since. For instance,
\textcite{fan96:_estim} and \textcite{gooijer03} 
consider local polynomial estimator, 
\textcite{hall99:_method} study a locally logistic estimator
that is later extended by \textcite{hyndman02:_nonpar}. In this
setting, pointwise convergence properties are considered, and extensions to
dependent data are often obtained.
The results depend however
on a critical bandwidth that should be chosen according to the
regularity of the unknown conditional density. Its
practical choice is rarely discussed with the notable exceptions of
\textcite{bashtannyk01:_bandw}, \textcite{fan04:_cross_method_estim_condit_densit} and \textcite{hall04:_cross_valid_estim_condit_probab_densit}. Extensions to censored cases have also
been discussed for instance by \textcite{keilegom02:_densit}. See for
instance \textcite{li07:_nonpar_econom} for a comprehensive review of
this topic.

In the approach of \textcite{stone94:_use_polyn_splin_their_tensor},
the conditional density is estimated through a parametrized
modelization. This idea has been reused since by \textcite{gyorfi07:_nonpar}
with a histogram based approach, by
\textcite{efromovich07:_condit,efromovich10:_oracl}  with a Fourier
basis, and by
\textcite{brunel07:_adapt_estim_condit_densit_presen_censor} and \textcite{akakpo11:_inhom}
with piecewise polynomial representation. Those
authors are able to control an integrated estimation error: with 
an integrated total variation loss for the first one and a quadratic distance loss for the others.
Furthermore, in the quadratic framework, they manage to construct
adaptive estimators, estimators that do not require the knowledge of
the regularity to be minimax optimal (up to a logarithmic factor),
using respectively a blockwise attenuation principle and a model
selection by penalization approach. Note that
\textcite{brunel07:_adapt_estim_condit_densit_presen_censor} extend their
result to censored cases while \textcite{akakpo11:_inhom} are able to
consider weakly dependent data. 

In this paper, we consider a direct estimation of the conditional
density function through a maximum likelihood approach. 
Although natural, this approach has been considered so far only by \textcite{stone94:_use_polyn_splin_their_tensor}
as mentioned before and by \textcite{blanchard97:_optim} in a classification
setting with histogram type estimators. Assume we have a set $\model_{\indm}$ of
candidate conditional densities, our estimate $\widehat{s}_{\indm}$ is simply the
maximum likelihood estimate
 \begin{align*}
   \widehat{s}_{\indm} = \argmin_{s_{\indm}\in\model_{\indm}} \left( - \sum_{i=1}^n
   \Log s_{\indm}(Y_i|X_i) \right).
 \end{align*}
Although this estimator may look like a maximum likelihood estimator of the
joint density of $(X_i,Y_i)$, it does not generally coincide, even when the $X_i$'s are assumed to be
i.i.d., with such
an estimator as every function of $\model_{\indm}$
is assumed to be a conditional density and not a density.
 The only exceptions are 
when the $X_i$'s are assumed to be i.i.d. uniform or non
random and equal.
Our aim is then to analyze the finite sample performance of such an
estimator in term of Kullback-Leibler type loss. As often, a trade-off
between a bias term measuring the closeness of $s_0$ to the set
$\model_{\indm}$ and a variance term depending on the complexity of the set
$\model_{\indm}$ and on the sample size appears. A good set $\model_{\indm}$ is
thus one for which this trade-off leads to a small risk bound. Using a
penalized model
selection approach, we propose a way to select the best model
$\model_{\widehat{\indm}}$ among a collection
$\Models=(\model_\indm)_{\indm\in\indmset}$. For
a given family of penalties
$\pen(\indm)$, we define
the \emph{best} model $\model_{\widehat{\indm}}$ as the one
that minimizes
\begin{align*}
  \widehat{\indm} = \argmin_{\indm\in\indmset} \left(- \sum_{i=1}^n
   \Log \widehat{s}_\indm(Y_i|X_i)\right) + \pen(\indm).
\end{align*}
The main result of this paper is a sufficient condition on the penalty
$\pen(\indm)$ such that for any density function $s_0$ and any sample
size $n$ the adaptive estimate $\widehat{s}_{\widehat{\indm}}$
performs almost as well as the best one in the family $\{
\widehat{s}_{\indm}\}_{\indm \in \indmset}$.

The very frequent use of conditional density estimation in
econometrics, see \textcite{li07:_nonpar_econom} for instance,
could have provided a sufficient motivation for this study. However it turns
out that this work stems from a completely different subject: unsupervised
hyperspectral image segmentation. Using the synchrotron beam of
Soleil, the IPANEMA platform~\cite{Bertrand:kv5096}, in which one of
the author works,  is able to acquire high quality
hyperspectral images, high resolution images for which a spectrum is
measured at each pixel location. This provides rapidly a huge amount of data
for which an automatic processing is almost necessary. One of these
processings is the segmentation of these images into homogeneous zones,
so that the spectral analysis can be performed on fewer places and the
geometrical structures can be exhibited. The most classical unsupervised
classification method relies on the density estimation of Gaussian
mixture by a maximum likelihood principle. Components of the
estimated mixtures correspond to classes. In the spirit of
\textcite{kolaczyk05:_multisimage} and \textcite{antoniadis08}, we
have extended this
method by taking into account the localization of the pixel in the
mixing proportions, going thus from density estimation to conditional
density estimation. As stressed by
\textcite{maugis09:_gauss,maugis11:_data}, understanding finely the density
estimator is crucial to be able to select the right
number of classes. This theoretical work has been motivated by a similar
issue for the conditional density estimation case.

\Cref{sec:single-model-maximum-1} is devoted to the analysis of the
maximum likelihood estimation in a single model. It starts by
\cref{sec:sett-maxim-likel} in which the setting and some notations are
given. The risk of the maximum likelihood in the classical case of \emph{misspecified} parametric model is
recalled in
\cref{sec:asympt-analys-param}. \Cref{sec:jens-kullb-leibl} provides
some tools required for the extension of this analysis to more general
setting presented in \cref{sec:single-model-maximum}.
We then focus in \ref{sec:model-select-penal} to the multiple model
case. The penalty used is described in \cref{sec:framework} while the
main theorem is given in \cref{sec:gener-theor-penal}.
We conclude by  \cref{sec:an-appl-spat} in which we
describe the application to unsupervised
segmentation that motivated this work and to what extent our
theoretical approach applies. 
The two main theorems of the
first sections are proved
in Appendix while, for sake of space, the proofs of technical lemmas
 are relegated to an extended
technical report~\cite{cohen11:_condit_densit_estim_penal_app}. For
the same reason, a companion
paper\cite{cohen11:_partit_based_condit_densit_estim} has been
dedicated to the study of the spatialized Gaussian mixture example, as
well as another example of \emph{partion}-based strategies.
}

\section{Single model maximum likelihood estimate}
\label{sec:single-model-maximum-1}

\subsection{Framework and notation}
\label{sec:sett-maxim-likel}

Our statistical framework is the following: we
observe $n$ independent pairs
$\left( (X_i,Y_i)\right) _{1 \leq i \leq n} \in \left(\SpaceX,\SpaceY\right)^n$ where the $X_i$'s are
independent, but not necessarily of the same law, and, conditionally to
$X_i$,
 each $Y_i$ is a random variable of unknown conditional
density $s_0(\cdot|X_i)$ with respect to a known reference measure $\ud \meas$.
For any model $\model_{\indm}$, a set comprising some candidate conditional densities,  
we estimate $s_0$
by the conditional
density $\widehat{s}_\indm$ that maximizes the likelihood
(conditionally to $\left( X_i\right) _{1 \leq i \leq n}$) or
equivalently that minimizes the opposite of the log-likelihood,
denoted -log-likelihood from now on:
\begin{align*}
  \widehat{s}_\indm = \argmin_{s_{\indm} \in \model_{\indm}} \left(  \sum_{i=1}^n - \Log(s_\indm(Y_i|X_i) ) \right). 
\end{align*}
To avoid existence issue, we should work with almost minimizer of this
quantity and define a $\rhoapp$ -log-likelihood minimizer as any
$\widehat{s}_\indm$ that satisfies
\begin{align*}
\sum_{i=1}^n - \Log(\widehat{s}_\indm(Y_i|X_i) ) \leq
\inf_{s_\indm \in
  \model_\indm} \left(  \sum_{i=1}^n - \Log(s_\indm(Y_i|X_i) ) \right)
+ \rhoapp.
\end{align*}

We should now specify our \emph{goodness} criterion. We are working
with a maximum likelihood approach, the most natural quality measure is thus
the Kullback-Leibler divergence $\KL$. As we consider law with densities with
respect to the known measure $\ud\meas$, we use the following notation
\begin{align*}
  \KL_{\meas}(s,t) = \KL(s \ud\meas,t \ud\meas) & = \begin{cases} -\int_{\Omega}  \Log \left(\frac{t}{s}\right) s\,\ud\meas &
    \text{if $s \ud\meas \ll t \ud\meas$}
\\
+ \infty & \text{otherwise}
\end{cases}
\end{align*}
where $s \ud\meas \ll t \ud\meas$ means $\Leftrightarrow\forall\Omega'
      \subset \Omega, \int_{\Omega'} t \ud\meas=0 \implies
      \int_{\Omega'} s \ud\meas =0$.
Remark that, contrary to the quadratic loss, this divergence is an
intrinsic quality measure between probability laws: it does not
depend on the reference measure $\ud\meas$. However, The densities
depend on this reference measure, this is stressed by the
index $\meas$ when we work with the non intrinsic densities instead of
the probability measures.
As we deal with conditional densities and not  classical
densities, the previous divergence should be adapted.
To take into account the structure of conditional densities and the
design of $(X_i)_{1\leq i \leq n}$, we 
use the following \emph{tensorized} divergence:
\begin{align*}
\KLtens_{\meas}(s,t) & 
= \E
\left[ \frac{1}{n}\sum_{i=1}^n
  \KL_{\meas}(s(\cdot|X_i),t(\cdot|X_i))\right].
\end{align*}
This divergence appears as the natural one in this setting and reduces
to classical ones in specific settings:
\begin{itemize}
\item If the law of $Y_i$ is independent of $X_i$, that is $s(\cdot|X_i)=s(\cdot)$ and $t(\cdot|X_i)=t(\cdot)$ do not depend
  on $X_i$, these divergences reduce to the classical
  $\KL_{\meas}(s,t)$.
\item If the $X_i$'s are not random but fixed, that is
  we consider a fixed design case, this divergence is the
  classical fixed design type divergence in which there is no expectation. 
\item If the $X_i$'s are i.i.d., this divergence is nothing but
\begin{alignI}
 \KLtens_{\meas}(s,t)  =
\E
\left[  \KL_{\meas}(s(\cdot|X_1),t(\cdot|X_1))\right].
\end{alignI}
\end{itemize}
Note that this divergence is an \emph{integrated} divergence as it is the
average over the index $i$ of the mean with respect to the law of $X_i$
of the divergence between the conditional densities for a given
covariate value. Remark in particular that more weight is given to
regions of high density of the covariates than to regions of low density
and, in particular, the values of the
divergence outside the supports of the $X_i$'s are not used.
In particular, if we assume that
each $X_i$ has a law with density with respect to a common finite positive
measure $\mu$ and that all those densities are lower and upper bounded
then all our results hold, up to modification in constants, by replacing the definition of $
\KLtens_{\meas}(s,t)$ (and their likes) by the more classical
\begin{align*}
   \KLtens_{\meas}(s,t) = \int_{\SpaceX} \KL(s(\cdot|x),t(\cdot|x)) \ud\mu.
\end{align*}
We stress that these types of loss is similar to the one used in the
machine-learning community (see for instance
\textcite{catoni07:_pac_bayes_super_class} that has inspired our
notations). Such kind of losses appears also, but less often, in regression with random
design (see for instance \textcite{birge04:_model_gauss}) or in
other conditional density estimation studies (see for instance
\textcite{brunel07:_adapt_estim_condit_densit_presen_censor} and \textcite{akakpo11:_inhom}).
When $\hat{s}$ is an
estimator, or any function that depends on the observation,
$\KLtens_{\meas}(s,\hat{s})$ measures this (random) integrated divergence between $s$ and
$\hat{s}$ conditionally to the observation while $\E\left[\KLtens_{\meas}(s,\hat{s})\right]$ is the
average of this random quantity with respect to the observations.

\subsection{Asymptotic analysis of a parametric model}
\label{sec:asympt-analys-param}

Assume that $\model_\indm$ is a parametric model of conditional densities, 
\begin{align*}
\model_{\indm} = \left\{
s_{\theta_\indm}(y|x) \middle| 
\theta_\indm \in \Theta_\indm \subset
\mathbb{R}^{\DimH_{\indm}} \right\},
\end{align*}
to which the true conditional density $s_0$ does not necessarily
belongs. 
In this case, if we let
\begin{align*}
  \widehat{\theta}_{\indm} = \argmin_{\theta_{\indm}\in\Theta_{\indm}} \left(  \sum_{i=1}^n - \Log(s_{\theta_{\indm}}(Y_i|X_i) ) \right)
\end{align*}
then $\widehat{s}_{\indm}=s_{\widehat{\theta}_{\indm }}$.
\textcite{white92:_maxim_likel_estim_missp_model} has studied
this \emph{misspecified model} setting for density estimation but its results can
easily been extended to the conditional density case.

If the model is identifiable and under some (strong) regularity assumptions on $\theta_\indm \mapsto
s_{\theta_\indm}$, provided the $\DimH_{\indm}\times\DimH_{\indm}$ matrices $A(\theta_m)$ and $B(\theta_m)$ defined by
\begin{align*}
  A(\theta_m)_{k,l} &= \E
\left[ \frac{1}{n}\sum_{i=1}^n
\int \frac{-\partial^2 \log
s_{\theta_{\indm}}}{\partial \theta_{\indm,k}\partial \theta_{\indm,l}} (y|X_i)\, s_0(y|X_i) \ud\meas
\right]\\
  B(\theta_m)_{k,l} &= \E
\left[ \frac{1}{n}\sum_{i=1}^n
\int \frac{\partial \log
s_{\theta_{\indm}}}{\partial \theta_{\indm,k}}(y|X_i)\,\frac{\partial \log
s_{\theta_{\indm}}}{\partial \theta_{\indm,l}}(y|X_i)\,  s_0(y|X_i) \ud\meas
\right]
\end{align*}
exists, the analysis of
\textcite{white92:_maxim_likel_estim_missp_model} implies that,
if we let
\begin{align*}
  \theta_{\indm}^{\star} = \argmin_{\theta_{\indm} \in \Theta_{\indm}}  \KLtens_{\lambda}(s_0,s_{\theta_{\indm}}),
\end{align*}
$\E\left[
\KLtens_{\lambda}(s_0,\widehat{s}_{\indm})
\right]
$ is asymptotically equivalent to
\begin{align*}
\KLtens_{\lambda}(s_0,s_{\theta_{\indm}^{\star}})
+ \frac{1}{2n}\mathop{\textrm{Tr}}(B(\theta_{\indm}^{\star})A(\theta_{\indm}^{\star})^{-1}).
\end{align*}
When $s_0$ belongs to the model, i.e. $s_0=s_{\theta_{\indm}^{\star}}$,
$B(\theta_{\indm}^{\star})=A(\theta_{\indm}^{\star})$ and 
thus the previous asymptotic equivalent of $\E\left[
\KLtens_{\lambda}(s_0,\widehat{s}_{\indm})
\right]
$ is 
 the classical parametric one
\begin{align*}
\min_{\theta_{\indm}} \KLtens_{\lambda}(s_0,s_{\theta_{\indm}})
+ \frac{1}{2n} \DimH_{\indm}.
\end{align*}
This simple expression does not hold when $s_0$ does not belong to
the parametric model as
$\mathop{\textrm{Tr}}(B(\theta_{\indm}^{\star})A(\theta_{\indm}^{\star})^{-1})$
cannot generally be simplified.  

A short glimpse on
the proof of the previous result shows that it depends heavily on the
asymptotic normality of
$\sqrt{n}(\widehat{\theta}_{\indm}-\theta_{\indm}^{\star})$. One may
wonder if extension of this result, often called the Wilk's
phenomenon~\cite{wilks38}, exists when this normality does not hold,
for instance in non parametric case or when the model is not identifiable. 
Along
these lines, \textcite{fan01:_gener_wilks} propose a
generalization of the corresponding Chi-Square goodness-of-fit test in
several settings and \textcite{boucheron10:_wilks} study the finite
sample deviation
of the corresponding empirical quantity in a bounded loss setting.

Our aim is to derive a non asymptotic
upper bound of type
\begin{align*}
  \E\left[
\KLtens_{\lambda}(s_0,\widehat{s}_{\indm})
\right] \leq  \left( \min_{s_{\indm} \in \model_{\indm}}
  \KLtens_{\lambda}(s_0,s_{\indm}) + \frac{1}{2n}
  \DimH_{\indm} \right) + \constoratwo \frac{1}{n}
\end{align*}
with as few assumptions on the conditional density set $\model_{\indm}$ as possible.
Note that we only aim at having an upper bound and do not focus on
the (important) question of the existence of a corresponding lower bound.

Our answer is far from definitive, the upper bound we obtained is the
following weaker one
\begin{align*}
\E\left[\JKLtens_{\propJKL,\meas}(s_0,\widehat{s}_{\indm})\right]
\leq (1+\epsilon) \left( \inf_{s_\indm
    \in \model_\indm} \KLtens_{\meas}(s_0,s_\indm) + \frac{\constpenmin}{n}
  \DIMH_{\indm} \right) + \constoratwo \frac{1}{n}
\end{align*}
in which the left-hand $\KLtens_{\meas}(s_0,\widehat{s}_{\indm})$
has been replaced by a smaller divergence
$\JKLtens_{\propJKL,\lambda}(s_0,\widehat{s}_{\indm})$ described
below, $\epsilon$ can be chosen arbitrary small, $\DIMH_{\indm}$ is a
model complexity term playing the role of the dimension $\DimH_{\indm}$ and
$\constpenmin$ is a constant that depends on $\epsilon$. This result
has nevertheless the right bias/variance trade-off flavor and can be
used to recover usual minimax properties of specific estimators.

\subsection{Jensen-Kullback-Leibler divergence and bracketing entropy}
\label{sec:jens-kullb-leibl}

The main visible loss is the use of a divergence smaller than the
Kullback-Leibler one (but larger than the squared Hellinger
distance and the squared $L_1$ loss whose definitions are recalled later).
Namely, we use the Jensen-Kullback-Leibler divergence
$\JKL_{\propJKL}$ with $\propJKL\in(0,1)$ defined by
\begin{align*}
  \JKL_{\propJKL}(s\ud\meas,t\ud\meas) = \JKL_{\propJKL,\meas}(s,t)= \frac{1}{\propJKL}
  \KL_{\meas}\left( s, (1-\propJKL) s+ \propJKL t \right).
\end{align*}
Note that this divergence appears explicitly with
$\propJKL=\frac{1}{2}$ in~\textcite{massart07:_concen}, but can also be found
implicitly in  \textcite{birge98:_minim} and \textcite{geer95}. We use
the name Jensen-Kullback-Leibler divergence in the same way
\textcite{lin91:_diver_shann} uses the name Jensen-Shannon divergence
for a sibling in his information theory work. The main tools in the
proof of the previous inequality are deviation inequalities for sums of
random variables and their suprema. Those tools require a boundness
assumption on the controlled functions that is not satisfied by the
-log-likelihood differences $-\Log \frac{s_{\indm}}{s_0}$. When
considering the Jensen-Kullback-Leibler divergence, those ratios are
implicitly replaced by ratios $-\frac{1}{\propJKL}\Log \frac{(1-\propJKL) s_0 + \propJKL
  s_{\indm}}{s_0}$ that are close to the -log-likelihood
differences when the $s_{\indm}$ are close to $s_0$ and always upper
bounded by
$-\frac{\Log(1-\propJKL)}{\propJKL}$.
This divergence is
smaller than the Kullback-Leibler one but larger, up to a constant
factor,
 than the squared Hellinger one,  $\d^{2}_{\meas}(s,t)=\int_{\Omega}
|\sqrt{s}-\sqrt{t}|^2 \ud\meas$, and the squared $L_1$ distance,
$\|s-t\|_{\meas,1}^2= \left(\int_{\Omega}
|s-t| \ud\meas\right)^2$, as proved
\ifthenelse{\boolean{extended}}{in Appendix}{
in our
technical report~\cite{cohen11:_condit_densit_estim_penal_app}}
\begin{proposition}
\label{prop:hellingerkull2} 
  For any probability measures $s \ud\meas$ and $t \ud\meas$ and any $\propJKL\in(0,1)$
  \begin{align*}
\ConstHellKL_{\propJKL}\,
  \d_{\meas}^2(s,t)  \leq \JKL_{\propJKL,\meas}(s,t) \leq \KL_{\meas}(s,t).
  \end{align*}
with 
\begin{alignI}
\ConstHellKL_{\propJKL}
=  \frac{1}{\propJKL} \min\left(\frac{1-\propJKL}{\propJKL},1\right)
\left( \Log \left(1 + \frac{\propJKL}{1-\propJKL}\right) -
  \propJKL\right)
\end{alignI}
while
  \begin{align*}
  \max(\ConstHellKL_{\propJKL}/4,\propJKL/2) \|s-t\|_{\meas,1}^2 \leq  \JKL_{\propJKL,\meas}(s,t) \leq \KL_{\meas}(s,t).
  \end{align*}

Furthermore, if $s \ud\meas \ll t \ud\meas$ then
\begin{align*}
  \d_{\meas}^2(s,t) &\leq \KL_{\meas}(s,t) \leq \left( 2 + \Log \left\| \frac{s}{t}
  \right\|_{\infty} \right) \d_{\meas}^2(s,t)
\end{align*}
while
\begin{align*}
  \frac{1}{2}\|s-t\|_{\meas,1}^2 &\leq \KL_{\meas}(s,t) \leq \left\|\frac{1}{t}\right\|_{\infty} \|s-t\|_{\meas,2}^2.
\end{align*}

\end{proposition}
More precisely, as we are in a conditional density setting, we use their \emph{tensorized} versions
\begin{align*}
\dtwotens_{\meas}(s,t) & =
\E \left[ \frac{1}{n}\sum_{i=1}^n
  \d^2_{\meas}(s(\cdot|X_i),t(\cdot|X_i))\right]
&
\text{and}&&
\JKLtens_{\propJKL,\meas}(s,t) & 
= \E \left[ \frac{1}{n}\sum_{i=1}^n
  \JKL_{\propJKL,\meas}(s(\cdot|X_i),t(\cdot|X_i))\right].
\end{align*}

We focus now on the definition of the model complexity $\DIMH_{\indm}$. 
It involves a  bracketing entropy condition on
the model $\model_\indm$ with respect to the Hellinger type  divergence
$\dtens_{\meas}(s,t)=\sqrt{\dtwotens_{\meas}(s,t)}$. A bracket $[t^-,t^+]$ 
is a pair of functions such that $\forall (x,y) \in
\SpaceX\times\SpaceY, t^-(y|x) \leq t^+(y|x)$. A conditional density function
$s$ is said to belong to the bracket $[t^-,t^+]$ if $\forall (x,y) \in
\SpaceX\times\SpaceY, t^-(y|x)\leq
s(y|x) \leq t^+(y|x)$. The bracketing entropy
$H_{[\cdot],\dtens_{\meas}}(\delta,\setmodel)$ of a set $\setmodel$ is defined as the logarithm of the minimum number
of brackets $[t^-,
t^+]$ of width $\dtens_{\meas}(t^-,t^+)$ smaller than $\delta$ 
such that every function of $\setmodel$ belongs to one of these brackets.
 $\DIMH_{\indm}$ depends on the bracketing entropies not of the global
 models $\model_{\indm}$ but of
the ones of smaller localized sets $\model_\indm(\widetilde{s},\sigma)=\left\{ s_\indm \in
\model_\indm \middle| \dtens_{\meas}(\widetilde{s},s_\indm) \leq \sigma
\right\}$. Indeed, we impose a structural assumption:
\begin{assumptionHm}
 There is
 a non-decreasing function
 $\phi_\indm(\delta)$ such
that $\delta\mapsto \frac{1}{\delta} \phi_\indm(\delta)$ is non-increasing on $(0,+\infty)$ and for
every $\sigma\in\R^+$ and every $s_\indm \in \model_\indm$
\begin{align*}
  \int_0^{\sigma} \sqrt{%
    H_{[\cdot],\dtens_{\meas}}\left(\delta,\model_\indm(s_\indm,\sigma)\right)}
  \, \ud\delta \leq \phi_\indm(\sigma).
\end{align*}
\end{assumptionHm}
Note that the function $\sigma\mapsto \int_0^{\sigma} \sqrt{%
    H_{[\cdot],\dtens_{\meas}}\left(\delta,\model_\indm\right)}
  \, \ud\delta$ does always satisfy this assumption.
$\DIMH_{\indm}$ is then defined as $n\sigma^2_{\indm}$ with $\sigma^2_{\indm}$ the unique root
of \begin{alignI}
\frac{1}{\sigma}\phi_\indm(\sigma)=\sqrt{n}\sigma.
\end{alignI}
A good choice of $\phi_{\indm}$ is one which leads to a small upper
bound of $\DIMH_{\indm}$.
This bracketing entropy integral, often call Dudley integral, plays an
important role in empirical processes theory, as stressed for instance
in \textcite{vaart96:_weak_conver} and in \textcite{kosorok08:_introd_empir_proces_semip_infer}. The equation defining
$\sigma_{\indm}$ corresponds to a crude optimization of a
supremum bound as shown explicitly in the proof.
This definition is obviously far from being very explicit but it turns
out that it can be related to an entropic dimension of the
model. Recall that the classical entropy dimension of a compact set $S$
with respect to a metric $d$
can be defined as the smallest non negative real $\DimH$ such that
there is a non negative $\ConstH$ such that
\begin{align*}
\forall \delta > 0, H_{d}(\delta,S) \leq  \ConstH+ \DimH \log \left(\frac{1}{\delta}\right)
\end{align*}
where $H_{d}$ is the classical entropy with respect to metric
$d$. The parameter $\ConstH$ can be interpreted as the logarithm of the volume of the set.
Replacing the classical entropy by a bracketing one, we define
the bracketing dimension $\DimH_{\indm}$ of a compact set as the
smallest real $\DimH$ such that there is a $\ConstH$ such 
\begin{align*}
\forall \delta > 0, H_{[\cdot],d}(\delta,S) \leq  \ConstH+ \DimH \log \left(\frac{1}{\delta}\right).
\end{align*}
As hinted by the notation, for parametric model, under mild assumption on the parametrization,
this bracketing dimension coincides with the usual one. Under such
assumption, one can prove that
$\DIMH_{\indm}$ is proportional to $\DimH_{\indm}$. More precisely, 
working
  with the localized set $\model_{\indm}(s,\sigma)$ instead of
  $\model_{\indm}$,
we obtain 
\ifthenelse{\boolean{extended}}{in Appendix,}{
in our technical report~\cite{cohen11:_condit_densit_estim_penal_app},}
 we obtain
\begin{proposition}
\label{prop:proporsimple}
\begin{itemize}
\item
if
\begin{alignI}
\exists \DimH_{\indm} \geq 0, \exists \ConstMultH_{\indm} \geq 0,
\forall \delta \in (0,\sqrt{2}],  H_{[\cdot],\dtens_{\meas}}(\delta,\model_{\indm}) \leq
\ConstH_{\indm} + \DimH_{\indm}  \Log \frac{1}{\delta}
\end{alignI}
then  
\begin{itemize}
\item if $\DimH_{\indm} > 0$,
\Hindm\ holds with
\begin{alignI}
  \DIMH_{\indm}  \leq \left( 2 \constspatgausslogindm
+ 1 + \left(\Log
    \frac{n}{e\constspatgausslogindm \DimH_{\indm}}\right)_+ \right)
\DimH_{\indm}
\end{alignI}
with $\constspatgausslogindm=\left(\sqrt{\frac{\ConstH_{\indm}}{\DimH_{\indm}}} +
    \sqrt{\pi}\right)^2$,
\item if $\DimH_{\indm}=0$, \Hindm\ holds with $\phi_\indm(\sigma)=\sigma\sqrt{\ConstH_{\indm}}$ such that \begin{alignI}
  \DIMH_{\indm}  = \ConstH_{\indm},
\end{alignI}
\end{itemize}
\item
if
\begin{alignI}
\exists \DimH_{\indm} \geq  0, \exists \ConstH_{\indm} \geq 0, \forall
\sigma \in (0,\sqrt{2}], \forall \delta \in (0,\sigma],   H_{[\cdot],\dtens_{\meas}}(\delta,\model_{\indm}(s_{\indm},\sigma)) \leq
 \ConstH_{\indm} + \DimH_{\indm} \Log
   \frac{\sigma}{\delta}
\end{alignI}
then 
\begin{itemize}
\item
if $\DimH_{\indm}>0$, \Hindm\  holds with $\phi_\indm$ such that \begin{alignI}
  \DIMH_{\indm}  = \constspatgausslogindm
    \DimH_{\indm}
\end{alignI}
with $\constspatgausslogindm=\left(\sqrt{\frac{\ConstH_{\indm}}{\DimH_{\indm}}} +
    \sqrt{\pi}\right)^2$,
\item if $\DimH_{\indm}=0$, \Hindm\ holds with
$\phi_\indm(\sigma)=\sigma\sqrt{\ConstH_{\indm}}$ such that \begin{alignI}
  \DIMH_{\indm}  = \ConstH_{\indm}.
\end{alignI}
\end{itemize}
\end{itemize}
\end{proposition}
Note that we assume bounds on the entropy only for $\delta$ and $\sigma$
smaller than $\sqrt{2}$, but, as for any conditional densities pair $(s,t)$
$\dtens_{\meas}(s,t)\leq \sqrt{2}$,
\begin{align*}
  H_{[\cdot],\dtens_{\meas}}(\delta,\model_{\indm}(s_{\indm},\sigma))
  = H_{[\cdot],\dtens_{\meas}}(\delta \wedge \sqrt{2}
  ,\model_{\indm}(s_{\indm},\sigma \wedge \sqrt{2}))
\end{align*}
which implies that those bounds are still useful when $\delta$ and
$\sigma$ are large.
Assume now that all models are such that
$\frac{\ConstH_{\indm}}{\DimH_{\indm}}\leq \ConstMultH$, i.e. their
log-volumes $\ConstH_{\indm}$ grow at most linearly with the dimension (as it is the
case for instance for hypercubes with the same width). One deduces
that Assumptions \Hindm\  hold simultaneously
for every model with a common constant 
 $\constspatgausslog=\left(\sqrt{\ConstMultH} +
    \sqrt{\pi}\right)^2$. The model complexity $\DIMH_{\indm}$ can
  thus be chosen roughly proportional to the dimension in this case, this justifies the
  notation as well as our claim at the end of the previous section. 

\ifthenelse{\boolean{extended}}{}
{\Cref{sec:an-appl-spat} provides an example for which the first
assumption holds while the technical
report\cite{cohen11:_condit_densit_estim_penal_app} and the companion
paper\cite{cohen11:_partit_based_condit_densit_estim} contain more examples.
An important aspect of the model complexity $\DIMH_{\indm}$ is that it is an
intrinsic measure of the complexity of $\model_{\indm}$: it is
obtained by looking only at the model $\model_{\indm}$ and does not
depend on external quantities (besides $n$).}

\subsection{Single model maximum likelihood estimation}
\label{sec:single-model-maximum}

For technical reason, we also need a separability assumption on our model:
\begin{assumptionSm}
There exist a countable subset
$\model'_\indm$ of $\model_\indm$ and a set $\SpaceY_\indm'$ with
$\meas(\SpaceY\setminus\SpaceY_\indm')=0$ such that for every $t\in\model_\indm$,
there exists a sequence $(t_k)_{k\geq 1}$ of elements of
$\model'_\indm$ such that for every $x$ and for every $y\in\SpaceY_\indm'$, $\Log\left(t_k(y|x)\right)$  goes
to $\Log\left(t(y|x)\right)$ as $k$ goes to infinity.
\end{assumptionSm}

We are now ready to state our risk bound theorem:
\begin{theorem}
\label{theo:single}
Assume we observe $(X_i,Y_i)$ with unknown conditional density $s_0$.
Assume $\model_\indm$ is a set of conditional densities for which
Assumptions \Hindm\  and \Sepindm\  hold and
let $\widehat{s}_\indm$ be a $\rhoapp$ -log-likelihood minimizer in
  $\model_\indm$
\[
\sum_{i=1}^n - \Log(\widehat{s}_\indm(Y_i|X_i) ) \leq
\inf_{s_\indm \in
  \model_\indm} \left(  \sum_{i=1}^n - \Log(s_\indm(Y_i|X_i) ) \right)
+ \rhoapp
\]

Then for any $\propJKL\in(0,1)$ and any $\constoraone>1$, there are two constants
$\constpenmin$ and $\constoratwo$ depending only on $\propJKL$ and
$\constoraone$ such that,
for $\DIMH_{\indm}=n\sigma^2_{\indm}$ with $\sigma_\indm$ the unique root of
\begin{alignI}
\frac{1}{\sigma}\phi_\indm(\sigma)=\sqrt{n}\sigma
\end{alignI},
the  likelihood estimate $\widehat{s}_{\indm}$
satisfies
\begin{align*}
\E\left[\JKLtens_{\propJKL,\meas}(s_0,\widehat{s}_{\indm})\right]
\leq \constoraone \left( \inf_{s_\indm \in \model_\indm}
  \KLtens_{\meas}(s_0,s_\indm) + \frac{\constpenmin}{n}
  \DIMH_{\indm} \right) + \constoratwo \frac{1}{n} + \frac{\rhoapp}{n}.
\end{align*}
\end{theorem}

This theorem holds without any assumption on the design $X_i$, in
particular we do not assume that the covariates admit upper or
lower bounded densities. The law of the design appears however in the
divergence $\JKLtens_{\meas}$ and $\KLtens_{\meas}$ used to assess the
quality of the estimate
as well as in the definition of the divergence $\dtens_{\meas}$ used
to measure the bracketing entropy. By construction, those quantities however
do not involve the values of the conditional densities outside the
support of the $X_i$s and put more focus on the regions of high
density of covariates than the other. 
Note that Assumption $\text{H}_{\indm}$
 could be further localized: it suffices to impose that the condition
 on the Dudley integral holds for a sequence of minimizer of $\dtwotens_{\meas}(s_0,s_{\indm})$.

We obtain thus a bound on the expected loss similar to the one
obtained in the parametric case that holds for finite sample
and that do not require the strong regularity assumptions of
\textcite{white92:_maxim_likel_estim_missp_model}. In particular, we
do not even require an identifiability condition in the parametric case.
As often in empirical processes theory, the constant $\constpenmin$
appearing in the bound is pessimistic. Even in a very simple
parametric model, the current best estimates are such that
$\constpenmin \DIMH_{\indm}$ is still much larger than the
variance of \cref{sec:asympt-analys-param}. Numerical experiments show 
there is a hope that this is only a technical issue.
 The obtained bound quantifies however  the
expected classical bias-variance trade-off: a good model should be large enough
so that the true conditional density is close from it but, at the same time, it
should also be small so that the $\DIMH_{\indm}$ term does not
dominate.

It should be stressed that a result of similar flavor could have been obtained by the
information theory technique
of~\textcite{barron08:_mdl_princ_penal_likel_statis_risk} and
\textcite{kolaczyk05:_multisimage}. Indeed, if we replace the set
$\model_{\indm}$ by a  \emph{discretized} version $\modeldisc_{\indm}$
so that
\begin{align*}
  \inf_{\sdisc_{\indm} \in \modeldisc_{\indm}} \KLtens_{\meas}(s_0,\sdisc_{\indm}) \leq
  \inf_{s_{\indm} \in \model_{\indm}} \KLtens_{\meas}(s_0,s_{\indm}) + \frac {1}{n},
\end{align*}
then, if we let $\widehat{\sdisc}_{\indm}$ be a -log-likelihood minimizer in
  $\modeldisc_\indm$,
\begin{align*}
   \E \left[ \Battwotens_{\meas}(s_0,\widehat{\sdisc}_{\indm}) \right] \leq
   \inf_{s_{\indm} \in \model_{\indm}}
   \KLtens_{\meas}(s_0,s_{\indm}) + \frac{1}{n} \Log |
   \modeldisc_{\indm}| + \frac{1}{n}
\end{align*}
where $\Battwotens_{\meas}$ is the tensorized Bhattacharyya-Renyi divergence,
another divergence smaller than $\KLtens$, $|
   \modeldisc_{\indm} |$ is the cardinality of
   $\modeldisc_{\indm}$ and expectation is taken conditionally
   to the covariates $(X_i)_{1\leq i \leq n}$. As verified by
   \textcite{barron08:_mdl_princ_penal_likel_statis_risk} and
   \textcite{kolaczyk05:_multisimage}, $\modeldisc_{\indm}$ can be
   chosen of cardinality of order $\Log n\, \DimH_{\indm}$ when the
   model is parametric. We obtain thus also a bound of type
   \begin{align*}
   \E \left[ \Battwotens_{\meas}(s_0,\widehat{\sdisc}_{\indm}) \right] \leq
   \inf_{s_{\indm} \in \model_{\indm}}
   \KLtens_{\meas}(s_0,s_{\indm}) + \frac{\constoraone}{n} \Log n \, \DimH_{\indm} + \frac{1}{n}.
   \end{align*}
with better constants but with a different divergence. The bound holds
 however only conditionally to the design, which can be an issue as 
soon as this design is random, and requires to compute an adapted
 discretization of the models.

\section{Model selection and penalized maximum likelihood}
\label{sec:model-select-penal}

\subsection{Framework}
\label{sec:framework}

A natural question is then the choice of the model.
In the model selection framework, instead of a single model
$\model_\indm$, we assume we have at hand a collection of models
$\Models=\left\{\model_\indm\right\}_{\indm\in\indmset}$. If we assume
that Assumptions \Hindm\  and \Sepindm\ 
hold for all models, then for every model $\model_{\indm}$
\begin{align*}
\E\left[\JKLtens_{\propJKL,\meas}(s_0,\widehat{s}_{\indm})\right]
\leq \constoraone \left( \inf_{s_\indm \in \model_\indm}
  \KLtens_{\meas}(s_0,s_\indm) + \frac{\constpenmin}{n}
  \DIMH_{\indm} \right) + \constoratwo \frac{1}{n} + \frac{\rhoapp}{n}.
\end{align*}
Obviously, one of the models minimizes the right hand side. Unfortunately, there
is no way to know which one without knowing $s_0$, i.e. without
an oracle. Hence, this \emph{oracle} model can not be used to estimate $s_0$.
 We nevertheless
propose a data-driven strategy to select an estimate
among the collection of estimates $\{
\widehat{s}_{\indm} \}_{\indm\in\indmset}$ according to a selection
rule that performs almost as well as if we had known this \emph{oracle}.

As always, using simply the -log-likelihood of the estimate in each model
\begin{align*}
\sum_{i=1}^n - \Log(\widehat{s}_{\indm}(Y_i|X_i) )
\end{align*}
as a criterion is not sufficient. It is an underestimation of the true
risk of the estimate and this leads to choose models that are too complex. By adding an adapted penalty $\pen(\indm)$, one
hopes to compensate for both the \emph{variance} term and the bias between $\frac{1}{n}\sum_{i=1}^n -
\Log\frac{\widehat{s}_{\widehat{\indm}}(Y_i|X_i)}{s_0(Y_i|X_i)} $ and
$\inf_{s_{\indm} \in \model_{\indm}} \KLtens_{\meas}(s_0,s_{\indm})$.
For a given choice of $\pen(m)$,
the \emph{best} model $\model_{\widehat{\indm}}$ is chosen
as the one whose index is an almost
 minimizer of the penalized $\rhoapp$ -log-likelihood :
\begin{align*}
\sum_{i=1}^n - \Log(\widehat{s}_{\widehat{\indm}}(Y_i|X_i) )
+ \pen(\widehat{\indm})
\leq \inf_{\indm\in\indmset} \left( \sum_{i=1}^n - \Log(\widehat{s}_\indm(Y_i|X_i) )
+ \pen(\indm)\right) + \rhomin.
\end{align*}

The analysis of the previous section turns out to be crucial as the intrinsic
complexity $\DIMH_{\indm}$ appears in the assumption on the penalty. It is no surprise
that the complexity of the model collection itself also appears.
We need an information theory type assumption on our collection;
we assume thus the existence of a Kraft type
inequality for the collection:
\begin{assumptionK}
There is a family $(x_\indm)_{\indm\in\indmset}$ of
non-negative number such that
\begin{align*}
\sum_{\indm\in\indmset} e^{-x_\indm} \leq \Sigma < +\infty
\end{align*}
\end{assumptionK}
It can be
interpreted as a coding condition as stressed by
\textcite{barron08:_mdl_princ_penal_likel_statis_risk} where a similar
assumption is used.
Remark that if this assumption holds, it also holds for any
permutation of the coding term $x_{\indm}$.
We should
try to mitigate this arbitrariness by favoring choice of $x_\indm$ for
which the ratio with the intrinsic entropy term $\DIMH_\indm$ is as
small as possible. Indeed, as the condition on the penalty is of
the form
\begin{align*}
  \pen(\indm) \geq \constpen \left( \DIMH_{\indm} + x_\indm \right),
\end{align*}
this ensures that this lower bound is dominated by the intrinsic
quantity $\DIMH_{\indm}$.

\subsection{A general theorem for penalized maximum likelihood
  conditional density estimation}
\label{sec:gener-theor-penal}

Our main theorem is then:
\begin{theorem}
\label{theo:select}
Assume we observe $(X_i,Y_i)$ with unknown conditional density $s_0$.
Let
$\Models=(\model_\indm)_{\indm\in\indmset}$ be at most countable
collection of conditional density sets.
Assume Assumption (K) holds while Assumptions \Hindm\ and \Sepindm\ 
hold for every model $\model_{\indm} \in \Models$.
Let $\widehat{s}_\indm$ be a $\rhoapp$ -log-likelihood minimizer in
  $\model_\indm$
\[
\sum_{i=1}^n - \Log(\widehat{s}_\indm(Y_i|X_i) ) \leq
\inf_{s_\indm \in
  \model_\indm} \left(  \sum_{i=1}^n - \Log(s_\indm(Y_i|X_i) ) \right)
+ \rhoapp
\]

Then for any $\propJKL\in(0,1)$ and any $\constoraone>1$, there are two constants
$\constpenmin$ and $\constoratwo$ depending only on $\propJKL$ and
$\constoraone$ such that,
as soon as  for every index $\indm\in\indmset$
\begin{align*}
\pen(\indm) \geq \constpen \left( \DIMH_{\indm} + x_\indm \right)
\quad\text{with $\constpen>\constpenmin$}
\end{align*}
where $\DIMH_{\indm}=n\sigma_{\indm}^2$ with $\sigma_\indm$ the unique root of
\begin{alignI}
\frac{1}{\sigma}\phi_\indm(\sigma)=\sqrt{n}\sigma,
\end{alignI}
the penalized likelihood estimate $\widehat{s}_{\widehat{\indm}}$
  with $\widehat{\indm}$ such that
  \begin{align*}
\sum_{i=1}^n - \Log(\widehat{s}_{\widehat{\indm}}(Y_i|X_i) )
+ \pen(\widehat{\indm})
\leq \inf_{\indm\in\indmset} \left( \sum_{i=1}^n - \Log(\widehat{s}_\indm(Y_i|X_i) )
+ \pen(\indm) \right) + \rhomin    
  \end{align*}
satisfies
\begin{align*}
\E\left[\JKLtens_{\propJKL,\meas}(s_0,\widehat{s}_{\widehat{\indm}})\right]
\leq \constoraone \inf_{\indm\in\indmset} \left( \inf_{s_\indm \in \model_\indm} \KLtens_{\meas}(s_0,s_\indm) +
  \frac{\pen(\indm)}{n} \right) + \constoratwo \frac{\Sigma}{n} + \frac{\rhoapp+\rhomin}{n}.
\end{align*}
\end{theorem}
Note that, as in \ref{sec:single-model-maximum}, the approach of of~\textcite{barron08:_mdl_princ_penal_likel_statis_risk} and
\textcite{kolaczyk05:_multisimage} could have been used to obtain a
similar result with the help of discretization.

This theorem extends Theorem 7.11
\textcite{massart07:_concen} which handles only density estimation. As in
this theorem, the cost of model selection with respect to the choice
of the best single model is proved to be very mild. Indeed,
let $\pen(m)=\constpen (\DIMH_{\indm} + x_{\indm})$ then one
obtains \pagebreak[1]
\begin{align*}
&\E\left[\JKLtens_{\propJKL,\meas}(s_0,\widehat{s}_{\widehat{\indm}})\right]\\
&\qquad\leq \constoraone \inf_{\indm\in\indmset} \left( \inf_{s_\indm \in \model_\indm} \KLtens_{\meas}(s_0,s_\indm) +
  \frac{\constpen}{n} (\DIMH_{\indm} + x_{\indm}) \right) + \constoratwo
\frac{\Sigma}{n} + \frac{\rhoapp+\rhomin}{n}\\
& \qquad\leq \constoraone \frac{\constpen}{\constpenmin}
\left(\max_{\indm\in\indmset} \frac{\DIMH_{\indm} + x_{\indm}}{\DIMH_{\indm}}\right) \inf_{\indm\in\indmset} \left( \inf_{s_\indm \in \model_\indm} \KLtens_{\meas}(s_0,s_\indm) +
  \frac{\constpenmin}{n} \DIMH_{\indm} \right) + \constoratwo
\frac{\Sigma}{n} + \frac{\rhoapp+\rhomin}{n}.
\end{align*}
As soon as the term $x_{\indm}$ is always small relatively to
$\DIMH_{\indm}$, we obtain thus an oracle inequality that show that
the penalized estimate satisfies, up to a small factor, the bound of
\cref{theo:single} for the estimate in the best model.
The price to pay for the use of a collection of model is thus
small. The gain is on the contrary very important: we do not have to
know the best model within a collection to almost achieve its performance. 

So far we do not have discussed the choice of the model collection,
it is however critical to obtain a \emph{good} estimator. There is
unfortunately no universal choice and it should be
adapted to the specific setting considered. Typically, if we consider
conditional density of \emph{regularity} indexed by a parameter $\alpha$, a
good collection is one such that for every parameter $\alpha$ there is
a model which achieves a quasi optimal bias/variance trade-off. 
\textcite{efromovich07:_condit,efromovich10:_oracl} considers Sobolev
type regularity and use thus models generated by the first elements of
Fourier basis.
\textcite{brunel07:_adapt_estim_condit_densit_presen_censor} and
\textcite{akakpo11:_inhom} considers anisotropic regularity spaces
for which they show that a collection of piecewise polynomial models
is adapted. Although those choices are justified, in these papers, in
a quadratic loss approach, they remain good choices in our maximum
likelihood approach with a Kullback-Leibler type loss.
Estimator associated to those collections
are thus \emph{adaptive} to the regularity: without knowing the
\emph{regularity} of the true conditional density, they select a model
in which the estimate performs almost as well as in the \emph{oracle} model,
the best choice if the regularity was known. In both cases, one could
prove that those estimators achieve the minimax rate for the
considered classes, up to a logarithmic factor.

As in \cref{sec:single-model-maximum}, the known estimate of constant
$\constpenmin$ and even of $\DIMH_{\indm}$ can be pessimistic. This
leads to a theoretical penalty which can be too large in practice. 
A natural question is thus whether the constant appearing in the
penalty can be estimated from the data without loosing a theoretical
guaranty on the performance? No definitive answer exists so far, but 
numerical experiment in specific case shows that the \emph{slope
  heuristic} proposed by \textcite{birge07:_minim_gauss} may yield a
solution.

\ifthenelse{\boolean{extended}}{
The assumptions of the previous theorem are as general as possible. It
is thus natural to question the existence of interesting model
collections that satisfy its assumptions. We have mention so far the 
Fourier based collection proposed by
\textcite{efromovich10:_oracl,efromovich07:_condit} and the piecewise
polynomial collection of \textcite{brunel07:_adapt_estim_condit_densit_presen_censor} and
\textcite{akakpo11:_inhom} considers anisotropic regularity. We focus
on a variation of this last strategy. Motivated by an application to
unsupervised image segmentation, we consider model collection in
which, in each model, the conditional densities depend on the
covariate only in a piecewise constant manner. After a general
introduction to these partition-based strategies, we study
two cases: a classical one in which the conditional density depends
in a piecewise polynomial manner of the variables and a newer one,
which correspond to the unsupervised segmentation application, in
which the conditional densities are Gaussian mixture with common
Gaussian components but mixing
proportions depending on the covariate.
}{

In the next section, we focus on the
  example that motivated our analysis: an estimator based on
  collection of Gaussian mixture with mixing proportions that depend
  on the covariate. We explain how this model is related to an
  application to unsupervised hyperspectral segmentation. More details
  can be found in our technical report\cite{cohen11:_condit_densit_estim_penal_app} and our companion paper\cite{cohen11:_partit_based_condit_densit_estim}.

\section{An application to unsupervised segmentation using spatialized
  Gaussian mixture models}
\label{sec:an-appl-spat}

At IPANEMA, we can measure hyperspectral images, images for which at each pixel
$X_i$ located on a regular lattice of $\SpaceX=[0,1]^2$ a spectrum
$Y$ belonging to $\SpaceY=\R^\dimSp$ is measured. To
derive automatically a segmentation of the image into
\emph{homogeneous} region, we use an extension of a
classical technique of unsupervised classification that takes into
account pixel location. We assume that each spectrum belongs to one of
$K$ classes, each modeled by a Gaussian distribution, whose mixing proportions
depend on the position. Once the parameters of this model have been
estimated, a class can be assigned to each spectrum by a simple
maximum likelihood principle, yielding a segmentation of the image.

We assume thus that the spectrum measured at a pixel $x\in\SpaceX$ follows a
law of density $s_0(\cdot|x)$ with respect to the Lebesgue measure and
model the
conditional density $s(\cdot|x)$ by Gaussian mixtures with varying
mixing proportions
\begin{align*}
s(\cdot|x) = \sum_{k=1}^K \pi_k(x)  \Gauss_{\theta_k}\left(\cdot \right)
\end{align*}
where 
with $K$ the number of mixture components, $\mu_k$ the mean of the
$k$th component, $\Sigma_k$ its covariance matrix,
$\theta_k=(\mu_k,\Sigma_k)$,  $\pi_k(x)$ its
proportion for the value $x$ of the covariate and  $
\Gauss_{\theta_k}(y)$ the density of a Gaussian of mean $\mu_k$ and
covariance $\Sigma_k$.  Once these parameters
have been estimated by respectively $\widehat{K}$,
$\widehat{\theta_k}$ and $\widehat{\pi_k}$, the segmentation of the
image is obtained by a maximum likelihood principle on the different
classes:
\begin{align*}
  \widehat{k}(y|x) = \argmax \widehat{\pi}_k(x) \Gauss_{\widehat{\theta_k}}(y). 
\end{align*}
We propose to use our model selection setting to estimate all these parameters.
Each model $\model_\indm$ will be specified by its number of class
$K$, the covariance structure of the Gaussian $K$-tuples appearing in the mixture
as described by \textcite{biernacki06:_model_mixmod} and a function set 
to which belong the mixing proportions. Following
\textcite{kolaczyk05:_multisimage} and \textcite{antoniadis08}, we
consider mixing proportions that are piecewise constant on a hierarchical
partition $\PartX$ induced by a tree structure, one of the
Recursive Dyadic Partition of \textcite{donoho97:_cart}.

The conditional densities  we consider are thus of the form
\begin{align*}
s_{\PartX,K,\theta,\pi}(\cdot|x)=    \sum_{k=1}^K \left( \sum_{\LeafXl
  \in \PartX} \pi_k[\LeafXl]\,
 \Charac{x \in \LeafXl}\right) \Gauss_{\theta_k}\left(\cdot \right)
\end{align*}
where $K$ is the number of component, $\PartX$ is a partition of
$\SpaceX$, $\theta_k$ is the parameter of the $k$th Gaussian and $\pi
= \left( \pi[\LeafXl] \right)_{\LeafXl \in \PartX}$ is the set of
proportions on each hyperrectangle $\LeafXl$. 

We then consider model $S_{\PartX,K,\Set}$ with a fixed number of
class $K$, a fixed recursive dyadic partition $\PartX$  and a given set $\Set$ 
for the $K$-tuples
$(\Phi_{\theta_1},\ldots,\Phi_{\theta_K})$ (or equivalently by a set
$\Theta_{\Set}$ for $\theta=(\theta_1,\ldots,\theta_K)$). Within this model,
the free
parameters are the mixing proportions $\pi[\LeafXl]$ on each hyperrectangle of the
partition and the parameters $\theta$ within $\Theta_\Set$. Our
collection is thus
\begin{align*}
  \model_{\PartX,K,\Set} =\bigg\{ s_{\PartX,K,\theta,\pi}(\cdot|x), \bigg|
&(\Phi_{\Space,\theta_{\Space,1}},\ldots,\Phi_{\Space,\theta_{\Space,K}}) \in \Set,
\forall \LeafXl \in \PartX, \pi[\LeafXl] \in
\Simplex_{K-1}
\bigg\}
\end{align*}
where $\Simplex_{K-1}$ is the $K-1$ dimensional simplex.
The space $\Set$ is chosen
among the \emph{classical} Gaussian $K$-tuples described in
\textcite{biernacki06:_model_mixmod}. 
We use thus some sets
\begin{align*}
\Set_{[\cdot]^K}= \left\{ \left(
    \Gauss_{\Space,\theta_1},\ldots,\Gauss_{\Space,\theta_K} \right) \middle|
  \theta=(\theta_1,\ldots,\theta_K)\in \Theta_{[\cdot]^K} \right\}
\end{align*}
where $\Theta_{[\cdot]^K}$ is defined by  some (mild) constraints on the means $\mu_k$ and some
(strong) constraints on the covariance matrices $\Sigma_k$. We refer
to our technical report~\cite{cohen11:_condit_densit_estim_penal_app}
 for more details. Note that
as a parametric model
\begin{align*}
  \dim(\model_{\PartX,K,\Set}) =
  \NbPartX(K-1)+\dim\left(\Theta_{[\cdot]^K}\right)
\end{align*}
where $\NbPartX$ is the number of regions in the partition.
Our aim is thus to estimate simultaneously the number of classes $K$,
the recursive dyadic partition $\PartX$, the set $\Set$ defining 
the Gaussian $K$-tuples, as well as the parameters within the
corresponding model. 

In the technical report~\cite{cohen11:_condit_densit_estim_penal_app},
we show that~\cref{theo:select} implies in this context
\begin{theorem}
\label{theo:spatgauss}
There exist a $\constspatgausslog>\pi$ and
a $\constspatgausstwo>0$, such that,  for any $\propJKL\in(0,1)$ and
for any $\constoraone>1$,
the penalized estimator of \cref{theo:select} within the collection of
the previous paragraph satisfies
\begin{align*}
\E\left[\JKLtens_{\propJKL,\meas}(s_0,\widehat{s}_{\widehat{{\PartX,K,\Set}}})\right]
&\leq \constoraone \inf_{(\PartX,K,\Set) \in\indmset} \left( \inf_{s_{\PartX,K,\Set} \in \model_{\PartX,K,\Set}} \KLtens_{\meas}(s_0,s_{\PartX,K,\Set}) +
  \frac{\pen({\PartX,K,\Set})}{n} \right)\\
&\qquad\qquad + \constoratwo \frac{1}{n} + \frac{\rhoapp+\rhomin}{n}
\end{align*}
as soon as
\begin{align*}
  \pen({\PartX,K,\Set}) &\geq 
 \constpen \left( 2 \constspatgausslog + \constspatgausstwo +\left( \Log
     \frac{n}{\constspatgausslog}\right)_+ 
\right)
 \dim(\model_{\PartX,K,\Set})
\end{align*}
with $\constpen>\constpenmin$ where  $\constpenmin$ and $\constoratwo$
are the constants of \cref{theo:select} that
depend only on $\propJKL$ and $\constoraone$.
\end{theorem}
A penalty that is proportional to the dimension of the
model is thus sufficient to ensure that the model selected 
performs almost as well as the best possible model in term of
conditional density estimation.
Combining the classical EM strategy for the
Gaussian parameter estimation (see for instance
\textcite{biernacki06:_model_mixmod}) and dynamic programming strategies
for the partition, as described for instance by
\textcite{kolaczyk05:_multisimage}, we have been able to implement this
penalized estimator and to test it on real
datasets~\cite{Bertrand:kv5096}. 
 As in the proof of
\textcite{antoniadis08}, we can also obtain that our proposed estimator yields
a minimax estimate for spatial Gaussian mixture with mixing proportions
having a geometrical regularity even without knowing the number of
classes. Numerical aspects of this method can be found in another study\cite{cohen12:_unsup_gauss_report}.

If the data were
generated according to a Gaussian mixture with varying mixing proportions, one
could obtain the asymptotic convergence of our class estimator
to the optimal Bayes one. However, our theoretical result on the
conditional density estimation does not implies immediately good
segmentation performance in general.
Figure~\ref{fig:glue} illustrates this methodology.
The studied sample is a thin cross-section of maple with a single
layer of hide glue on top of it, prepared recently using materials and
processes from the Cité de la Musique, using materials of the same
type and quality that is used for lutherie.  We show here the result for a low
signal to noise ratio acquisition requiring only two minutes of scan.
Using piecewise constant mixing proportions instead of constant
  mixing proportions leads to a better geometry of the segmentation, with
  less isolated points and more structured boundaries. As described in
  a more applied study~\cite{cohen12:_unsup_gauss_report}, this
  methodology permits to work with a much lower signal to noise ratio
  and thus allows to reduce significantly the acquisition time.

\begin{figure}
  \centering

\begin{tabular}{ccc}
\includegraphics[width=.3\linewidth]{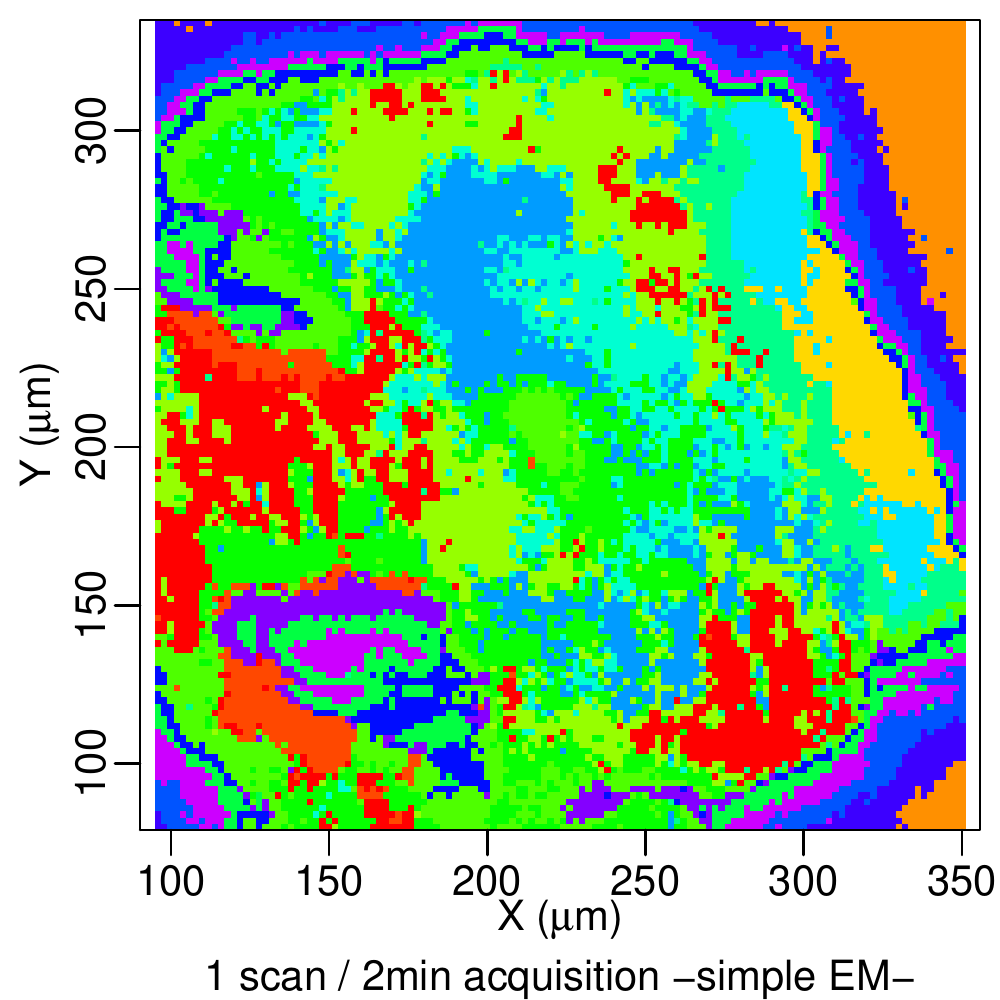}
&
&\includegraphics[width=.3\linewidth]{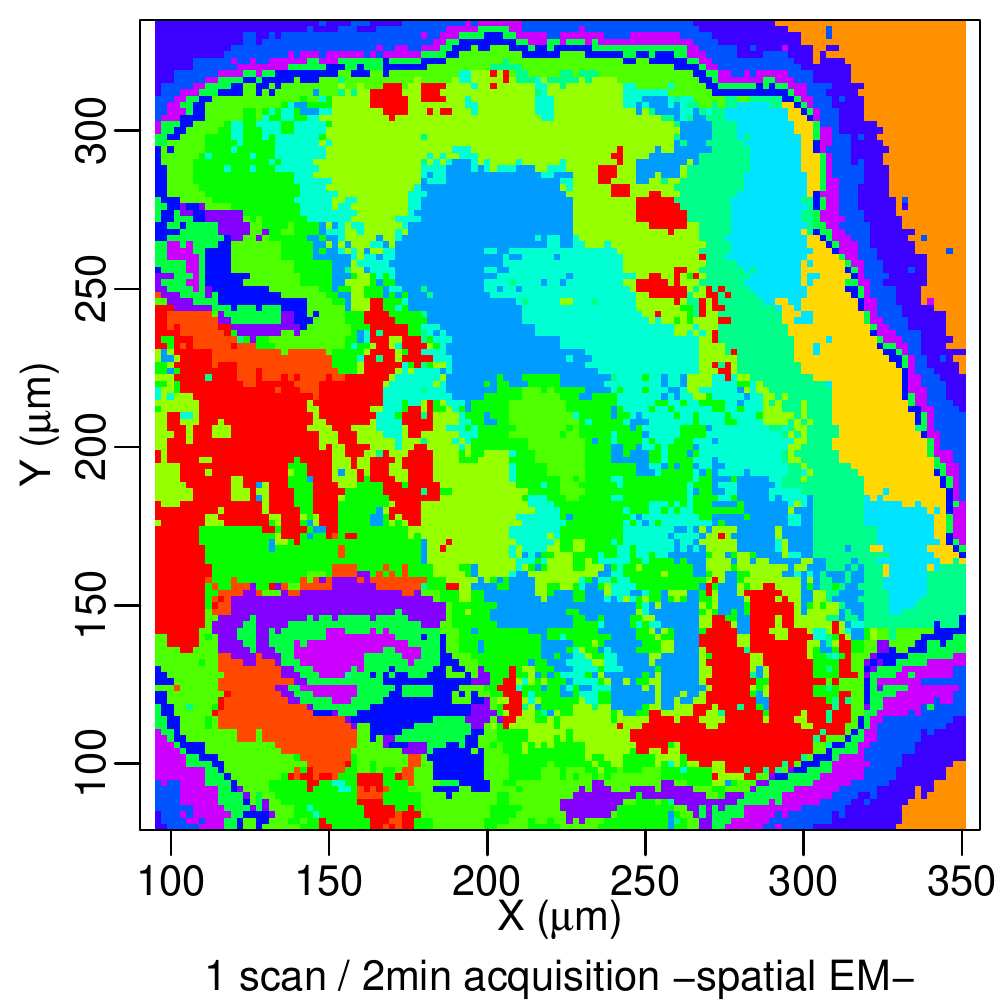}\\
(a) && (b)
\end{tabular}

  \caption{Unsupervised segmentation result: a) with constant mixing
    proportions b) with piecewise constant mixing proportions.}
  \label{fig:glue}
\end{figure}
}


\ifthenelse{\boolean{extended}}{}{
\section{Introduction}

Assume we observe $n$ pairs $\left((X_i,Y_i)\right)_{1 \leq i \leq n}$ of random variables, we are
interested in estimating the law of the second one $Y_i \in
\SpaceY$, called variable,
conditionally to the first one $X_i \in \SpaceX$, called covariate.
In this paper, we assume that
the pairs $(X_i,Y_i)$ are independent while $Y_i$ depends on $X_i$
through its law. More precisely, we assume that the covariates $X_i$'s
are independent but not necessarily identically
distributed. Assumption 
on the $Y_i$'s is stronger: we assume that, conditionally to
the $X_i$'s, they are
independent and  each variable $Y_i$ follows a law of density
$s_0(\cdot|X_i)$ with respect to a common known measure $\ud\meas$.
Our goal is to estimate this two-variable conditional density
function $s_0(\cdot|\cdot)$ from the observations. In this paper, we
apply a penalized maximum likelihood model selection result of~\cite{cohen11:_condit_densit_estim_penal_likel_model_selec} to
partition-based collection in which the conditional densities
depend on covariate in a piecewise constant manner.

The original conditional density estimation problem has been introduced by \textcite{rosenblatt69:_condit} in
the late 60's. In a stationary framework, he used a link
between $s_0(y|x)$ and the supposed existing 
densities $s_{0'}(x)$ and $s_{0''}(x,y)$  of respectively $X_i$
and $(X_i,Y_i)$,
\begin{align*}
  s_0(y|x)=\frac{s_{0''}(x,y)}{s_{0'}(x)},
\end{align*}
and proposed a plugin estimate based on kernel estimation of both
$s_{0''}(x,y)$ and $s_{0'}(x)$. Few other references on this subject seem to
exist before the mid 90's with a study of a spline tensor based
maximum likelihood estimator
proposed by \textcite{stone94:_use_polyn_splin_their_tensor}
and a bias correction of
\citeauthor{rosenblatt69:_condit}'s estimator due to
\textcite{hyndman96:_estim}. 
Kernel based method have been much studied since as stressed by \textcite{li07:_nonpar_econom}. To name a few,
\textcite{fan96:_estim} and \textcite{gooijer03} 
consider local polynomial estimator, 
\textcite{hall99:_method} study a locally logistic estimator
later extended by \textcite{hyndman02:_nonpar}. Pointwise convergence properties are considered, and extensions to
dependent data are often obtained.
Those results are however non  adaptive: their performances depend on
a critical bandwidth choice that should be chosen according to the
regularity of the unknown conditional density. Its
practical choice is rarely discussed with the notable exception of
\textcite{bashtannyk01:_bandw}. Extensions to censored cases have also
been discussed for instance by \textcite{keilegom02:_densit}.
In the approach of \textcite{stone94:_use_polyn_splin_their_tensor},
the conditional density is estimated using a representation,
a parametrized
modelization. This idea has been reused by \textcite{gyorfi07:_nonpar}
with a histogram based approach, by
\textcite{efromovich07:_condit,efromovich10:_oracl}  with a Fourier
basis, and by
\textcite{brunel07:_adapt_estim_condit_densit_presen_censor} and \textcite{akakpo11:_inhom}
with piecewise polynomial representation. Risks of those estimators
are controlled  with a total variation loss for the first one and a quadratic
distance for the others. Furthermore within the quadratic framework,
almost minimax adaptive estimators are constructed using respectively
a blockwise attenuation principle and a penalized model selection
approach. Kullback-Leibler type loss, and thus
maximum likelihood approach, has only been considered by \textcite{stone94:_use_polyn_splin_their_tensor}
as mentioned before and by \textcite{blanchard97:_optim} in a classification
setting with histogram type estimators.

In~\cite{cohen11:_condit_densit_estim_penal_likel_model_selec}, we propose
a penalized maximum likelihood model selection approach to
 estimate  $s_0$. Given a
collection of models $\Models=(\model_\indm)_{\indm\in\indmset}$
comprising conditional densities and their maximum
likelihood estimates
 \begin{align*}
   \widehat{s}_{\indm} = \argmin_{s_{\indm}\in\model_\indm} - \sum_{i=1}^n
   \Log s_\indm(Y_i|X_i),
 \end{align*} we define, for a given penalty $\pen(\indm)$, 
the \emph{best} model $\model_{\widehat{\indm}}$ as the one
that minimizes a penalized likelihood:
\begin{align*}
  \widehat{\indm} = \argmin_{\indm\in\indmset} - \sum_{i=1}^n
   \Log \widehat{s}_\indm(Y_i|X_i) + \pen(\indm).
\end{align*}
The main result of~\cite{cohen11:_condit_densit_estim_penal_likel_model_selec} is a sufficient condition on the penalty
$\pen(\indm)$ such that an oracle type inequality holds for the
conditional density estimation error. In this paper, we show how this
theorem can be used to derive results for two interesting partition-based
conditional density models, inspired by~\textcite{kolaczyk05:_multis},
\textcite{kolaczyk05:_multisimage} and \textcite{antoniadis08}.

Both are based on a recursive partitioning of space $\SpaceX$,
assumed for sake of simplicity
to be equal to $[0,1]^{\dimX}$, they differ by the choice of the
density used, once conditioned by covariates: in the first case, we
consider traditional piecewise polynomial models, while, in the
second case, we use Gaussian mixture models with common mixture
components. The first case is motivated by the work of
\textcite{willet07:_multis_poiss_inten_densit_estim} where they propose a
similar model for Poissonian intensities. The second one is drived
 by an application to unsupervised segmentation, which was our
 original motivation for this work.
For both examples, we prove that the penalty can be chosen roughly proportional
to the dimension of the model.
 
In \cref{sec:gener-penal-maxim}, we summarize the setting and the
results of
\cite{cohen11:_condit_densit_estim_penal_likel_model_selec}. We
describe the loss considered, explain the penalty structure and
present a general penalized maximum likelihood theorem we have
proved. This will be a key tool for the study of the partition-based strategy conducted in \cref{sec:emphp-const-cond}. We
describe first our general partition based approach in
\cref{sec:covar-part-cond} and exemplify it with piecewise polynomial
density with respect to the variable in \cref{sec:piec-polyn-cond} and
with Gaussian mixture with varying proportion in
\cref{sec:spat-mixt-models}. Main proofs are given in Appendix
while proofs of the most technical lemmas are relegated to our technical
report~\cite{cohen11:_condit_densit_estim_penal_app}.

\section{A general penalized maximum likelihood theorem}
\label{sec:gener-penal-maxim}

\subsection{Framework and notation}
\label{sec:sett-maxim-likel}

As in~\cite{cohen11:_condit_densit_estim_penal_likel_model_selec}, we
observe $n$ independent pairs
$\left( (X_i,Y_i)\right) _{1 \leq i \leq n} \in \left(\SpaceX,\SpaceY\right)^n$ where the $X_i$'s are
independent, but not necessarily of same law, and, conditionally to
$X_i$,
 each $Y_i$ is a random variable of unknown conditional
density $s_0(\cdot|X_i)$ with respect to a known reference measure $\ud \meas$.
For any model $\model_{\indm}$, a set of candidate conditional densities,  
we estimate $s_0$
by the conditional
density $\widehat{s}_\indm$ that maximizes the likelihood
(conditionally to $\left( X_i\right) _{1 \leq i \leq n}$) or
equivalently that minimizes the opposite of the log-likelihood,
denoted -log-likelihood from now on:
\begin{align*}
  \widehat{s}_\indm = \argmin_{s_{\indm} \in \model_{\indm}} \left(  \sum_{i=1}^n - \Log(s_\indm(Y_i|X_i) ) \right). 
\end{align*}
To avoid existence issue, we should work with almost minimizer of this
quantity and define a $\rhoapp$ -log-likelihood minimizer as any
$\widehat{s}_\indm$ that satisfies
\begin{align*}
\sum_{i=1}^n - \Log(\widehat{s}_\indm(Y_i|X_i) ) \leq
\inf_{s_\indm \in
  \model_\indm} \left(  \sum_{i=1}^n - \Log(s_\indm(Y_i|X_i) ) \right)
+ \rhoapp.
\end{align*}
Given a collection $\Models=(\model_\indm)_{\indm\in\indmset}$ of
models, we construct a penalty $\pen(\indm)$ and select the best
model $\widehat{\indm}$ as the one that minimizes
\begin{align*}
  \sum_{i=1}^n - \Log(\widehat{s}_\indm(Y_i|X_i) ) + \pen(\indm).
\end{align*}
In~\cite{cohen11:_condit_densit_estim_penal_likel_model_selec}, we give
conditions on penalties ensuring that the resulting estimate $\widehat{s}_{\widehat{\indm}}$
is a \emph{good} estimate of the true conditional density.

We should now specify our \emph{goodness} criterion. As we are working
in a maximum likelihood approach, the most natural quality measure is
the Kullback-Leibler divergence $\KL$. As we consider law with densities with
respect to a known measure $\ud\meas$, we use the following notation
\begin{align*}
  \KL_{\meas}(s,t) = \KL(s \ud\meas,t \ud\meas) & = \begin{cases} \int_{\Omega} \frac{s}{t} \Log \frac{s}{t} t \ud\meas &
    \text{if $s \ud\meas \ll t \ud\meas $}
\\
+ \infty & \text{otherwise}.
\end{cases}
\end{align*}
where $s \ud\meas \ll t \ud\meas $ means $\forall\Omega'
      \subset \Omega, \int_{\Omega'} t \ud\meas =0 \implies
      \int_{\Omega'} s \ud \meas=0$.
Remark that, contrary to the quadratic loss, this divergence is an
intrinsic quality measure between probability laws: it does not
depend on the reference measure $\ud\meas$. However, the densities
depend on this reference measure, this is stressed by the
index $\meas$ when we work with the non intrinsic densities instead of
the probability measures.
As we study conditional densities and not classical
densities, the previous divergence should be further adapted.
To take into account the structure of conditional densities and the
design of $(X_i)_{1\leq i \leq n}$, we 
use the following \emph{tensorized} divergence:
\begin{align*}
\KLtens_{\meas}(s,t) & 
= \E
\left[ \frac{1}{n}\sum_{i=1}^n
  \KL_{\meas}(s(\cdot|X_i),t(\cdot|X_i))\right].
\end{align*}
This divergence appears as the natural one in this setting and reduces
to classical ones in specific settings:
\begin{itemize}
\item If the law of $Y_i$ is independent of $X_i$, that is $s(\cdot|X_i)=s(\cdot)$ and $t(\cdot|X_i)=t(\cdot)$ do not depend
  on $X_i$, this divergence reduces to the classical
  $\KL_{\meas}(s,t)$.
\item If the $X_i$'s are not random but fixed, that is
  we consider a fixed design case, this divergence is the
  classical fixed design type divergence in which there is no expectation. 
\item If the $X_i$'s are i.i.d., this divergence can be rewritten
  as
\begin{alignI}
 \KLtens_{\meas}(s,t)  =
\E
\left[  \KL_{\meas}(s(\cdot|X_1),t(\cdot|X_1))\right].
\end{alignI}
\end{itemize}
Note that this divergence is an \emph{integrated} divergence as it is the
average over the index $i$ of the mean with respect to the law of $X_i$
of the divergence between the conditional densities for a given
covariate value. Remark that more weight is given to
regions of high density of the covariates than to regions of low density
and, in particular, divergence values outside the supports of the $X_i$'s are not used.
When $\hat{s}$ is an
estimator, or any function that depends on the observations,
$\KLtens_{\meas}(s,\hat{s})$ measures this (random) integrated divergence between $s$ and
$\hat{s}$ conditionally to the observations while $\E\left[\KLtens_{\meas}(s,\hat{s})\right]$ is the
average of this random quantity with respect to the observations.

As often in density estimation, we are not able to control this
loss but only a smaller one.
Namely, we use the Jensen-Kullback-Leibler divergence
$\JKL_{\propJKL}$ with $\propJKL\in(0,1)$ defined by
\begin{align*}
  \JKL_{\propJKL}(s\ud\meas,t\ud\meas) = \JKL_{\propJKL,\meas}(s,t)= \frac{1}{\propJKL}
  \KL_{\meas}\left( s, (1-\propJKL) s+ \propJKL t \right).
\end{align*}
Note that this divergence appears explicitly with
$\propJKL=\frac{1}{2}$ in~\textcite{massart07:_concen}, but can also be found
implicitly in  \textcite{birge98:_minim} and \textcite{geer95}. We use
the name Jensen-Kullback-Leibler divergence in the same way
\textcite{lin91:_diver_shann} use the name Jensen-Shannon divergence
for a sibling in an information theory work. This divergence is
smaller than the Kullback-Leibler one but larger, up to a constant
factor,
 than the squared Hellinger one,  $\d^{2}_{\meas}(s,t)=\int_{\Omega}
|\sqrt{s}-\sqrt{t}|^2 \ud\meas$, and the squared $L_1$ distance, $\|s-t\|^2_{\meas,1}=\left(\int_{\Omega}
|s-t| \ud\meas\right)^2$, as proved
\ifthenelse{\boolean{extended}}{in Appendix}{in our
  technical report~\cite{cohen11:_condit_densit_estim_penal_app}}. More
  precisely, we use their tensorized counterparts:
\begin{align*}
\dtwotens_{\meas}(s,t) & =
\E \left[ \frac{1}{n}\sum_{i=1}^n
  \d^2_{\meas}(s(\cdot|X'_i),t(\cdot|X'_i))\right]
 &
 \text{and}&&
 \JKLtens_{\propJKL,\meas}(s,t) & 
 = \E \left[ \frac{1}{n}\sum_{i=1}^n
   \JKL_{\propJKL,\meas}(s(\cdot|X'_i),t(\cdot|X'_i))\right].
\end{align*}

\subsection{Penalty, bracketing entropy and Kraft inequality}

Our condition on the penalty is given as a lower bound on its value:
\begin{align*}
  \pen(\indm) \geq \kappa_0 \left( \DIMH_{\indm} + x_{\indm} \right)
\end{align*}
where $\kappa_0$ is an absolute constant, $\DIMH_{\indm}$ is a
quantity, depending only on the model $\model_\indm$, that measures its
complexity (and is often almost proportional to its dimension) while
$x_{\indm}$  is a non intrinsic coding term that depends on the
structure of the whole model collection.

The complexity term $\DIMH_{\indm}$ is related to the bracketing
entropy of the model $\model_\indm$ with respect to the Hellinger type  divergence
$\dtens_{\meas}(s,t)=\sqrt{\dtwotens_{\meas}(s,t)}$, or more precisely to the
bracketing entropies of its subsets $\model_\indm(\widetilde{s},\sigma)=\left\{ s_\indm \in
\model_\indm \middle| \dtens_{\meas}(\widetilde{s},s_\indm) \leq \sigma
\right\}$. We recall that a bracket $[t^-,t^+]$ 
is a pair of functions such that $\forall (x,y) \in
\SpaceX\times\SpaceY, t^-(y|x) \leq t^+(y|x)$ and that a conditional density function
$s$ is said to belong to the bracket $[t^-,t^+]$ if $\forall (x,y) \in
\SpaceX\times\SpaceY, t^-(y|x)\leq
s(y|x) \leq t^+(y|x)$. The bracketing entropy
$H_{[\cdot],\dtens_{\meas}}(\delta,\setmodel)$ of a set $\setmodel$ is defined as the logarithm of the minimum number
$N_{[\cdot],\dtens_{\meas}}(\delta,\setmodel)$ of brackets $[t^-,
t^+]$ of width $\dtens_{\meas}(t^-,t^+)$ smaller than $\delta$ 
such that every function of $\setmodel$ belongs to one of these brackets.
To define $\DIMH_{\indm}$, we first impose a structural assumption:
\begin{assumptionHm}
 There is
 a non-decreasing function
 $\phi_\indm(\delta)$ such
that $\delta\mapsto \frac{1}{\delta} \phi_\indm(\delta)$ is non-increasing on $(0,+\infty)$ and for
every $\sigma\in\R^+$ and every $s_\indm \in \model_\indm$
\begin{align*}
  \int_0^{\sigma} \sqrt{%
    H_{[\cdot],\dtens_{\meas}}\left(\delta,\model_\indm(s_\indm,\sigma)\right)}
  \, \ud\delta \leq \phi_\indm(\sigma).
\end{align*}
\end{assumptionHm}
Note that the function $\sigma\mapsto \int_0^{\sigma} \sqrt{%
    H_{[\cdot],\dtens_{\meas}}\left(\delta,\model_\indm\right)}
  \, \ud\delta$ does always satisfy this assumption.
$\DIMH_{\indm}$ is then defined as $n\sigma^2_{\indm}$ with $\sigma^2_{\indm}$ the unique root
of \begin{alignI}
\frac{1}{\sigma}\phi_\indm(\sigma)=\sqrt{n}\sigma.
\end{alignI}
A good choice of $\phi_{\indm}$ is one which leads to a small upper
bound of $\DIMH_{\indm}$.
The bracketing entropy integral appearing in the assumption, often call Dudley integral, plays an
important role in empirical processes theory, as stressed for instance
in \textcite{vaart96:_weak_conver}. The equation defining
$\sigma_{\indm}$ corresponds to an approximate optimization of a
supremum bound as shown explicitly in the proof.
This definition is obviously far from being very explicit but it turns
out that it can be related to an entropic dimension of the
model. Recall that the classical entropic dimension of a compact set $S$
with respect to a metric $d$
can be defined as the smallest real $\DimH$ such that there is a $\ConstMultH$ such 
\begin{align*}
\forall \delta > 0, H_{d}(\delta,S) \leq  \DimH (\log \left(\frac{1}{\delta}\right) + \ConstMultH)
\end{align*}
where $H_{d}$ is the classical entropy with respect to metric
$d$. Replacing the classical entropy by a bracketing one, we define
the bracketing dimension $\DimH_{\indm}$ of a compact set as the
smallest real $\DimH$ such that there is a $\ConstMultH$ such 
\begin{align*}
\forall \delta > 0, H_{[\cdot],d}(\delta,S) \leq  \DimH (\log \left(\frac{1}{\delta}\right) + \ConstMultH).
\end{align*}
As hinted by the notation, for parametric model, under mild assumption on the parametrization,
this bracketing dimension coincides with the usual one.
It turns out that if this bracketing dimension exists then
$\DIMH_{\indm}$ can be thought as
roughly proportional to $\DimH_{\indm}$.
More precisely, \ifthenelse{\boolean{extended}}{in Appendix,}{
in our technical report~\cite{cohen11:_condit_densit_estim_penal_app},}
 we obtain
\begin{proposition}
\label{prop:proporsimple}
\begin{itemize}
\item
if
\begin{alignI}
\exists \DimH_{\indm} \geq 0, \exists \ConstMultH_{\indm} \geq 0,
\forall \delta \in (0,\sqrt{2}],  H_{[\cdot],\dtens_{\meas}}(\delta,\model_{\indm}) \leq
\ConstH_{\indm} + \DimH_{\indm}  \Log \frac{1}{\delta}
\end{alignI}
then  
\begin{itemize}
\item if $\DimH_{\indm} > 0$,
\Hindm\ holds with a function $\phi_\indm$ such that
\begin{alignI}
  \DIMH_{\indm}  \leq \left( 2 \constspatgausslogindm
+ 1 + \left(\Log
    \frac{n}{e\constspatgausslogindm \DimH_{\indm}}\right)_+ \right)
\DimH_{\indm}
\end{alignI}
with $\constspatgausslogindm=\left(\sqrt{\frac{\ConstH_{\indm}}{\DimH_{\indm}}} +
    \sqrt{\pi}\right)^2$,
\item if $\DimH_{\indm}=0$, \Hindm\ holds with the function
$\phi_\indm(\sigma)=\sigma\sqrt{\ConstH_{\indm}}$ which is such \begin{alignI}
  \DIMH_{\indm}  = \ConstH_{\indm},
\end{alignI}
\end{itemize}
\item
if
\begin{alignI}
\exists \DimH_{\indm} \geq  0, \exists \ConstH_{\indm} \geq 0, \forall
\sigma \in (0,\sqrt{2}], \forall \delta \in (0,\sigma],   H_{[\cdot],\dtens_{\meas}}(\delta,\model_{\indm}(s_{\indm},\sigma)) \leq
 \ConstH_{\indm} + \DimH_{\indm} \Log
   \frac{\sigma}{\delta}
\end{alignI}
then 
\begin{itemize}
\item
if $\DimH_{\indm}>0$, \Hindm\  holds with a function $\phi_\indm$ such that \begin{alignI}
  \DIMH_{\indm}  = \constspatgausslogindm
    \DimH_{\indm}
\end{alignI}
with $\constspatgausslogindm=\left(\sqrt{\frac{\ConstH_{\indm}}{\DimH_{\indm}}} +
    \sqrt{\pi}\right)^2$,
\item if $\DimH_{\indm}=0$, \Hindm\ holds with the function
$\phi_\indm(\sigma)=\sigma\sqrt{\ConstH_{\indm}}$ which is such \begin{alignI}
  \DIMH_{\indm}  = \ConstH_{\indm}.
\end{alignI}
\end{itemize}
\end{itemize}
\end{proposition}
We assume bounds on the entropy only for $\delta$ and $\sigma$
smaller than $\sqrt{2}$, but, as for any conditional density pair $(s,t)$ $\dtens_{\meas}(s,t)\leq \sqrt{2}$,
\begin{align*}
  H_{[\cdot],\dtens_{\meas}}(\delta,\model_{\indm}(s_{\indm},\sigma))
  = H_{[\cdot],\dtens_{\meas}}(\delta \wedge \sqrt{2}
  ,\model_{\indm}(s_{\indm},\sigma \wedge \sqrt{2}))
\end{align*}
which implies that those bounds are still useful when $\delta$ and
$\sigma$ are large.

The coding term $x_{\indm}$ is constrained by a Kraft type assumption:
\begin{assumptionK}
There is a family $(x_\indm)_{\indm\in\indmset}$ of
non-negative number such that
\begin{align*}
\sum_{\indm\in\indmset} e^{-x_\indm} \leq \Sigma < +\infty
\end{align*}
\end{assumptionK}
This condition is an information theory type condition and thus can be
interpreted as a coding condition as stressed by
\textcite{barron08:_mdl_princ_penal_likel_statis_risk}.

\subsection{A penalized maximum likelihood theorem}

For technical reason, we also have to assume a separability condition
on our models:
\begin{assumptionSm}
There exist a countable subset
$\model'_\indm$ of $\model_\indm$ and a set $\SpaceY_\indm'$ with
$\meas(\SpaceY\setminus\SpaceY_\indm')=0$ such that for every $t\in\model_\indm$,
there exists a sequence $(t_k)_{k\geq 1}$ of elements of
$\model'_\indm$ such that for every $x$ and for every $y\in\SpaceY_\indm'$, $\Log\left(t_k(y|x)\right)$  goes
to $\Log\left(t(y|x)\right)$ as $k$ goes to infinity.
\end{assumptionSm}

The main result of
\cite{cohen11:_condit_densit_estim_penal_likel_model_selec} is
\begin{theorem}
\label{theo:select}
Assume we observe $(X_i,Y_i)$ with unknown conditional density $s_0$.
Let
$\Models=(\model_\indm)_{\indm\in\indmset}$ an at most countable model
collection.
Assume Assumption (K) holds while Assumptions ($\text{H}_{\indm}$) and ($\text{Sep}_{\indm}$)
hold for every model $\model_{\indm} \in \Models$.
Let $\widehat{s}_\indm$ be a $\rhoapp$ -log-likelihood minimizer in
  $\model_\indm$
\[
\sum_{i=1}^n - \Log(\widehat{s}_\indm(Y_i|X_i) ) \leq
\inf_{s_\indm \in
  \model_\indm} \left(  \sum_{i=1}^n - \Log(s_\indm(Y_i|X_i) ) \right)
+ \rhoapp
\]

Then for any $\propJKL\in(0,1)$ and any $\constoraone>1$, there are two constants
$\constpenmin$ and $\constoratwo$ depending only on $\propJKL$ and
$\constoraone$ such that,
as soon as  for every index $\indm\in\indmset$
\begin{align*}
\pen(\indm) \geq \constpen \left( \DIMH_{\indm} + x_\indm \right)
\quad\text{with $\constpen>\constpenmin$}
\end{align*}
where $\DIMH_{\indm}=n\sigma_{\indm}^2$ with $\sigma_\indm$ the unique root of
\begin{alignI}
\frac{1}{\sigma}\phi_\indm(\sigma)=\sqrt{n}\sigma
\end{alignI},
the penalized likelihood estimate $\widehat{s}_{\widehat{\indm}}$
  with $\widehat{\indm}$ such that
  \begin{align*}
\sum_{i=1}^n - \Log(\widehat{s}_{\widehat{\indm}}(Y_i|X_i) )
+ \pen(\widehat{\indm})
\leq \inf_{\indm\in\indmset} \left( \sum_{i=1}^n - \Log(\widehat{s}_\indm(Y_i|X_i) )
+ \pen(\indm) \right) + \rhomin    
  \end{align*}
satisfies
\begin{align*}
\E\left[\JKLtens_{\propJKL,\meas}(s_0,\widehat{s}_{\widehat{\indm}})\right]
\leq \constoraone \inf_{\indm\in\indmset} \left( \inf_{s_\indm \in \model_\indm} \KLtens_{\meas}(s_0,s_\indm) +
  \frac{\pen(\indm)}{n} \right) + \constoratwo \frac{\Sigma}{n} + \frac{\rhoapp+\rhomin}{n}.
\end{align*}
\end{theorem}

This theorem extends Theorem 7.11
of \textcite{massart07:_concen}, which handles only density
estimation, and reduces to it if all conditional densites considered
do not depend on the covariate. The cost of model selection with respect to the choice
of the best single model is proved to be very mild. Indeed,
let $\pen(m)=\constpen (\DIMH_{\indm} + x_{\indm})$ then one
obtains
\begin{align*}
\E\left[\JKLtens_{\propJKL,\meas}(s_0,\widehat{s}_{\widehat{\indm}})\right]
&\leq \constoraone \inf_{\indm\in\indmset} \left( \inf_{s_\indm \in \model_\indm} \KLtens_{\meas}(s_0,s_\indm) +
  \frac{\constpen}{n} (\DIMH_{\indm} + x_{\indm}) \right) + \constoratwo
\frac{\Sigma}{n} + \frac{\rhoapp+\rhomin}{n}\\
& \leq \constoraone \frac{\constpen}{\constpenmin}
\left(\max_{\indm\in\indmset} \frac{\DIMH_{\indm} + x_{\indm}}{\DIMH_{\indm}}\right) \inf_{\indm\in\indmset} \left( \inf_{s_\indm \in \model_\indm} \KLtens_{\meas}(s_0,s_\indm) +
  \frac{\constpenmin}{n} \DIMH_{\indm} \right) + \constoratwo
\frac{\Sigma}{n} + \frac{\rhoapp+\rhomin}{n}.
\end{align*}
where 
\begin{align*}
  \inf_{\indm\in\indmset} \left( \inf_{s_\indm \in \model_\indm} \KLtens_{\meas}(s_0,s_\indm) +
  \frac{\constpenmin}{n} \DIMH_{\indm} \right) + \constoratwo
\frac{\Sigma}{n} + \frac{\rhoapp}{n}
\end{align*}
is the best known bound for a generic single model, as explained 
in \cite{cohen11:_condit_densit_estim_penal_likel_model_selec}:
As soon as the term $x_{\indm}$ remains small relatively to
$\DIMH_{\indm}$, we have thus an oracle inequality:
the penalized estimate satisfies up to a small factor the same bound 
as the estimate in the best model.
The price to pay for the use of a collection of model is thus
small. The gain is on the contrary huge: we do not have to
know the best model within a collection to almost achieve its performance. 
Note that as there exists a constant $c_\propJKL>0$ such that
$c_{\propJKL}\|s-t\|_{\meas,1}^{\tens,2} \leq
\JKLtens_{\propJKL,\meas}(s,t)$, as proved in our technical report~\cite{cohen11:_condit_densit_estim_penal_app}, this theorem implies a bound for the
squared $L_1$ loss of the estimator. 

For sake of generality, this theorem is relatively abstract.
A natural question is the existence of
interesting model collections that satisfy these assumptions. Motivated
by an application to unsupervised hyperspectral image segmentation,
already mentioned in
\cite{cohen11:_condit_densit_estim_penal_likel_model_selec}, we consider the case where the covariate $X$ belongs to $[0,1]^{\dimX}$
and use collections for which the conditional densities depend
on the covariate only in a piecewise constant manner.
}

\section{Partition-based conditional density models}
\label{sec:emphp-const-cond}

\subsection{Covariate partitioning and conditional density estimation}
\label{sec:covar-part-cond}

Following an idea developed by \textcite{kolaczyk05:_multisimage}, we
partition the covariate domain and consider candidate
conditional density estimates that depend on the covariate only
through the region it belongs. We are thus
interested in conditional densities that can be written as
\begin{align*}
  s(y|x) = \sum_{\LeafXl\in\PartX} s(y|\LeafXl) \Charac{x\in\LeafXl}
\end{align*}
where $\PartX$ is partition of $\SpaceX$, $\LeafXl$ denotes a generic region in this partition,
 $\Characonly$ denotes the characteristic function
of a set and $s(y|\LeafXl)$ is a density for any
$\LeafXl\in\PartX$. Note that this strategy, called as in
\textcite{willet07:_multis_poiss_inten_densit_estim} partition-based, shares a lot with the CART-type 
strategy proposed by \textcite{donoho97:_cart} in an image processing setting. 

Denoting $\NbPartX$ the number of regions in
this partition, the model we consider are thus specified by a partition $\PartX$ and a
set $\DSet$ of $\NbPartX$-tuples of densities into which
$(s(\cdot|\LeafXl))_{\LeafXl \in \PartX}$ is chosen. This set
$\DSet$ can be a product of density sets,
yielding an independent choice on each region of the partition, or have a
more complex structure. We study two examples: in the first one, 
$\DSet$ is indeed a product of piecewise
polynomial density sets, while in the second one $\DSet$ is a set of $\NbPartX$-tuples of Gaussian mixtures sharing the same
mixture components. Nevertheless, denoting with a slight abuse of notation
$\model_{\PartX,\DSet}$ such a model, our
$\rhoapp$-log-likelihood estimate in this model is  any conditional density
$\widehat{s}_{\PartX,\DSet}$ such that
 \begin{align*} 
\left(  \sum_{i=1}^n -
   \Log(\widehat{s}_{\PartX,\DSet}(Y_i|X_i) ) \right) & \leq \min_{s_{\PartX,\DSet} \in \model_{\PartX,\DSet}} \left(  \sum_{i=1}^n -
   \Log(s_{\PartX,\DSet}(Y_i|X_i) ) \right) + \rhoapp.
 \end{align*}

We first specify the partition collection we consider.
For the sake of simplicity we restrict our description to the case where the covariate space
$\SpaceX$ is simply $[0,1]^{\dimX}$. We stress that 
the proposed strategy can easily
be adapted to more general settings including discrete variable
ordered or not. We impose a strong structural assumption on the
partition collection considered that allows to control their
\emph{complexity}. We only consider 
five specific hyperrectangle based collections of partitions of $[0,1]^{\dimX}$:
\begin{itemize}
\item Two are recursive dyadic partition collections.
  \begin{itemize}
  \item The uniform dyadic partition collection ($\UDPX$) in which all
    hypercubes are subdivided in $2^{\dimX}$ hypercubes of equal
    size at each step. In this collection, in the partition obtained
    after $J$ step, all
    the $2^{\dimX J}$ hyperrectangles $\{ \LeafXl \}_{1\leq l \leq
      \NbPartX}$ are thus hypercubes whose measure $|\LeafXl|$
    satisfies $|\LeafXl|=2^{- \dimX J}$. We stop the
    recursion as soon as the number of steps $J$ satisfies
 $\frac{2^{\dimX}}{n} \geq |\LeafXl| \geq\frac{1}{n}$.
\item The recursive dyadic partition collection ($\RDPX$) in which at each
  step a hypercube of measure $|\LeafXl| \geq
  \frac{2^{\dimX}}{n}$ is subdivided in $2^{\dimX}$
  hypercubes of equal size.
  \end{itemize}
\item Two are recursive split partition collections.
  \begin{itemize}
  \item The recursive dyadic split partition ($\RDSPX$) in which at each
    step a hyperrectangle of measure $|\LeafXl|\geq
  \frac{2}{n}$ can be subdivided in $2$
  hyperrectangles of equal size by an even split along one of the
  $\dimX$ possible directions.
\item The recursive split partition ($\RSPX$) in which at each step a
  hyperrectangle of measure $|\LeafXl|\geq\frac{2}{n}$ can be
  subdivided in $2$ hyperrectangles of measure larger than
  $\frac{1}{n}$ by a split along one a point of the grid
  $\frac{1}{n}\Z$ in one the $\dimX$ possible
  directions.
  \end{itemize}
\item The last one does not possess a hierarchical structure. The
  hyperrectangle partition collection ($\HRPX$) is the full collection of all
  partitions into hyperrectangles whose corners are located on the grid
  $\frac{1}{n}\Z^{\dimX}$ and whose volume is larger than $\frac{1}{n}$.
\end{itemize}
We denote by $\CollPart^{\starX}$ the corresponding partition
collection where  $\starX$ is either
$\UDPX$, $\RDPX$, $\RDSPX$, $\RSPX$ or $\HRPX$. 

As noticed by \textcite{kolaczyk05:_multis}, \textcite{huang06:_fast} or
\textcite{willet07:_multis_poiss_inten_densit_estim}, the first four
partition collections,  ($\CollPart^{\UDPX}$, $\CollPart^{\RDPX}$,
$\CollPart^{\RDSPX}$, $\CollPart^{\RSPX}$), have a tree
structure. Figure~\ref{fig:quadtree} illustrates this structure for a
$\RDPX$ partition. This specific structure is mainly
used to obtain an efficient numerical algorithm performing the model
selection. For sake of completeness, we have also added the much more
complex to deal with collection $\CollPart^{\HRPX}$, for which only
exhaustive search algorithms exist.

\begin{figure}
  \centering
  \includegraphics[width=.2\textwidth]{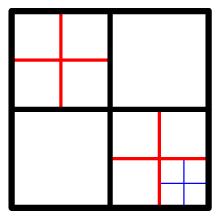}
\hspace*{1cm}
  \includegraphics[width=.2\textwidth]{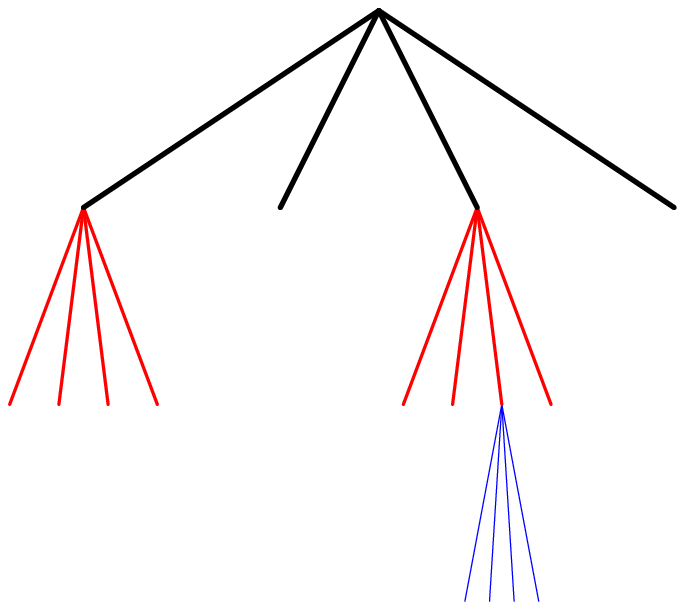}
\caption{Example of a recursive dyadic partition with its associated
  dyadic tree.}
\label{fig:quadtree}
\end{figure}

As proved \ifthenelse{\boolean{extended}}{in Appendix}{in our technical
report~\cite{cohen11:_condit_densit_estim_penal_app}}, those partition
collections satisfy Kraft type inequalities with weights constant
for the $\UDPX$ partition collection and proportional to the number $\NbPartX$ of
hyperrectangles for the other collections. Indeed,
\begin{proposition}
\label{prop:codingpolpart}
For any of the five described partition collections
$\CollPart^{\starX}$,
$\exists \PartCstA^{\star}, \PartCstB^{\star}, \PartCstc^{\star}
\text{ and } \PartCstSigma$ such that
for all $c \geq \PartCstc^{\starX}$:
\begin{align*}
  \sum_{\PartX \in \CollPart^{\starX}} e^{-c \left( \PartCstA^{\starX}
      + \PartCstB^{\starX} \NbPartX \right)} \leq \PartCstSigma^{\starX}
  e^{-c\max\left(\PartCstA^{\starX},\PartCstB^{\starX}\right)}.
\end{align*}
\end{proposition}
Those constants can be chosen as follow:\\
{\small
\newcolumntype{C}{>{$\displaystyle}c<{$}}
 \begin{tabular}{C|C|C|C|C|C}
& \star=\UDPX & \star=\RDPX & \star=\RDSPX &
\star=\RSPX & \star=\HRPX\\
\hline
\PartCstA^{\star} & \Log \left( \max\left(2,1 + \frac{\Log n}{\dimX \Log 2}\right) \right)  & 0 & 0 & 0 & 0 \\
\PartCstB^{\star} & 0 & \Log 2 & \lceil \Log (1 + \dimX)\rceil_{\Log
  2} & \lceil \Log (1+\dimX)\rceil_{\Log
  2} &  \dimX \lceil \Log n \rceil_{\Log
  2}\\ 
& &  & & + \lceil \Log
n \rceil_{\Log
  2}& \\
\PartCstc^{\star} & 0 & \frac{2^\dimX}{2^\dimX-1} & 2 & 2 & 1\\
\PartCstSigma^{\star} &
1 + \frac{\Log n}{\dimX \Log 2} & 2 & 2(1+ \dimX)
& 4(1+\dimX) n & (2n)^{\dimX}
\end{tabular}
}\\
where $\lceil x \rceil_{\Log
  2}$  is the smallest multiple of $\Log 2$ larger than $x$. 
Furthermore, as soon as $c\geq 2 \Log 2$ the right hand term of the
bound is smaller than $1$.
This will
prove 
useful to verify Assumption $(K)$ for the model collections of the next
sections.

In those sections, we study the two different choices proposed
above for the set $\DSet$. We first consider a piecewise polynomial
strategy similar to the one proposed by
\textcite{willet07:_multis_poiss_inten_densit_estim} defined for 
$\SpaceY=[0,1]^{\dimY}$ in which the set $\DSet$ is a product of
sets. We then consider a Gaussian mixture strategy with varying
mixing proportion but common mixture components that extends the work
of \textcite{maugis09:_gauss} and has been the
original motivation of this work. In both cases, we prove that the
penalty can be chosen roughly proportional to the dimension.

\subsection{Piecewise polynomial conditional density estimation}
\label{sec:piec-polyn-cond}

In this section, we let $\SpaceX=[0,1]^{\dimX}$, $\SpaceY=[0,1]^{\dimY}$ and
$\meas$ be the Lebesgue measure $\ud y$. Note that, in this case,
$\meas$  is a probability measure on
$\SpaceY$. Our candidate density $s(y|x\in\LeafXl)$ is then  chosen
among piecewise polynomial densities. More precisely,
we reuse a hyperrectangle partitioning strategy this time for $\SpaceY=[0,1]^{\dimY}$ and impose
that our candidate conditional density $s(y|x\in\LeafXl)$ is a square of polynomial on each
hyperrectangle $\LeafYlk$ of the partition $\PartYLeafXl$. This differs from
the choice of \textcite{willet07:_multis_poiss_inten_densit_estim} in
which the candidate density is simply a polynomial. The two choices
coincide however when the polynomial is chosen among the constant
ones.
Although our choice of using squares of polynomial is less natural, it
already ensures the positiveness of the candidates so that we only
have to impose that the integrals of the piecewise polynomials are equal
to $1$ to obtain conditional densities. It turns out to be also crucial
to obtain a control of the local bracketing entropy of our models. 
Note that this setting differs from the one of
\textcite{blanchard97:_optim} in which $\SpaceY$ is a finite discrete
set.

We should now define the sets $\DSet$ we consider for a given
partition $\PartX=\{ \LeafXl \}_{1 \leq l \leq \NbPartX}$ of
$\SpaceX=[0,1]^{\dimX}$. Let
$\dimPYmax=(\dimPYmaxi{1},\ldots,\dimPYmaxi{\dimY})$, we first define  for any partition
$\PartY= \{ \LeafY_{k}\}_{1 \leq k \leq \NbPartY}$ of
$\SpaceY=[0,1]^{\dimY}$ the set $\DDSet_{\PartY,\dimPYmax}$ of
squares of piecewise polynomial densities of maximum degree $\dimPYmax$
defined in the partition $\PartY$:
\begin{align*}
  \DDSet_{\PartY,\dimPYmax} = \left\{ s(y) = \sum_{\LeafY_{k}
    \in \PartY} P^2_{\LeafY_k}(y) \Charac{y \in \LeafY_{k}}
  \middle| 
\begin{array}{l}
\forall \LeafY_{k} \in \PartY, P_{\LeafY_k} \text{ polynomial of degree at most $\dimPYmax$},
\\
 \sum_{\LeafY_{k} \in \PartY} \int_{\LeafY_k} P^2_{\LeafY_k}(y) = 1 
\end{array}
\right\}
\end{align*}
For any partition collection $\PartXY = \left( \PartYLeafXl \right)_{1 \leq l
  \leq \NbPartX} = \left( \{ \LeafY_{l,k}\}_{1 \leq k \leq \NbPartYLeafXl} \right)_{1 \leq l
  \leq \NbPartX}$ of $\SpaceY=[0,1]^{\dimY}$, we can thus defined the
set
$\DSet_{\PartXY,\dimPYmax}$ of $\NbPartX$-tuples of piecewise polynomial
densities as
\begin{align*}
 \DSet_{\PartXY,\dimPYmax} = \left\{ \left(s(\cdot|\LeafXl)\right)_{\LeafXl\in\PartX} \middle| \forall
       \LeafXl \in \PartX, s(\cdot|\LeafXl) \in \DDSet_{\PartYLeafXl,\dimPYmax}\right\}.
\end{align*}
The model $\model_{\PartX,\DSet_{\PartXY,\dimPYmax}}$, that is
denoted $\model_{\PartXY,\dimPYmax}$ with a slight abuse of notation,
is thus the set
\begin{align*}
  \model_{\PartXY,\dimPYmax} & = \left\{ 
s(y|x)  = \sum_{\LeafXl\in\PartX} s(y|\LeafXl) \Charac{x\in\LeafXl}
\middle| \left( s(y|\LeafXl\right)_{\LeafXl\in\PartX} \in
\DSet_{\PartXY,\dimPYmax}
\right\}\\
& = \left\{
s(y|x)  = \sum_{\LeafXl\in\PartX} \sum_{\LeafYlk \in \PartYLeafXl}
P^2_{\LeafXl\times\LeafYlk}(y)
\Charac{y\in\LeafYlk} \Charac{x\in\LeafXl}
\middle|
\begin{array}{l}
\forall \LeafXl \in \PartX, \forall \LeafYlk
\in \PartYLeafXl,\\ 
\quad P_{\LeafXl\times\LeafYlk} \text{ polynomial of degree at most $\dimPYmax$},
\\
\forall \LeafXl \in \PartX, \sum_{\LeafYlk \in \PartYLeafXl}
\int_{\LeafYlk} P^2_{\LeafXl\times\LeafYlk}(y) = 1 
\end{array}
\right\}
\end{align*}
Denoting $\LeafXYlk$ the product $\LeafXl\times\LeafYlk$, the
conditional densities of the previous set can be advantageously rewritten as
\begin{align*}
  s(y|x)  = \sum_{\LeafXl\in\PartX} \sum_{\LeafYlk \in \PartYLeafXl} P^2_{\LeafXYlk}(y)
\Charac{(x,y)\in\LeafXYlk}
\end{align*}
As shown by \textcite{willet07:_multis_poiss_inten_densit_estim},
the maximum likelihood estimate in this model can be obtained by an
independent computation on each subset $\LeafXYlk$:
\begin{align*}
  \widehat{P}_{\LeafXYlk} =  \frac{\sum_{i=1}^{n} \Charac{(X_i,Y_i)
      \in \LeafXYlk}}{\sum_{i=1}^{n} \Charac{X_i
      \in \LeafXl}}
\argmin_{P, \deg(P)\leq \dimPYmax, \int_{\LeafYlk} P^2(y) \ud y =1}
\sum_{i=1}^n \Charac{(X_i,Y_i)
      \in \LeafXYlk} \Log\left( P^2(Y_i) \right).
\end{align*}
This property is important to be able to  use  the efficient
optimization algorithms of \textcite{willet07:_multis_poiss_inten_densit_estim} and \textcite{huang06:_fast}.

Our model collection is obtained by considering all partitions $\PartX$
within one of the $\UDPX$, $\RDPX$, $\RDSPX$, $\RSPX$ or
$\HRPX$ partition collections  with respect to $[0,1]^{\dimX}$ and, for a fixed $\PartX$, all partitions
$\PartYLeafXl$ within one of the
$\UDPY$, $\RDPY$, $\RDSPY$, $\RSPY$ or $\HRPY$ partition 
collections with respect to $[0,1]^{\dimY}$.
By construction, in any cases,
\begin{align*}
\dim(\model_{{\PartXY,\dimPYmax}})
& = \sum_{\LeafXl\in\PartX} \left(
\NbPartYLeafXl \prod_{d=1}^{\dimY} (\dimPYmaxi{d}+1)
 - 1\right).
\end{align*}
To define the penalty, we use a slight upper bound of this dimension
\begin{align*}
\DimH_{{\PartXY,\dimPYmax}}  &= \sum_{\LeafXl\in\PartX} \NbPartYLeafXl
\prod_{d=1}^{\dimY} (\dimPYmaxi{d}+1)
= \NbPartXY \prod_{d=1}^{\dimY} (\dimPYmaxi{d}+1)
\end{align*}
where  \begin{alignI}
  \NbPartXY = \sum_{\LeafXl \in \PartX} \NbPartYLeafXl.
\end{alignI} is the total number of hyperrectangles in all the
partitions:
\begin{theorem}
\label{theo:polypart}
Fix a collection $\starX$ among
$\UDPX$, $\RDPX$, $\RDSPX$, $\RSPX$ or
$\HRPX$ for $\SpaceX=[0,1]^{\dimX}$, a collection $\starY$ among
$\UDPY$, $\RDPY$, $\RDSPY$, $\RSPY$ or
$\HRPY$ and a maximal degree for the polynomials $\dimPYmax \in \N^{\dimY}$.

Let
\begin{align*}
  \Models = \left\{ \model_{\PartXY,\dimPYmax} \middle| \PartX=\{\LeafXl\}  \in
  \CollPart^{\starX} \text{ and } \forall \LeafXl
  \in \PartX, \PartYLeafXl \in \CollPart^{\starY}
 \right\}.
\end{align*}

 Then there exist  a
$\constpolyparttwo>0$ and a
$\constpolypartthree>0$ independent of $n$,
 such that for any $\propJKL$ and
for any $\constoraone>1$, the 
penalized estimator of \cref{theo:select} satisfies
\begin{align*}
\E\left[\JKLtens_{\propJKL,\meas}(s_0,\widehat{s}_{\widehat{{\PartXY,\dimPYmax}}})\right]
&\leq \constoraone \inf_{\model_{\PartXY,\dimPYmax}
  \in \Models} \left( \inf_{s_{\PartXY,\dimPYmax} \in \model_{\PartXY,\dimPYmax}} \KLtens_{\meas}(s_0,s_{\PartXY,\dimPYmax}) +
  \frac{\pen({\PartXY,\dimPYmax})}{n} \right) \\
& \qquad\qquad+ \constoratwo
\frac{1}{n} + \frac{\rhoapp+\rhomin}{n}
\end{align*}
as soon as 
\begin{align*}
\pen({\PartXY,\dimPYmax}) \geq \constpensimple
\,\DimH_{\PartXY,\dimPYmax}
\end{align*}
for
\begin{align*}
\constpensimple
> \constpenmin  
\left(\constpolyparttwo 
+  \constpolypartthree \left( \PartCstA^{\starX}
  +\PartCstB^{\starX}  + \PartCstA^{\starY}
    + \PartCstB^{\starY}
\right) 
 + 2 \Log n
\right).
\end{align*}
where $\constpenmin$ and $\constoratwo$ are the constants of \cref{theo:select} that
depend only on $\propJKL$ and $\constoraone$.
Furthermore $\constpolyparttwo \leq \frac{1}{2}\Log(8\pi e) +
\sum_{d=1}^{\dimY} \Log\left(\sqrt2 (\dimPYmaxi{d}+1)\right)$ and
$\constpolypartthree\leq 2 \Log 2$.
\end{theorem}
A penalty chosen proportional to the dimension of
the model, the multiplicative factor $\constpensimple$
being constant over $n$ up to a logarithmic factor, is thus sufficient
to guaranty the estimator performance.
Furthermore, one can use a penalty which
is a sum of penalties for each hyperrectangle of the partition:
\begin{align*}
\pen({\PartXY,\dimPYmax})
= \sum_{\LeafXYlk \in \PartXY} \constpensimple
\left( \prod_{d=1}^{\dimY} (\dimPYmaxi{d}+1) \right).
\end{align*}
This additive structure of the penalty allows to use the fast
partition optimization algorithm of \textcite{donoho97:_cart} and
\textcite{huang06:_fast} as soon as the partition collection is tree structured.

In Appendix, we obtain a weaker requirement on the penalty
\begin{align*}
  \pen({\PartXY,\dimPYmax}) \geq 
\constpen \Bigg( &\left(\constpolyparttwo + 2 \Log
  \frac{n}{\sqrt{\NbPartXY}}\right)
\DimH_{\PartXY,\dimPYmax}
  \\
  &
+ \constpolypartthree \left(
    \PartCstA^{\starX}
      +  \left( \PartCstB^{\starX}  + \PartCstA^{\starY}
      \right)\NbPartX 
      + \PartCstB^{\starY} \sum_{\LeafXl\in\Part}
      \NbPartYLeafXl 
\right) \Bigg)
\end{align*}
in which the complexity part and the coding part appear more
explicitly. This smaller penalty is no longer proportional to the dimension
but still sufficient to guaranty the estimator performance.
Using the crude bound $\NbPartXY\geq 1$, one sees that such a penalty
penalty can still be upper bounded by a sum of penalties over each hyperrectangle.
The loss with respect to the original penalty is of order
$\constpen\log\NbPartXY\DimH_{\PartXY,\dimPYmax}$, which is negligible
as long as the number of hyperrectangle remains small with respect to
$n^2$.

Some variations around this Theorem can  be obtained
through simple modifications of its proof as explained in Appendix.
 For example, the term $2\Log
(n/\sqrt{\NbPartXY})$ disappears if
$\PartX$ belongs to $\CollPart^{\UDPX}$ while $\PartYLeafXl$ is independent of $\LeafXl$ and belongs to
$\CollPart^{\UDPX}$. 
Choosing the degrees $\dimPYmax$ of the polynomial among a family $\dimPYfamily$ either globally or
locally as proposed by
\textcite{willet07:_multis_poiss_inten_densit_estim} is also
possible. The constant $\constpolyparttwo$ is replaced by its maximum
over the family considered, while
the coding part is modified
  by replacing respectively 
$\PartCstA^{\starX}$ by  $\PartCstA^{\starX} + \Log
        |\dimPYfamily|$ for a global optimization and
      $\PartCstB^{\starY}$ by  $\PartCstB^{\starY} +
        \Log |\dimPYfamily| $ a the local optimization. 
Such a penalty can be further modified into an additive one with only
minor loss. Note that even if the family and its maximal
degree grows with $n$, the \emph{constant} $\constpolyparttwo$ grows
at a logarithic rate in $n$ as long as the maximal degree grows at most
polynomially with $n$.

Finally, if we assume that the true conditional density is lower
bounded, then
\begin{align*}
  \KLtens_{\meas}(s,t) \leq \left\|\frac{1}{t}\right\|_{\infty} \|s-t\|_{\meas,2}^{\tens,2}
\end{align*}
as shown by \textcite{kolaczyk05:_multis}.
We can thus reuse ideas from
\textcite{willet07:_multis_poiss_inten_densit_estim},
\textcite{akakpo10:_adapt} or \textcite{akakpo11:_inhom} to infer
the quasi optimal minimaxity of this estimator
for anisotropic Besov spaces (see for instance in
\textcite{karaivanov03:_nonlin_besov} for a definition) whose regularity indices are smaller than
$1$ along the axes of $\SpaceX$ and smaller than $\dimPYmax+1$ along
the axes of $\SpaceY$.

\subsection{Spatial Gaussian mixtures, models, bracketing entropy and penalties}
\label{sec:spat-mixt-models}

In this section, we consider an extension of Gaussian mixture that
takes account into the covariate into the mixing proportion. This model
has been motivated by the unsupervised hyperspectral image
segmentation problem mentioned in the introduction. We recall first
some basic facts about Gaussian mixtures and their uses in
unsupervised classification.

In a classical Gaussian mixture model, the observations are assuming
to be drawn from several different classes, each class having a
Gaussian law.
 Let $K$ be the number of different
Gaussians, often call the number of clusters, the density $s_0$ of $Y_i$
with respect to the Lebesgue measure is thus modeled as
\begin{align*}
s_{K,\theta,\pi}(\cdot) = \sum_{k=1}^K \pi_k  \Gauss_{\theta_k}\left(\cdot \right)
\end{align*}
where
\begin{align*}
\Gauss_{\theta_k}(y) = \frac{1}{\left(2\pi\det \Sigma_k\right)^{p/2}}\,
  e^{-\frac{1}{2} (y-\mu_k)' \Sigma_k^{-1} (y-\mu_k)}
\end{align*}
with $\mu_k$ the mean of the
$k$th component, $\Sigma_k$ its covariance matrix,
$\theta_k=(\mu_k,\Sigma_k)$ and $\pi_k$ its mixing proportion. A model
$\model_{K,\Set}$ is obtained by specifying the number of component
$K$ as well as a set $\Set$ to which should belong the $K$-tuple of
Gaussian $(\Gauss_{\theta_1},\ldots,\Gauss_{\theta_K})$.
Those Gaussians can share for instance the same shape, the same volume
or the same diagonalization basis. The classical choices are described
for instance in \textcite{biernacki06:_model_mixmod}. Using the EM
algorithm, or one of its extension, one can efficiently
 obtain the proportions $\widehat{\pi}_k$ and the Gaussian
 parameters $\widehat{\theta}_k$ of the 
maximum likelihood estimate within such a model. Using tools also
derived from \textcite{massart07:_concen},
\textcite{maugis09:_gauss} show how to choose
the number of classes by a penalized maximum likelihood
principle.
These Gaussian mixture models are often used in unsupervised
classification application: one
observes a collection of $Y_i$ and tries to split them into
homogeneous classes. Those classes are chosen as the Gaussian
components of an estimated Gaussian mixture close to the density of the
observations. Each observation can then be assigned to a class by a
simple maximum likelihood principle:
\begin{align*}
  \widehat{k}(y) = \argmax_{1 \leq k \leq \widehat{K}}
  \widehat{\pi}_k \Gauss_{\widehat{\theta}_k}(y).
\end{align*}
This methodology can be applied directly to an hyperspectral image and
yields a segmentation method, often called spectral method in the
image processing communit. This method however fails to exploit the spatial organization
of the pixels.

To overcome this issue, \textcite{kolaczyk05:_multisimage} and
\textcite{antoniadis08} propose to use mixture model in which the
mixing proportions depend on the covariate $X_i$ while
the mixture components remain constant. We propose to estimate
simultaneously those mixing proportions and the mixture components
with our partition-based strategy. In a semantic analysis context, in which documents replace
pixels, a similar
Gaussian mixture with varying weight, but without the partition structure, has been proposed by
\textcite{si05:_adjus_mixtur_weigh_gauss_mixtur} as an extension of a
general mixture based semantic analysis model introduced by
\textcite{hofmann99:_probab}
under the name \emph{Probabilistic Latent Semantic Analysis}. A
similar model has also been considered in the work of \textcite{youg10:_mixtur}.
In our approach, for a given partition $\PartX$, the
conditional density $s(\cdot|x)$ are modeled as
\begin{align*}
s_{\Part,K,\theta,\pi}(\cdot|x) & = \sum_{\LeafXl\in\PartX} \left( \sum_{k=1}^K
  \pi_k[\LeafXl]  \Gauss_{\theta_k}\left(\cdot \right) \right)
\Charac{x\in\LeafXl} \\
\intertext{which, denoting \begin{alignI}
  \pi[\LeafX(x)] = \sum_{\LeafXl \in \PartX}  \pi[\LeafXl]\,\Charac{x\in\LeafXl}
\end{alignI}, can advantageously be rewritten }
& = \sum_{k=1}^K \pi_k[\LeafX(x)]  \Gauss_{\theta_k}\left(\cdot \right).
\end{align*}
The $K$-tuples of Gaussian can be chosen is the same way as in the
classical Gaussian mixture case. Using a penalized maximum likelihood
strategy, a partition $\widehat{\Part}$, a number of Gaussian components
$\widehat{K}$, their parameters $\widehat{\theta}_k$ and all the mixing
proportions $\widehat{\pi}[\widehat{\LeafXl}]$ can be estimated. Each
pair of pixel position and spectrum $(x,y)$ can then be
assigned to one of the estimated mixture components by a maximum likelihood principle:
\begin{align*}
  \widehat{k}(x,y) = \argmax_{1 \leq k \leq \widehat{K}}
  \widehat{\pi}_k[\widehat{\LeafXl}(x)] \Gauss_{\widehat{\theta}_k}(y).
\end{align*}
This is the strategy we have used at IPANEMA~\cite{Bertrand:kv5096} to
segment, in an unsupervised manner, hyperspectral images. In these
images, a spectrum $Y_i$, with around 1000 frequency bands, is
measured at each pixel location $X_i$ and our aim was to derive a
partition in \emph{homogeneous} regions without any human
intervention. This is a precious help for users of this imaging
technique as this allows to focus the study on a few representative
spectrums. Combining the classical EM strategy for the
Gaussian parameter estimation (see for instance
\textcite{biernacki06:_model_mixmod}) and dynamic programming strategies
for the partition, as described for instance by
\textcite{kolaczyk05:_multisimage}, we have been able to implement this
penalized estimator and to test it on real
datasets.
\ifthenelse{\boolean{extended}}{%
Figure~\ref{fig:glue} illustrates this methodology.
The studied sample is a thin cross-section of maple with a single
layer of hide glue on top of it, prepared recently using materials and
processes from the Cité de la Musique, using materials of the same
type and quality that is used for lutherie. This sample is to serve as
reference material to study the spectral variation of the hide glue at
the various steps of the process. We present here the result for a low
signal to noise ratio acquisition requiring only two minutes of scan.
Using piecewise constant mixing proportions instead of constant
  mixing proportions leads to a better geometry of the segmentation, with
  less isolated points and more structured boundaries. As described in
  a more applied study~\cite{cohen12:_unsup_gauss_report}, this
  methodology permits to work with a much lower signal to noise ratio
  and thus allows to reduce significantly the acquisition time.

\begin{figure}
  \centering

\begin{tabular}{ccc}
\includegraphics[width=.4\linewidth]{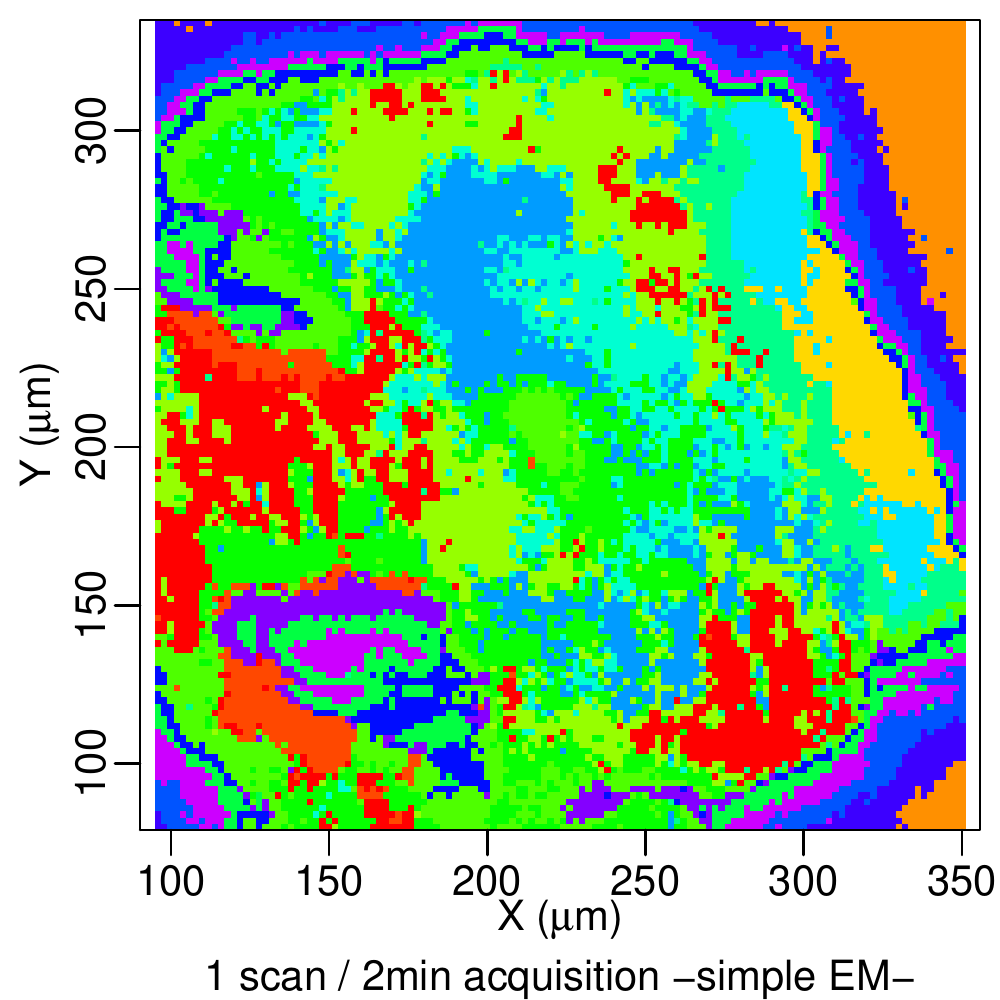}
&
&\includegraphics[width=.4\linewidth]{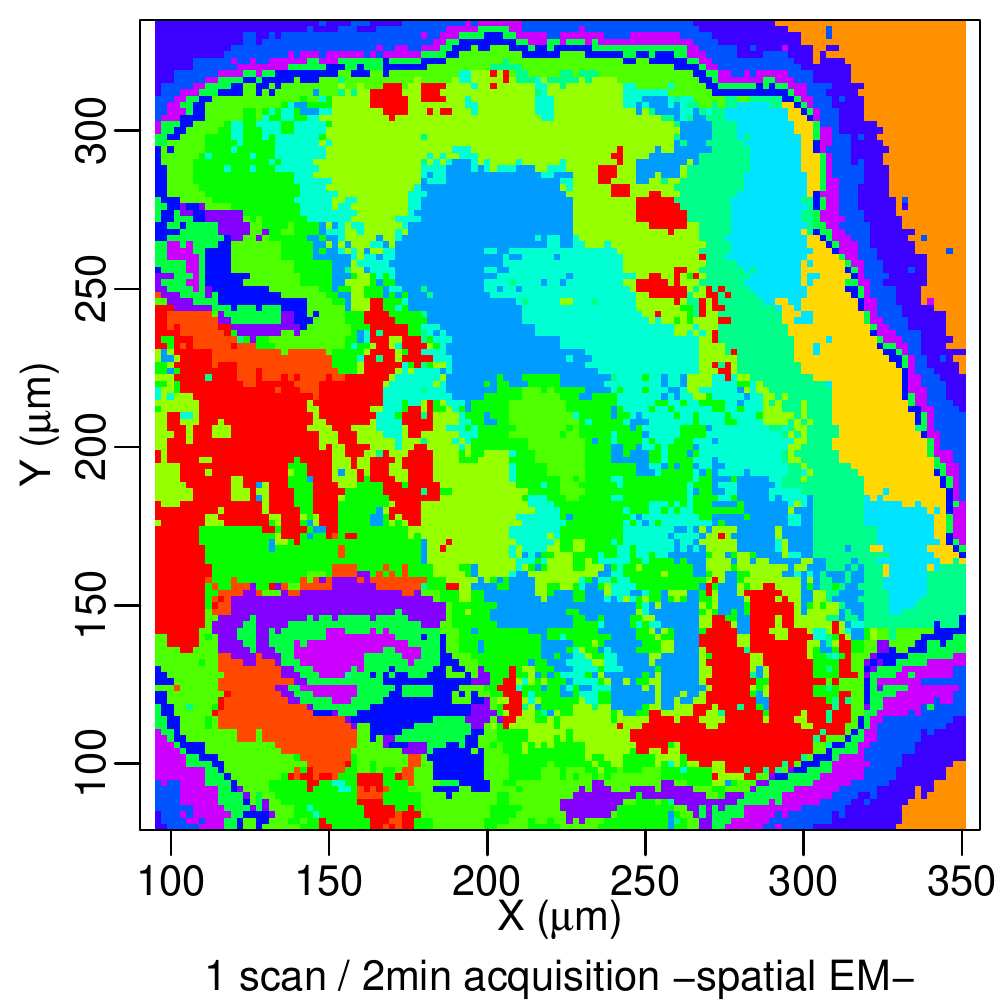}\\
(a) && (b)
\end{tabular}

  \caption{Unsupervised segmentation result: a) with constant mixing
    proportions b) with piecewise constant mixing proportions.}
  \label{fig:glue}
\end{figure}

}{
Figure~\ref{fig:glue} illustrates this methodology.
The studied sample is a thin cross-section of maple with a single
layer of hide glue on top of it, prepared recently using materials and
processes from the Cité de la Musique, using materials of the same
type and quality that is used for lutherie.  We present 
here the result for a low
signal to noise ratio acquisition requiring only two minutes of scan.
Using piecewise constant mixing proportions instead of constant
  mixing proportions leads to a better geometry of the segmentation, with
  less isolated points and more structured boundaries. As described in
  a more applied study~\cite{cohen12:_unsup_gauss_report}, this
  methodology permits to work with a much lower signal to noise ratio
  and thus allows to reduce significantly the acquisition time.

\begin{figure}
  \centering

\begin{tabular}{ccc}
\includegraphics[width=.4\linewidth]{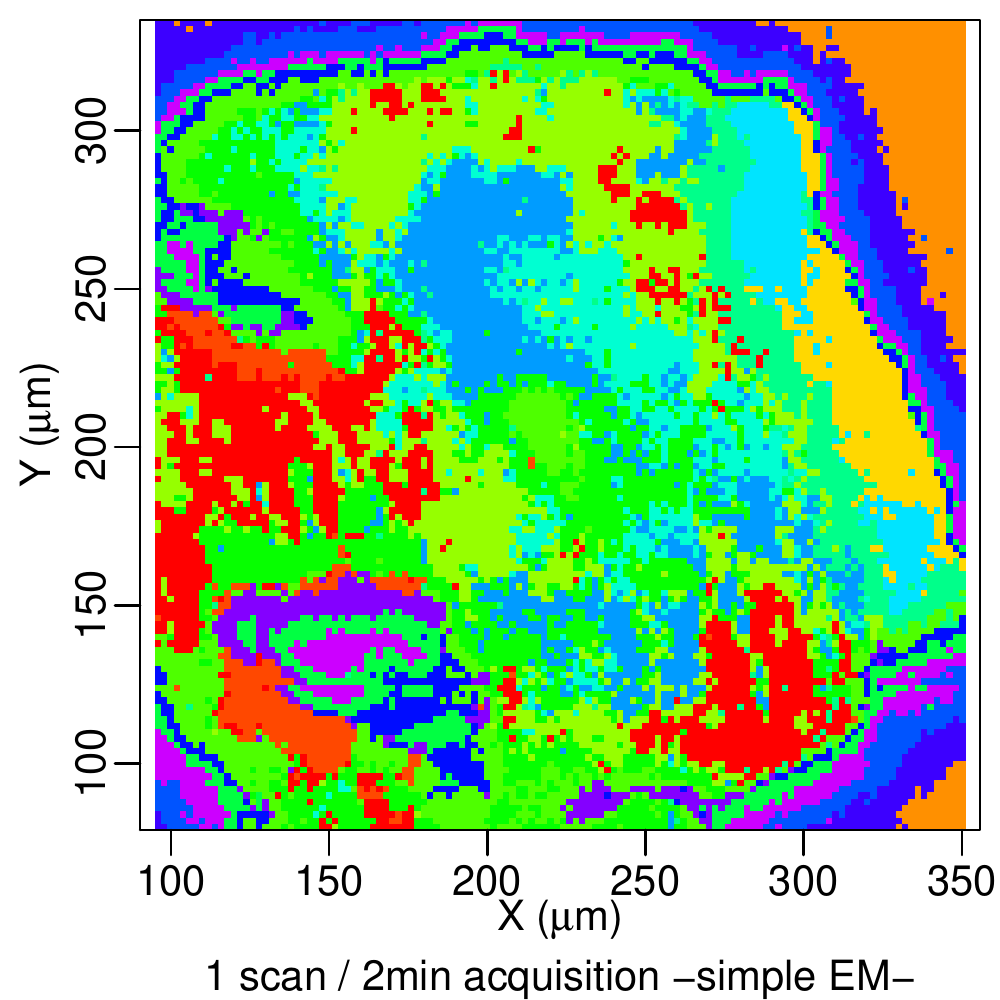}
&
&\includegraphics[width=.4\linewidth]{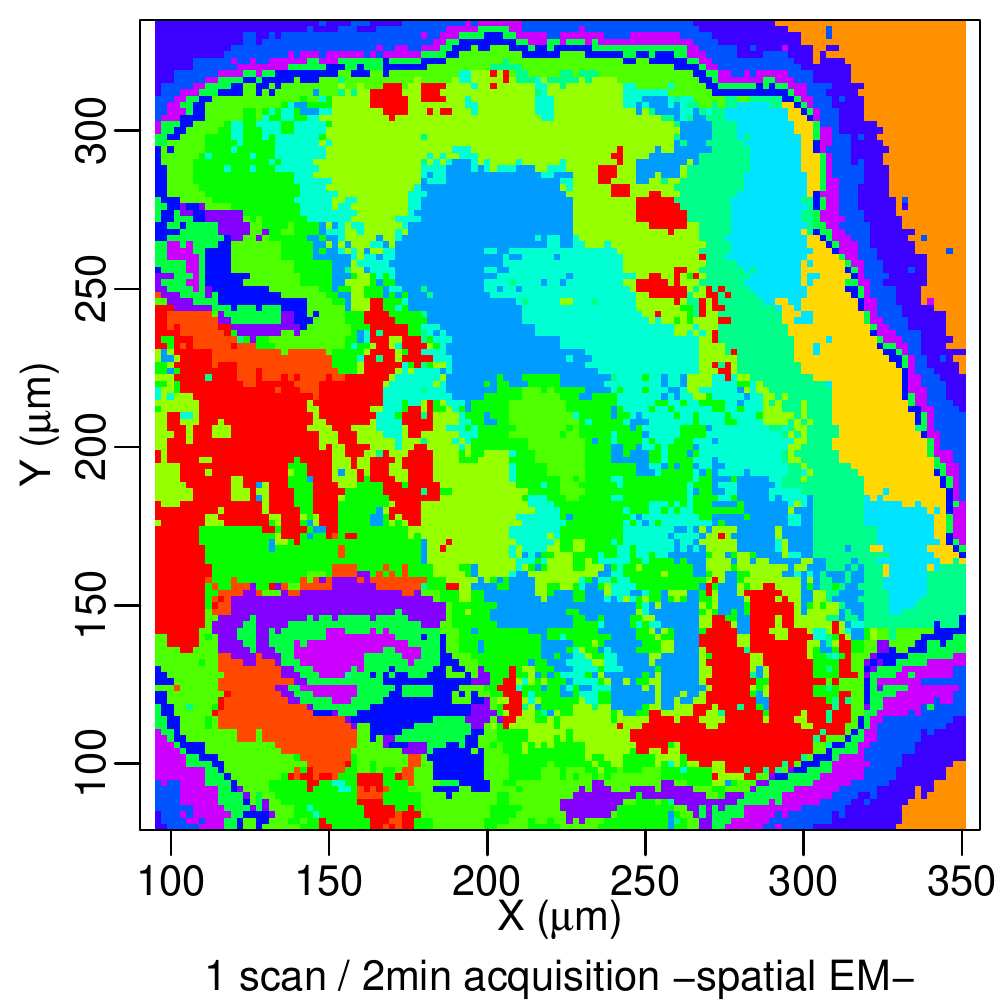}\\
(a) && (b)
\end{tabular}

  \caption{Unsupervised segmentation result: a) with constant mixing
    proportions b) with piecewise constant mixing proportions.}
  \label{fig:glue}
\end{figure}
}

We should now specify the
models we consider. As we follow the construction of
\cref{sec:covar-part-cond}, for a given segmentation $\PartX$, this
amounts to specify the set $\DSet$ to which belong the $\NbPartX$-tuples
of densities $\left( s(y|\LeafXl) \right)_{\LeafXl\in\PartX}$. As
described above, we assume that $s(y|\LeafXl)= \sum_{k=1}^K
  \pi_k[\LeafXl]  \Gauss_{\theta_k}(y)$. The mixing proportions within
  the region $\LeafXl$, $\pi[\LeafXl]$, are chosen freely among
  all vectors of the $K-1$ dimensional simplex $\Simplex_{K-1}$:
\begin{align*}
  \Simplex_{K-1} = \left\{ \pi = (\pi_1, \ldots ,\pi_k) \middle| \forall k, 1
  \leq k \leq K, \pi_k \geq 0, \sum_{k=1}^K \pi_k = 1\right\}.
\end{align*}
 As we assume the
  mixture components are the same in each region, for a given
  number of components $K$, the set $\DSet$ is
 entirely specified by the set $\Set$ of $K$-tuples of Gaussian
  $(\Gauss_{\theta_1},\ldots,\Gauss_{\theta_K})$ (or equivalently by a set
$\Theta$ for $\theta=(\theta_1,\ldots,\theta_K)$).

To allow variable selection, we follow \textcite{maugis09:_gauss}
and let $\Space$ be an arbitrary subspace of $\SpaceY=\R^\dimSp$,
that is expressed differently for the different classes, and let
$\Space^\perp$ be its orthogonal, in which all classes behave
similarly. We assume thus that
\[
\Gauss_{\theta_k}(y)=\Gauss_{\theta_{\Space,k}}(y_{\Space})\Gauss_{\theta_{\Space^\perp}}(y_{\Space^\perp})
\]
where $y_\Space$ and $y_{\Space^\perp}$ denote, respectively, the
projection of $y$ on $\Space$ and $\Space^{\perp}$,
$\Gauss_{\theta_{\Space,k}}$ is a Gaussian whose parameters depend on
$k$ while $\Gauss_{\theta_{\Space^\perp}}$ is independent of $k$.
A model is then specified by the choice of a set $\Set_{\Space}^{K}$
for the $K$-tuples
$(\Gauss_{\theta_{\Space,1}},\ldots,\Gauss_{\theta_{\Space,K}})$ (or
equivalently a set $\Theta_{\Space}^{K}$ for the $K$-tuples of parameters
$(\theta_{\Space,1},\ldots,\theta_{\Space,K})$) and a set
$\Set_{\Space^\perp}$ for the Gaussian
$\Gauss_{\theta_{\Space^\perp}}$ (or equivalently a set
$\Theta_{\Space^\perp}$ for its parameter $\theta_{\Space^\perp}$).
 The resulting model is denoted
$\model_{\PartX,K,\Set}$
\begin{align*}
  \model_{\PartX,K,\Set} =\left\{ s_{\PartX,K,\theta,\pi}(y|x) =
  \sum_{k=1}^K \pi_k[\LeafX(x) ]\,
  \Gauss_{\theta_{\Space,k}}\left(y_{\Space}\right)\,\Gauss_{\theta_{\Space^\perp}}\left(y_{\Space^\perp}\right)
  \middle|
\begin{array}{l}
(\Gauss_{\theta_{\Space,1}},\ldots,\Gauss_{\theta_{\Space,K}}) \in \Set_{\Space}^{K},
\\ 
\Gauss_{\theta_{\Space^\perp}} \in \Set_{\Space^\perp},
\\
\forall \LeafXl \in \PartX, \pi[\LeafXl] \in
\Simplex_{K-1}
\end{array}
\right\}.
\end{align*}

The sets $\Set_{\Space}^{K}$ and $\Set_{\Space^\perp}$ are chosen
among the \emph{classical} Gaussian $K$-tuples, as described for
instance in
\textcite{biernacki06:_model_mixmod}. For a space $\Space$ of
dimension $\dimSpace$  and a fixed number $K$ of classes, we specify
the set
\begin{align*}
\Set = \left\{ \left(
    \Gauss_{\Space,\theta_1},\ldots,\Gauss_{\Space,\theta_K} \right) \middle|
  \theta=(\theta_1,\ldots,\theta_K)\in \Theta_{[\cdot]^K_{\dimSpace}} \right\}
\end{align*}
through a parameter set $\Theta_{[\cdot]^K_{\dimSpace}}$ defined by  some (mild) constraints on the means $\mu_k$ and some
(strong) constraints on the covariance matrices $\Sigma_k$. 

The $K$-tuple  of means $\mu=(\mu_1,\ldots,\mu_K)$ is either known or unknown without
any restriction. 
A stronger structure is imposed on the $K$-tuple of covariance
matrices $(\Sigma_1,\ldots,\Sigma_K)$. To define it, we need to
introduce a decomposition of any covariance matrix $\Sigma$
into $L D A D'$ where, denoting $|\Sigma|$ the determinant of
$\Sigma$,
 $L=|\Sigma|^{1/\dimSpace}$ is a positive scalar corresponding to the
volume, $D$ is the matrix of eigenvectors of $\Sigma$ and $A$ the
diagonal matrix of renormalized eigenvalues of $\Sigma$ (the
eigenvalues of $|\Sigma|^{-1/\dimSpace}\Sigma$). Note that this
decomposition is not unique as, for example, $D$ and $A$ are defined
up to a permutation. We impose nevertheless  a structure on the
$K$-tuple
$(\Sigma_1,\ldots,\Sigma_K)$
through structures on the corresponding $K$-tuples of $(L_1,\ldots,L_K)$,
$(D_1,\ldots,D_K)$ and $(A_1,\ldots,A_K)$. They are either known,
unknown but with a common value or unknown without any restriction.
The corresponding set is indexed by
$[\mumodany \, \Lmodany\,  \Dmodany\,
\Amodany]^K_{\dimSpace}$ where $\modany=\modknown$ means that the quantity is known,
$\modany=\modfree$ that the quantity is unknown without any
restriction
and possibly different for every class and its lack means
that there is a common unknown value over all classes.

To have a set with finite bracketing entropy, we
further restrict the values of the means $\mu_k$, the volumes $L_k$
and the renormalized eigenvalue matrix $A_k$.
The means are  assumed to satisfy
\begin{alignI}
    \forall 1\leq k \leq K,  |\mu_k| \leq a
\end{alignI}
for a known $a$ while the volumes satisfy
\begin{alignI}
  \forall 1\leq k \leq K, \Lm \leq L_k \leq \LM 
\end{alignI}
for some known positive values $\Lm$ and $\LM$. To describe the
constraints on the renormalized eigenvalue matrix $A_k$, we define the set
$\Diag(\lambdam,\lambdaM,\dimSpace)$ of diagonal matrices $A$ such that
$|A|=1$ and $\forall 1 \leq i \leq
  \dimSpace, \lambdam \leq A_{i,i} \leq \lambdaM$. Our assumption
  is that all the $A_k$ belong to
  $\Diag(\lambdam,\lambdaM,\dimSpace)$ for some known values
  $\lambdam$ and $\lambdaM$.

Among the $3^4=81$ such possible sets, six of them have been already studied by
\textcite{maugis09:_gauss,maugis11:_data} in their classical Gaussian mixture model
analysis:
\ifthenelse{\boolean{extended}}{ 
\begin{itemize}
\item $[\mumodknown\,\Lmodfree\,\Dmodknown\,\Amodknown]^K_{\dimSpace}$ in which only the volume of the
  variance of a class
  is unknown. They use this model with a single class to model the non
  discriminant variables in $\Space^\perp$.
\item $[\mumodfree\,\Lmodfree\,\Dmodknown\, \Amodfree]^K_{\dimSpace}$ in which one assumes that the unknown
  variances $\Sigma_k$ can be diagonalized in the same known basis
  $\Dmodknown$.
\item $[\mumodfree\,\Lmodfree\,
\Dmodfree\, \Amodfree]^K_{\dimSpace}$ in which everything is free,
\item $[\mumodfree\,\Lmodsame\,
\Dmodknown\, \Amodsame]^K_{\dimSpace}$ in which the variances $\Sigma_k$ are assumed to be equal
and diagonalized in the known basis $\Dmodknown$.
\item $[\mumodfree\,\Lmodsame\,
\Dmodknown\, \Amodfree]^K_{\dimSpace}$ in which the volumes $\L_k$ are assumed to be equal
and the variance can be diagonalized in the known basis $\Dmodknown$
\item $[\mumodfree\,\Lmodsame\,
\Dmodsame\, \Amodsame]^K_{\dimSpace}$ in which the variances
$\Sigma_k$ are only assumed to be equal
\end{itemize}}{$[\mumodknown\,\Lmodfree\,\Dmodknown\,\Amodknown]^K_{\dimSpace}$,
$[\mumodfree\,\Lmodfree\,\Dmodknown\, \Amodfree]^K_{\dimSpace}$, $[\mumodfree\,\Lmodfree\,
\Dmodfree\, \Amodfree]^K_{\dimSpace}$, $[\mumodfree\,\Lmodsame\,
\Dmodknown\, \Amodfree]^K_{\dimSpace}$, $[\mumodfree\,\Lmodsame\,
\Dmodknown\, \Amodsame]^K_{\dimSpace}$ and $[\mumodfree\,\Lmodsame\,
\Dmodsame\, \Amodsame]^K_{\dimSpace}$.}
All these cases, as well as the others, are covered by our analysis
with a single proof.

To summarize, our models $\model_{\PartX,K,\Set}$ are parametrized by
a partition $\PartX$, a number of components $K$, a set $\Set$ of
$K$-tuples of Gaussian specified by a space $\Space$ and two parameter
sets, a set $\Theta_{[\mumodany\,\Lmodany\,\Dmodany\,\Amodany
  ]^K_{\dimSpace}}$ of $K$-tuples of Gaussian parameters for the differentiated
  space $\Space$ and a set
$\Theta_{[\mumodany\,\Lmodany\,\Dmodany\,\Amodany
  ]_{\dimSpaceperp}}$ of Gaussian parameters for its orthogonal
  $\Space^\perp$. Those two sets are chosen among the ones described
  above with the same constants $a$, $\Lm$, $\LM$,
$\lambdam$ and $\lambdaM$. One verifies that
\begin{align*}
  \dim(\model_{\PartX,K,\Set}) =
  \NbPartX(K-1)+\dim\left(\Theta_{[\mumodany\,\Lmodany\,\Dmodany\,\Amodany
    ]^K_{\dimSpace}}\right) + \dim\left(\Theta_{[\mumodany\,\Lmodany\,\Dmodany\,\Amodany
    ]_{\dimSpaceperp}}\right).
\end{align*}

Before stating a model selection theorem, we should specify the
collections $\Models$ considered. We consider sets of model
$\model_{\PartX,K,\Set}$ with $\PartX$ chosen among one of the
partition collections $\CollPart^\star$, $K$ smaller than $K_M$, which
can be theoretically chosen equal to $+\infty$, a space $\Space$ chosen as
$\Span \{ e_i\}_{i\in I}$ where $e_i$ is the canonical
basis of $\R^\dimSp$
and $I$ a subset of $\{1,\ldots,\dimSp\}$ is either known, equal to
$\{1,\ldots,\dimSpace\}$ or free and the indices 
$[\mumodany\,\Lmodany\,\Dmodany\,\Amodany]$ of $\Theta_\Space$ and
$\Theta_{\Space^\perp}$ are chosen freely among a subset of the
possible combinations.

Without any assumptions on the design, we obtain
\begin{theorem}
\label{theo:spatgauss}
Assume the collection $\Models$ is one of the collections of the
previous paragraph.

Then, there exist a $\constspatgausslog>\pi$ and
a $\constspatgausstwo>0$, such that,  for any $\propJKL$ and
for any $\constoraone>1$,
the penalized estimator of \cref{theo:select} satisfies
\begin{align*}
\E\left[\JKLtens_{\propJKL,\meas}(s_0,\widehat{s}_{\widehat{{\PartX,K,\Set}}})\right]
\leq \constoraone \inf_{\model_{\PartX,K,\Set} \in\Models} \left( \inf_{s_{\PartX,K,\Set} \in \model_{\PartX,K,\Set}} \KLtens_{\meas}(s_0,s_{\PartX,K,\Set}) +
  \frac{\pen({\PartX,K,\Set})}{n} \right) +  \frac{\constoratwo}{n} + \frac{\rhoapp+\rhomin}{n}
\end{align*}
as soon as 
\begin{align*}
\pen({\PartX,K,\Set}) \geq \constpensimple_1
\dim(\model_{\PartX,K,\Set}) + \constpensimple_2 \CstSpace
\end{align*}
for
\begin{align*}
\constpensimple_1
 &\geq 
 \constpen  
\left(\left( 2\constspatgausslog+ 1 + \left( \Log
     \frac{n}{e\constspatgausslog}\right)_+ 
+
\constspatgausstwo \left(
     \PartCstA^{\starX}
       +  \PartCstB^{\starX}
  + 1
 \right)
\right) \right)
&\text{and}&&\constpensimple_2 \geq \constpen \constspatgausstwo
\end{align*}
with $\constpen>\constpenmin$ where  $\constpenmin$ and $\constoratwo$
are the constants of \cref{theo:select} that
depend only on $\propJKL$ and $\constoraone$ and
\begin{align*}
\CstSpace = 
\begin{cases}
0  & \text{if $\Space$ is known,}\\
 \dimSpace & \text{\parbox{6cm}{if
  $\Space$ is chosen among spaces spanned by the first
  coordinates,}}\\
(1+ \Log 2 + \Log \frac{\dimSp}{\dimSpace}) \dimSpace & \text{if
  $\Space$ is free.}
\end{cases}  
\end{align*}
\end{theorem}

As in the previous section, the penalty term can thus be chosen, up to the variable selection term $\CstSpace$, proportional to
the dimension of the model, with a proportionality factor constant up to a
logarithmic term with $n$. 
A penalty proportional to the dimension of the
model is thus sufficient to ensure that the model selected
performs almost as well as the best possible model in term of
conditional density estimation.
 As in the proof of
\textcite{antoniadis08}, we can also obtain that our proposed estimator yields
a minimax estimate for spatial Gaussian mixture with mixture proportions
having a geometrical regularity even without knowing the number of
classes.

Moreover, again as in the previous section,
the penalty can have an additive structure, it can be chosen as a sum
of penalties over each hyperrectangle plus one corresponding
to $K$ and the set $\Set$. Indeed
\begin{align*}
\pen({\PartX,K,\Set}) = \sum_{\LeafXl \in \PartX} \constpensimple_1
(K-1) + \constpensimple_1 \left( \dim\left(\Theta_{[\mumodany\,\Lmodany\,\Dmodany\,\Amodany
    ]^K_{\dimSpace}}\right) + \dim\left(\Theta_{[\mumodany\,\Lmodany\,\Dmodany\,\Amodany
    ]_{\dimSpaceperp}}\right) \right)
+\constpensimple_2 \CstSpace
\end{align*}
satisfies the requirement of~\cref{theo:spatgauss}. This structure is
the key for our numerical minimization algorithm in which one optimizes
alternately the Gaussian parameters with an EM algorithm and the
partition with the same fast optimization strategy as in the previous section.

In Appendix, we obtain a weaker requirement
\begin{align*}
  \pen({\PartX,K,\Set}) &\geq 
 \constpen  \Bigg(\left( 2 \constspatgausslog+ 1 +\left( \Log
     \frac{n}{e\constspatgausslog\dim(\model_{\PartX,K,\Set})}\right)_+ \right)
 \dim(\model_{\PartX,K,\Set})
 \\
  &\qquad \qquad + \constspatgausstwo \left(
     \PartCstA^{\starX}
       +  \PartCstB^{\starX} \NbPartX 
  + (K-1) + \CstSpace
 \right) \Bigg)
\end{align*}
in which the complexity and the coding terms are more explicit. Again
up to a logarithmic term in $\dim(\model_{\PartX,K,\Set})$, this
requirement can be satisfied by a penalty having the same additive
structure as in the previous paragraph.

Our theoretical result on the conditional density estimation does not
guaranty good segmentation performance. If data are
generated according to a Gaussian mixture with varying mixing proportions, one
could nevertheless obtain the asymptotic convergence of our class estimator
to the optimal Bayes one. We have nevertheless observed in our
numerical experiments at IPANEMA that the proposed methodology allow to reduce
the signal to noise ratio while keeping meaningful segmentations. 

Two major questions remain nevertheless open. Can we calibrate
the penalty (choosing the constants) in a datadriven way while
guaranteeing the theoretical performance in this specific setting? 
Can we derive a non asymptotic
classification result from this conditional density result? 
The \emph{slope heuristic}, proposed by \textcite{birge07:_minim_gauss},
we have used in our numerical experiments, seems a promising
direction. Deriving a theoretical justification in this conditional estimation
setting would be much better. Linking the non asymptotic estimation
behavior to a non asymptotic classification behavior appears even more challenging.

\ifthenelse{\boolean{extended}}{
\subsection{Bracketing entropy of Gaussian families}

A key ingredient in the proof of \ref{theo:spatgauss} is a generalization of a result of \textcite{maugis09:_gauss,maugis10:_errat_ae_gauss} controlling the bracketing
entropy the Gaussian families $\Set_{[\cdot]^K_\Space}$ with respect
to the $\dmax_{\meas}$ distance defined by
\begin{align*}
  \dtwomax_{\meas}\left(\left(s_1,\ldots,s_K\right),\left(t_1,\ldots,t_K\right)\right)
= \sup_{1 \leq k \leq K} d^2(s_k,t_k).
\end{align*}
Here, $[(t_1^-,\ldots,t_K^-),(t_1^+,\ldots,t_K^+)]$ is a bracket
containing
 $(s_1,\ldots,s_K)$ if 
 \begin{align*}
      \forall 1 \leq k \leq K, \forall y \in \Space,\quad t_k^-(y)
      \leq  s_k(y) \leq t_k^+(y).
 \end{align*}
 As it can be
of interest on its own, we state it here: 
\begin{proposition}
\label{prop:entcoll}
For any $\delta \in (0, \sqrt{2}]$, 
  \begin{align*}
    H_{[\cdot],\dmax_{\meas}}(\delta/9, \Set_{[\mumodany,\Lmodany,\Dmodany,\Amodany]^K_{\Space}})
 \leq  \ConstH_{[\mumodany,\Lmodany,\Dmodany,\Amodany]^K_{\dimSpace}}  + \DimH_{[\mumodany,\Lmodany,\Dmodany,\Amodany]^K_{\dimSpace}}
 \Log \frac{1}{\delta}
\end{align*}
where 
\begin{alignI}
  \DimH_{[\mumodany,\Lmodany,\Dmodany,\Amodany]^K_{\dimSpace}}
= \dim\left(\Theta_{[\mumodany,\Lmodany,\Dmodany,\Amodany]^K_{\dimSpace}}\right) = c_{\mumodany} \DimH_{\mu,\dimSpace} + c_{\Lmodany} \DimH_{L}+
 c_{\Dmodany} \DimH_{D,\dimSpace} + c_{\Amodany} \DimH_{A,\dimSpace}
\end{alignI}
and
\begin{alignI}
\ConstH_{[\mumodany,\Lmodany,\Dmodany,\Amodany]^K_{\dimSpace}} 
=
c_{\mumodany} \ConstH_{\mumod,\dimSpace}  + c_{\Lmodany} \ConstH_{L,\dimSpace} +
 c_{\Dmodany} \ConstH_{D,\dimSpace}  + c_{\Amodany} \ConstH_{A,\dimSpace} 
\end{alignI}
with
\begin{alignI}
\begin{cases}
  c_{\mumodknown}=c_{\Lmodknown}=c_{\Dmodknown}=c_{\Amodknown}=0\\
  c_{\mumodfree}=c_{\Lmodfree}=c_{\Dmodfree}=c_{\Amodfree}=K\\
  c_{\mumodsame}=c_{\Lmodsame}=c_{\Dmodsame}=c_{\Amodsame}=1
\end{cases},
\end{alignI}
\begin{align*}
\begin{cases}
  \DimH_{\mumod,\dimSpace} = \dimSpace \\
\DimH_{\Lmod}=1\\
\DimH_{\Dmod,\dimSpace}=\frac{\dimSpace(\dimSpace-1)}{2}\\
\DimH_{\Amod,\dimSpace}  = \dimSpace - 1
\end{cases}
\quad\text{and}\quad
\begin{cases}
\ConstH_{\mumod,\dimSpace}=   \dimSpace \left( \Log\left( 1 + 
108\frac{a}{\sqrt{\Lm \lambdam
   \frac{\lambdam}{\lambdaM}}}\dimSpace
\right)  \right) \\
\ConstH_{\Lmod,\dimSpace}= \Log\left( 1 + 39 \Log\left(\frac{\LM}{\Lm}\right)\dimSpace\right) \\
 \ConstH_{\Dmod,\dimSpace}= \frac{\dimSpace(\dimSpace-1)}{2}
\left( \frac{2 \Log \ConstSzarek}{\dimSpace(\dimSpace-1)} +
   \left( 
\Log\left( 
252 \frac{\lambdaM}{\lambdam}\dimSpace\right)
\right) \right)\\
 \ConstH_{\Amod,\dimSpace} =(\dimSpace-1) \left( \Log \left(
     2 + 255 \frac{\lambda_M}{\lambdam}\Log\left(\frac{\lambdaM}{\lambdam}\right)\dimSpace\right) \right)
\end{cases}
\end{align*}
where $\ConstSzarek$ is an universal constant.
\end{proposition}
}{}


\appendix

\ifthenelse{\boolean{extended}}{
\section{Proofs for \cref{sec:single-model-maximum-1}
  (\nameref{sec:single-model-maximum-1})}}{}

\ifthenelse{\boolean{extended}}{
\subsection{Proof of \cref{prop:hellingerkull2}}

\begin{proof}[Proof of \cref{prop:hellingerkull2}]
We first notice that, by convexity of the Kullback-Leibler divergence,
\begin{align*}
  \JKL_{\propJKL,\meas}(s,t) =  \frac{1}{\propJKL}\KL_{\meas}\left( s,
    (1-\propJKL) s+ \propJKL t \right) \leq \frac{1}{\propJKL} \left(
    (1-\propJKL) \KL_{\meas}(s,s) + \propJKL \KL(s,t) \right) = \KL_{\meas}(s,t).
\end{align*}
Then let $\ud\meas'= ( (1-\propJKL) s+ \propJKL t ) \ud\meas$, the function 
\begin{alignI}
  u = \frac{s-t}{(1-\propJKL) s + \propJKL t}
\end{alignI}
remains in $[-1/\propJKL,1/(1-\propJKL)]$, and is 
such that
  \begin{alignI}
    \frac{s\ud\meas}{\ud\meas'} = 1 + \propJKL u \quad\text{and}\quad \frac{t\ud\meas}{\ud\meas'} = 1 -
    (1-\propJKL) u.
  \end{alignI}

\begin{align*}
\text{Now,}  &&\JKL_{\propJKL}(s\ud\meas,t\ud\meas)&=\frac{1}{\propJKL} \KL(s\ud\meas, (1-\propJKL) s +
  \propJKL t \ud\meas)
= \frac{1}{\propJKL}\KL((1+ \propJKL u)\ud\meas',\ud\meas')\\ 
&&& = \frac{1}{\propJKL} \KL_{\meas'}(1+ \propJKL u,1)
= \frac{1}{\propJKL} \int (1+ \propJKL u) \Log (1+ \propJKL u) \ud\meas'\\
&\text{and as $\int u \ud\meas'=0$}
&& = \frac{1}{\propJKL} \int \left( (1+\propJKL u) \Log (1+\propJKL u) - \propJKL u \right) \ud\meas.
\end{align*}
\begin{align*}
\text{Similarly,}&&
  \d^2(s\ud\meas,t\ud\meas) &= \d^2((1+\propJKL u)\ud\meas',(1-
  (1-\propJKL) u)\ud\meas ')
= \d^2_{\meas'}(1+\propJKL u,1- (1-\propJKL) u)\\
&&
& = 2 - 2 \int \sqrt{1+ \propJKL u}\sqrt{1- (1-\propJKL) u} \ud\meas'
 = 2 \int \left(1 - \sqrt{1+(2\propJKL-1)u-\propJKL(1-\propJKL)u^2} \right) \ud\meas'\\
&&
& = 2 \int \left(1 - \sqrt{1+(2\propJKL-1)u-\propJKL(1-\propJKL)u^2} + (\propJKL -
  \frac{1}{2}) u \right) \ud\meas'
\end{align*}
Now let $\Phi(x)=(1+ x)\Log(1+ x)- x$, one can verify that
$\Phi(x)/x^2$ is non increasing on $[-1,+\infty]$, so that
$\forall u \in [-1/\propJKL,1/(1-\propJKL)]$, $\Phi(\propJKL u)=
\frac{\Phi(\propJKL u)}{\propJKL^2 u^2} \propJKL^2 u^2 \geq
\frac{\Phi(\frac{\propJKL}{1-\propJKL})}{\propJKL^2/(1-\propJKL)^2}
\propJKL^2 u^2$
so that
\begin{align*}
 (1+\propJKL u) \Log (1+\propJKL u) -
\propJKL u &\geq
\left( (1+ \frac{\propJKL}{1-\propJKL}) \Log \left(1 + \frac{\propJKL}{1-\propJKL}\right) -
\frac{\propJKL}{1-\propJKL} \right)  (1-\propJKL)^2 u^2\\
& \geq (1-\propJKL) \left( \Log \left(1 + \frac{\propJKL}{1-\propJKL}\right) -
  \propJKL\right) u^2
\end{align*}
Along the same lines, one can verify that $\forall u \in
[-1/\propJKL,1/(1-\propJKL)]$
\begin{align*}
  1 - \sqrt{1+(2\propJKL-1)u-\propJKL(1-\propJKL)u^2} + (\propJKL -
  \frac{1}{2}) u  \leq \frac{\max(\propJKL,1-\propJKL)}{2}  u^2.
\end{align*}
This implies
thus 
\begin{align*}
 \frac{1}{\propJKL} &\left( (1+\propJKL u) \Log (1+\propJKL u) -
\propJKL u \right) \\
&\geq \frac{1}{\propJKL}\frac{1}{\max(\propJKL,1-\propJKL)}
(1-p) \left( \Log \left(1 + \frac{\propJKL}{1-\propJKL}\right) -
  \propJKL\right)
2\left( 1 - \sqrt{1+(2\propJKL-1)u-\propJKL(1-\propJKL)u^2} + (\propJKL -
  \frac{1}{2}) u \right)\\
& \geq  \frac{1}{\propJKL} \min(\frac{1-\propJKL}{\propJKL},1)
\left( \Log \left(1 + \frac{\propJKL}{1-\propJKL}\right) -
  \propJKL\right)
2\left( 1 - \sqrt{1+(2\propJKL-1)u-\propJKL(1-\propJKL)u^2} + (\propJKL -
  \frac{1}{2}) u \right)
\end{align*}
which
yields the first inequality.

Recall now that $\KL(s\ud\meas,t\ud\meas) \geq \frac{1}{2}
\|s-t\|_{\meas,1}^2$ so that
\begin{align*}
JKL_{\propJKL}(s\ud\meas,t\ud\meas)&=\frac{1}{\propJKL} \KL(s\ud\meas, (1-\propJKL) s +
  \propJKL t \ud\meas)\\
&\geq \frac{1}{2\propJKL} \|s-\left( (1-\propJKL)s + \propJKL
  t\right)\|_{\meas,1}^2\\
& \geq \frac{1}{2\propJKL} \|\propJKL (s-t)\|_{\meas,1}^2\\
& \geq \frac{\propJKL}{2} \|s-t\|_{\meas,1}^2.  
\end{align*}
Combining this result with $\d^2_{\meas}(s,t)\geq \frac{1}{4}
\|s-t\|_{\meas,1}^2$ allows to conclude.

For the third series of inequalities, 
\begin{align*}
  \d^2(s\ud\meas,t\ud\meas) &= d^2_{t\ud\meas}(\frac{s}{t},1)
 = \int \left( \sqrt{\frac{s}{t}} - 1 \right)^2 t\ud\meas,
\end{align*}
while
\begin{align*}
  \KL(s\ud\meas,t\ud\meas) 
= \KL_{t\ud\meas}(\frac{s}{t},1)
= \int \frac{s}{t} \Log \frac{s}{t}  t\ud\meas
 = \int \left( \frac{s}{t} \Log \frac{s}{t} - \frac{s}{t} +1
\right) t\ud\meas.
\end{align*}
It turns out that $\forall x \in [0,M], $
\begin{align*}
(\sqrt{x}-1)^2 \leq x \Log x - x +1 \leq (2 + (\Log M )_+ )  (\sqrt{x}-1)^2 
\end{align*}
which yields the result.

For the last bound, we use an idea of \textcite{kolaczyk05:_multis}:
\begin{align*}
  \KL(s\ud\meas,t\ud\meas) 
& 
= \int s \Log\left( \frac{s}{t} \right) d\meas\\
& = \KL(s\ud\meas,t\ud\meas) 
& 
= \int \left(t -s +s \Log\left( \frac{s}{t} \right)\right) d\meas\\
\intertext{and as $\log x \leq x -1$}\\
& 
\leq \int \left(t -s +s \left( \frac{s}{t} - 1\right)\right) d\meas\\
\intertext{and assuming that $t$ does not vanish}
& \leq \int \frac{1}{t}\left(t^2 - 2st +s^2)\right) d\meas\\
& \leq \left\|\frac{1}{t}\right\|_{\infty} \|t-s\|_{\meas,2}^2
\end{align*}

\end{proof}

\subsection{Proof of \cref{prop:proporsimple}}

For sake of simplicity, we remove from now on the subscript
reference to the common measure $\meas$ from all notations.

\Cref{prop:proporsimple} is split into three propositions:
\cref{prop:complexdimzero} handles the cases of bracketing dimension $0$, \cref{prop:complexnonlocal} applies when
one control the bracketing entropy of the models $\model_{\indm}$ while \ref{prop:complexlocal}
applies using bounds on
the bracketing entropy of 
the local models $\model_{\indm}(s_{\indm},\sigma)$.
Recall that Assumption \Hindm is the existence of a non-decreasing
function $\phi_{\indm}$ such that $\delta \mapsto \frac{1}{\delta}
\phi_{\indm}(\delta)$ is non-increasing and
\begin{alignI}
 \int_0^{\sigma} \sqrt{
    H_{[\cdot],\dtens}(\delta,\model_{\indm}(s_{\indm},\sigma))}
  \, \ud\delta \leq \phi_{\indm}(\sigma).
\end{alignI}
The complexity term $\DIMH_{\indm}$ is then defined by
$n\sigma_{\indm}^2$ where $\sigma_{\indm}$ is the unique root of 
\begin{alignI}
    \frac{1}{\sigma}   \phi_{\indm}(\sigma) =
    \sqrt{n}  \sigma. 
  \end{alignI}

For the case of bracketing dimension 0, it suffices to show that the
property holds for the local models as $
H_{[\cdot],\dtens}(\delta,\model_{\indm}(s_{\indm},\sigma)) \leq   H_{[\cdot],\dtens}(\delta,\model_{\indm})$.
\begin{proposition}
\label{prop:complexdimzero}
Assume for any $\sigma \in (0,\sqrt{2}]$ and any $\delta\in(0,\sigma]$
\begin{align*}
  H_{[\cdot],\dtens}(\delta,\model_{\indm}(s_{\indm},\sigma)) \leq
\ConstH_{\indm}
\end{align*}
then the function
\begin{align*}
  \phi_{\indm}(\sigma) = \sigma \sqrt{\ConstH_{\indm}}
\end{align*}
 satisfies the properties required in Assumption \Hindm.

Furthermore, $\DIMH_{\indm}= \ConstH_{\indm}$.
\end{proposition}
\begin{proof}
One check easily that $\phi_{\indm}$ is non-decreasing while
$\delta\mapsto\frac{1}{\delta}
\phi_{\indm}(\delta)=\sqrt{\ConstH_{\indm}}$ is constant and thus
non-increasing.

Using the assumption on the entropy, 
\begin{align*}
 \int_0^{\sigma}
 \sqrt{H_{[\cdot],\dtens}(\delta,\model_{\indm}(s_{\indm},\sigma))} \, \ud \delta
& =  \int_0^{\sigma}
 \sqrt{H_{[\cdot],\dtens}(\delta\wedge
   \sqrt{2},\model_{\indm}(s_{\indm},\sigma\wedge \sqrt{2}))} \, \ud \delta
\\
& \leq \int_0^{\sigma}
 \sqrt{\ConstH_{\indm}} \, \ud \delta \\
& \leq \sigma \sqrt{\ConstH_{\indm}} = \phi_{\indm}(\sigma).
\end{align*}

Finally, the unique root of \begin{alignI}
    \frac{1}{\sigma}   \phi_{\indm}(\sigma) =
    \sqrt{n}  \sigma
  \end{alignI} is $\sigma_{\indm} = \sqrt{\frac{\ConstH_{\indm}}{n}}$
  which implies $\DIMH_{\indm}=n\sigma^2_{\indm}=\ConstH_{\indm}$.
\end{proof}

If one is only able to bound  the bracketing entropy of the
global model, one has:
\begin{proposition}
\label{prop:complexnonlocal}
Assume for any $\delta \in (0, \sqrt{2}]$,
  \begin{align*}
    H_{[\cdot],\dtens}(\delta,\model_{\indm}) \leq 
\DimH_{\indm} \left( \ConstMultH_{\indm} + \Log \frac{1}{\delta}\right).
\end{align*} 

Then the function
\begin{align*}
  \phi_{\indm}(\sigma) =
\sigma \sqrt{\DimH_{\indm}}  \left( \sqrt{\ConstMultH_{\indm}} +
\sqrt{\pi} + \sqrt{\Log
    \frac{1}{\sigma\wedge e^{-1/2}}}  \right).
\end{align*}
 satisfies the properties required in Assumption \Hindm.

Furthermore, $\DIMH$
satisfies
\begin{align*}
\DIMH_{\indm} & \leq \left( 2 \left(\sqrt{\ConstMultH_{\indm}} + \sqrt{\pi} \right)^2 
+ 1 + \left(\Log
    \frac{n}{e \left( \sqrt{\ConstMultH_{\indm}} +
 \sqrt{\pi} \right)^2 \DimH_{\indm}}\right)_+ \right)
\DimH_{\indm}
\end{align*}
where $(x)_+=x$ if $x\geq 0$ and $(x)_+=0$ otherwise.
\end{proposition}

\begin{proof}
When $\sigma\geq e^{-1/2}$,
\begin{align*}
  \phi_{\indm}(\sigma) =
\sigma \sqrt{\DimH_{\indm}}  \left( \sqrt{\ConstMultH_{\indm}} +
\sqrt{\pi} \right)
\end{align*}
which is non-decreasing and such that $\delta\mapsto\frac{1}{\delta}\phi_{\indm}(\delta)=\sqrt{\DimH_{\indm}}  \left( \sqrt{\ConstMultH_{\indm}} +
\sqrt{\pi} \right)$ is constant and thus non-increasing.

When $\sigma \leq e^{-1/2}$,
\begin{align*}
  \phi_{\indm}(\sigma) &=
\sigma \sqrt{\DimH_{\indm}}  \left( \sqrt{\ConstMultH_{\indm}} +
\sqrt{\pi} + \sqrt{\Log
    \frac{1}{\sigma}}  \right)\\
\intertext{and thus}
  \phi_{\indm}'(\sigma) &=
\sqrt{\DimH_{\indm}}  \left( \left( \sqrt{\ConstMultH_{\indm}} +
\sqrt{\pi} + \sqrt{\Log
    \frac{1}{\sigma}} \right)
-\frac{1}{2\sqrt{\Log
    \frac{1}{\sigma}}}
\right)
\\
& = \sqrt{\DimH_{\indm}}  \left( \sqrt{\ConstMultH_{\indm}} +
\sqrt{\pi} + \frac{1}{\sqrt{\Log
    \frac{1}{\sigma}}} \left(\Log
    \frac{1}{\sigma} -1/2 \right)
\right) \geq 0
\end{align*}
as $\Log
    \frac{1}{\sigma} -1/2 \geq 0$ when $\sigma \leq e^{-1/2}$.
$\phi_{\indm}$ is thus non-decreasing. The function
\begin{align*}
\delta \mapsto \frac{1}{\delta} \phi_{\indm}(\delta) =
 \sqrt{\DimH_{\indm}} \left( \sqrt{\ConstMultH_{\indm}} + \sqrt{\pi}
+ \sqrt{\Log
    \frac{1}{\delta}}  \right).
\end{align*}
is  strictly decreasing and thus non-increasing.

Now
  \begin{align*}
 \int_0^{\sigma}
 \sqrt{H_{[\cdot],\dtens}(\delta,\model_{\indm})} \, \ud \delta
& = \int_0^{\sigma}
 \sqrt{H_{[\cdot],\dtens}(\delta \wedge \sqrt{2},\model_{\indm})} \, \ud
 \delta 
\\
& \leq \int_0^{\sigma}
 \sqrt{H_{[\cdot],\dtens}(\delta \wedge e^{-1/2},\model_{\indm})} \, \ud \delta
\\
& \leq \int_0^{\sigma \wedge e^{-1/2}}
 \sqrt{\DimH_{\indm}} \sqrt{\ConstMultH_{\indm} +  \Log
   \frac{1}{\delta}} \, \ud \delta
+
\int_{\sigma \wedge e^{-1/2}}^{\sigma}
  \sqrt{\DimH_{\indm}} \sqrt{\ConstMultH_{\indm}+ \Log \frac{1}{e^{-1/2}}}\, \ud \delta\\
& \leq \left(\sigma \sqrt{\ConstMultH_{\indm} } +
 \int_0^{\sigma\wedge e^{-1/2}} \sqrt{\Log
  \frac{1}{\delta}} \, \ud \delta  
+
 \int_{\sigma\wedge e^{-1/2}}^{\sigma} \sqrt{\Log
  \frac{1}{e^{-1/2}}} \, \ud \delta  
 \right) \sqrt{\DimH_{\indm}}.
\end{align*}
We now rely on
\begin{lemma}\label{lem:boundinteg} For any $\sigma\in[0,1]$,
  \begin{alignI}
    \int_{0}^\sigma \sqrt{\Log \frac{1}{\delta}} \, \ud \delta \leq
    \sigma \left( \sqrt{\Log \frac{1}{\sigma} }+\sqrt{\pi} \right).
  \end{alignI}
  \end{lemma}
proved in \textcite{maugis09:_gauss} to deduce
\begin{align*}
 \int_0^{\sigma}
 \sqrt{H_{[\cdot],\dtens}(\delta,\model_{\indm})} \, \ud \delta
& \leq 
\bigg(  \sigma \sqrt{\ConstMultH_{\indm}} +
 ( \sigma\wedge e^{-1/2}) \left( \sqrt{\Log
    \frac{1}{\sigma\wedge e^{-1/2}}} + \sqrt{\pi} \right)\\
&\qquad\qquad\qquad\qquad
+ (\sigma-\sigma\wedge e^{-1/2})_+ \sqrt{\Log
    \frac{1}{e^{-1/2}}}
\bigg)  \sqrt{\DimH_{\indm}}
\\
& \leq  \sigma \left( \sqrt{\ConstMultH_{\indm}} + \sqrt{\pi} +
 \sqrt{\Log
    \frac{1}{\sigma\wedge e^{-1/2}}} \right) \sqrt{\DimH_{\indm}}
\end{align*}
\begin{align*}
&\frac{1}{\sigma} \phi_{\indm}(\sigma)=
 \sqrt{n}  \sigma
 \quad\Leftrightarrow \quad
 \left( \sqrt{\ConstMultH_{\indm}} + \sqrt{\pi} +
 \sqrt{\Log
    \frac{1}{\sigma\wedge e^{-1/2}}} \right) \sqrt{\DimH_{\indm}}
=     \sqrt{n}  \sigma
\\
\quad\Leftrightarrow\quad
&
\sigma = \frac{1}{\sqrt{n}} 
 \left( \sqrt{\ConstMultH_{\indm}} + \sqrt{\pi} +
 \sqrt{\Log
    \frac{1}{\sigma\wedge e^{-1/2}}} \right) \sqrt{\DimH_{\indm}}
 \end{align*}

This implies
\begin{align*}
\sigma_{\indm} \geq \frac{1}{\sqrt{n}} 
\left( \sqrt{\ConstMultH_{\indm}} + \sqrt{\pi} \right) \sqrt{\DimH_{\indm}}
\end{align*}
which implies by inserting this bound in the initial equality
\begin{align*}
\sigma_{\indm} & \leq  \frac{1}{\sqrt{n}} \left(
  \sqrt{\ConstMultH_{\indm}} + \sqrt{\pi} + \sqrt{\Log \frac{1}{ \frac{1}{\sqrt{n}} \left(
        \sqrt{\ConstMultH_{\indm}} + \sqrt{\pi} \right)
\sqrt{\DimH_{\indm}} \wedge e^{-1/2}}} \right) \sqrt{\DimH_{\indm}}\\
& \leq  \frac{1}{\sqrt{n}} \left(
  \sqrt{\ConstMultH_{\indm}} + \sqrt{\pi} + \sqrt{\Log \frac{e^{1/2}}{ \frac{e^{1/2}}{\sqrt{n}} \left(
        \sqrt{\ConstMultH_{\indm}} + \sqrt{\pi} \right)
\sqrt{\DimH_{\indm}} \wedge 1}} \right) \sqrt{\DimH_{\indm}}\\
& \leq  \frac{1}{\sqrt{n}} \left(
  \sqrt{\ConstMultH_{\indm}} + \sqrt{\pi} + \sqrt{\frac{1}{2} \left( 1
      +\left(
     \Log
    \frac{n}{  e \left(
        \sqrt{\ConstMultH_{\indm}} + \sqrt{\pi}
      \right)^2\DimH_{\indm}} \right)_+\right)} \right) \sqrt{\DimH_{\indm}}
  \end{align*}
Proposition's bound is obtained by squaring this inequality,
 using the inequality $(\sqrt{a}+\sqrt{b})^2
\leq 2 (a+b)$ and multiplying by $n$.
\end{proof}

If one is able to bound the bracketing entropy of the local models,
one can use:
\begin{proposition}
\label{prop:complexlocal}
Assume for any $\sigma \in (0,\sqrt{2}]$ and any $\delta\in(0,\sigma]$
\begin{align*}
  H_{[\cdot],\dtens}(\delta,\model_{\indm}(s_{\indm},\sigma)) \leq
\DimH_{\indm} \left( \ConstMultH_{\indm} +  \Log \frac{\sigma}{\delta}\right).
\end{align*}

Then the function
\begin{align*}
  \phi_{\indm}(\sigma) = \sigma \sqrt{\DimH_{\indm}} \left( \sqrt{\ConstMultH_{\indm}} +
     \sqrt{\pi} \right) 
\end{align*}
 satisfies the properties required in Assumption \Hindm.

Furthermore, $\DIMH_{\indm}$
  satisfies
\begin{align*}
\DIMH_{\indm}
& = \left( \sqrt{\ConstMultH_{\indm}} +
    \sqrt{\pi} \right)^2
    \DimH_{\indm}.
\end{align*}
\end{proposition}

\begin{proof}[Proof of \cref{prop:complexlocal}]
By construction, the function $\phi_{\indm}$ is non decreasing while the function
\begin{align*}
  \delta \mapsto \frac{1}{\delta} \phi_{\indm}(\delta) =
  \left( \sqrt{\ConstMultH_{\indm}} +  \sqrt{\pi} \right) \sqrt{\DimH_{\indm}}
\end{align*}
is non increasing.

Now,
\begin{align*}
 \int_0^{\sigma}
 \sqrt{H_{[\cdot],\dtens}(\delta,\model_{\indm}(s_{\indm},\sigma))} \, \ud \delta
& \leq \int_0^{\sigma}
 \sqrt{H_{[\cdot],\dtens}(\delta \wedge
   \sqrt{2},\model_{\indm}(s_{\indm},\sigma \wedge \sqrt{2}))} \\
& \leq  \int_0^{\sigma}
 \sqrt{\DimH_{\indm} \left( \ConstMultH_{\indm} +  \Log
     \frac{\sigma \wedge \sqrt{2}}{\delta \wedge \sqrt{2}} \right)}
 \, \ud \delta
\\&
 \leq \int_0^{\sigma}
 \left(\sqrt{\ConstMultH_{\indm}} + \sqrt{\Log \frac{\sigma}{\delta}}\right)
 \, \ud \delta \sqrt{\DimH_{\indm}} \\
&
 \leq \sigma \int_0^{1}
 \left(\sqrt{\ConstMultH_{\indm}} + \sqrt{\Log \frac{1}{\delta}}\right)
 \, \ud \delta \sqrt{\DimH_{\indm}}
\\
\intertext{
We now use \cref{lem:boundinteg}
to obtain}
& \leq \sigma \left( 
 \sqrt{\ConstMultH_{\indm}} + \sqrt{\pi}
\right) \sqrt{\DimH_{\indm}} 
\end{align*}

By definition of $\phi_{\indm}(\sigma)$:
\begin{align*}
&\frac{1}{\sigma} \phi_{\indm}(\sigma)=
 \sqrt{n}  \sigma
\quad \Leftrightarrow \quad
 \left( \sqrt{\ConstMultH_{\indm}} +
\sqrt{\pi}\right) \sqrt{\DimH_{\indm}} 
=    \sqrt{n}  \sigma
\quad\Leftrightarrow\quad
\sigma = \frac{1}{\sqrt{n}}   \left( \sqrt{\ConstMultH_{\indm}} +
\sqrt{\pi} \right) \sqrt{\DimH_{\indm}} 
\end{align*}
Squaring this equality and multiplying by $n$ yields the equality of the Proposition.
\end{proof}

}{}

\subsection{Proof of \cref{theo:single}}

\ifthenelse{\boolean{extended}}{}{For sake of simplicity, we remove from now on the subscript
reference to the common measure $\meas$ from all notations.}

\begin{proof}[Proof of \cref{theo:single}]
For any function $g$, which may depend on the observed $(X_i,Y_i)$, 
we define its empirical process $\Pemp(g)$ by
\begin{align*}
\Pemp(g) = \frac{1}{n} \sum_{i=1}^n g(X_i,Y_i)
\end{align*}
and its mean $\Ptens(g)$ 
by
\begin{align*}
\Ptens(g)=\E\left[\Pemp(g) 
\right]= \E\left[\frac{1}{n} \sum_{i=1}^n g(X'_i,Y'_i) 
\right] = \frac{1}{n} \sum_{i=1}^n \E\left[g(X'_i,Y'_i) 
\right] 
\end{align*}
where $(X'_i,Y'_i)$ is an independent copy of $(X_i,Y_i)$. Note that
when $g$ depends on the $(X_i,Y_i)$, $\Ptens(g)$ is a random variable.
Let $\Prec(g)$ denote the recentred process $\Pemp(g)-\Ptens(g)$.

Using this definition, 
\begin{align*}
\KLtens(s_0,t) & =\Ptens\left( - \Log\left(
    \frac{t}{s_0} \right) \right)&\text{and}&&
\JKLtens_{\propJKL}(s_0,t) & =
\Ptens\left(  -  \frac{1}{\propJKL} \Log\left(
    \frac{(1-\propJKL)s_0+ \propJKL t}{s_0} \right) \right).
\end{align*}

By construction, $\widehat{s}_\indm$ satisfies
\begin{align*}
 \Pemp(-\Log \widehat{s}_\indm) &\leq
  \inf_{s_\indm\in\model_\indm} \Pemp(-\Log s_\indm) +
  \frac{\rhoapp}{n}\\
\intertext{We let $\widebar{s}_\indm$ be a function 
  such that}
\KLtens(s_0,\widebar{s}_{\indm}) &\leq 
  \inf_{s_\indm\in\model_\indm} \KLtens(s_0,s_\indm) +\frac{\paramrhoKL}{n}.
\end{align*}
We then define the functions $\klbarsindm$, 
$\klhatsindm$, and $\jklhatsindm$ by
\begin{align*}
  \klbarsindm &= - \Log\left(
    \frac{\widebar{s}_\indm}{s_0} \right)&
  \klhatsindm &= - \Log\left(
    \frac{\widehat{s}_\indm}{s_0} \right)&
 \jklhatsindm &= -  \frac{1}{\propJKL} \Log\left(
    \frac{(1-\propJKL)s_0+ \propJKL \widehat{s}_\indm}{s_0} \right)
\end{align*}

By construction
\begin{align*}
  \Pemp(\klhatsindm) 
\leq \Pemp(\klbarsindm) + \frac{\rhoapp}{n}
\end{align*}
Since, by concavity of the logarithm, 
\begin {align*}
\jklhatsindm = -  \frac{1}{\propJKL} \Log\left(
    \frac{(1-\propJKL)s_0 + \propJKL \widehat{s}_{\indm'}}{s_0} \right) \leq 
- \frac{1}{\propJKL} \left(
    (1-\propJKL)  \Log \frac{s_0}{s_0} + \propJKL \Log
      \frac{\widehat{s}_{\indm'}}{s_0} \right) = - \Log \frac{\widehat{s}_{\indm'}}{s_0} =
\klhatsindm,
\end{align*}
\begin{align*}
  \Pemp(\jklhatsindm) \leq \Pemp(\klbarsindm)   +\frac{\rhoapp}{n}
\end{align*}
and thus
\begin{align*}
  \Ptens(\jklhatsindm) - \Prec(\klbarsindm)
& \leq \Ptens(\klbarsindm) -
 \Prec(\jklhatsindm) + \frac{\rhoapp}{n}\\
\intertext{using the definition of $\jklhatsindm$ and of
 $\klbarsindm$, we deduce}
\JKLtens_{\propJKL}(s_0,\widehat{s}_{\indm}) - \Prec(\klbarsindm)
& \leq  \inf_{s_{\indm} \in \model_\indm} \KLtens(s_0,s_\indm) 
 -\Prec(\jklhatsindm)
 +\frac{\rhoapp}{n} + \frac{\paramrhoKL}{n}
\end{align*}
where $\JKLtens_{\propJKL}(s_0,\widehat{s}_{\indm})$ is still a random variable.

We now rely on a control on the deviation of
$\Prec(\jklhatsindm)$ through its conditional expectation. For
any random variable $Z$ and any event $A$ such that $\Prob\{A\}>0$, we
let $\E^A\left[ Z \right] =
\frac{E\left[Z\Charac{A}\right]}{\Prob\{A\}}$. It is sufficient to
control those quantities for all $A$ to obtain a control of the
deviation. More precisely,
\begin{lemma}\label{lemma:controlcond}
Let $Z$ be a random variable, assume there exists a non decreasing
$\Psi$ such that for all $A$ such that $\Prob\{ A \}>0$,
\begin{alignI}
  \E^A[ Z ] \leq \Psi\left(\Log\left(\frac{1}{\Prob\{A\}} \right)\right).
\end{alignI}
  then for all $x$
  \begin{alignI}
    \Prob\{ Z > \Psi(x) \} \leq e^{-x}.
  \end{alignI}
\end{lemma}

Here, we can prove
\begin{lemma}
\label{lemma:core} There exist three absolute constants $\constlemmathree>4$,
  $\constlemmaone$ and $\constlemmatwo$ such that, under Assumption (H), for all
  $\indm\in\indmset$, for every  $y_\indm>\sigma_\indm$ and every
  event $A$ such that $\Prob\{A\}>0$,
\begin{align*}
  \E^A &\left[
\Prec\left(
  \frac{-\jklhatsindm}{y_\indm^2 +
    \constlemmathree\dtwotens(s_0,\widehat{s}_\indm)} \right)
\right]
& \leq 
 \frac{\constlemmaone \sigma_\indm}{y_\indm}
+ \constlemmatwo
\frac{1}{\sqrt{ny_m^2}}
\sqrt{\Log\left(\frac{1}{\Prob\{A\}}\right)} + \frac{18}{ny_m^2\propJKL}
\Log\left(\frac{1}{\Prob\{A\}}\right).
\end{align*}
\end{lemma}

Combining \cref{lemma:controlcond} and \cref{lemma:core} implies that,
except on a set of probability   less than
$e^{-x}$, for any $y_{\indm}> \sigma_{\indm}$,
\begin{align*}
\frac{-\Prec(\jklhatsindm)}
{y_{\indm}^2 + \constlemmathree \dtwotens(s_0,\widehat{s}_{\indm})}
\leq 
\frac{\constlemmaone \sigma_{\indm}}{y_{\indm}} +
  \constlemmatwo \sqrt{\frac{x}{n y_{\indm}^2}} + \frac{18}{\propJKL} \frac{x}{n y_{\indm}^2}.
\end{align*}
Choosing $y_{\indm} = \constcore \sqrt{\sigma_{\indm}^2 +
  \frac{x}{n}}$ with $\constcore>1$ to be fixed later, we
deduce that, except on a set of probability less than
$e^{-x}$, \begin{align*}
\frac{-\Prec(\jklhatsindm)}
{y_{\indm}^2 + \constlemmathree \dtwotens(s_0,\widehat{s}_{\indm})}
\leq \frac{\constlemmaone+ \constlemmatwo}{\constcore} + \frac{18}{\constcore^2\propJKL}
\end{align*}

Thus, except on the same set,
\begin{align*}
\JKLtens_{\propJKL}(s_0,\widehat{s}_{\indm}) - \Prec(\klbarsindm)
& \leq \inf_{s_{\indm} \in \model_\indm}
\KLtens(s_0,s_\indm) 
+\left( 
\frac{\constlemmaone+ \constlemmatwo}{\constcore} +
\frac{18}{\constcore^2\propJKL}
\right)
\left( y_{\indm}^2 + \constlemmathree
  \dtwotens(s_0,\widehat{s}_{\indm})\right)\\
& \qquad
 +\frac{\rhoapp}{n} + \frac{\paramrhoKL}{n}.   
\end{align*}
Let $\paramepsipen>0$,
we define $\constcoreepsipen$ by $
\left( \frac{\constlemmaone+\constlemmatwo}{\constcoreepsipen}  + \frac{18}{\constcoreepsipen^2\propJKL}
\right)\constlemmathree
=\ConstHellKL_{\propJKL}\, \paramepsipen$ 
with 
$\ConstHellKL_{\propJKL}$ defined in \cref{prop:hellingerkull2}
and as $\widehat{s}_{\indm}$ is a conditional density
$\ConstHellKL_{\propJKL}\,
\dtwotens(s_0,\widehat{s}_{\indm})\leq\JKLtens_{\propJKL}(s_0,\widehat{s}_{\indm})$. Thus,
 we obtain
\begin{align*}
(1-\paramepsipen) \JKLtens_{\propJKL}(s_0,\widehat{s}_{\indm}) - \Prec(\klbarsindm)
& \leq \inf_{s_{\indm} \in \model_\indm}
\KLtens(s_0,s_\indm) 
+\frac{\ConstHellKL_{\propJKL}\paramepsipen y_{\indm}^2}{\constlemmathree}
 + \frac{\rhoapp}{n} + \frac{\paramrhoKL}{n}\\
& \leq  \inf_{s_{\indm} \in \model_\indm}
\KLtens(s_0,s_\indm) 
+ \frac{\ConstHellKL_{\propJKL}\paramepsipen \constcoreepsipen^2}{\constlemmathree}
  \left( \sigma_m^2 + \frac{x}{n}\right)\\
&\qquad\qquad
 + \frac{\rhoapp}{n} + \frac{\paramrhoKL}{n}
\end{align*}

Let
$\constpenmin=\frac{\ConstHellKL_{\propJKL}\paramepsipen\constcoreepsipen^2}{\constlemmathree}$,
we obtain that, with probability smaller than  $e^{-x}$,
 \begin{align*}
 \JKLtens_{\propJKL}(s_0,\widehat{s}_{\widehat{\indm}})
& >
\frac{1}{1-\paramepsipen} \left( \inf_{s_{\indm} \in \model_\indm}
\KLtens(s_0,s_\indm) +
\constpenmin
\sigma_m^2
\right)
+ \frac{\rhoapp}{n} + \frac{\paramrhoKL}{n}
\\
&\qquad
+\frac{1}{1-\paramepsipen}
 \Prec(\klbarsindm)
+
\frac{\constpenmin}{1-\paramepsipen}
\frac{x}{n}
\end{align*}
which can be rewritten as, with probability smaller than $e^{-x}$,
\begin{align*}
\JKLtens_{\propJKL}(s_0,\widehat{s}_{\widehat{\indm}})- \left(
\frac{1}{1-\paramepsipen} \left( \inf_{s_{\indm} \in \model_\indm}
\KLtens(s_0,s_\indm) +
\constpenmin
\sigma_m^2
\right)
+ \frac{\rhoapp}{n} + \frac{\paramrhoKL}{n}
\right)&\\
+\frac{1}{1-\paramepsipen}
 \Prec(\klbarsindm) &> \frac{\constpenmin}{1-\paramepsipen}
\frac{x}{n}
\end{align*}
For any non negative random variable $Z$ and any $a>0$, $\E\left[Z\right]= a \int_{z\geq 0}
\Prob\{Z> a z\} dz$ so
\begin{align*}
 \E\left[  \JKLtens_{\propJKL}(s_0,\widehat{s}_{\widehat{\indm}})
- \left(
\frac{1}{1-\paramepsipen} \left( \inf_{s_{\indm} \in \model_\indm}
\KLtens(s_0,s_\indm) +
\constpenmin
\sigma_m^2
\right)
+ \frac{\rhoapp}{n} + \frac{\paramrhoKL}{n} \right)\right]\\
+ \E\left[\frac{1}{1-\paramepsipen}
 \Prec(\klbarsindm)\right] &\leq \frac{1}{1-\paramepsipen}\constpenmin
\frac{1}{n}
\end{align*}
As by construction 
$\Prec(\klbarsindm)$ is integrable and
$\E\left[\Prec(\klbarsindm)\right]=0$), we derive
\begin{align*}
\E\left[\JKLtens_{\propJKL}(s_0,\widehat{s}_{\widehat{\indm}})\right]
& \leq \frac{1}{1-\paramepsipen} \left( 
\inf_{s_{\indm} \in \model_\indm} \KLtens(s_0,s_\indm) 
+ \constpenmin \sigma_\indm^2
\right) +
\frac{\constpenmin}{1-\paramepsipen}
\frac{1}{n} + \frac{\rhoapp}{n} + \frac{\paramrhoKL}{n}.
\end{align*}
As $\paramrhoKL$ can be chosen arbitrary small this implies
\begin{align*}
\E\left[\JKLtens_{\propJKL}(s_0,\widehat{s}_{\widehat{\indm}})\right]
& \leq \frac{1}{1-\paramepsipen} \left( 
\inf_{s_{\indm} \in \model_\indm} \KLtens(s_0,s_\indm) + \constpenmin \sigma_\indm^2
\right) +
\frac{\constpenmin}{1-\paramepsipen}
\frac{1}{n} + \frac{\rhoapp+\rhomin}{n}
\end{align*}
and thus 
$\constoraone= \frac{1}{1-\paramepsipen}$
and $\constoratwo= \frac{\constpenmin}{1-\paramepsipen}$.
\end{proof}

\ifthenelse{\boolean{extended}}{\section{Proofs for \cref{sec:model-select-penal} (\nameref{sec:model-select-penal})}
}{}

\subsection{Proof of \cref{theo:select}}


\begin{proof}[Proof of \cref{theo:select}]
For any model $\model_\indm$, 
we let
$\widebar{s}_\indm$ be a function such that 
\begin{align*}
\KLtens(s_0,\widebar{s}_{\indm}) &\leq 
  \inf_{s_\indm\in\model_\indm} \KLtens(s_0,s_\indm) +\frac{\paramrhoKL}{n}.
\end{align*}

Let $\indm\in\indmset$ such that $\KLtens(s,\widebar{s}_m)<+\infty$ and
let 
\begin{align*}\indmset'= \left\{ \indm' \in \indmset \middle| \Pemp(-\Log \widehat{s}_{m'}) +
\frac{\pen(\indm')}{n} \leq \Pemp(-\Log \widehat{s}_{m}) +
\frac{\pen(\indm)}{n} + \frac{\rhomin}{n}\right\}.
\end{align*}

For every $\indm'\in\indmset'$,
\begin{align*}
  \Pemp(\klhatsindmp) + \frac{\pen(\indm')}{n} \leq
  \Pemp(\klhatsindm) + \frac{\pen(\indm)}{n} + \frac{\rhomin}{n}
\leq \Pemp(\klbarsindm) + \frac{\pen(\indm)}{n} + \frac{\rhoapp+\rhomin}{n}
\end{align*}
Since, by concavity of the logarithm, 
$\jklhatsindmp \le
\klhatsindmp$,
\begin{align*}
  \Pemp(\jklhatsindmp) + \frac{\pen(\indm')}{n} \leq \Pemp(\klbarsindm) + \frac{\pen(\indm)}{n}  +\frac{\rhoapp+\rhomin}{n}
\end{align*}
and thus
\begin{align*}
  \Ptens(\jklhatsindmp) - \Prec(\klbarsindm)
& \leq \Ptens(\klbarsindm) + \frac{\pen(\indm)}{n} -
 \Prec(\jklhatsindmp) - \frac{\pen(\indm')}{n} + \frac{\rhoapp+\rhomin}{n}\\
\intertext{using the definition of $\jklhatsindmp$ and of
 $\klbarsindm$, we deduce}
\JKLtens_{\propJKL}(s_0,\widehat{s}_{\indm'}) - \Prec(\klbarsindm)
& \leq  \inf_{s_{\indm} \in \model_\indm} \KLtens(s_0,s_\indm) + \frac{\pen(\indm)}{n} -
 \Prec(\jklhatsindmp) - \frac{\pen(\indm')}{n}\\
&\qquad\qquad
 +\frac{\rhoapp+\rhomin}{n} + \frac{\paramrhoKL}{n}
\end{align*}

Combining again \cref{lemma:controlcond} and \cref{lemma:core}, we deduce that,
except on a set of probability   less than
$e^{-x_{\indm'}-x}$, for any $y_{\indm'}> \sigma_{\indm'}$,
\begin{align*}
\frac{-\Prec(\jklhatsindmp)}
{y_{\indm'}^2 + \constlemmathree \dtwotens(s_0,\widehat{s}_{\indm'})}
\leq 
\frac{\constlemmaone \sigma_{\indm'}}{y_{\indm'}} +
  \constlemmatwo \sqrt{\frac{x_{\indm'}+x}{n y_{\indm'}^2}} + \frac{18}{\propJKL} \frac{x_{\indm'}+x}{n y_{\indm'}^2}.
\end{align*}
Choosing this time $y_{\indm'} = \constcore \sqrt{\sigma_{\indm'}^2 +
  \frac{x_{\indm'}+x}{n}}$ with $\constcore>1$ to be fixed later, we
deduce that, except on a set of probability less than
$e^{-x_{\indm'}-x}$, \begin{align*}
\frac{-\Prec(\jklhatsindmp)}
{y_{\indm'}^2 + \constlemmathree \dtwotens(s_0,\widehat{s}_{\indm'})}
\leq \frac{\constlemmaone+ \constlemmatwo}{\constcore} + \frac{18}{\constcore^2\propJKL}
\end{align*}
Using the Kraft condition of Assumption (K), we deduce that if we make this
choice of $y_{\indm'}$ for all models $\indm'$, this properties hold simultaneously for all
$\indm'\in\indmset$
except on a set of probability less than $\Sigma e^{-x}$.

Thus, except on the same set, simultaneously for all
$\indm\in\indmset'$,
\begin{align*}
\JKLtens_{\propJKL}(s_0,\widehat{s}_{\indm'}) - \Prec(\klbarsindm)
& \leq \inf_{s_{\indm} \in \model_\indm}
\KLtens(s_0,s_\indm) + \frac{\pen(\indm)}{n}\\
&\qquad\qquad
+\left( 
\frac{\constlemmaone+ \constlemmatwo}{\constcore} +
\frac{18}{\constcore^2\propJKL}
\right)
\left( y_{\indm'}^2 + \constlemmathree \dtwotens(s_0,\widehat{s}_{\indm'})\right)
- \frac{\pen(\indm')}{n}\\
&\qquad\qquad\qquad +\frac{\rhoapp+\rhomin}{n} + \frac{\paramrhoKL}{n}.   
\end{align*}
Let $\paramepsipen>0$,
we define $\constcoreepsipen$ by $
\left( \frac{\constlemmaone+\constlemmatwo}{\constcoreepsipen}  + \frac{18}{\constcoreepsipen^2\propJKL}
\right)\constlemmathree
=\ConstHellKL_{\propJKL}\, \paramepsipen$ 
with 
$\ConstHellKL_{\propJKL}$ defined in \cref{prop:hellingerkull2}
and, as $\widehat{s}_{\indm'}$ is a conditional density
$\ConstHellKL_{\propJKL}\, \dtwotens(s_0,\widehat{s}_{\indm'})\leq\JKLtens_{\propJKL}(s_0,\widehat{s}_{\indm'})$,
we obtain
\begin{align*}
(1-\paramepsipen) \JKLtens_{\propJKL}(s_0,\widehat{s}_{\indm'}) - \Prec(\klbarsindm)
& \leq \inf_{s_{\indm} \in \model_\indm}
\KLtens(s_0,s_\indm) + \frac{\pen(\indm)}{n}\\
&\qquad
+\frac{\ConstHellKL_{\propJKL}\paramepsipen y_{\indm'}^2}{\constlemmathree}
- \frac{\pen(\indm')}{n} + \frac{\rhoapp+\rhomin}{n} + \frac{\paramrhoKL}{n}. 
\end{align*}
We should now study $\frac{\ConstHellKL_{\propJKL}\paramepsipen y_{\indm'}^2}{\constlemmathree}-\frac{\pen(\indm')}{n}$:
\begin{align*}
  \frac{\ConstHellKL_{\propJKL}\paramepsipen y_{\indm'}^2}{\constlemmathree}-\frac{\pen(\indm')}{n} & = \frac{\ConstHellKL_{\propJKL}\paramepsipen \constcoreepsipen^2}{\constlemmathree}
  \left( \sigma_m^2 + \frac{x_m+x}{n}\right) - \frac{\pen(\indm')}{n}\\
\intertext{and by construction if we let $\constpenmin = \frac{\ConstHellKL_{\propJKL}\paramepsipen \constcoreepsipen^2}{\constlemmathree}$}
 \frac{\ConstHellKL_{\propJKL}\paramepsipen y_{\indm'}^2}{\constlemmathree}-\frac{\pen(\indm')}{n} & \leq \constpenmin
 \frac{x}{n} - (1-\frac{\constpenmin}{\kappa}) \frac{\pen(\indm')}{n}.
\end{align*}
We deduce thus, except on a set of probability smaller than $\Sigma
e^{-x}$, simultaneously for any $\indm'\in\indmset'$
\begin{align*}
(1-\paramepsipen) \JKLtens_{\propJKL}(s_0,\widehat{s}_{\indm'}) 
&+ (1-\frac{\constpenmin}{\kappa}) \frac{\pen(\indm')}{n}
- \Prec(\klbarsindm)\\
& \leq  \inf_{s_{\indm} \in \model_\indm}
\KLtens(s_0,s_\indm) + \frac{\pen(\indm)}{n} + \constpenmin \frac{x}{n}  + \frac{\rhoapp+\rhomin}{n} + \frac{\paramrhoKL}{n}
\end{align*}

As $\Prec(\klbarsindm)$ is integrable (and of mean $0$), we
derive that $M=\sup_{\indm'\in\indmset'} \frac{\pen(\indm')}{n}$ is almost surely
finite, so that as $\kappa \frac{x_{\indm'}}{n} \leq M$ for every
$\indm' \in \indmset'$, one has
\begin{align*}
  \Sigma \geq \sum_{\indm'\in\indmset'} e^{-x_{\indm'}} \geq |\indmset'|
  e^{-\frac{Mn}{\kappa}}
\end{align*}
and thus $\indmset'$ is almost surely finite. This implies that the minimizer $\hat{m}$ of $\Pemp\left(-\Log(\widehat{s}_\indm)\right)
+ \frac{\pen(\indm)}{n}$ exists.

For this minimizer, one has with probability greater than $1-\Sigma e^{-x}$,
\begin{align*}
(1-\paramepsipen)
\JKLtens_{\propJKL}(s_0,\widehat{s}_{\widehat{\indm}}) &
+ (1-\frac{\constpenmin}{\kappa}) \frac{\pen(\widehat{\indm})}{n} 
- \Prec(\klbarsindm)\\
& \leq  \inf_{s_{\indm} \in \model_\indm}
\KLtens(s_0,s_\indm) + \frac{\pen(\indm)}{n} + \constpenmin \frac{x}{n}  + \frac{\rhoapp+\rhomin}{n} + \frac{\paramrhoKL}{n}  
\end{align*}
which yields by the same integration technique that in the proof of
the previous theorem
\begin{align*}
\E\left[\JKLtens_{\propJKL}(s_0,\widehat{s}_{\widehat{\indm}})
+ \frac{1-\frac{\constpenmin}{\kappa}}{1-\paramepsipen} \frac{\pen(\widehat{\indm})}{n} 
\right]
& \leq \frac{1}{1-\paramepsipen} \left( 
\inf_{s_{\indm} \in \model_\indm} \KLtens(s_0,s_\indm) + \frac{\pen(\indm)}{n}
\right) 
\\
& \qquad +
\frac{\constpenmin}{1-\paramepsipen}
\frac{\Sigma}{n} 
+ \frac{\rhoapp+\rhomin}{n} + \frac{\paramrhoKL}{n}.
\end{align*}
As $\paramrhoKL$ can be chosen arbitrary small this implies
\begin{align*}
\E\left[\JKLtens_{\propJKL}(s_0,\widehat{s}_{\widehat{\indm}})+ \frac{1-\frac{\constpenmin}{\kappa}}{1-\paramepsipen} \frac{\pen(\widehat{\indm})}{n} \right]
& \leq \frac{1}{1-\paramepsipen} \left( 
\inf_{s_{\indm} \in \model_\indm} \KLtens(s_0,s_\indm) + \frac{\pen(\indm)}{n}
\right)\\
&\qquad +
\frac{\constpenmin}{1-\paramepsipen}
\frac{\Sigma}{n} + \frac{\rhoapp+\rhomin}{n}
\end{align*}
which is sligthly stronger than the result stated in the theorem with
$\constoraone= \frac{1}{1-\paramepsipen}$
and $\constoratwo= \frac{\constpenmin}{1-\paramepsipen}$ as the
penalty of the select model appears in the right-hand side with a
positive weight.
\end{proof}

\subsection{Proof of \cref{lemma:controlcond}}

\begin{proof}[Proof of  \cref{lemma:controlcond}]
Let $A=\{ Z > \Psi(x) \}$. Either $P\{ A \} = 0 \leq e^{-x}$
or
\begin{align*}
  \E^A [ Z ] \leq \Psi\left(\Log\left(\frac{1}{\Prob\{A\}} \right)\right).
\end{align*}
Now in the later case,
\begin{align*}
  \E^A [ Z ] = \frac{\E \left[ Z \Charac{Z > \Psi(x)}
    \right]}{\Prob\{ Z > \Psi(x)\}}
\geq \Psi(x).
\end{align*}
Hence $\Psi(x) \leq \Psi\left(\Log\left(\frac{1}{\Prob\{A\}}
  \right)\right)$ which implies $x \leq \Log\left(\frac{1}{\Prob\{A\}}
\right)$ as $\Psi$ is not decreasing. This last inequality yields
$\Prob\{A\} \leq e^{-x}$ which concludes the proof.
\end{proof}

\subsection{Proof of \cref{lemma:core}}
We should now prove \cref{lemma:core} which contains most of the
differences with \textcite{massart07:_concen}'s proof.


\begin{proof}[Proof of \cref{lemma:core}]
In this lemma, we want to control the deviation of
\begin{align*}
  \Prec( - \jklhatsindm) = \Prec \left( \frac{1}{\propJKL}
\Log \left(
\frac{(1-\propJKL)s_0 + \propJKL \widehat{s}_\indm}{s_0}
 \right)   
\right).
\end{align*}
Note that for any $\widetilde{s}_{\indm}$ to be fixed later, if we let
\begin{alignI}
  \jkltildes= -\frac{1}{\propJKL}
\Log \left(
\frac{(1-\propJKL)s_0 + \propJKL \widetilde{s}_{\indm}}{s_0}
 \right),
\end{alignI} then $-\jklhatsindm=-\jkltildes+(-\jklhatsindm+\jkltildes)$ with
\begin{align*}
 - \jklhatsindm + \jkltildes =  \frac{1}{\propJKL}
\Log \left(
\frac{(1-\propJKL)s_0 + \propJKL \widehat{s}_{\indm}}{(1-\propJKL)s_0 + \propJKL \widetilde{s}_{\indm}}
 \right)
\end{align*}

To control the behavior of these quantities, we use the following key
properties of Jensen-Kullback-Leibler related quantities (a rewriting of Lemma 7.26 of
\textcite{massart07:_concen})
\begin{lemma}
  Let $P$ be a probability measure with density $s_0$ with respect to
  a measure $\meas$ and $s,t$ be some non-negative and $\meas$
  integrable functions, then 
 one has for every integer $k\geq 2$
  \begin{align*}
    P\left( \left|\Log \left(
\frac{s_0+s}{s_0+t} \right)\right|^k \right) \leq \frac{k!}{2} \left(
\frac{9 \|\sqrt{s}-\sqrt{t}\|_{\meas,2}^2}{8}
\right) 2^{k-2} 
  \end{align*}
where $\|\cdot\|_{\meas,2}$ is the $\meas$-$L^2$ norm so that
$\|\sqrt{s}-\sqrt{t}\|_{\meas,2}^2$ is nothing but the extended
Hellinger distance.
\end{lemma}
In this lemma, $P\left( g \right)$ stands for $\int g s_0 \ud\meas$
i.e. the expectation with respect to the probability $s_0\ud\meas$.
In our context this implies, conditioning first by $(X_i)_{1\leq  i\leq n}$, applying the previous inequality for each $(s_0(\cdot|X_i),s(\cdot|X_i),t(\cdot|X_i))$ and then taking the expectation, that
\begin{align*}
  \Ptens\left(\left| \frac{1}{\propJKL} \Log \left(
\frac{s_0+ \frac{\propJKL}{1-\propJKL} s}{s_0+ \frac{\propJKL}{1-\propJKL} t} \right)\right|^k \right)
\leq  \frac{k!}{2} \left(
\frac{9  \dtwotens(s,t)}{8 \propJKL(1-\propJKL)}
\right) \left( \frac{2}{\propJKL} \right) ^{k-2}.
\end{align*}
We now use
\begin{theorem}
Assume $f$ is a function such that
\begin{align*}
  &\Ptens\left(|f|^2\right) \leq V\\
\forall k \geq 3,\quad &\Ptens\left((f)_+^k\right) \leq \frac{k!}{2} V b^{k-2}.
\end{align*}

Then for all $A$ such that $\Prob\{A\}>0$
\begin{align*}
  \E^A(\Prec(f)) \leq \frac{\sqrt{2V}}{\sqrt{n}}
  \sqrt{\Log\left(\frac{1}{\Prob\{A\}}\right)}
 + \frac{b}{n}   \Log\left(\frac{1}{\Prob\{A\}}\right).
\end{align*}
\end{theorem}
These bounds are sufficient to obtain a Bernstein type control
for $\jkltildes$
\begin{align*}
\E^A \left[
-\Prec(\jkltildes)
\right]
\leq
\frac{3}{2 \sqrt{\propJKL(1-\propJKL)}}
\frac{\sqrt{\dtwotens(s_0,\widetilde{s}_{\indm})}}{\sqrt{n}}
\sqrt{\Log\left(\frac{1}{\Prob\{A\}}\right)} + \frac{2}{n\propJKL}
\Log\left(\frac{1}{\Prob\{A\}}\right).
\end{align*}

To cope with the randomness of $\widehat{s}_{\indm}$, we rely on the
following much
more involved theorem (a rewriting of Theorem 6.8 of \textcite{massart07:_concen})
\begin{theorem}\label{theo:deviation}
Let $\Set$ be a countable class of real valued and measurable
functions. Assume that there exist some positive numbers
$V$ and $b$ such that for all $f \in
\Set$ and all integers $k\geq 2$
\begin{align*}
\Ptens(|f|^k) \leq \frac{k!}{2} V b^{k-2}
\end{align*}
Assume furthermore that for any positive number $\delta$, there exists a finite
set $B(\delta)$ of brackets covering $\Set$ 
such that for any bracket $[g^- , g^+ ] \in B(\delta)$ and all integer
$k\geq2$
\begin{align*}
\Ptens(|g^+-g^-|^k) \leq \frac{k!}{2} \delta^2 b^{k-2}
\end{align*}
Let $e^{H(\delta)}$ denote
the minimal cardinality of such a covering.
There exists an absolute constant $\kappalemma$ such that, for any
$\epsilon \in (0, 1]$ and any
measurable set $A$ with $\Prob\{A\} > 0$,
\begin{align*}
&&E^A\left[
\sup_{f\in\Set} \Prec(f)
\right]&\leq 
E + \frac{(1+6\epsilon)\sqrt{2V}}{\sqrt{n}}
\sqrt{\Log\left(\frac{1}{\Prob\{A\}}\right)}
+ \frac{2b}{n} \Log\left(\frac{1}{\Prob\{A\}}\right)\\
\text{where}&&
&E = \frac{\kappalemma}{\epsilon} \frac{1}{\sqrt{n}} \int_0^{\epsilon\sqrt{V}}
\sqrt{H(\delta) \wedge n} \ud \delta + \frac{2(b+\sqrt{V})}{n} H(\sqrt{V}).
\end{align*}
 Furthermore $\kappalemma \leq 27$.
\end{theorem}

If  we consider
\begin{align*}
\Set_{\indm}(\widetilde{s}_{\indm},\sigma)&=\left\{ 
-\jklsindm +\jkltildes=
\frac{1}{\propJKL}\Log \left(
\frac{s_0+\frac{\propJKL}{1-\propJKL}s_\indm}{s_0+\frac{\propJKL}{1-\propJKL}\widetilde{s}_{\indm}}
\right) \middle| s_{\indm} \in \model_\indm,
\dtwotens(\widetilde{s}_{\indm},s_{\indm})\leq\sigma\right\} \\
& = \left\{\frac{1}{\propJKL}\Log \left(
\frac{s_0+\frac{\propJKL}{1-\propJKL}s_{\indm}}{s_0+\frac{\propJKL}{1-\propJKL}\widetilde{s}_{\indm}}
\right) \middle|
s_{\indm} \in \model_\indm(\widetilde{s}_{\indm},\sigma) \right\}.
\end{align*}
then the first assumption of \cref{theo:deviation} holds
with $V=\left(\frac{3\sigma}{2 \sqrt{2\propJKL(1-\propJKL)}}\right)^2$ and $b=\frac{2}{\propJKL}$.

We are thus focusing on
\begin{align*}
W_{\indm}(\widetilde{s}_{\indm},\sigma) 
&= \sup_{f \in \Set_{\indm}(\widetilde{s}_{\indm},\sigma)} \Prec(f)
= \sup_{ s_\indm \in\model_\indm(\widetilde{s}_{\indm},\sigma)} \Prec\left(
-\jklsindm  +\jkltildes
\right)\\
&= \sup_{ s_\indm \in\model_\indm(\widetilde{s}_{\indm},\sigma)} \Prec\left(
-\jklsindm\right)   + \Prec\left(\jkltildes
\right)
\end{align*}

Now if  $[t^-,t^+]$ is a bracket containing $s$, then
\begin{align*}
 g^-= \frac{1}{\propJKL} \Log \left(
\frac{s_0+ \frac{\propJKL}{1-\propJKL}t^-}{s_0+\frac{\propJKL}{1-\propJKL}\widetilde{s}_{\indm}} \right) \leq \frac{1}{\propJKL}\Log \left(
\frac{s_0+\frac{\propJKL}{1-\propJKL}s}{s_0+\frac{\propJKL}{1-\propJKL}\widetilde{s}_{\indm}} \right) \leq \frac{1}{\propJKL} \Log \left(
\frac{s_0+\frac{\propJKL}{1-\propJKL}t^+}{s_0+\frac{\propJKL}{1-\propJKL}\widetilde{s}_{\indm}}
\right) = g^+
\end{align*}
and
\begin{align*}
 g^+ - g^- = \frac{1}{\propJKL} \Log \left(
\frac{s_0+\frac{\propJKL}{1-\propJKL}t^+}{s_0+\frac{\propJKL}{1-\propJKL}\widetilde{s}_{\indm}} \right) -   \frac{1}{\propJKL}\Log \left(
\frac{s_0+\frac{\propJKL}{1-\propJKL}t^-}{s_0+\frac{\propJKL}{1-\propJKL}\widetilde{s}_{\indm}} \right) =    \frac{1}{\propJKL} \Log \left(
\frac{s_0+\frac{\propJKL}{1-\propJKL}t^+}{s_0+\frac{\propJKL}{1-\propJKL}t^-} \right)
\end{align*}
So that
\begin{align*}
\Ptens(|g^+-g^-|^k) \leq \frac{k!}{2} \delta^2 b^{k-2}
\end{align*}
as soon as $\frac{3\dtens(t^-,t^+)}{2\sqrt{2 \propJKL(1-\propJKL)}}
\leq \delta$. This implies that, for any $\delta>0$, 
 one can construct a set of brackets satisfying the second assumption
 of \cref{theo:deviation} from a set of brackets of $\dtens$ width smaller than
$\frac{2\sqrt{2 \propJKL(1-\propJKL)}}{3} \delta$ covering
 $\model_{\indm}(\widetilde{s}_{\indm},\sigma)$. That is
\begin{align*}
H(\delta)\leq H_{[\cdot],\dtens}
\left(
\frac{2\sqrt{2 \propJKL(1-\propJKL)}}{3} \delta
, \model_{\indm}(\widetilde{s}_{\indm},\sigma)
\right).
\end{align*}

\Cref{theo:deviation} can not be used directly with the set
$\Set_{\indm}(\widetilde{s}_{\indm},\sigma)$ as it is not necessarily
countable.
However, Assumption \Sepindm\  implies the existence of a countable family
$\model'_\indm$ such that 
\begin{align*}
\Set_{\indm}'(\widetilde{s}_{\indm},\sigma)&=\left\{ 
-\jklsindm +\jkltildes=
\frac{1}{\propJKL}\Log \left(
\frac{s_0+\frac{\propJKL}{1-\propJKL}s_\indm}{s_0+\frac{\propJKL}{1-\propJKL}\widetilde{s}_{\indm}}
\right) \middle| s_{\indm} \in \model'_\indm,
\dtwotens(\widetilde{s}_{\indm},s_{\indm})\leq\sigma\right\}
\end{align*}
is countable, and thus for which the conclusion of \Cref{theo:deviation} holds,  while $\sup_{\Set_{\indm}'(\widetilde{s}_{\indm},\sigma)} \Prec(f)
=\sup_{\Set_{\indm}(\widetilde{s}_{\indm},\sigma)} \Prec(f) $ with
probability $1$. We deduce thus that
for every measurable set $A$ with $\Prob\{A\}>0$,
\begin{align*}
&&&\E^A\left[
W_m(\widetilde{s}_{\indm},\sigma)
\right]\leq 
E + \frac{(1+6\epsilon)3 \sigma }{2\sqrt{\propJKL(1-\propJKL)} \sqrt{n}}
\sqrt{\Log\left(\frac{1}{\Prob\{A\}}\right)}
+ \frac{4}{\propJKL n} \Log\left(\frac{1}{\Prob\{A\}}\right)\\
\text{where}&&
E &= \frac{\kappalemma}{\epsilon} \frac{1}{\sqrt{n}} \int_0^{\epsilon\frac{3 \sigma }{2\sqrt{2\propJKL(1-\propJKL)}}}
\sqrt{%
  H_{[\cdot],\dtens}\left( \frac{2\sqrt{2\propJKL(1-\propJKL)}}{3}
\delta,\model_\indm(\widetilde{s}_{\indm},\sigma)\right) \wedge n} \ud \delta \\
&&& \qquad
+ \frac{2(\frac{2}{\propJKL}+\frac{3 \sigma
  }{2\sqrt{2\propJKL(1-\propJKL)}})}{n}
H_{[\cdot],\dtens}\left(\frac{2\sqrt{2\propJKL(1-\propJKL)}}{3}
\frac{3 \sigma }{2\sqrt{2\propJKL(1-\propJKL)}}
,\model_\indm(\widetilde{s}_{\indm},\sigma)\right)\\
 &&& = \frac{3 \kappalemma}{2 \epsilon \sqrt{2\propJKL(1-\propJKL)}}
  \frac{1}{\sqrt{n}} \int_0^{\epsilon\sigma}
 \sqrt{
   H_{[\cdot],\dtens}\left(\delta,\model_\indm(\widetilde{s}_{\indm},\sigma)\right) \wedge n} 
\ud \delta\\
&&&\qquad+ \frac{2(\frac{2}{\propJKL}+\frac{3 \sigma
  }{2\sqrt{2\propJKL(1-\propJKL)}})}{n}
H_{[\cdot],\dtens}\left(\sigma
,\model_\indm(\widetilde{s}_{\indm},\sigma)\right)
\end{align*}

Choosing $\epsilon=1$ leads to
\begin{align*}
\E^A\left[
W_m(\widetilde{s}_{\indm},\sigma)
\right]&\leq 
E + \frac{21 \sigma }{2\sqrt{\propJKL(1-\propJKL)} \sqrt{n}}
\sqrt{\Log\left(\frac{1}{\Prob\{A\}}\right)}
+ \frac{4}{\propJKL n} \Log\left(\frac{1}{\Prob\{A\}}\right)
\end{align*}
where
\begin{align*}
E &= \frac{3\kappalemma}{2\sqrt{2\propJKL(1-\propJKL)}} \frac{1}{\sqrt{n}} \int_0^{\sigma}
\sqrt{%
  H_{[\cdot],\dtens}\left(\delta,\model_\indm(\widetilde{s}_\indm,\sigma)\right) \wedge n} \ud \delta 
+ \frac{2(\frac{2}{\propJKL}+\frac{3 \sigma }{2\sqrt{2\propJKL(1-\propJKL)}})}{n} H_{[\cdot],\dtens}\left(\sigma,\model_\indm(\widetilde{s}_\indm,\sigma)\right)
\end{align*}

By Assumption \Hindm, if we assume $\widetilde{s}_{\indm}\in\model_{\indm}$, $\int_0^{\sigma}
\sqrt{
  H_{[\cdot],\dtens}\left(\delta,\model_\indm(\widetilde{s}_\indm,\sigma)\right)
  \wedge n} \ud \delta\leq \phi_\indm(\sigma)$ , 
as well as $\delta\mapsto
H_{[\cdot],\dtens}\left(\delta,\model_\indm(\widetilde{s}_\indm,\sigma)\right)$
is non-increasing. This implies
\begin{align*}
  H_{[\cdot],\dtens}\left(\sigma,\model_\indm(\widetilde{s}_\indm,\sigma)\right)
& \leq \left( \frac{1}{\sigma} \int_0^{\sigma}
\sqrt{H_{[\cdot],\dtens}\left(\delta,\model_\indm(\widetilde{s}_\indm,\sigma)\right)}
\ud \delta \right)^2
\leq \frac{\phi^2_\indm(\sigma)}{\sigma^2}.
\end{align*}
Inserting these bounds in the previous inequality yields
\begin{align*}
  E & \leq \frac{3\kappalemma}{2\sqrt{2\propJKL(1-\propJKL)}}  \frac{\phi_{\indm}(\sigma)}{\sqrt{n}}
   +
\left(\frac{4}{\propJKL}+\frac{3\sigma}{\sqrt{2\propJKL(1-\propJKL)}}\right) 
\frac{\phi^2_\indm(\sigma)}{n\sigma^2} \\
& \leq \left( \frac{3\kappalemma}{2\sqrt{2\propJKL(1-\propJKL)}}  +
\left(\frac{4}{\propJKL}+\frac{3\sigma}{\sqrt{2\propJKL(1-\propJKL)}}\right) 
\frac{\phi_\indm(\sigma)}{\sqrt{n}\sigma^2}
\right)  \frac{\phi_{\indm}(\sigma)}{\sqrt{n}}.
\end{align*}
As $\delta\mapsto\delta^{-1}\phi_\indm(\delta)$ is also
non-increasing, so is
$\delta\mapsto\delta^{-2}\phi_\indm(\delta)$. 
The definition of $\sigma_{\indm}$ can be rewritten as the equation 
$\frac{\phi_\indm(\sigma_{\indm})}{\sqrt{n}\sigma_{\indm}^2}=1$. The
right-hand side of the previous inequality is thus  an $O(\frac{\phi_\indm(\sigma)}{\sqrt{n}})$ as soos
as $\sigma\geq\sigma_\indm$. Indeed under this assumption,
\begin{align*}
  E \leq \left(\frac{3\kappalemma}{2\sqrt{2\propJKL(1-\propJKL)}}+\frac{4}{\propJKL}+\frac{3\sigma}{\sqrt{2\propJKL(1-\propJKL)}}\right) \frac{\phi_\indm(\sigma)}{\sqrt{n}} 
\end{align*}
and
\begin{align*}
\E^A\left[
W_m(\widetilde{s}_{\indm},\sigma)
\right]&\leq 
\left(\frac{3\kappalemma}{2\sqrt{2\propJKL(1-\propJKL)}}+\frac{4}{\propJKL}+\frac{3\sigma}{\sqrt{2\propJKL(1-\propJKL)}}\right) \frac{\phi_\indm(\sigma)}{\sqrt{n}}
 + \frac{21 \sigma }{2\sqrt{\propJKL(1-\propJKL)}\sqrt{n}}
\sqrt{\Log\left(\frac{1}{\Prob\{A\}}\right)}\\
& \qquad
+ \frac{4}{\propJKL n} \Log\left(\frac{1}{\Prob\{A\}}\right)
\end{align*}
Using now $\sigma\leq\sqrt{2}$, we let
$\constlemmaonebis=\left(\frac{3\kappalemma}{2\sqrt{2\propJKL(1-\propJKL)}}+\frac{4}{\propJKL}+\frac{3}{\sqrt{\propJKL(1-\propJKL)}}\right)\leq\left(\frac{81}{2\sqrt{\propJKL(1-\propJKL)}}+\frac{4}{\propJKL}+\frac{3}{\sqrt{\propJKL(1-\propJKL)}}\right)$
as $\kappalemma\leq 27$,
$\constlemmatwobis=\frac{21}{2\sqrt{\propJKL(1-\propJKL)}}$,
so that $\forall \sigma > \sigma_\indm$,
\begin{align*}
\E^A\left[
\sup_{ s_\indm \in\model_\indm(\widetilde{s}_{\indm},\sigma)} \Prec\left(
-\jklsindm  +\jkltildes
\right)
\right]&\leq 
\constlemmaonebis   \frac{\phi_\indm(\sigma)}{\sqrt{n}}
 + \frac{\constlemmatwobis\sigma}{\sqrt{n}}
\sqrt{\Log\left(\frac{1}{\Prob\{A\}}\right)}
+ \frac{4}{\propJKL n} \Log\left(\frac{1}{\Prob\{A\}}\right).
\end{align*}

Thanks to Assumption \Sepindm, we can use the \emph{pealing} lemma (Lemma
4.23 of 
\cite{massart07:_concen}):
\begin{lemma}
Let $\model$ be a countable set, $\widetilde{s} \in \model$ and $a: \model \to \R^+$ such
that $a(\widetilde{s})=\inf_{s\in\model} a(s)$. Let $Z$ be a random process
indexed by $\model$ and let 
\begin{align*}
B(\sigma) = \left\{ s \in \model \middle| a(s)\leq\sigma \right\},
\end{align*}
assume that for any positive $\sigma$ the non-negative random variable
$\sup_{s\in B(\sigma)} \left( Z(s)-Z(\widetilde{s})\right)$ has finite expectation.
Then, for any function $\psi$ on $\R^+$ such that $\psi(x)/x$ is
non-increasing on $\R^+$ and
\begin{align*}
  \E\left[ 
\sup_{s\in B(\sigma)} \left( Z(s)-Z(\widetilde{s})\right)
\right] \leq \psi(\sigma),\quad\text{for any $\sigma\geq\sigma_\star\geq0$,}
\end{align*}
one has for any positive number $x\geq\sigma_\star$
\begin{align*}
\E \left[ \sup_{s\in\model} \frac{Z(s)-Z(\widetilde{s})}{x^2+a^2(s)} \right]
\leq 4x^{-2} \psi(x).
\end{align*}
\end{lemma}

With $\model=\model_\indm$, $\widetilde{s}=\widetilde{s}_\indm \in \model_\indm$ to be
specified with
$a(s)=\dtwotens(\widetilde{s}_\indm,s)$ and $
  Z(s) = -\jkls $. Provided $y_\indm\geq\sigma_\indm$, one
  obtains
\begin{align*}
\E^A \left[
\sup_{s_\indm \in\model_\indm} \Prec\left(
\frac{  -\jklsindm+\jkltildesindm}{y_\indm^2 +
    \dtwotens(\widetilde{s}_\indm,s_{\indm})} \right)
\right]
\leq 4 \constlemmaonebis \frac{\phi_\indm(y_\indm)}{\sqrt{n}y_\indm^2}
 + \frac{4\constlemmatwobis\sigma}{\sqrt{n}y_{\indm}^2}
\sqrt{\Log\left(\frac{1}{\Prob\{A\}}\right)}
+ \frac{16}{\propJKL n y_{\indm}^2} \Log\left(\frac{1}{\Prob\{A\}}\right).
\end{align*}
Now using again the monotonicity of
$\delta\mapsto\delta^{-1}\phi_\indm(\delta)$ and the definition of
$\sigma_\indm$, $\forall y_\indm \geq \sigma_{\indm}$,
\begin{align*}
  \frac{\phi_\indm(y_\indm)}{\sqrt{n} y_\indm} &\leq 
  \frac{\phi_\indm(\sigma_\indm)}{\sqrt{n} \sigma_\indm}
= \sigma_{\indm}
\end{align*}
and therefore
\begin{align*}
\E^A \left[
\sup_{s_\indm \in\model_\indm} \Prec\left(
  \frac{-\jklsindm+\jkltildesindm}{y_\indm^2 +
    \dtwotens(\widetilde{s}_\indm,s_{\indm})} \right)
\right]
\leq \frac{4\constlemmaonebis\sigma_\indm}{y_\indm}
 + \frac{4\constlemmatwobis}{\sqrt{ny_{\indm}^2}}
\sqrt{\Log\left(\frac{1}{\Prob\{A\}}\right)}
+ \frac{16}{\propJKL n y_{\indm}^2} \Log\left(\frac{1}{\Prob\{A\}}\right).
\end{align*}

We can now choose $\widetilde{s}_\indm$ such that for every $s_\indm\in\model_\indm$
\begin{align*}
  \dtwotens(s_0,\widetilde{s}_\indm) \leq (1+\paramepsid) \dtwotens(s_0,s_{\indm})
\end{align*}
so that
\begin{align*}
  \dtwotens(\widetilde{s}_\indm,s_\indm) &= \Ptens\left(\d^2(\widetilde{s}_\indm,s_{\indm})\right)
\leq \Ptens\left( \left(d(\widetilde{s}_\indm,s_0)+d(s_0,s_{\indm})\right)^2\right)
\\
& \leq 2 \Ptens\left( \d^2(\widetilde{s}_\indm,s_0) +
  \d^2(s_0,s_{\indm}) \right) \leq 2(2+\paramepsid) \dtwotens(s_0,s_{\indm}).
\end{align*}
For this choice,  one obtains
\begin{align*}
\E^A \left[
\sup_{s_{\indm}\in\model_\indm} \Prec\left(
  \frac{%
-\jklsindm+\jkltildesindm
}{y_\indm^2 +
    2(2+\paramepsid)\dtwotens(s_0,s_{\indm})} \right)
\right]
&\leq \frac{4\constlemmaonebis\sigma_\indm}{ y_\indm} + 
 + \frac{4\constlemmatwobis}{\sqrt{ny_{\indm}^2}}
\sqrt{\Log\left(\frac{1}{\Prob\{A\}}\right)}\\
& \qquad\qquad
+ \frac{16}{\propJKL n y_{\indm}^2} \Log\left(\frac{1}{\Prob\{A\}}\right)
\end{align*}
which implies
\begin{align*}
\E^A \left[
\Prec\left(
  \frac{%
-\jklhatsindm+\jkltildesindm
}{y_\indm^2 +
    2(2+\paramepsid)\dtwotens(s_0,\widehat{s}_\indm)} \right)
\right]
\leq \frac{4\constlemmaonebis\sigma_\indm}{y_\indm}
 + \frac{4\constlemmatwobis}{\sqrt{ny_{\indm}^2}}
\sqrt{\Log\left(\frac{1}{\Prob\{A\}}\right)}
+ \frac{16}{\propJKL n y_{\indm}^2} \Log\left(\frac{1}{\Prob\{A\}}\right).
\end{align*}

We turn back to the control of $-\Prec(\jkltildesindm)$. Our Berstein
type control yields
\begin{align*}
\E^A \left[
-\Prec(\jkltildesindm)
\right]
\leq
\frac{3}{2 \sqrt{\propJKL(1-\propJKL)}}
\frac{\sqrt{\dtwotens(s_0,\widetilde{s}_\indm)}}{\sqrt{n}}
\sqrt{\Log\left(\frac{1}{\Prob\{A\}}\right)} + \frac{2}{\propJKL n}
\Log\left(\frac{1}{\Prob\{A\}}\right)
\end{align*}
or for any $y_\indm >0 $ and any $\paramB > 0$:
\begin{align*}
\E^A \left[ 
\frac{-\Prec(\jkltildesindm)}{y_\indm^2+\paramB^2\dtwotens(s_0,\widetilde{s}_\indm)}
\right]
&\leq
\frac{1}{y_\indm^2+\paramB^2\dtwotens(s_0,\widetilde{s}_\indm)}
\frac{3}{2 \sqrt{\propJKL(1-\propJKL)}}
\frac{\sqrt{\dtwotens(s_0,\widetilde{s}_\indm)}}{\sqrt{n}}
\sqrt{\Log\left(\frac{1}{\Prob\{A\}}\right)}\\
&\qquad\qquad + \frac{1}{y_\indm^2+\paramB^2\dtwotens(s_0,\widetilde{s}_\indm)}
\frac{2}{n\propJKL}
\Log\left(\frac{1}{\Prob\{A\}}\right)\\
& \leq \frac{3}{4\paramB \sqrt{\propJKL(1-\propJKL)}}
\frac{1}{\sqrt{ny_m^2}}
\sqrt{\Log\left(\frac{1}{\Prob\{A\}}\right)} + \frac{2}{\propJKL ny_m^2}
\Log\left(\frac{1}{\Prob\{A\}}\right).
\end{align*}
We derive thus
\begin{align*}
  \E^A &\left[
\Prec\left(
  \frac{-\jklhatsindm + \jkltildesindm}{y_\indm^2 +
    2(2+\paramepsid)\dtwotens(s_0,\widehat{s}_\indm)} \right)
+
\frac{-\Prec(\jkltildesindm)}{y_\indm^2+\paramB^2\dtwotens(s_0,\widetilde{s}_\indm)}
\right]\\
& \leq 
\frac{4\constlemmaonebis\sigma_\indm}{y_\indm}
+
\left( 4\constlemmatwobis + \frac{3}{4\paramB
    \sqrt{\propJKL(1-\propJKL)}} \right)
\frac{1}{\sqrt{ny_m^2}}
\sqrt{\Log\left(\frac{1}{\Prob\{A\}}\right)} + \frac{18}{\propJKL ny_m^2}
\Log\left(\frac{1}{\Prob\{A\}}\right)
\end{align*}
Let $\paramBd$ such that $\paramBd^2=2(2+\paramepsid)/(1+\paramepsid)$, using 
$\dtwotens(s_0,\widehat{s}_\indm) \geq
\dtwotens(s_0,\widetilde{s}_\indm)/(1+\paramepsid)$,
\begin{align*}
\Prec\left(
  \frac{-\jklhatsindm+\jkltildesindm}{y_\indm^2 +
    2(2+\paramepsid)\dtwotens(s_0,\widehat{s}_\indm)} \right)
+
\frac{-\Prec(\jkltildesindm)}{y_\indm^2+\paramBd^2\dtwotens(s_0,\widetilde{s}_\indm)}
& \geq   \Prec\left(
  \frac{-\jklhatsindm}{y_\indm^2 +
    2(2+\paramepsid)\dtwotens(s_0,\widehat{s}_\indm)} \right)
\end{align*}
and thus
\begin{align*}
  \E^A &\left[
\Prec\left(
  \frac{-\jklhatsindm}{y_\indm^2 +
    \constlemmathree\dtwotens(s_0,\widehat{s}_\indm)} \right)
\right]
& \leq 
 \frac{\constlemmaone \sigma_\indm}{y_\indm}
+ \constlemmatwo
\frac{1}{\sqrt{ny_m^2}}
\sqrt{\Log\left(\frac{1}{\Prob\{A\}}\right)} + \frac{18}{ny_m^2\propJKL}
\Log\left(\frac{1}{\Prob\{A\}}\right).
\end{align*}
where
$\constlemmathree = 2(2+\paramepsid)/(1+\paramepsid)$,
$\constlemmaone=4 \constlemmaonebis$ and
$\constlemmatwo=4\constlemmatwobis+3/(4\sqrt{\propJKL(1-\propJKL)}\paramBd)$.
\end{proof}

\ifthenelse{\boolean{extended}}{
\subsection{Behavior of the constants of \cref{theo:single} and \cref{theo:select}}

We now explain the behavior of the constants $\constpenmin$ and
$\constoratwo$ with respect to $\constoraone$ and $\propJKL$. As shown
in the proof, if we let $\paramepsipen=1-\frac{1}{\constoraone}$ then
 $\constoraone=\frac{1}{1-\paramepsipen}$ and
$\constoratwo=\frac{\constpenmin}{1-\paramepsipen}= \constpenmin
\constoraone$ so that it suffices to study the behavior of
$\constpenmin$.

Now $\constpenmin$ is defined as equal to
$\frac{\ConstHellKL_{\propJKL}\paramepsipen\constcoreepsipen^2}{\constlemmathree}$
with
$\constcoreepsipen$ the root of
$\left( \frac{\constlemmaone+\constlemmatwo}{\constcoreepsipen} + \frac{18}{\constcoreepsipen^2\propJKL}
\right)\constlemmathree
=\ConstHellKL_{\propJKL}\, \paramepsipen$ where we use the constants
appearing in \cref{lemma:core}. This implies
\begin{align*}
  \constpenmin&=\frac{\ConstHellKL_{\propJKL}\paramepsipen\constcoreepsipen^{2}}{\constlemmathree}
= \constcoreepsipen^{2}\left( \frac{\constlemmaone+\constlemmatwo}{\constcoreepsipen} +\frac{18}{\constcoreepsipen^2\propJKL}
    \right)
= \constcoreepsipen\left(\constlemmaone+\constlemmatwo\right) +\frac{18}{\propJKL}.
\end{align*}

Solving the implied quadratic equation 
$\constcoreepsipen (\constlemmaone+\constlemmatwo) + \frac{18}{\propJKL} =
    \constcoreepsipen^2 \frac{\ConstHellKL_{\propJKL}\paramepsipen}{\constlemmathree}$
yields
\begin{align*}
\constcoreepsipen = \frac{\constlemmathree
  (\constlemmaone+\constlemmatwo)
\left(
\sqrt{1 + \frac{72 \ConstHellKL_{\propJKL}\paramepsipen}{\propJKL\constlemmathree
  (\constlemmaone+\constlemmatwo)^2}} +1
\right)
}
{
2\ConstHellKL_{\propJKL}\paramepsipen
 }
\end{align*}
and thus
\begin{align*}
  \constpenmin &
= \frac{\constlemmathree
  (\constlemmaone+\constlemmatwo)^2
\left(
\sqrt{1 + \frac{72 \ConstHellKL_{\propJKL}\paramepsipen}{\propJKL\constlemmathree
  (\constlemmaone+\constlemmatwo)^2}} +1
\right)
}
{
2\ConstHellKL_{\propJKL}\paramepsipen
 }+ \frac{18}{\propJKL}
\end{align*}
Now
\begin{align*}
  \constlemmaone&=4\constlemmaonebis=
4
\left(
    \frac{3\kappalemma}{2\sqrt{2\propJKL(1-\propJKL)}}+
      \frac{4}{\propJKL} + \frac{3}{\sqrt{\propJKL(1-\propJKL)}}
  \right)
= \frac{1}{\sqrt{\propJKL(1-\propJKL)}} \left( 3\kappalemma\sqrt{2} +
  12 + 16
  \sqrt{\frac{1-\propJKL}{\propJKL}} \right)
\\
\intertext{and using that for any $\paramepsiconstlemmathree>0$, once $\paramepsid$ is small
enough, 
$2 > \paramBd \geq 2(1-\paramepsiconstlemmathree) $}
\constlemmatwo & = 4 \constlemmatwobis +
\frac{3}{4\sqrt{\propJKL(1-\propJKL)}\paramBd}\leq 
\frac{42}{\sqrt{\propJKL(1-\propJKL)}}
+\frac{3}{8\sqrt{\propJKL(1-\propJKL)}(1-\paramepsiconstlemmathree)}
= \frac{1}{\sqrt{\propJKL(1-\propJKL)}} \left( 42 + \frac{3}{8(1-\paramepsiconstlemmathree)}\right)
\end{align*}
so that
\begin{align*}
  \left(\constlemmaone + \constlemmatwo\right)^2 \leq
\frac{1}{\propJKL(1-\propJKL)} \left( 3 \kappalemma \sqrt{2} +
  54+  \frac{3}{8
  (1-\paramepsiconstlemmathree)} +  16 \sqrt{\frac{1-\propJKL}{\propJKL}} \right)^2.
\end{align*}
Now using $4 < \constlemmathree\leq 4(1+\paramepsiconstlemmathree)$
\begin{align*}
\constpenmin & \leq 
\frac{4( 1+ \paramepsiconstlemmathree) 
\left( 3 \kappalemma \sqrt{2} +
  54+  \frac{3}{8
  (1-\paramepsiconstlemmathree)} +  16 \sqrt{\frac{1-\propJKL}{\propJKL}} \right)^2
\left(
\sqrt{1 + \frac{72 \ConstHellKL_{\propJKL}\paramepsipen}{\propJKL\constlemmathree
  (\constlemmaone+\constlemmatwo)^2}} +1
\right)
}
{
2 \propJKL(1-\propJKL) \ConstHellKL_{\propJKL}\paramepsipen
 }+ \frac{18}{\propJKL}\\
& \leq
\frac{1}{\ConstHellKL_{\propJKL}\propJKL(1-\propJKL)\paramepsipen}\\
& \qquad\times
\left(
2( 1+ \paramepsiconstlemmathree)
\left( 3 \kappalemma \sqrt{2} +
  54+  \frac{3}{8
  (1-\paramepsiconstlemmathree)} +  16
\sqrt{\frac{1-\propJKL}{\propJKL}} \right)^2
\left(
\sqrt{1 + \frac{72 \ConstHellKL_{\propJKL}\paramepsipen}{\propJKL\constlemmathree
  (\constlemmaone+\constlemmatwo)^2}} +1
\right)
+
18 \ConstHellKL_{\propJKL}(1-\propJKL)\paramepsipen
 \right)
\end{align*}

This implies that $\constpenmin$ scales when $\propJKL$ is close to $1$
proportionally to
\begin{align*}
\frac{1}{\ConstHellKL_{\propJKL}\propJKL(1-\propJKL)\paramepsipen}=
\frac{\propJKL}{(1-\propJKL)^2\left( \Log \left(1 + \frac{\propJKL}{1-\propJKL}\right) -
  \propJKL\right) \paramepsipen}
\end{align*}
and thus explodes when $\propJKL$ goes to $1$ as well as when
$\paramepsipen$ goes to $0$.

Note that, as it is almost always the case in density estimation,
 these constants are rather large, mostly because of the
crude constant appearing in \cref{theo:deviation}. Indeed let $\sigma_{\indmset}$  
denote the supremum over all models of the collection,
the right hand side of the previous bound on $\constpenmin$ can already
be replaced by
\begin{align*}
&\frac{1}{\ConstHellKL_{\propJKL}\propJKL(1-\propJKL)\paramepsipen}\\
& \quad\times
\left(
2( 1+ \paramepsiconstlemmathree)
\left( 3 \kappalemma \sqrt{2} + 42  + 6 \sqrt{2} \sigma_{\indmset
} +  \frac{3}{8
  (1-\paramepsiconstlemmathree)} +  16
\sqrt{\frac{1-\propJKL}{\propJKL}} \right)^2
\left(
\sqrt{1 + \frac{72 \ConstHellKL_{\propJKL}\paramepsipen}{\propJKL\constlemmathree
  (\constlemmaone+\constlemmatwo)^2}} +1
\right)
+
18 \ConstHellKL_{\propJKL}(1-\propJKL)\paramepsipen
 \right).
\end{align*}
which
is much smaller than the previous quantity as soon as
$\sigma_{\indmset}$ is much smaller than $\sqrt{2}$, which can be
ensured in the models of \ifthenelse{\boolean{extended}}{\cref{sec:emphp-const-cond}}{\cref{sec:an-appl-spat}} provided we limit
their maximum dimension well below $n$, for instance to $n/\Log^2(n)$.}{}


\ifthenelse{\boolean{extended}}{
\section{Proof for \cref{sec:covar-part-cond}
  (\nameref{sec:covar-part-cond})}


\begin{proof}[Proof of \cref{prop:codingpolpart}]
We start by the UDP case, as we stop as soon as $  \frac{2^\dimX}{n} >
2^{-\dimX J} \leq \frac{1}{n}$, $J \leq \frac{\Log n}{\dimX \Log 2}$
and thus there is at most $1 + \frac{\Log n}{\dimX \Log 2}$ different
partitions in the collection, which allows to prove the proposition in
this case.

Proofs for the RDP, RSDP and RSP cases are handled
simultaneously. Indeed all these partition collections are recursive
partition collections and thus correspond to tree structures.
More precisely, any RDP can be represented by a $2^\dimX$-ary tree in
which a node has value $0$ if it has no child or value
$1$ otherwise. Similarly, any RSDP (respectively RSP) can be
represented by a dyadic tree in which a node has  value $0$ if it
has no child or $1$ plus the number of the dimension of the split
(respectively $1$ plus the number of the dimension and the position
of the split). Such a tree can be encoded by an ordered list of the
values of its nodes. The total length of the code is thus given by the
product of the
number of nodes $N(\PartX)$ by their encoding cost (respectively
$\left\lceil \frac{\Log
  2}{\Log 2} \right\rceil$ bits, $\left\lceil\frac{\Log
(1+ \dimX)}{\Log 2}\right\rceil$ bits and $\left\lceil\frac{\Log
(1+ \dimX)}{\Log 2} \right\rceil  + \left\lceil \frac{\Log n}{\Log 2}\right\rceil $).
As this code is decodable, it satisfies the Kraft inequality and thus,
using the definition of $\PartCstB^{\starX}$,
\begin{align*}
  \sum_{\PartX \in \CollPart^{\starX}} 2^{-N(\PartX) \frac{\PartCstB^{\starX}}{\Log 2}}
\leq 1 \Leftrightarrow \sum_{\PartX \in \CollPart^{\starX}} e^{- N(\PartX)
  \PartCstB^{\starX}}
\leq 1.
\end{align*}

It turns out that the number of nodes $N(\PartX)$ can be computed from
the number of hyperrectangles of the partition $\NbPartX$, which is also the
number of leaves in the tree. Indeed, each inner node has exactly
$2^\dimX$ children in the RDP case and only $2$ in the RDSP and RSP
case, while, in all cases, every node but the root has a single
parent. Let
$\dimPart=\dimX+\dimY$ in the RDP case and $d=1$ in the RDSP and RSP case then
$2^\dimPart (N(\PartX)-\NbPartX) = N(\PartX) - 1$ and thus 
\begin{align*}
N(\PartX) = \frac{2^\dimPart
  \NbPartX-1}{2^\dimPart-1} = \frac{2^\dimPart}{2^\dimPart-1} \NbPartX +
\left( 1 - \frac{2^\dimPart}{2^\dimPart-1} \right) = \PartCstc^{\starX} \NbPartX + (1
- \PartCstc^{\starX})
\end{align*}
with $\PartCstc^{\starX}$ as defined in the proposition. Plugging this in the Kraft inequality
 leads to
 \begin{align*}
   \sum_{\PartX \in \CollPart^{\starX}} e^{- \PartCstc^{\starX} \PartCstB^{\starX} \NbPartX + \PartCstB^{\starX} (\PartCstc^{\starX}-1)}
\leq 1 \Leftrightarrow
   \sum_{\PartX \in \CollPart^{\starX}} e^{- \PartCstc^{\starX} \PartCstB^{\starX} \NbPartX} \leq e^ {\PartCstB^{\starX} (1-\PartCstc^{\starX})}.
 \end{align*}
Let now $c\geq \PartCstc^{\starX}$,
\begin{align*}
  \sum_{\PartX \in \CollPart^{\starX}} e^{- c \PartCstB^{\starX} \NbPartX} & \leq \sum_{\PartX
    \in \CollPart^{\starX}} e^{- (c-\PartCstc^{\starX}) \PartCstB^{\starX} \NbPartX} e^{- \PartCstc^{\starX} \PartCstB^{\starX}
    \NbPartX}\\
\intertext{and as $\NbPartX \geq 1$}
& \leq e^{- (c-\PartCstc^{\starX}) \PartCstB^{\starX}  } \sum_{\PartX
    \in \CollPart^{\starX}}  e^{- \PartCstc^{\starX} \PartCstB^{\starX}
    \NbPartX}\\
& \leq e^{- (c-\PartCstc^{\starX}) \PartCstB^{\starX} } e^ {(1-\PartCstc^{\starX})\PartCstB^{\starX} } = e^{\PartCstB^{\starX}}
e^{-c\PartCstB^{\starX}} = \PartCstSigma^{\starX} e^{-c\PartCstC^{\starX}}
\end{align*}
which concludes these three cases.

For the HRP cases, it is sufficient to give the uppermost coordinate
of the hyperrectangles ordered in a uniquely decodable way based on
the following observation: assume we have a current list of hyperrectangles, the complementary of
the union of these hyperrectangles is either empty if the list contains
all the hyperrectangles of the partition or contains a lowermost point
that is the lowermost corner of a unique hyperrectangle. Furthermore, this
hyperrectangle is completely specified by its uppermost corner
coordinates. Starting with an empty list, an HRP partition can thus be
entirely specified by the list of uppermost corner coordinates
obtained through this scheme.

 This leads to a code with $\NbPartX \times \dimX \left\lceil \frac{\Log
  n}{\Log 2} \right\rceil$ bits for each partition that satisfies a Kraft
inequality
\begin{align*}
  \sum_{\PartX \in \CollPart^{HRP}} 2^{-\NbPartX \frac{\PartCstB^{\HRPX}}{\Log 2}}
\leq 1 \Leftrightarrow \sum_{\PartX \in \CollPart^{\HRPX}} e^{- \PartCstc^{\HRPX} \PartCstB^{\HRPX} \NbPartX }
\leq 1
\end{align*}
Now for any $c\geq \PartCstc^{\HRPX}$,
\begin{align*}
\sum_{\PartX \in \CollPart^{\HRPX}} e^{- c \PartCstB^{\HRPX} \NbPartX }
& = \sum_{\PartX \in \CollPart^{\HRPX}} e^{- (c-\PartCstc^{\HRPX}) \PartCstB^{\HRPX} \NbPartX } e^{-\PartCstc^{\HRPX}
  \PartCstB^{\HRPX} \NbPartX}\\
& \leq e^{- (c-\PartCstc^{\HRPX}) \PartCstB^{\HRPX}} \sum_{\PartX \in \CollPart^{\HRPX}} e^{- \PartCstc^{\HRPX}
  \PartCstB^{\HRPX} \NbPartX }\\
& \leq e^{- (c-\PartCstc^{\HRPX}) \PartCstB^{\HRPX}} =
e^{\PartCstB^{\HRPX}} e^{-c\PartCstB^{\HRPX}}
 = \PartCstSigma^{\HRPX} e^{-c\PartCstC^{\HRPX}}.
\end{align*}

It is then only a matter of calculation to check that if $c$ is larger than $1$
in the UDP and RDP cases and larger than $2\Log 2$ in the other cases
then all these sums can be bounded by $1$.

\end{proof}
}{}

\section{Proof for \cref{sec:piec-polyn-cond} (\nameref{sec:piec-polyn-cond})}

\Cref{theo:polypart} is obtained by proving that Assumption ($\mathrm
{H}_{\PartXY,\dimPYmax}$) and ($\text{S}_{\PartXY,\dimPYmax}$) hold for any model
$\model_{\PartXY,\dimPYmax}$ while Assumption (K) holds for any model collection.
\Cref{theo:polypart}  is then a consequence of \cref{theo:select}.

One easily verifies that Assumption ($\text{S}_{\PartXY,\dimPYmax}$)
holds whatever the partition choice. Concerning the first assumption, 
\begin{proposition}
\label{prop:polypartH}
Under the assumptions of \cref{theo:polypart}, there exists a
$\constpolyparttwobis$ such that for any model
$\model_{\PartXY,\dimPYmax}$ Assumption
($\text{H}_{\PartXY,\dimPYmax}$) is satisfied with a function $\phi$
such that
\begin{align*}
 \DIMH_{\PartXY,\dimPYmax}\leq \left(\constpolyparttwo + \Log
  \frac{n^2}{\NbPartXY}\right)
\DimH_{\PartXY,\dimPYmax}
\end{align*}
with $\constpolyparttwo=2 \constpolyparttwobis + 2 \pi$.
\end{proposition}
The proof relies on the combination of \cref{prop:proporsimple}
and 
\begin{proposition}
\label{prop:dimpol}
$\forall \model_{\PartXY,\dimPYmax}, \forall s_{\PartXY,\dimPYmax} \in \model_{\PartXY,\dimPYmax}$,
  \begin{align*}
    H_{[\cdot],\dtens} 
\left(\delta, \model_{\PartXY,\dimPYmax}(s_{\PartXY,\dimPYmax},\sigma)
\right)
& \leq \DimH_{{\PartXY,\dimPYmax}}  
\left( \frac{1}{2}\Log \frac{n^2}{\NbPartXY} +
  \constpolyparttwobis
 + \Log
  \frac{\sigma}{\delta} \right).
\end{align*}
\end{proposition}
 \ifthenelse{\boolean{extended}}{Remark that we also use
 the inequality
\begin{align*}
\left( \sqrt{\frac{1}{2}\Log \frac{n^2}{\sum_{\LeafXl \in \PartX} \NbPartYLeafXl} +
  \constpolyparttwobis} + \sqrt{\pi} \right)^2 \leq 
\Log \frac{n^2}{\NbPartXY} + 2
\constpolyparttwobis + 2\pi.
\end{align*}}{}

By using \cref{prop:codingpolpart} for  both $\PartX$
and $\PartY$, we obtain
the Kraft type assumption:
\begin{proposition}
\label{prop:polypartK}
Under the assumptions of \cref{theo:polypart}, for any collection
$\Models$, there exists a $\constpolypartthree>0$ such that for
\begin{align*}
  x_{\PartXY,\dimPYmax} =  \constpolypartthree \left( \PartCstA^{\starX}
      +  \left( \PartCstB^{\starX}  + \PartCstA^{\starY}
      \right)\NbPartX 
      + \PartCstB^{\starY} \sum_{\LeafXl\in\Part}
      \NbPartYLeafXl \right)
\end{align*}
Assumption (K) is satisfied with
\begin{alignI}
  \sum_{\model_{\PartXY,\dimPYmax}\in\Models}
  e^{-x_{\PartXY,\dimPYmax}} \leq 1.
\end{alignI}
\end{proposition}
\ifthenelse{\boolean{extended}}{The complete proof  is
  postponed after the one \cref{prop:dimpol}.}
{Its complete proof can be found in the 
technical report~\cite{cohen11:_condit_densit_estim_penal_app}.}

\subsection{Proof of \cref{prop:dimpol}}

\ifthenelse{\boolean{extended}}{
For sake of simplicity, we remove from now on the subscript
reference to the common measure $\meas$ from all notations.}

We rely on a link between
$\|\cdot\|_2$ and $\|\cdot\|_\infty$ structures of the square roots of the
models and a relationship between bracketing entropy and metric
entropy for $\|\cdot\|_\infty$  norms.

 Following \textcite{massart07:_concen}, we define the
following tensorial \emph{norm} on functions $u(y|x)$
\begin{align*}
  \|u\|_{2}^{2\tens} = \E \left[ \frac{1}{n} \sum_{i=1}^n
    \|u(\cdot|X_i)\|_2^2 \right]
\quad\text{and}\quad
  \|u\|_{\infty}^{2,\tens} = \E \left[ \frac{1}{n} \sum_{i=1}^n
    \|u(\cdot|X_i)\|_{\infty}^2 \right].
\end{align*}
As the reference measure is the Lebesgue measure on $[0,1]^\dimY$,
$\|u\|_{\infty}^{2\tens} \geq \|u\|_{2}^{2\tens}$.
By definition $\dtens(s,t)=\|\sqrt{s}-\sqrt{t}\|_{2}^{\tens}$ and thus
for any model $\model_{\indm}$ and any function $s_\indm \in \model_{\indm}$
\begin{align*}
  H_{[\cdot],\dtens}(\delta,\model_{\indm}(s_\indm,\sigma))
& = H_{[\cdot],\|.\|_{2}^{\tens}} 
\left(\delta, \left\{ 
u \in \sqrt{\model_\indm} \middle| \|u-\sqrt{s_\indm}\|_2^{\tens} \leq \sigma
 \right\} \right) 
\end{align*}
If
$\sqrt{\model_{\indm}}$ is a subset of a linear space
$\widebar{\sqrt{\model_{\indm}}}$ of dimension
$\DimH_{\indm}$, as in our model,
\begin{align*}
  H_{[\cdot],\dtens}(\delta,\model_{\indm}(s_\indm,\sigma))
&\leq H_{[\cdot],\|.\|_{2}^{\tens}} 
\left(\delta, \left\{ 
u \in \widebar{\sqrt{\model_\indm}} \middle| \|u-\sqrt{s_\indm}\|_2^{\tens} \leq \sigma
 \right\} \right) 
\intertext{so that one can replace, without loss of generality,
  $\sqrt{s_\indm}$ by $0$ and use}
  H_{[\cdot],\dtens}(\delta,\model_{\indm}(s_\indm,\sigma))
&\leq H_{[\cdot],\|.\|_{2}^{\tens}} 
\left(\delta, \left\{ 
u \in \widebar{\sqrt{\model_\indm}} \middle| \|u\|_2^{\tens} \leq \sigma
 \right\} \right).\\
\intertext{Using now $\|\cdot\|_{\infty}^{\tens} \geq
\|\cdot\|_{2}^{\tens}$, one deduces}
  H_{[\cdot],\dtens}(\delta,\model_{\indm}(s_\indm,\sigma))
& \leq  H_{[\cdot],\|.\|_{\infty}^{\tens}} 
\left(\delta, \left\{ 
u \in \widebar{\sqrt{\model_\indm}}\middle| \|u\|_2^{\tens} \leq \sigma
 \right\} \right).
\end{align*}
As for any $u$, $[u-\delta/2, u+ \delta/2]$ is a
$\delta$-bracket for the $\|\cdot\|_{\infty}^{\tens}$ norm, any
covering of $\left\{ 
u \in \widebar{\sqrt{\model_\indm}}\middle| \|u\|_2^{\tens} \leq \sigma
 \right\}$ by $\|\cdot\|_{\infty}^{\tens}$ ball of radius $\delta/2$
 yields a covering by the corresponding brackets.
This implies 
\begin{align*}
  H_{[\cdot],\dtens}(\delta,\model_{\indm}(s_\indm,\sigma))
& \leq  H_{\|.\|_{\infty}^{\tens}} 
\left(\frac{\delta}{2}, \left\{ 
u \in \widebar{\sqrt{\model_\indm}}\middle| \|u\|_2^{\tens} \leq \sigma
 \right\} \right)
\end{align*}
where $H_{d}(\delta,\setmodel)$, the classical entropy, is defined as
the logarithm of the minimum number of ball of radius $\delta$ with
respect to norm $d$ covering the set $\setmodel$.

The following proposition, proved in
  next section, is similar to
a proposition of
\textcite{massart07:_concen}. It provides a bound for this last
entropy term under an
assumption on a link
between $\|\cdot\|_{\infty}^{2\tens}$ and $\|\cdot\|_{2}^{2\tens}$ structures:
\begin{proposition}
\label{prop:inftyentropy}
For any  basis $\{\phi_k\}_{1\leq k \leq \DimH_{\indm}}$ of
$\widebar{\sqrt{\model_{\indm}}}$ such that 
\begin{align*}
\forall \beta \in \R^{\DimH_{\indm}},\quad
 \| \sum_{k=1}^{\DimH_{\indm}} \beta_k
    \phi_k\|_{2}^{2\tens} \geq \|\beta\|_2^2,
\end{align*}
 let
\begin{align*}
  \widebar{r}_{\indm}(\{\phi_k\}) = \sup_{\sum_{k=1}^{\DimH_{\indm}} \beta_k
    \phi_k \neq 0 }
  \frac{1}{\sqrt{\DimH_{\indm}}} \frac{\| \sum_{k=1}^{\DimH_{\indm}} \beta_k
    \phi_k\|_{\infty}^{\tens}}{\|\beta\|_{\infty}}.
\end{align*}
and let $\widebar{r}_{\indm}$ be the infimum over all
suitable bases.

Then $\widebar{r}_{\indm}\geq 1$ and
  \begin{align*}
    H_{\|.\|_{\infty}^{\tens}} 
\left(\frac{\delta}{2}, \left\{ 
u \in \widebar{\sqrt{\model_\indm}} \middle| \|u\|_2^{\tens} \leq \sigma
 \right\} \right)
& \leq \DimH_{\indm} \left( \ConstMultH_{\indm} +  \Log
  \frac{\sigma}{\delta} \right)
  \end{align*}
with $\ConstMultH_{\indm} = \Log \left( \kappa_{\infty}
  \widebar{r}_{\indm}\right) $ and $\kappa_{\infty}\leq 2\sqrt{2\pi
  e}$.
\end{proposition}

In our setting, using a basis of Legendre polynomials, we are able to derive
from \cref{prop:inftyentropy}
\begin{proposition}
\label{prop:entropypoly}
For any model of \cref{sec:piec-polyn-cond},
\begin{align*}
  \widebar{r}_{{\PartXY,\dimPYmax}}
& \leq 
 \prod_{d=1}^{\dimY} 
\left(
\sqrt{ \dimPYmaxi{d} + 1}
\sqrt{2 \dimPYmaxi{d} +1}
\right)
\sup_{\LeafXYlk \in \PartXY}
  \frac{1}
{
\sqrt{\NbPartXY}
\sqrt{|\LeafXYlk|}
}
\end{align*}
so that $\forall s_{\PartXY,\dimPYmax} \in \model_{\PartXY,\dimPYmax}$,
  \begin{align*}
    H_{[\cdot],\dtens} 
\left(\delta, \model_{\PartXY,\dimPYmax}(s_{\PartXY,\dimPYmax},\sigma)
\right)
& \leq \DimH_{{\PartXY,\dimPYmax}}  \left( \ConstMultH_{{\PartXY,\dimPYmax}} + \Log
  \frac{\sigma}{\delta} \right)
  \end{align*}
with $\ConstMultH_{{\PartXY,\dimPYmax}} = \Log \left( \kappa_{\infty}
  \widebar{r}_{{\PartXY,\dimPYmax}}\right) $ and $\kappa_{\infty}\leq 2\sqrt{2\pi
  e}$.
\end{proposition}
\ifthenelse{\boolean{extended}}{}{
A proof, essentially computational, can be found in our technical report~\cite{cohen11:_condit_densit_estim_penal_app}.}
One easily verifies that
\begin{align*}
\sup_{\LeafXYlk \in \PartXY}
  \frac{1}
{
\sqrt{\NbPartXY}
\sqrt{|\LeafXYlk|}
}
\leq
&  \begin{cases} 1 & \text{{
if  all hyperrectangles have same sizes}}\\
\sqrt{\frac{n^2}{\NbPartXY}}
& \text{otherwise}.
\end{cases}
\end{align*}
Remark that when $\starX=\UDPX$,
$\starY=\UDPY$ and $\PartYLeafXl$ is independent of
$\LeafXl$, all the hyperrectangles have same sizes and that the $n^2$ corresponds to the arbitrary limitation imposed on the
minimal size of the segmentations. If we limit this minimal size to
$\frac{1}{\sqrt{n}}$ instead of $\frac{1}{n}$ this factor becomes $n$.

Let 
\begin{align*}
\constpolyparttwobis=\Log \left( \kappa_{\infty}   \prod_{k=1}^{\dimY} \left( \sqrt{\dimPYmaxi{k}+1}\sqrt{2
     \dimPYmaxi{k}+1} \right) \right)
\end{align*}
we have slightly more than \cref{prop:dimpol} as
$\forall s_{\PartXY,\dimPYmax} \in \model_{\PartXY,\dimPYmax}$,
  \begin{align*}
    H_{[\cdot],\dtens} 
\left(\delta, \model_{\PartXY,\dimPYmax}(s_{\PartXY,\dimPYmax},\sigma)
\right)
& \leq \DimH_{{\PartXY,\dimPYmax}}  
\begin{cases}
\left( \constpolyparttwobis
 + \Log
  \frac{\sigma}{\delta} \right) & \text{for the same
  size case}\\
\left( \frac{1}{2}\Log \frac{n^2}{\NbPartXY} +
  \constpolyparttwobis
 + \Log
  \frac{\sigma}{\delta} \right) & \text{otherwise}
\end{cases}
\end{align*}

\subsection{Proofs of \cref{prop:inftyentropy} and
  \cref{prop:entropypoly}}


\begin{proof}[Proof of \cref{prop:inftyentropy}]
Let $(\phi_k)_{1 \leq k \leq \DimH_{\indm}}$ be a basis of
$\widebar{\sqrt{\model_\indm}}$ satisfying
\begin{align*}
  \forall \beta \in \R^{\DimH_{\indm}},\quad
\left\| \sum_{k=1}^{\DimH_{\indm}} \beta_k
    \phi_k\right\|_{2}^{2,\tens} \geq \|\beta\|_2^2.
\end{align*}
Note that for $\beta$ defined by $\forall 1 \leq k \leq \DimH_{\indm}, \beta_k=1$
\begin{align*}
\left\| \sum_{k=1}^{\DimH_{\indm}} \beta_k
    \phi_k\right\|_{\infty}^{2,\tens} \geq  \left\| \sum_{k=1}^{\DimH_{\indm}} \beta_k
    \phi_k\right\|_{2}^{2,\tens} \geq \|\beta\|_2^2 = \DimH_{\indm} =
    \DimH_{\indm} \|\beta\|_ {\infty}^2
\end{align*}
so that $\widebar{r}_{\indm}(\phi)\geq 1$.

Let the grid $\Grid_{\indm}(\delta,\sigma)$:
\begin{align*}
  \left\{ \beta \in \R^{\DimH_{\indm}}\middle|
\, \forall 1 \leq k \leq \DimH_{\indm}, \beta_k \in
\frac{\delta}{\sqrt{\DimH_{\indm}} \widebar{r}_{\indm}(\phi)} \Z 
\  \text{and}\ \min_{\beta', \|\beta'\|_2 \leq
      \sigma} \|\beta-\beta'\|_{\infty} \leq \frac{\delta}{2\sqrt{\DimH_{\indm}} \widebar{r}_{\indm}(\phi)}
\right\}.
\end{align*}
By definition, for any $u' \in \widebar{\sqrt{\model_\indm}}$ such
that $\|u'\|_{2}^{\tens}\leq \sigma$ there is a $\beta'$ such that
 $u'=\sum_{k=1}^{\DimH_{\indm}} \beta'_k
\phi_k$ and $\|\beta'\|_2 \leq \sigma$. By construction, there is a
$\beta \in \Grid_{\indm}(\delta,\sigma)$ such that
\begin{align*}
  \|\beta-\beta'\|_{\infty} \leq \frac{\delta}{2\sqrt{\DimH_{\indm}} \widebar{r}_{\indm}(\phi)}.
\end{align*}
Definition of $\widebar{r}_{\indm}$ implies then that
\begin{align*}
 \left\| \sum_{k=1}^{\DimH_{\indm}} \beta_k
    \phi_k - 
\sum_{k=1}^{\DimH_{\indm}} \beta'_k
    \phi_k
\right\|_{\infty}^{\tens} & \leq \widebar{r}_{\indm}(\phi) \sqrt{\DimH_{\indm}}
\|\beta-\beta'\|_{\infty}\\
& \leq \frac{\delta}{2}.
\end{align*}
The set $\left\{ \sum_{k=1}^{\DimH_{\indm}} \beta_k
    \phi_k \middle| \beta \in \Grid_{\indm}(\delta,\sigma) \right\}$ is
    thus a $\frac{\delta}{2}$ covering of $\left\{ 
u \in \widebar{\sqrt{\model_\indm}} \middle| \|u\|_2^{\tens} \leq \sigma
 \right\}$ for the $\|\cdot\|_{\infty}^{\tens}$ norm. It remains thus only
 to bound the cardinality of $\Grid_{\indm}(\delta,\sigma)$.

Let $\widebar{\Grid_{\indm}(\delta,\sigma)}$ be the union of all hypercubes of width
$\frac{\delta}{\sqrt{\DimH_{\indm}}\widebar{r}_{\indm}(\phi)}$ centered on the
grid $\Grid_{\indm}(\delta,\sigma)$, by construction, for any
$\beta\in\widebar{\Grid_{\indm}(\delta,\sigma)}$ there is a $\beta'$
with $\|\beta'\|_2 \leq \sigma$ such that $\|\beta'-\beta\|_{\infty} \leq 
\frac{\delta}{\sqrt{\DimH_{\indm}}\widebar{r}_{\indm}(\phi)}$.
As $\|\beta'-\beta\|_{2} \leq \sqrt{\DimH_{\indm}}
\|\beta'-\beta\|_{\infty}$, this implies $\|\beta\|_{2} \leq \sigma + \frac{\delta}{\widebar{r}_{\indm}(\phi)}$.
We then deduce
\begin{align*}
  \text{Vol}\left(\widebar{\Grid_{\indm}(\delta,\sigma)}\right)
= \left| \Grid_{\indm}(\delta,\sigma) \right|
\left(
\frac{\delta}{\sqrt{\DimH_{\indm}}\widebar{r}_{\indm}(\phi)}
\right)^{\DimH_{\indm}}
&\leq \text{Vol}\left(
\left\{ \beta \in \R^{\DimH_{\indm}}\middle| \|\beta\|_2 \leq \sigma + \frac{\delta}{\widebar{r}_{\indm}(\phi)}
\right\} 
\right)\\
& \leq
\left( \sigma + \frac{\delta}{\widebar{r}_{\indm}(\phi)}
\right)^{\DimH_{\indm}}
\text{Vol}\left(
\left\{ \beta \in \R^{\DimH_{\indm}}\middle| \|\beta\|_2 \leq 1
\right\} 
\right)
\end{align*}
and thus
\begin{align*}
  \left| \Grid_{\indm}(\delta,\sigma) \right| & \leq \left(1+
  \frac{\sigma \widebar{r}_\indm(\phi)}{\delta}\right)^{\DimH_{\indm}}
\DimH_{\indm}^{\DimH_{\indm}/2}
\text{Vol}\left(
\left\{ \beta \in \R^{\DimH_{\indm}}\middle| \|\beta\|_2 \leq 1
\right\} 
\right)
\intertext{%
and as $\frac{\sigma \widebar{r}_\indm(\phi)}{\delta}\geq 1$ and
$\text{Vol}\left(
\left\{ \beta \in \R^{\DimH_{\indm}}\middle| \|\beta\|_2 \leq 1
\right\} 
\right)
\leq 
\left(\frac{2\pi e}{\DimH_{\indm}}\right)^{\DimH_{\indm}/2}
$}
  \left| \Grid_{\indm}(\delta,\sigma) \right| &
\leq \left(
\frac{2\sqrt{2\pi e} \widebar{r}_{\indm}(\phi) \sigma}{\delta}
\right)^{\DimH_{\indm}}
\end{align*}
which concludes the proof.
\end{proof}

Instead of \cref{prop:entropypoly}, by mimicking a proof of
\textcite{massart07:_concen}, we prove
\ifthenelse{\boolean{extended}}{}{in our technical
  report~\cite{cohen11:_condit_densit_estim_penal_app}} an extended version of it
in which the degree of the conditional densities may depend on the hyperrectangle. More
precisely, we reuse the partition $\PartX \in \CollPart^{\starX}$
and the partitions $\PartYLeafXl \in \CollPart^{\starY}$ for
$\LeafXl \in \PartX$ and define now the model
$\model_{\PartXY,\dimPYmaxPartXY}$ as the set of conditional densities
such that
\begin{align*}
  s(y|x) & = \sum_{\LeafXYlk\in\PartXY}
P^2_{\LeafXYlk}(y)
\Charac{(x,y)\in\LeafXYlk}
\end{align*}
where $P_{\LeafXYlk}$ is a polynomial of degree at most
\begin{alignI}
\dimPYmax(\LeafXYlk)
=\left(\dimPYmaxi{1}(\LeafXYlk),\ldots,\dimPYmaxi{\dimY}(\LeafXYlk)\right)
\end{alignI} 
which depends on the leaf.

\ifthenelse{\boolean{extended}}{
By construction,
\begin{align*}
 \dim(\model_{\PartXY,\dimPYmaxPartXY}) = \sum_{\LeafXl
   \in \PartX} \left( 
\left( \sum_{\LeafYlk \in \PartYLeafXl} 
\prod_{d=1}^{\dimY}
\left(\dimPYmaxi{d}(\LeafXYlk)+1\right)
\right) - 1
\right).
\end{align*}
The corresponding linear space
$\widebar{\sqrt{\model_{\PartXY,\dimPYmaxPartXY}}}$ is 
\begin{align*}
\left\{ 
  \sum_{\LeafXYlk\in\PartXY}
P_{\LeafXYlk}(y)
\Charac{(x,y)\in\LeafXYlk}
\middle| \deg\left( P_{\LeafXYlk}\right) \leq \dimPYmax(\LeafXYlk) 
\right\}.
\end{align*}}{}
Instead of the true dimension, we use a slight upper bound
\begin{align*}
\DimH_{\PartXY,\dimPYmaxPartXY} =   
\sum_{\LeafXl
   \in \PartX}  \sum_{\LeafYlk \in \PartYLeafXl} 
\prod_{d=1}^{\dimY}
\left(\dimPYmaxi{d}(\LeafXYlk)+1\right)
= \sum_{\LeafXYlk
   \in \PartXY} \prod_{d=1}^{\dimY}
\left(\dimPYmaxi{d}(\LeafXYlk)+1\right)
\end{align*}
Note that the space $\model_{\PartXY,\dimPYmax}$
introduced in the main part of the paper 
corresponds to the case 
where the degree $\dimPmax(\LeafXYlk)$ does
not depend on the hyperrectangle $\LeafXYlk$.

\begin{proposition}
\label{prop:entropypolyext}
There exists
\begin{align*}
  \widebar{r}_{{\PartXY,\dimPYmaxPartXY}}
& \leq
\frac{
 \sup_{\LeafXYlk\in\PartXY}
\prod_{d=1}^{\dimY} 
\left(
\sum_{\dimPYi{d} \leq \dimPYmaxi{d}(\LeafXYlk)}
\sqrt{2 \dimPYi{d}+1}
\right)
}{\inf _{\LeafXYlk
   \in \PartXY} 
\prod_{d=1}^{\dimY}
\sqrt{\dimPYmaxi{d}(\LeafXYlk)+1}}
\sup_{\LeafXYlk \in \PartXY}
\frac{1}{\sqrt{
\NbPartX}\sqrt{|\LeafXYlk|}
}
\end{align*}
such that $\forall s_{\PartXY,\dimPYmaxPartXY} \in \model_{\PartXY,\dimPYmaxPartXY}$,
  \begin{align*}
    H_{[\cdot],\dtens} 
\left(\delta, \model_{\PartXY,\dimPYmaxPartXY}(s_{\PartXY,\dimPYmaxPartXY},\sigma)
\right)
& \leq \DimH_{{\PartXY,\dimPYmaxPartXY}}  \left( \ConstMultH_{{\PartXY,\dimPYmaxPartXY}} + \Log
  \frac{\sigma}{\delta} \right)
  \end{align*}
with $\ConstMultH_{{\PartXY,\dimPYmaxPartXY}} = \Log \left( \kappa_{\infty}
  \widebar{r}_{{\PartXY,\dimPYmaxPartXY}}\right) $ and $\kappa_{\infty}\leq 2\sqrt{2\pi
  e}$.
\end{proposition}

\Cref{prop:entropypoly} is deduced from this proposition with the help
of the simple upper bound
\begin{align*}
  \sum_{\dimPYi{d} \leq \dimPYmaxi{d}(\LeafXYlk)}
\sqrt{2 \dimPYi{d}+1}
& \leq (\dimPYmaxi{d}(\LeafXYlk) + 1 ) \sqrt{2
  \dimPYmaxi{d}(\LeafXYlk) + 1 }.
\end{align*}

As 
\begin{align*}
\frac{
 \sup_{\LeafXYlk\in\PartXY}
\prod_{d=1}^{\dimY} 
\left(
\sum_{\dimPYi{d} \leq \dimPYmaxi{d}(\LeafXYlk)}
\sqrt{2 \dimPYi{d}+1}
\right)
}{\inf _{\LeafXYlk
   \in \PartXY} 
\prod_{d=1}^{\dimY}
\sqrt{\dimPYmaxi{d}(\LeafXYlk)+1}}
\leq \prod_{d=1}^{\dimY} \max \sqrt{2}(\dimPYmaxi{d} + 1 ),
\end{align*}
once a maximal degree is chosen along each axis, the equivalent of constant 
$\constpolyparttwo$ of \ref{theo:polypart} depends only on this
maximal degrees. Assumption
 $\text{H}_{\PartXY,\dimPYmax}$ holds then, with the same constants, simultaneaously for all
 models of both
global choice and local choice strategies. Obtaining the Kraft type
assumption, Assumption (K) is only a matter of taking into account the
augmentation of the number of models within the collection. Replacing respectively 
$\PartCstA^{\starX}$ by  $\PartCstA^{\starX} + \Log
        |\dimPYfamily|$ for global optimization and
      $\PartCstB^{\starY}$ by  $\PartCstB^{\starY} +
        \Log |\dimPYfamily| $ for local optimization, where
        $|\dimPYfamily|$ denotes the size of the family of possible
        degrees, turns out to be sufficient as mentioned earlier.

\ifthenelse{\boolean{extended}}{
\begin{proof}[Proof of \cref{prop:entropypolyext}]
Let $L_\dimP$ be the one dimensional Legendre polynomial of degree
$\dimP$ on $[0,1]$
and $G_{\dimP}= \sqrt{2 \dimP +1} L_{\dimP}$ its rescaled version, we
  recall that, by definition,
  \begin{align*}
\forall \dimP \in \N,\quad    \|G_{\dimP}\|_{\infty}&= \sqrt{2\dimP+1} & \text{and}  &&\forall (\dimP,\dimP') \in
\N^2,\quad \int G_\dimP(t) G_{\dimP'}(t) \ud t &= \delta_{\dimP,\dimP'} 
  \end{align*}
 Let $\dimPY\in\N^{\dimY}$, we define
 $G_{\dimPY}$ as
   the  polynomial
   \begin{align*}
  G_{\dimPYi{1},\ldots,\dimPYi{\dimY}}(y)
  = 
  G_{\dimPYi{1}}(y_1)\times \cdots  \times G_{\dimPYi{\dimY}}(y_{\dimY}),
   \end{align*}
by construction
  \begin{align*}
\forall \dimPY \in \N^{\dimY},\quad    \|G_{\dimPY}\|_{\infty}&= 
\prod_{1 \leq d \leq \dimY} \sqrt{2 \dimPYi{d}+1}
\end{align*}
and
\begin{align*}
\forall (\dimPY,\dimPY') \in
\N^{\dimY\times 2},\quad \int_{y \in [0,1]^{\dimY}} G_{\dimPY}(y) G_{\dimPY'}(y) \ud y &=
\delta_{\dimPY,\dimPY'}.
  \end{align*}

Now for any hyperrectangle $\LeafXYlk$, we define $G_{\dimPY}^{\LeafXYlk}(x,y)=
\frac{1}{\sqrt{|\LeafXYlk|}} G_{\dimPY}(T^{\LeafYlk}(y))
\Charac{(x,y) \in  \LeafXYlk}(x,y)$
where $T^{\LeafYlk}$ is the affine transform that maps
$\LeafYlk$ into
$[0,1]^{\dimY}$ so that
  \begin{align*}
\forall \LeafXYlk \in \PartXY,
 \forall \dimPY \in \N^{\dimY},\quad
\|G^{\LeafXYlk}_\dimP\|_{\infty}&=
\frac{1}{\sqrt{|\LeafXYlk|}} 
  \prod_{1 \leq d \leq \dimY} \sqrt{2 \dimPYi{k}+1}
\end{align*}
and
\begin{align*}
\forall (\LeafXYlk,\LeafXYlkp) \in \left(\PartXY\right)^2,
 &\forall (\dimPY,\dimPY') \in
\N^{\dimY\times 2},\\
&\quad
\int_{x \in [0,1]^{\dimX}}
\int_{y \in [0,1]^{\dimY}}
 G_\dimP^{\LeafXYlk}(x,y)
G_{\dimP'}^{\LeafXYlkp}(x,y) \ud y \ud x =
\delta_{\LeafXYlk,\LeafXYlkp} \delta_{\dimPY,\dimPY'}.
  \end{align*}
Using the piecewise structure, one deduces
\begin{align*}
\E &\left[ \left\| 
 \sum_{\LeafXYlk \in \PartXY}
\sum_{\dimPY\leq \dimPYmax(\LeafXYlk)}
\beta_{\dimPY}^{\LeafXYlk}
G_{\dimPY}^{\LeafXYlk}(X_i,\cdot) \right\|_2^2
\right] \\
\ifthenelse{\boolean{extended}}{
  & = \E \left[ \sum_{\LeafXl \in \PartX}
\frac{\Charac{X_1 \in \LeafXl}}{|\LeafXl|} 
\sum_{\LeafYlk \in \PartYLeafXl}
\int_{(x,y) \in \LeafXYlk}
\left|
\sum_{\dimPY\leq \dimPYmax(\LeafXl,\LeafYlk)}
\beta_{\dimPY}^{\LeafXl,\LeafYlk}G_{\dimPY}^{\LeafXl,\LeafYlk}(x,y)
\right|^2
\ud y  \ud x \right]\\
& =  \E \left[ \sum_{\LeafXl \in \PartX}
\frac{\Charac{X_1 \in \LeafXl}}{|\LeafXl|} 
\sum_{\LeafYlk \in \PartYLeafXl}
\sum_{\dimPY\leq \dimPYmax(\LeafXYlk)}
\left| \beta_{\dimPY}^{\LeafXYlk}\right|^2 
 \right]\\
}{}
& = \sum_{\LeafXl \in \PartX} \frac{\Prob\{X_i \in \LeafXl\}}{|\LeafXl|} \sum_{\LeafYlk \in \PartYLeafXl}
\sum_{\dimPY\leq \dimPYmax(\LeafXYlk)}
\left| \beta_{\dimPY}^{\LeafXYlk}\right|^2.
\end{align*}

The space
$\widebar{\sqrt{\model_{\PartXY,\dimPYmaxPartXY}}}$ is
spanned by
\begin{align*}
  \left\{
 G_{\dimPY}^{\LeafXYlk}
\middle|
\LeafXYlk \in \PartXY, \dimPY \leq \dimPYmax(\LeafXYlk)
\right\}
\end{align*}
but also by the rescaled
$\phi_{\dimP}^{\LeafXYlk}=\frac{1}{\sqrt{\mu_X(\LeafXl)}}
G_\dimP^{\LeafXYlk}$ where $\mu_X(\LeafXl)= \frac{1}{n}
\sum_{i=1}^n
\frac{\Prob\{ X_i \in \LeafXl\}}{|\LeafXl|}$. For these functions, one has
\begin{align*}
\Bigg\|  &\sum_{\LeafXYlk \in \PartXY} 
\sum_{\dimPY\leq \dimPYmax(\LeafXYlk)} \beta_{\dimPY}^{\LeafXYlk}  \phi_{\dimPY}^{\LeafXYlk}
\Bigg\|_2^{2\tens}\\
\ifthenelse{\boolean{extended}}{
& = \E \left[ \frac{1}{n} \sum_{i=1}^n
\left\| 
 \sum_{\LeafXYlk \in \PartXY}
\sum_{\dimPY\leq \dimPYmax(\LeafXYlk)}
\beta_{\dimPY}^{\LeafXYlk}
\phi_{\dimPY}^{\LeafXYlk}(X_i,\cdot) \right\|^2
\right]\\
}{}
& = \frac{1}{n} \sum_{i=1}^n \E \left[ \left\| 
 \sum_{\LeafXYlk \in \PartXY}
\sum_{\dimPY\leq \dimPYmax(\LeafXYlk)}
\frac{\beta_{\dimPY}^{\LeafXYlk}}{\sqrt{\mu_X(\LeafXl)}}
G_{\dimPY}^{\LeafXYlk}(X_i,\cdot) \right\|^2
\right]\\
& = \sum_{\LeafXYlk \in \PartXY}
\sum_{\dimPY\leq \dimPYmax(\LeafXYlk)}
\left|\beta_{\dimPY}^{\LeafXYlk}\right|^2 = \left\|\beta_{\dimPY}^{\LeafXYlk}\right\|_2^2.
\end{align*}

For $\|\cdot\|_\infty$ type norm,
\begin{align*}
\Bigg\|  &\sum_{\LeafXYlk \in \PartXY}
\sum_{\dimPY\leq \dimPYmax(\LeafXYlk)} \beta_{\dimPY}^{\LeafXYlk}  \phi_{\dimPY}^{\LeafXYlk}
\Bigg\|_\infty^{2\tens}\\
\ifthenelse{\boolean{extended}}{
& = \E \left[
\frac{1}{n} \sum_{i=1}^n \left\|
\sum_{\LeafXYlk \in \PartXY}
\sum_{\dimPY\leq \dimPYmax(\LeafXYlk)} \beta_{\dimPY}^{\LeafXYlk}  \phi_{\dimPY}^{\LeafXYlk}(X_i,\cdot)
\right\|_{\infty}^2
\right]\\
&=
\frac{1}{n} \sum_{i=1}^n
\E \left[
\left\|
\sum_{\LeafXYlk \in \PartXY}
\sum_{\dimPY\leq \dimPYmax(\LeafXYlk)} \beta_{\dimPY}^{\LeafXYlk}  \phi_{\dimPY}^{\LeafXYlk}(X_i,\cdot)
\right\|_{\infty}^2
\right]\\
}{}
&= \frac{1}{n} \sum_{i=1}^n
\E \left[
\sum_{\LeafXl \in \PartX}
\Charac{X_i\in\LeafXl}
\sup_{\LeafYlk \in \PartYLeafXl}
\left\|
\sum_{\dimPY\leq \dimPYmax(\LeafXYlk)} \beta_{\dimPY}^{\LeafXYlk}  \phi_{\dimPY}^{\LeafXYlk}(X_i,\cdot)
\right\|_{\infty}^2
\right]\\
&\leq \frac{1}{n} \sum_{i=1}^n \E \left[
\sum_{\LeafXl \in \PartX}
\Charac{X_i\in\LeafXl}
\sup_{x\in\LeafXl}
\sup_{\LeafYlk \in \PartYLeafXl}
\left(
\sum_{\dimPY\leq \dimPYmax(\LeafXYlk)} \left|
  \beta_{\dimPY}^{\LeafXYlk} \right|
  \left\|\phi_{\dimPY}^{\LeafXYlk}(x,\cdot)\right\|_{\infty}
\right)^2\right]\\
\ifthenelse{\boolean{extended}}{
& \leq \frac{1}{n} \sum_{i=1}^n
\E \left[
\sum_{\LeafXl \in \PartX}
\Charac{X_i\in\LeafXl}
\sup_{\LeafYlk \in \PartYLeafXl}
\frac{1}{\mu_X(\LeafXl)|\LeafXYlk|}
\left(
\sum_{\dimPY\leq \dimPYmax(\LeafXYlk)} 
  \left\|G_{\dimPY}\right\|_{\infty}
\right)^2 \left\| \beta_{\dimPY}^{\LeafXYlk}
\right\|_{\infty}^2\right]\\
}{}
& \leq 
\sum_{\LeafXl \in \PartX}
|\LeafXl|
\sup_{\LeafYlk \in \PartYLeafXl}
\frac{1}{|\LeafXYlk|}
\left(
\sum_{\dimPY\leq \dimPYmax(\LeafXYlk)} 
  \left\|G_{\dimPY}\right\|_{\infty}
\right)^2 
\left\| \beta_{\dimPY}^{\LeafXYlk}
\right\|_{\infty}^2.
\end{align*}

Now
\begin{align*}
   \sum_{\dimPY \leq \dimPYmax(\LeafXYlk)}
 \|G_{\dimPY}\|_{\infty}  
& =  \sum_{\dimPY \leq \dimPYmax(\LeafXYlk)}
  \prod_{d=1}^{\dimY} \|G_{\dimPYi{d}}\|_{\infty}
= \prod_{d=1}^{\dimY} 
\left(
\sum_{\dimPYi{d} \leq \dimPYmaxi{d}(\LeafXYlk)}
\|G_{\dimPYi{d}}\|_{\infty}
\right)
\\& 
= \prod_{d=1}^{\dimY} 
\left(
\sum_{\dimP_{d} \leq \dimPYmaxi{d}(\LeafXYlk)}
\sqrt{2 \dimPYi{d}+1}
\right)
\leq \sup_{\LeafXYlkp\in\PartXY}
\prod_{d=1}^{\dimY} 
\left(
\sum_{\dimPYi{d} \leq \dimPYmaxi{d}(\LeafXYlkp)}
\sqrt{2 \dimPYi{d}+1}
\right)
 \end{align*}
while
\begin{align*}
\DimH_{\PartXY,\dimPYmaxPartXY} & \geq   
\sum_{\LeafXl
   \in \PartX}  \sum_{\LeafYlk \in \PartYLeafXl} 
\inf _{\LeafXYlkp
   \in \PartXY} 
\prod_{d=1}^{\dimY}
\left(\dimPYmaxi{d}(\LeafXYlk)+1\right)
 \geq 
\left(
\inf _{\LeafXYlkp
   \in \PartXY} 
\prod_{d=1}^{\dimY}
\left(\dimPYmaxi{d}(\LeafXYlk)+1\right)
\right)
\NbPartXY.
\end{align*}
This implies
\begin{align*}
&\mspace{-30mu}\frac{\Bigg\|\sum_{\LeafXYlk \in \PartXY} 
\sum_{\dimPY\leq \dimPYmax(\LeafXYlk)} \beta_{\dimPY}^{\LeafXYlk}  \phi_{\dimPY}^{\LeafXYlk}
\Bigg\|_\infty^{2\tens}}{\DimH_{\PartXY,\dimPYmaxPartXY} 
\left\| \beta_{\dimPY}^{\LeafXYlk}
\right\|_{\infty}^2}\\
& \leq
\frac{\left(
\sup_{\LeafXYlkp\in\PartXY}
\prod_{d=1}^{\dimY} 
\left(
\sum_{\dimPYi{d} \leq \dimPYmaxi{d}(\LeafXYlkp)}
\sqrt{2 \dimPYi{d}+1}
\right)
\right)^2}{
\inf _{\LeafXYlkp
   \in \PartXY}
\prod_{d=1}^{\dimY}
\left(\dimPYmaxi{d}(\LeafXYlkp)+1\right)}
\sum_{\LeafXl \in \PartX}
|\LeafXl|
\sup_{\LeafYlk \in \PartYLeafXl}
\frac{1}{ \NbPartXY |\LeafXYlk|}\\
\ifthenelse{\boolean{extended}}{
& \leq
\left( 
\frac{\sup_{\LeafXYlkp\in\PartXY}
\prod_{d=1}^{\dimY} 
\left(
\sum_{\dimPYi{d} \leq \dimPYmaxi{d}(\LeafXYlkp)}
\sqrt{2 \dimPYi{d}+1}
\right)}{
\inf _{\LeafXYlkp
   \in \PartXY}
\prod_{d=1}^{\dimY}
\sqrt{\dimPYmaxi{d}(\LeafXYlkp)+1}}
\sum_{\LeafXl \in \PartX}
|\LeafXl|
\sup_{\LeafYlk \in \PartYLeafXl}
\frac{1}{\sqrt{\NbPartXY} \sqrt{|\LeafXYlk|}}
\right)^2\\}{}
& \leq
\left( 
\frac{\sup_{\LeafXYlkp\in\PartXY}
\prod_{d=1}^{\dimY} 
\left(
\sum_{\dimPYi{d} \leq \dimPYmaxi{d}(\LeafXYlkp)}
\sqrt{2 \dimPYi{d}+1}
\right)}{
\inf _{\LeafXYlkp
   \in \PartXY}
\prod_{d=1}^{\dimY}
\sqrt{\dimPYmaxi{d}(\LeafXYlkp)+1}}
\sup_{\LeafXYlk \in \PartXY}
\frac{1}{\sqrt{\NbPartXY} \sqrt{|\LeafXYlk|}}
\right)^2.
\end{align*}
Proposition is then obtained by a simple application of \Cref{prop:inftyentropy}.
\end{proof}
}{The proof of~\cref{prop:entropypolyext} is essentially computational
and thus relegated to our extended technical report.}

\ifthenelse{\boolean{extended}}{
\subsection{Proof of \cref{prop:polypartK}}

\begin{proof}
By construction
\begin{align*}
  \sum_{\model_{\PartXY,\dimPYmax}\in\Models}
  e^{-x_{\PartXY,\dimPYmax}}
& = \sum_{\PartX \in \CollPart^{\starX}} \sum_{\LeafXl \in \PartX}
\sum_{\PartYLeafXl \in \CollPart^{\starY}}
  e^{-x_{\PartXY,\dimPYmax}}\\
& = \sum_{\PartX \in \CollPart^{\starX}} \sum_{\LeafXl \in \PartX}
\sum_{\PartYLeafXl \in \CollPart^{\starY}}
  e^{-\constpolypartthree \left( \PartCstA^{\starX}
      +  \left( \PartCstB^{\starX}  + \PartCstA^{\starY}
      \right)\NbPartX 
      + \PartCstB^{\starY} \sum_{\LeafXl\in\PartX}
      \NbPartYLeafXl \right)}\\
& = \sum_{\PartX \in \CollPart^{\starX}} e^{-\constpolypartthree \left( \PartCstA^{\starX}
      +   \PartCstB^{\starX} 
      |\Part| \right)}
\prod_{\LeafXl \in \PartX} \left( \sum_{\PartYLeafXl \in \CollPart^{\starY}}
  e^{-\constpolypartthree \left( \PartCstA^{\starY}
      + \PartCstB^{\starY}
      \NbPartYLeafXl \right)} \right)
\end{align*}
By \cref{prop:codingpolpart}, one can find $\constpolypartthree \geq
\max(1,\PartCstc^{\starX},\PartCstc^{\starY})$ such that
\begin{align*}
  \sum_{\PartYLeafXl \in \CollPart^{\starY}}
  e^{-\constpolypartthree \left( \PartCstA^{\starY}
      + \PartCstB^{\starY}
      \NbPartYLeafXl \right)} \leq 1
\end{align*}
and
\begin{align*}
\sum_{\PartX \in \CollPart^{\starX}} e^{-\constpolypartthree \left( \PartCstA^{\starX}
      +   \PartCstB^{\starX} 
      \NbPartX \right)} \leq 1.
\end{align*}
Plugging these bounds in the previous equality yields
\begin{align*}
  \sum_{\model_{\PartXY,\dimPYmax}\in\Models}
  e^{-x_{\PartXY,\dimPYmax}}
& \leq \sum_{\PartX \in \CollPart^{\starX}} e^{-\constpolypartthree \left( \PartCstA^{\starX}
      +   \PartCstB^{\starX} 
      \NbPartX \right)}
 \leq 1 .
\end{align*}
Proposition holds with the modified weights for polynomial 
as
\begin{align*}
  \sum_{\dimPYmax \in \dimPYfamily} e^{-\constpolypartthree \Log
    |\dimPYfamily|} = |\dimPYfamily|^{1-\constpolypartthree} \leq 1
\end{align*}
as soon as $\constpolypartthree \geq 1$.
\end{proof}
}

\section{Proofs for \cref{sec:spat-mixt-models} (\nameref{sec:spat-mixt-models})}

As in the piecewise polynomial density case, \Cref{theo:spatgauss} is
obtained by showing that Assumptions ($\text{H}_{\PartX,K,\Set}$),
($\text{S}_{\PartX,K,\Set}$) and ($\text{K}$) hold for any collection.

Again, one easily verifies that Assumption ($\text{S}_{\PartX,K,\Set}$)
holds. For the complexity assumption, 
combining \ref{prop:proporsimple} with a bound on the
bracketing entropy
of the models of type
\begin{align*}
H_{[\cdot],\dsup}(\delta, \model_{\PartX,K,\Set}) & \leq
    \dim(\model_{\PartX,K,\Set}) \left(\constUentropy + \Log \frac{1}{\delta} \right),
  \end{align*}
one obtains
\begin{proposition}
\label{prop:spatgaussH}
There exists a constant $\constUentropy$ depending only on $a$, $\Lm$, $\LM$, $\lambdam$ and
$\lambdaM$ such that for any model
$\model_{\PartX,K,\Set}$ of \cref{theo:spatgauss}
Assumption
$(\text{H}_{\PartX,K,\Set})$ is satisfied with a function $\phi$ such that
\begin{align*}
 \DIMH_{\PartX,K,\Set}  & \leq \left( 2\left( \sqrt{\constUentropy} + \sqrt{\pi} \right)^2 +
1 + \left( \Log  \frac{n}{e\left(
      \sqrt{\constUentropy} + \sqrt{\pi} \right)^2 \dim(\model_{\PartX,K,\Set})} \right)_+\right) \dim(\model_{\PartX,K,\Set}).
\end{align*}
\end{proposition}
For the Kraft assumption, one can verify that
\begin{proposition}
\label{prop:spatgaussK}
For any collections $\Models$ of \cref{theo:spatgauss}, there is a
$\constspatgausstwo$ such that for the choice
\begin{align*}
  x_{\PartX,K,\Set} =  \constspatgausstwo \left(
     \PartCstA^{\starX}
       +  \PartCstB^{\starX} \NbPartX 
  + (K-1) + \CstSpace
 \right),
\end{align*}
Assumption (K) holds with
\begin{alignI}
\sum_{\model_{\PartX,K,\Set} \in \Models}
   e^{-x_{\PartX,K,\Set}} \leq 1.
 \end{alignI}
\end{proposition}
As for the piecewise
    polynomial case section, the main
difficulty lies in controlling the bracketing entropy of the
models. A proof of \cref{prop:spatgaussK} can be found in our
technical report~\cite{cohen11:_condit_densit_estim_penal_app}.

We focus thus on the proof of \cref{prop:spatgaussH}.
Due to the complex structure of spatial mixture, we did not
manage to bound the bracketing entropy of local model. We derive 
only an upper bound  of the
bracketing entropy      
$H_{[\cdot],\dtens}(\delta,\model_{\PartX,K,\Set})$, but one that is
independent of the distribution law of $(X_i)_{1\leq i \leq n}$: the  bracketing
entropy with a $\sup$ norm Hellinger distance $\dsup=\sqrt{\dtwosup}$,
$H_{[\cdot],\dsup}(\delta,\model_{\PartX,K,\Set})$, 
where $\dtwosup$ is defined by
\[
\dtwosup(s,t)=\sup_{x} \d^2 \left(s(\cdot|x),t(\cdot|x)\right).
\]
Obviously $\dtwosup \geq \dtwotens$ and thus
$H_{[\cdot],\dsup}(\delta,\model_{\PartX,K,\Set}) \geq
H_{[\cdot],\dtens}(\delta,\model_{\PartX,K,\Set})$. This upper bound
is furthermore design independent.

\cref{prop:spatgaussH} is a direct consequence of
\cref{prop:proporsimple}
and 
\begin{proposition}
\label{prop:entropyU}
There exists a constant $\constUentropy$ depending only on $a$, $\Lm$, $\LM$, $\lambdam$ and
$\lambdaM$ such that for any model
$\model_{\PartX,K,\Set}$ of \cref{theo:spatgauss}:
\begin{align*}
H_{[\cdot],\dsup}(\delta, \model_{\PartX,K,\Set}) & \leq
    \dim(\model_{\PartX,K,\Set}) \left(\constUentropy + \Log \frac{1}{\delta} \right).
  \end{align*}
\end{proposition}

\ifthenelse{\boolean{extended}}{
\subsection{Model coding}


\begin{proof}[Proof of \cref{prop:spatgaussK}]
This proposition is a simple combination of 
\cref{theo:polypart}, 
crude bounds on the
number of different models indexed by
$[\mumodany\,\Lmodany\,\Dmodany\,\Amodany]^K$ and
$[\mumodany\,\Lmodany\,\Dmodany\,\Amodany]$ and
of classical Kraft type inequalities for order selection and variable selection (see
for instance in the book of \textcite{massart07:_concen}):
\begin{lemma}
\label{lem:coding}
\begin{itemize}
\item For the selection of model order  $K$, let
$x_K = (K-1)$, for $c>0$
  \begin{align*}
    \sum_{K\geq 1} e^{-c x_K} = \frac{1}{1-e^{-c}}
  \end{align*}
\item For the ordered variable selection case,
  $\Space= \Span\{ e_i \}_{i\in 
I}$ with $I=\{1,\ldots,\dimSpace\}$, let $\theta_\Space = \dimSpace$,
 for $c>0$
\begin{align*}
  \sum_{\Space} e^{-c\theta_\Space} = \frac{1}{e^c-1} \leq 1.
\end{align*}
\item For the non ordered variable selection case, $\Space= \Span\{
  e_i \}_{i\in I}$ with $I \subset
  \{1,\ldots,\dimSp\}$, let $\theta_\Space =  \left( 1 + \theta + \Log
    \frac{\dimSp}{\dimSpace} \right) \dimSpace$, for $c\geq 1$,
\begin{align*}
  \sum_{\Space} e^{-c\theta_\Space} = \frac{e^{-(c-1)(1+\theta)}}{1-e^{-\theta}}.
\end{align*}
\end{itemize}
\end{lemma} 
 Using that
there is at most $3\times 3\times 3 \times 3$
different type of models $[\mumodany\,\Lmodany\,\Dmodany\,\Amodany]^K$
and $2\times 2\times 2\times 2$ different type of models
$[\mumodany\,\Lmodany\,\Dmodany\,\Amodany]$, and $3^4 \times 2^4 =
1296$, we obtain
\begin{align*}
   \sum_{\model_{K,\PartX,\Set} \in \Models}
   e^{-x_{K,\PartX,\Set}} & = \sum_{K \in \N^*} \sum_{\PartX
     \in \CollPart^{\star}} \sum_{\Space}
   \sum_{[\mumodany\,\Lmodany\,\Dmodany\,\Amodany]^K}
   \sum_{[\mumodany\,\Lmodany\,\Dmodany\,\Amodany]} e^{-\constspatgausstwo \left(
     \PartCstA^{\starX}
       +  \PartCstB^{\starX} \NbPartX 
  + (K-1) + \CstSpace
 \right)}\\
& = \left( \sum_{K \in \N^*} e^{-\constspatgausstwo (K-1)} \right)
\left( \sum_{\PartX
     \in \CollPart^{\star}} e^{-\constspatgausstwo \left( \PartCstA^{\starX}
       +  \PartCstB^{\starX} \NbPartX \right)} \right)\\
& \qquad \times
   \left( \sum_{\Space} e^{-\constspatgausstwo \CstSpace} \right)
   \left( \sup_{K\in\N^*} \sum_{[\mumodany\,\Lmodany\,\Dmodany\,\Amodany]^K}
   \sum_{[\mumodany\,\Lmodany\,\Dmodany\,\Amodany]} \right)\\
& \leq 1296 \frac{1}{1-e^{-\constspatgausstwo}} \PartCstSigma^\star
e^{-\constspatgausstwo \PartCstC^\star} 
\begin{cases}
1  & \text{if $\Space$ is known,}\\
\frac{1}{e^\constspatgausstwo-1} & \text{\parbox{4cm}{if
  $\Space$ is chosen amongst spaces spanned by the first
  coordinates,}}\\
2e^{-(\constspatgausstwo-1)(1+\Log 2)} & \text{if
  $\Space$ is free.}
\end{cases}  
\end{align*}
Choosing $\constspatgausstwo$ slightly larger than
$\max(1,\PartCstc^{\star})$ yields the result. 
\end{proof}
}{}

\subsection{Entropy of spatial mixtures}


\begin{proof}[Proof of \cref{prop:entropyU}]

While we  use classical Hellinger distance to measure the
complexity of the simplex $\Simplex_{K-1}$ and the set
$\Set_{E^\perp}$, we use a $\sup$ norm Hellinger distance on
$\Set_{\Space}^K$ defined by
\begin{align*}
  \dtwomax\left( (s_1,\ldots,s_K),(t_1,\ldots,t_K)\right)
& = \sup_k \d^2 (s_k,t_k).
\end{align*}
We say that $\left[(s_1,\ldots,s_K),(t_1,\ldots,t_K)\right]$ is a
bracket of $\Set_{\Space}^K$ if $\forall 1 \leq k \leq K, s_k \leq t_k$.

Using a similar proof than \textcite{genovese00:_rates_gauss}, we decompose the entropy in
three parts with:
\begin{lemma}
\label{lemma:entropy} For any $\delta\in (0,\sqrt{2}]$,
  \begin{align*}
    H_{[\cdot],\dsup}(\delta,\model_{\PartX,K,\Set}) \leq \| \PartX\|
    H_{[\cdot],\d}(\delta/3, \Simplex_{K-1}) +
    H_{[\cdot],\dmax}(\delta/9, \Set_{\Space}^K) +
H_{[\cdot],\d}(\delta/9, \Set_{\Space^\perp}). 
  \end{align*}
\end{lemma}

We bound those bracketing entropies with the help of two
results. We first use a Lemma proved in \textcite{genovese00:_rates_gauss} that implies the existence of a universal constant $\ConstMultH_{\Simplex}$
such that
\begin{align*}
  H_{[\cdot],\d}(\delta/3, \Simplex_{K-1}) & \leq (K-1) \left(
    \ConstMultH_{\Simplex} + \Log \frac{1}{\delta} \right):
\end{align*}
\begin{lemma}
\label{lem:entsimp}
 For any $\delta \in (0,\sqrt{2}]$,
  \begin{align*}
    H_{[\cdot],\d}(\delta/3, \Simplex_{K-1})
 \leq   \DimHSimpKm \left( \ConstMultH_{\Simplex_{K-1}}  + 
 \Log \frac{1}{\delta} \right)
\end{align*}
with
\begin{alignI}
\ConstMultH_{\Simplex_{K-1}}
=  \frac{1}{K-1} 
\Log K + \frac{K}{2(K-1)} \Log ( 2\pi e) + \Log 3\sqrt{2} 
  \end{alignI}

Furthermore, uniformly on $K$: 
\begin{alignI}
\ConstMultH_{\Simplex_{K-1}}
\leq \Log 2 +
\frac{1}{2} \Log (2 \pi e) + \Log 3\sqrt{2} = \ConstMultH_{\Simplex}
\end{alignI}
\end{lemma}
We then rely on \cref{prop:entcoll} to
handle the bracketing
entropy of Gaussian $K$-tuples collection. It implies the existence of
two constants $\ConstMultH_{[\starmuLDA]^\starK}$
and $\ConstMultH_{[\starmuLDA]}$ depending only on $a$, $\Lm$, $\LM$, $\lambdam$ and
$\lambdaM$ such that
\begin{align*}
      H_{[\cdot],\dmax} \left(\delta/9, \Set_{\Space}^K\right) & \leq
      \dim(\Set_{\Space}^K) \left(\ConstMultH_{[\starmuLDA]^\starK} + \Log
      \frac{1}{\delta} \right)\\   
H_{[\cdot],\d}(\delta/9, \Set_{\Space^\perp}) 
&     \leq \dim(\Set_{\Space^\perp}) \left(\ConstMultH_{[\starmuLDA]} + \Log
      \frac{1}{\delta} \right).
\end{align*}
As $\dim(\model_{K,\Part,\Set}) = \NbPartX(K-1)  +
\dim(\Set_{\Space}^K) + \dim(\Set_{\Space^\perp})$, we obtain
\cref{prop:entropyU} with
$\constUentropy=\max(\ConstMultH_{\Simplex},\ConstMultH_{[\starmuLDA]^\starK},\ConstMultH_{[\starmuLDA]})$.
\end{proof}

\subsection{Entropy of Gaussian families}

\ifthenelse{\boolean{extended}}{Instead of \cref{prop:entcoll}, we prove
the slightly stronger
\begin{proposition}\label{prop:entcollext}
Let $\kappag \geq
 \frac{3}{4}$ and
\begin{align*}
  \gammakappa  &=
 \min\left(\frac{3(\kappag-\frac{3}{4})}{2(1+\frac{\kappag}{6})(1+\frac{1}{6})(1+\frac{1}{12})},\frac{(\kappag-\frac{1}{2})}{2(1+\frac{\kappag}{6})(1+\frac{1}{6})}\right)
&\constcosh
  = \sqrt{\kappag^2\cosh(\frac{\kappag}{6}) +
    \frac{1}{2}}
\end{align*}
Then for any $\delta \in (0,\sqrt{2}]$, 
  \begin{align*}
    H_{[\cdot],\dmax}(\delta/9, \Set_{[\mumodany,\Lmodany,\Dmodany,\Amodany]^K_{\Space}})
 \leq  \ConstH_{[\mumodany,\Lmodany,\Dmodany,\Amodany]^K_{\dimSpace}}  + \DimH_{[\mumodany,\Lmodany,\Dmodany,\Amodany]^K_{\dimSpace}}
 \Log \frac{1}{\delta}
\end{align*}
where 
\begin{alignI}
  \DimH_{[\mumodany,\Lmodany,\Dmodany,\Amodany]^K_{\dimSpace}}
= \dim\left(\Theta_{[\mumodany,\Lmodany,\Dmodany,\Amodany]^K_{\dimSpace}}\right) = c_{\mumodany} \DimH_{\mu,\dimSpace} + c_{\Lmodany} \DimH_{L}+
 c_{\Dmodany} \DimH_{D,\dimSpace} + c_{\Amodany} \DimH_{A,\dimSpace}
\end{alignI}
and
\begin{alignI}
\ConstH_{[\mumodany,\Lmodany,\Dmodany,\Amodany]^K_{\dimSpace}} 
=
c_{\mumodany} \ConstH_{\mumod,\dimSpace}  + c_{\Lmodany} \ConstH_{L,\dimSpace} +
 c_{\Dmodany} \ConstH_{D,\dimSpace}  + c_{\Amodany} \ConstH_{A,\dimSpace} 
\end{alignI}
with
\begin{alignI}
\begin{cases}
  c_{\mumodknown}=c_{\Lmodknown}=c_{\Dmodknown}=c_{\Amodknown}=0\\
  c_{\mumodfree}=c_{\Lmodfree}=c_{\Dmodfree}=c_{\Amodfree}=K\\
  c_{\mumodsame}=c_{\Lmodsame}=c_{\Dmodsame}=c_{\Amodsame}=1
\end{cases},
\end{alignI}
\begin{align*}
\begin{cases}
  \DimH_{\mumod,\dimSpace} = \dimSpace \\
\DimH_{\Lmod}=1\\
\DimH_{\Dmod,\dimSpace}=\frac{\dimSpace(\dimSpace-1)}{2}\\
\DimH_{\Amod,\dimSpace}  = \dimSpace - 1
\end{cases}
\quad\text{and}\quad
\begin{cases}
\ConstH_{\mumod,\dimSpace}=   \dimSpace 
\Log\left(
1+\frac{18\constcosh a\dimSpace}{\sqrt{\gammakappa \Lm \lambdam
   \frac{\lambdam}{\lambdaM}}}
\right)
 \\
\ConstH_{\Lmod,\dimSpace}= \Log\left( 1+\frac{39}{2}\constcosh\Log\left(\frac{\LM}{\Lm}\right)\dimSpace\right) \\
 \ConstH_{\Dmod,\dimSpace}= \frac{\dimSpace(\dimSpace-1)}{2}
\left( \frac{2\Log \ConstSzarek}{\dimSpace(\dimSpace-1)} +
   \left( 
\Log\left( 
126\constcosh\frac{\lambdaM}{\lambdam}\dimSpace\right)
\right) \right)\\
 \ConstH_{\Amod,\dimSpace} =(\dimSpace-1) \Log \left( 2+\frac{255}{2}\constcosh\frac{\lambdaM}{\lambdam}\Log\left(\frac{\lambdaM}{\lambdam}\right)\dimSpace\right) 
\end{cases}
\end{align*}
where $\ConstSzarek$ is a universal constant.

Furthermore,
for any $\dimSpace\leq\dimSp$
\begin{align*}
 \ConstH_{\mumod,\dimSpace}& \leq \ConstMultH_{\mumod,\dimSp} \DimH_{\mumod,\dimSpace}  \\
\ConstH_{\Lmod,\dimSpace}& \leq \ConstMultH_{\Lmod,\dimSp} \DimH_{\Lmod,\dimSpace}\\
 \ConstH_{\Dmod,\dimSpace}& \leq  \ConstMultH_{\Dmod,\dimSp} \DimH_{\Dmod,\dimSpace}\\
 \ConstH_{\Amod,\dimSpace} & \leq \ConstMultH_{\Amod,\dimSp} \DimH_{\Amod,\dimSpace}
\end{align*}
with
\begin{align*}
 \ConstMultH_{\mumod,\dimSp}&=  \Log\left(
1+\frac{18\constcosh a\dimSp}{\sqrt{\gammakappa \Lm \lambdam
   \frac{\lambdam}{\lambdaM}}}
\right)\\
\ConstMultH_{\Lmod,\dimSp}&= \Log\left( 1+\frac{39}{2}\constcosh\Log\left(\frac{\LM}{\Lm}\right)\dimSp\right)\\
 \ConstMultH_{\Dmod,\dimSp}&= \left( 2 \Log \ConstSzarek +
   \left( 
\Log\left( 
126\constcosh\frac{\lambdaM}{\lambdam}\dimSp\right)
\right) \right)\\
 \ConstMultH_{\Amod,\dimSp} &= \Log \left( 2+\frac{255}{2}\constcosh\frac{\lambdaM}{\lambdam}\Log\left(\frac{\lambdaM}{\lambdam}\right)\dimSp\right)
\end{align*}
and, uniformly over $K$,
\begin{align*}
  \ConstH_{[\mumodany,\Lmodany,\Dmodany,\Amodany]^K_{\dimSpace}} & \leq
\max_{\mumodany',\Lmodany,\Dmodany',\Amodany',K'}\bigg( 
\ConstMultH_{\mumod,\dimSp} 
\frac{c_{\mumodany'} 
K'}
{c_{\mumodany'} K' + c_{\Lmodany'}
 + c_{\Dmodany'} \frac{K'(K'-1)}{2} + c_{\Amodany'} (K'-1)}\\
&\qquad\qquad\qquad\qquad
+\ConstMultH_{\Lmod,\dimSp}
\frac{c_{\Lmodany'}}
{c_{\mumodany'} K' + c_{\Lmodany'}
 + c_{\Dmodany'} \frac{K'(K'-1)}{2} + c_{\Amodany'} (K'-1)}\\
&\qquad\qquad\qquad\qquad
 +
\ConstMultH_{\Dmod,\dimSp}
\frac{c_{\Dmodany'} \frac{K'(K'-1)}{2}}
{c_{\mumodany'} K' + c_{\Lmodany'}
 + c_{\Dmodany'} \frac{K'(K'-1)}{2} + c_{\Amodany'} (K'-1)}\\
& \qquad\qquad\qquad\qquad +
\ConstMultH_{\Amod,\dimSp}
\frac{c_{\Amodany'} (K'-1)}
{c_{\mumodany'} K' + c_{\Lmodany'}
 + c_{\Dmodany'} \frac{K'(K'-1)}{2} + c_{\Amodany'} (K'-1)}
\bigg) \DimH_{[\mumodany,\Lmodany,\Dmodany,\Amodany]^K_{\dimSpace}}
\\
& \leq
\max(\ConstMultH_{\mu,\dimSp},\ConstMultH_{L,\dimSp},\ConstMultH_{D,\dimSp},\ConstMultH_{A,\dimSp})
\DimH_{[\mumodany,\Lmodany,\Dmodany,\Amodany]^K_{\dimSpace}}
\end{align*}
where the $\max$ is taken over every Gaussian set type and every
number of classes considered.
\end{proposition}
\Cref{prop:entcoll} is obtained by setting $\kappa=1$ and using the
crude bounds $1/9\leq\gammakappa\leq 1/4$, $1\leq\constcosh\leq 2$.
}{
\begin{proposition}
\label{prop:entcoll}
For any $\delta \in (0,\sqrt{2}]$, 
  \begin{align*}
    H_{[\cdot],\dmax}(\delta/9, \Set_{[\mumodany,\Lmodany,\Dmodany,\Amodany]^K_{\Space}})
 \leq  \ConstH_{[\mumodany,\Lmodany,\Dmodany,\Amodany]^K_{\dimSpace}}  + \DimH_{[\mumodany,\Lmodany,\Dmodany,\Amodany]^K_{\dimSpace}}
 \Log \frac{1}{\delta}
\end{align*}
where 
\begin{alignI}
  \DimH_{[\mumodany,\Lmodany,\Dmodany,\Amodany]^K_{\dimSpace}}
= \dim\left(\Theta_{[\mumodany,\Lmodany,\Dmodany,\Amodany]^K_{\dimSpace}}\right) = c_{\mumodany} \DimH_{\mu,\dimSpace} + c_{\Lmodany} \DimH_{L}+
 c_{\Dmodany} \DimH_{D,\dimSpace} + c_{\Amodany} \DimH_{A,\dimSpace}
\end{alignI}
and
\begin{alignI}
\ConstH_{[\mumodany,\Lmodany,\Dmodany,\Amodany]^K_{\dimSpace}} 
=
c_{\mumodany} \ConstH_{\mumod,\dimSpace}  + c_{\Lmodany} \ConstH_{L,\dimSpace} +
 c_{\Dmodany} \ConstH_{D,\dimSpace}  + c_{\Amodany} \ConstH_{A,\dimSpace} 
\end{alignI}
with
\begin{alignI}
\begin{cases}
  c_{\mumodknown}=c_{\Lmodknown}=c_{\Dmodknown}=c_{\Amodknown}=0\\
  c_{\mumodfree}=c_{\Lmodfree}=c_{\Dmodfree}=c_{\Amodfree}=K\\
  c_{\mumodsame}=c_{\Lmodsame}=c_{\Dmodsame}=c_{\Amodsame}=1
\end{cases},
\end{alignI}
\begin{align*}
\begin{cases}
  \DimH_{\mumod,\dimSpace} = \dimSpace \\
\DimH_{\Lmod}=1\\
\DimH_{\Dmod,\dimSpace}=\frac{\dimSpace(\dimSpace-1)}{2}\\
\DimH_{\Amod,\dimSpace}  = \dimSpace - 1
\end{cases}
\quad\text{and}\quad
\begin{cases}
\ConstH_{\mumod,\dimSpace}=   \dimSpace \left( \Log\left( 1 + 
108\frac{a}{\sqrt{\Lm \lambdam
   \frac{\lambdam}{\lambdaM}}}\dimSpace
\right)  \right) \\
\ConstH_{\Lmod,\dimSpace}= \Log\left( 1 + 39 \Log\left(\frac{\LM}{\Lm}\right)\dimSpace\right) \\
 \ConstH_{\Dmod,\dimSpace}= \frac{\dimSpace(\dimSpace-1)}{2}
\left( \frac{2 \Log \ConstSzarek}{\dimSpace(\dimSpace-1)} +
   \left( 
\Log\left( 
252 \frac{\lambdaM}{\lambdam}\dimSpace\right)
\right) \right)\\
 \ConstH_{\Amod,\dimSpace} =(\dimSpace-1) \left( \Log \left(
     2 + 255 \frac{\lambdaM}{\lambdam}\Log\left(\frac{\lambdaM}{\lambdam}\right)\dimSpace\right) \right)
\end{cases}
\end{align*}
where $\ConstSzarek$ is a universal constant.
\end{proposition}
}

\ifthenelse{\boolean{extended}}{
\begin{proof}[Proof of
    \cref{prop:entcollext}]}{
\begin{proof}[Proof of
    \cref{prop:entcoll}]}

We consider all models $\Set_{[\mumodany\, \Lmodany\,  \Amodany\,  \Dmodany]^K_\Space}$ at once by
a ``tensorial'' construction of a suitable  $\delta/9$ bracket collection.

We first define a set of grids for the mean $\mu$, the volume $L$, the
eigenvector matrix $D$ and the renormalized eigenvalue matrix $A$ from
which one constructs the bracket collection.
\begin{itemize}
\item For any $\delta_\mumod$, the grid
  $\Grid_{\mumod}(a,\dimSpace,\delta_\mumod)$ of $[-a,a]^{\dimSpace}$:
\begin{align*}
\Grid_{\mumod}(a,\dimSpace,\delta_\mumod) = \left\{ g\delta_\mumod \middle|  g\in \Z^{\dimSpace},
\|g\|_\infty  \leq \frac{a}{\delta_\mumod} \right\}.
\end{align*}
\item For any $\delta_\Lmod$, the grid $\Grid_{\Lmod}(\Lm,\LM,\delta_\Lmod)$ of $[\Lm,\LM]$:
  \begin{align*}
    \Grid_{\Lmod}(\Lm,\LM,\delta_\Lmod) = \left\{ \Lm (1+\delta_\Lmod)^g
\middle| g \in \N, \Lm (1+\delta_\Lmod)^g \leq \LM
\right\}.
  \end{align*}
\item For any $\delta_\Dmod$, the grid $\Grid_{\Dmod}(\dimSpace,\delta_\Dmod)$ of $SO(\dimSpace)$
made of the elements of a $\delta_\Dmod$-net  with
respect to the $\|\cdot\|_2$ operator norm (as described by \textcite{szarek98:_metric_gdans_banac_center_publ}).
\item  For any
$\delta_\Amod$, the grid $\Grid_{\Amod}(\lambdam,\lambdaM,\dimSpace,\delta_\Amod)$ of
$\Diag(\lambdam,\lambdaM(1+\delta_\Amod),\dimSpace)$:
\begin{align*}
  \Grid_{\Amod}(\lambdam,\lambdaM,\dimSpace,\delta_\Amod) = \left\{
A \in \Diag(\lambdam,\lambdaM(1+\delta_\Amod),\dimSpace) \middle| \forall 1 \leq i < \dimSpace, \exists g_i \in \N,
A_i = \lambdam (1+\delta_\Amod)^{g_i} \right\}.
\end{align*}
\end{itemize}

Obviously, for any  $\mu \in [-a,a]$, there is a
$\tilde{\mu}\in\Grid_{\mumod}(a,\dimSpace,\delta_\mumod)$ such that 
\[
  \|\tilde{\mu}-\mu\|^2 \leq \dimSpace \delta_\mumod^2 
\]
while 
\[
\left|\Grid_{\mumod}(a,\dimSpace,\delta_\mumod)\right| \leq \left( 1
  + 2\frac{a}{\delta_\mumod} \right)^{\dimSpace} \leq \max\left(2^{\dimSpace},
  \left( \frac{4a}{\delta_\mumod} \right)^{\dimSpace}\right).
\]
In the same fashion, for any $L$ in $[\Lm,\LM]$, there is a $\tilde{L}\in\Grid_{\Lmod}(\Lm,\LM,\delta_\Lmod)$ such that
\(
  (1+\delta_\Lmod)^{-1} L_{j_L} < L \leq  L_{j_L}
\)
while
\[
\left| \Grid_{\Lmod}(\Lm,\LM,\delta_\Lmod) \right| \leq 1 + \frac{\Log\left(\frac{\LM}{\Lm}\right)}{\Log(1+\delta_\Lmod)}.
\]
If we further assume that $\delta_\Lmod \leq \frac{1}{12}$ then
$\Log(1+\delta_\Lmod) \geq \frac{12}{13} \delta_\Lmod$ and
\[
\left| \Grid_{\Lmod}(\Lm,\LM,\delta_\Lmod) \right| \leq 1 +
\frac{13\Log\left(\frac{\LM}{\Lm}\right)}{12\delta_\Lmod}.
\]

By definition on a $\delta_\Dmod$-net, for any $D \in SO(\dimSpace)$ there
is a $\tilde{D} \in \Grid_{\Dmod}(\dimSpace,\delta_\Dmod)$ such that
\begin{align*}
  \forall x, \|(\tilde{D}-D)x\|_2 \leq \delta_\Dmod \|x\|_2. 
\end{align*}
As proved by \textcite{szarek98:_metric_gdans_banac_center_publ}, it
exists a universal constant $\ConstSzarek$ such that, as soon as
$\delta_\Dmod \leq 1$
\begin{align*}
  \left| \Grid_{\Dmod}(\dimSpace,\delta_\Dmod) \right| \leq \ConstSzarek \left(
    \frac{1}{\delta_\Dmod} \right)^{\frac{\dimSpace(\dimSpace-1)}{2}} 
\end{align*}
where $\frac{\dimSpace(\dimSpace-1)}{2}$ is the intrinsic dimension of $SO(\dimSpace)$.

The structure of the grid $\Grid_{\Amod}(\lambdam,\lambdaM,\dimSpace,\delta_\Amod)$
is more complex. Although, looking at condition on the
$\dimSpace-1$ first diagonal values, 
\begin{align*}
  \left| \Grid_{\Amod}(\lambdam,\lambdaM,\dimSpace,\delta_\Amod) \right|
\leq \left( 2 + \frac{\Log\left(\frac{\lambdaM}{\lambdam}\right)}{\Log(1+\delta_\Amod)}\right)^{\dimSpace-1}
\end{align*}
where $\dimSpace-1$ is the intrinsic dimension of
$\Diag(\lambdam,\lambdaM,\dimSpace)$. If we further assume that $\delta_\Amod \leq
\frac{1}{84}$ then $\Log(1+\delta_\Amod) \geq \frac{84}{85} \delta_\Amod$ and thus
\begin{align*}
    \left| \Grid_{\Amod}(\lambdam,\lambdaM,\dimSpace,\delta_\Amod) \right|
\leq \left( 2 +
  \frac{85\Log\left(\frac{\lambdaM}{\lambdam}\right)}{84\delta_\Amod}\right)^{\dimSpace-1}
.
\end{align*}
A key to the succes of this construction is the following
approximation property of this grid
 \ifthenelse{\boolean{extended}}{proved later}{obtained in our
  technical report~\cite{cohen11:_condit_densit_estim_penal_app} with a
  calculatory proof}:
\begin{lemma}
\label{lem:gridA}
For $A\in\Diag(\lambdam,\lambdaM,\dimSpace)$ there is
$\tilde{A}\in\Grid_{\Amod}(\lambdam,\lambdaM,\dimSpace,\delta_\Amod)$ such that
\begin{align*}
  |\tilde{A}_{i,i}^{-1}-A_{i,i}^{-1}|\leq \delta_\Amod \lambdam^{-1}.
\end{align*}
\end{lemma}

Define $\displaystyle c_{\mumodknown}=c_{\Lmodknown}=c_{\Dmodknown}=c_{\Amodknown}=0,
  c_{\mumodfree}=c_{\Lmodfree}=c_{\Dmodfree}=c_{\Amodfree}=K,
  c_{\mumod}=c_{\Lmod}=c_{\Dmod}=c_{\Amod}=1$.
Let $f_{K,\mumodany,\dimSpace}$ be the application
from \begin{alignI}
\left(\R^{\dimSpace}\right)^{c_{\mumodany}}
\end{alignI}
 to 
\begin{alignI}
\R^K
\end{alignI}
defined by
\begin{align*}
  \begin{cases}
  0 \mapsto (\mu_{0,1},\ldots,\mu_{0,K}) & \text{if $\mumodany=\mumodknown$}\\
 (\mu_1,\ldots,\mu_K) \mapsto (\mu_1,\ldots,\mu_K) & \text{if
   $\mumodany=\mumodfree$}\\
 \mu \mapsto (\mu,\ldots,\mu) & \text{if
   $\mumodany=\mumodsame$}
  \end{cases},
\end{align*}
and $f_{K,\Lmodany}$ (respectively 
$f_{K,\Dmodany,\dimSpace}$ and $f_{K,\Amodany,\dimSpace}$)
be the similar
application from $\left(\R^+\right)^{c_{\Lmodany}}$ into
$\left(\R^+\right)^K$ (respectively from $
\left( SO(\dimSpace) \right)^{c_{\Dmodany}}$ into $\left(
  SO(\dimSpace) \right)^K$ and 
from $\left(\Diag(0,+\infty,\dimSpace)\right)^{c_{\Amodany}}$ into $\left(\Diag(0,+\infty,\dimSpace)\right)^K$).

By definition, the image of
\begin{align*}
\left([-a,a]^{\dimSpace}\right)^{c_{\mumodany}}
\times \left([\Lm,\LM]\right)^{c_{\Lmodany}}
\times \left( SO(\dimSpace) \right)^{c_{\Dmodany}}
\times \left(\Diag(\lambdam,\lambdaM,\dimSpace)\right)^{c_{\Amodany}}
\end{align*}
by $\left( f_{K,\mumodany,\dimSpace} \otimes
f_{L_{K,\cdot},\dimSpace}\otimes f_{K,\Dmodany,\dimSpace} \otimes 
f_{K,\Amodany} \right)$ 
is, up to reordering, the set of parameters
of all $K$-tuples of Gaussian densities of type $[\mumodany\,\Lmodany,\Dmodany,\Amodany]^K$.

We construct our $\delta/9$ bracket covering with a grid on those parameters.
For any $K$-tuple of Gaussian parameters 
\begin{alignI}
( (\mu_1,\Sigma_1),\ldots,(\mu_K,\Sigma_K) )
\end{alignI}
and any $\deltaVar$, we associate the $K$-tuple of pairs
\begin{align*}
\bigg(&
\left((1+\kappag\deltaVar)^{-\dimSpace}
  \Gauss_{\mu_1,(1+\deltaVar)^{-1} \Sigma_1},
(1+\kappag\deltaVar)^{\dimSpace} \Gauss_{\mu_1,(1+\deltaVar) \Sigma_1 }
\right),\ldots, \\
&\qquad\qquad\left((1+\kappag\deltaVar)^{-\dimSpace}
  \Gauss_{\mu_K,(1+\deltaVar)^{-1} \Sigma_K},
(1+\kappag\deltaVar)^{\dimSpace} \Gauss_{\mu_K,(1+\deltaVar) \Sigma_K }
\right)
 \bigg).
\end{align*}

We prove \ifthenelse{\boolean{extended}}{}{in our technical
  report~\cite{cohen11:_condit_densit_estim_penal_app}} that, 
\ifthenelse{\boolean{extended}}{
for $\gammakappa$ and $\constcosh$ defined in
\cref{prop:entcollext} and any $\kappag\geq\frac{3}{4}$, }{
for $\gammakappa = 18/49$ and $\constcosh=\sqrt{\cosh(\frac{1}{6})+
    \frac{1}{2}}$,}
  the choice
\begin{align*}
 \delta_\mumod =  \frac{\sqrt{\gammakappa \Lm \lambdam
   \frac{\lambdam}{\lambdaM}}}{9\constcosh}\frac{\delta}{\dimSpace},\ 
 \delta_\Lmod  = \frac{1}{18\constcosh}\frac{\delta}{\dimSpace} \leq
 \frac{1}{12},\ 
\delta_\Dmod = \delta_\Amod  =
\frac{1}{126\constcosh}
\frac{\lambdam}{\lambdaM}  \frac{\delta}{\dimSpace} \leq
\frac{1}{84},\ 
\deltaVar = \frac{1}{9\constcosh}\frac{\delta}{\dimSpace} \leq \frac{1}{8}
\end{align*}
is such that
 the image of
\begin{align*}
\left(
\Grid_{\mumod}(a,\dimSpace,\delta_\mumod)
\right)^{c_{\mumodany}}
\times \left(
\Grid_{\Lmod}(\Lm,\LM,\delta_\Lmod)
\right)^{c_{\Lmodany}}
\times \left(
\Grid_{\Dmod}(\dimSpace,\delta_\Dmod)
 \right)^{c_{\Dmodany}}
\times \left(
\Grid_{\Amod}(\lambdam,\lambdaM,\dimSpace,\delta_\Amod)
\right)^{c_{\Amodany}}
\end{align*}
by $f_{K,\mumodany,\dimSpace} \otimes
f_{L_{K,\cdot},\dimSpace}\otimes f_{K,\Dmodany,\dimSpace} \otimes 
f_{K,\Amodany}$ is a set of parameters corresponding to a set
of pairs that is a
$\delta/9$-bracket covering of 
$\Set_{[\mumodany\,\Lmodany\,\Dmodany\,\Amodany]^K_{\Space}}$ for the
$\dmax$ norm.

\ifthenelse{\boolean{extended}}{
Indeed, as proved later,
\begin{lemma}
\label{lem:gaussbracketnet}
Let $\kappag\geq
\frac{3}{4}$,
\( \displaystyle
  \gammakappa  =
 \min\left(\frac{3(\kappag-\frac{3}{4})}{2(1+\frac{\kappag}{6})(1+\frac{1}{6})(1+\frac{1}{12})},\frac{(\kappag-\frac{1}{2})}{2(1+\frac{\kappag}{6})(1+\frac{1}{6})}\right)
\) and $\constcosh= \sqrt{\kappag^2\cosh(\frac{\kappag}{6}) +
    \frac{1}{2}}$. For any $0<\delta\leq \sqrt{2}$, any $\dimSpace\geq 1$ 
 and any $\deltaVar \leq
\frac{1}{9\constcosh}\frac{\delta}{\dimSpace}$, 

Let 
$(\tilde{\mu}, \tilde{L},\tilde{A},\tilde{D})\in [-a,a]^{\dimSpace} \times [\Lm,\LM] \times 
\Diag(\lambdam,+\infty) \times SO(\dimSpace)$, define
 $\tilde{\Sigma}=\tilde{L} \tilde{D}
\tilde{A} \tilde{D}'$,
\begin{align*}
  t^-(x) = (1+\kappag\deltaVar)^{-\dimSpace} \Gauss_{\tilde{\mu},(1+\deltaVar)^{-1} \tilde{\Sigma}}(x)
\quad\text{and}\quad
  t^+(x) = (1+\kappag\deltaVar)^{\dimSpace} \Gauss_{\tilde{\mu},(1+\deltaVar) \tilde{\Sigma}}(x).
\end{align*}
then $[t^-,t^+]$ is a $\delta/9$ Hellinger bracket.

Furthermore, let $(\mu,L,A,D) \in [-a,a]^{\dimSpace} \times  [\Lm,\LM] \times 
\Diag(\lambdam,\lambdaM) \times SO(\dimSpace)$ and define
$\Sigma=L D A D'$.
If
\begin{align*}
\begin{cases}
\|\mu-\tilde{\mu}\|^2 \leq \dimSpace \gammakappa \Lm \lambdam
  \frac{\lambdam}{\lambdaM} \deltaVar^2\\
(1+\frac{\deltaVar}{2})^{-1} \tilde{L} \leq L \leq \tilde{L}\\
\forall 1\leq i \leq \dimSpace,\quad |A_{i,i}^{-1}-\tilde{A}_{i,i}^{-1}| \leq \frac{1}{14}
\frac{1}{\lambdaM} \deltaVar\\
 \forall x \in \R^{\dimSpace},\quad \|Dx-\tilde{D}x\| \leq
\frac{1}{14} \frac{\lambdam}{\lambdaM}
\deltaVar \|x\|
\end{cases}
\end{align*}
then \(\displaystyle
t^-(x) \leq \Gauss_{\mu,\Sigma}(x) \leq t^+(x).
\)
\end{lemma}
 By
definition of $\dmax$, this implies that our choice of
$\delta_\mumod$, $\delta_\Lmod$, $\delta_\Dmod$,  $\delta_\Amod$ and $\deltaVar$
is such that every $K$-tuple of pairs of the collections is a
$\delta/9$-bracket and they cover the whole set.}{}

The cardinality of this $\delta/9$-bracket covering is bounded by
\begin{align*}
\ifthenelse{\boolean{extended}}{
&\left(
\left(1+\frac{2a}{\frac{\sqrt{\gammakappa \Lm \lambdam
   \frac{\lambdam}{\lambdaM}}}{9\constcosh}\frac{\delta}{\dimSpace}} \right)^{\dimSpace}
 \right)^{c_{\mumodany}}
\times 
\left(
\left(1+ \frac{13\Log\left(\frac{\LM}{\Lm}\right)}{12\frac{1}{18\constcosh}\frac{\delta}{\dimSpace}}\right)
\right)^{c_{\Lmodany}}\\
& \qquad
\times
\left(
\ConstSzarek \left(
    \frac{1}{\frac{1}{126\constcosh} \frac{\lambdam}{\lambdaM}  \frac{\delta}{\dimSpace} 
} \right)^{\frac{\dimSpace(\dimSpace-1)}{2}} 
\right)^{c_{\Dmodany}}\\
& \qquad
\times
\left(
\left( 2 + \left( \frac{85\Log\left(\frac{\lambdaM}{\lambdam}\right)}{84\frac{1}{126\constcosh} \frac{\lambdam}{\lambdaM}  \frac{\delta}{\dimSpace} 
}\right)\right)^{\dimSpace-1}
\right)^{c_{\Amodany}}
\\
& \leq}{
&}
  \left(
\left( 1+
  \frac{18a\constcosh\dimSpace}{\sqrt{\gammakappa \Lm \lambdam
   \frac{\lambdam}{\lambdaM}}\delta} \right)^{\dimSpace}
  \right)^{c_{\mumodany}}
\times 
\left(
\left(1+\frac{39\constcosh\Log\left(\frac{\LM}{\Lm}\right)\dimSpace}{2\delta}\right)
\right)^{c_{\Lmodany}}\\
& \qquad
\times
\left(
\ConstSzarek \left(
    \frac{126\constcosh\frac{\lambdaM}{\lambdam}\dimSpace}{\delta}
 \right)^{\frac{\dimSpace(\dimSpace-1)}{2}} 
\right)^{c_{\Dmodany}}
\times
\left(
\left( 2 +  \left( \frac{255\constcosh\frac{\lambdaM}{\lambdam}\Log\left(\frac{\lambdaM}{\lambdam}\right)\dimSpace}
{2\delta }\right)\right)^{\dimSpace-1}
\right)^{c_{\Amodany}}
\end{align*}
So
\begin{align*}
     H_{[\cdot],\dmax}&(\delta/9, \Set_{[\mumodany,\Lmodany,\Dmodany,\Amodany]^K_{\Space}})\\
 & \leq c_{\mumodany} \dimSpace \left( \Log\left(
1+\frac{18\constcosh a\dimSpace}{\sqrt{\gammakappa \Lm \lambdam
   \frac{\lambdam}{\lambdaM}}}
\right) +
   \Log\frac{1}{\delta} \right)
+ c_{\Lmodany} \left( \Log\left( 1+\frac{39}{2}\constcosh\Log\left(\frac{\LM}{\Lm}\right)\dimSpace\right)+ \Log \frac{1}{\delta} \right)\\
& \quad + c_{\Dmodany} \frac{\dimSpace(\dimSpace-1)}{2}
\left( \frac{2 \Log \ConstSzarek}{\dimSpace(\dimSpace-1)} +
\Log\left( 
126\constcosh\frac{\lambdaM}{\lambdam}\dimSpace\right)
+ \Log \frac{1}{\delta}
 \right)
\\
& \quad 
+c_{\Amodany} (\dimSpace-1) \left( \Log \left( 2+\frac{255}{2}\constcosh\frac{\lambdaM}{\lambdam}\Log\left(\frac{\lambdaM}{\lambdam}\right)\dimSpace\right) + \Log \frac{1}{\delta} \right)
\end{align*}
which concludes the proof\ifthenelse{\boolean{extended}}{.}{ as soon
  as one notices that $1/9\leq\gammakappa\leq1/3$ and $1\leq\constcosh\leq2$.}
\end{proof}

\ifthenelse{\boolean{extended}}{
\subsection{Entropy of spatial mixtures (Lemmas)}


  \begin{proof}[Proof of \cref{lemma:entropy}]
    This a variation around the proof of \textcite{genovese00:_rates_gauss}.

    Let $\{ \left[\pi^-_1,\pi^+_1\right], \ldots,
    \left[\pi^-_{N_{\Simplex_{K-1}}},\pi^+_{N_{\Simplex_{K-1}}}\right] \}$ be a minimal
    covering of $\delta/3$
    Hellinger bracket of the simplex $\Simplex_{K-1}$.
Let
\begin{align*} 
\left\{ 
\left[(t^-_{\Space,1,1},\ldots,t^-_{\Space,K,1}),
(t^+_{\Space,1,1},\ldots,t^+_{\Space,K,1})\right]
,\ldots,
\left[(t^-_{\Space,1,N_{\Space,K}},\ldots,t^-_{\Space,K,N_{\Space,K}}),
(t^+_{\Space,1,N_{\Space,K}},\ldots,t^+_{\Space,K,N_{\Space,K}})\right]
\right\}
\end{align*}
be a minimal covering of $\delta/9$ sup norm Hellinger bracket of
$\Set_{\Space,K}$ and
$\{ \left[t^-_{\Space^\perp,1},t^+_{\Space^\perp,1}\right], \ldots,  \left[t^-_{\Space_\perp,N_{\space_\perp}},t^+_{\Space_\perp,N_{\Space_\perp}}\right]
\}$ be a minimal covering of $\delta/9$ Hellinger bracket of $\Set_{\Space^\perp}$.
By definition,
$\Log N_{\Simplex_{K-1}} =    H_{[\cdot],\d}(\delta/3,
\Simplex_{K-1})$,
$\Log N_{\Space,K} = H_{[\cdot],\dmax}(\delta/9,
\Set_{\Space,K})$ and $\Log N_{\Space^\perp} = H_{[\cdot],\d}(\delta/9,
\Set_{\Space^\perp})$.

By construction,
\begin{align*}
  \Bigg\{ \Bigg[ 
\sum_{\LeafXl\in\PartX}\left( \sum_{k=1}^K \pi^{-}_{i[\LeafXl],k}\,
  t^{-}_{\Space,k,j}\left(y\right) \,
  t^{-}_{\Space^\perp,j_{\Space^\perp}}\left(y\right)\right) \Charac{x\in\LeafXl},
&\sum_{\LeafXl\in\PartX}\left( \sum_{k=1}^K \pi^+_{i[\LeafXl],k}\,
  t^{+}_{\Space,k,j}\left(y\right) \,
  t^{+}_{\Space^\perp,l}\left(y\right) \right) \Charac{x\in\LeafXl}
\Bigg] \Bigg|\\
& 1 \leq i[\LeafXl] \leq N_{\Simplex_{K-1}}, 1 \leq j \leq N_{\Space,K},
1 \leq l \leq N_{\Space^\perp}
\Bigg\}
\end{align*}
is a covering of model $\model_{K,\PartX,\Set}$ of cardinality 
\begin{alignI}
\exp\left( 
| \PartX|
    H_{[\cdot],\d}(\delta/3, \Simplex_{K-1}) +
    H_{[\cdot],\dmax}(\delta/9, \Set_{\Space,K}) +
H_{[\cdot],\d}(\delta/9, \Set_{\Space^\perp}) 
\right).
\end{alignI}

It remains thus only to prove that each bracket is of sup norm Hellinger $\dsup$
width smaller than $\delta$.

Using 
\begin{lemma}
\label{lemma:bracketproduct}
 For any $\delta$ Hellinger brackets $\left[t^-(x),t^+(x)\right]$,
if for any $x$ $\left[u^-(x,y),u^+(x,y)\right]$ is a $\delta$ bracket
 and $\delta\leq \sqrt{2}/3$,
  then
 $\left[t^-(x)\, u^-(x,y),t^+(x)\, u^+(x,y)\right]$ is a
  $3\delta$ Hellinger bracket.
\end{lemma}
we obtain immediately
\begin{align*}
  \d^2\left(t^{-}_{\Space,k,j_{\Space}}\left(\cdot\right)\,t^{-}_{\Space^\perp,l}\left(\cdot\right),
t^{+}_{\Space,k,j_{\Space}}\left(\cdot\right)\,t^{+}_{\Space^\perp,l}\left(\cdot\right)\right)
\leq 9 (\delta/9)^2 = (\delta/3)^2.
\end{align*}
Let $\left[t^{--}_{k,j,l},t^{++}_{k,j,l}\right]$ denote  the corresponding $\delta/3$
Hellinger bracket.

By definition,
\begin{align*}
  \dtwosup& \left( \sum_{\LeafXl\in\PartX}
\left( \sum_{k=1}^K
  \pi^{-}_{i[\LeafXl],k}\,t^{--}_{k,j,l}\left(y\right)\right)
\Charac{x\in \LeafXl} ,\right.
\left.\sum_{\LeafXl\in\PartX}
\left( \sum_{k=1}^K
  \pi^+_{i[\LeafXl],k}\,t^{++}_{k,j,l}\left(y\right)\right)\Charac{x \in \LeafXl}
\right)\\
 & = \sup_{\LeafXl\in\PartX} \d^2\left(
   \sum_{k=1}^K \pi^{-}_{i[\LeafXl],k}\,t^{--}_{k,j,l} , \sum_{k=1}^K
   \pi^+_{i[\LeafXl],k}\,t^{++}_{k,j,l} \right)\\
& \leq \sup_{i, j,l} \, \d^2\left(
  \sum_{k=1}^K \pi^{-}_{i,k} \, t^{--}_{k,j,l} , \sum_{k=1}^K
  \pi^+_{i,k} \, t^{++}_{k,j,l} \right)
\end{align*}
Seeing $ \pi_{i,k} g_{k,j,l}(y)$ as a function of $k$ and
  $y$, we can use 
\begin{lemma}\label{lem:helprod}
For any brackets $\left[t^-(x),t^+(x)\right]$
  and if for any $x$ $\left[u^-(x,y),u^+(x,y)\right]$ is a  bracket
then
\begin{align*}
  \d^2_y\left(\int_x t^-(x) \, u^-(x,y) \, \ud\meas_{x}(x),\int_x t^+(x)
    \,u^+(x,y) \, \ud\meas_{x}(x)\right)
& \leq \d^2_{x,y}\left( t^-(x) \, u^-(x,y), t^+(x) \, u^+(x,y) \right)
\end{align*}
 \end{lemma}
to obtain
\begin{align*}
  \dtwosup& \left( \sum_{\LeafXl\in\PartX}
\left( \sum_{k=1}^K
  \pi^{-}_{i[\LeafXl],k}\,t^{--}_{k,j,l}\left(y\right)\right)
\Charac{x\in \LeafXl} ,\right.
\left.\sum_{\LeafXl\in\PartX}
\left( \sum_{k=1}^K
  \pi^+_{i[\LeafXl],k}\,t^{++}_{k,j,l}\left(y\right)\right)\Charac{x \in \LeafXl}
\right)\\
& \leq \sup_{i, j,l } \d^2_{k,y}\left( \pi^{-}_{i,k} \, t^{--}_{k,j,l}(y) , \pi^{+}_{i,k}\,
  t^{++}_{k,j,l}(y) \right)
\intertext{and then using  again \cref{lemma:bracketproduct}}
& \leq 9 (\delta/3)^2 = \delta^2.
\end{align*}
  \end{proof}

\begin{proof}[Proof of \cref{lemma:bracketproduct}]
  \begin{align*}
 \d^2&(t^-(x) \, u^-(x,y),t^+(x) \, u^+(x,y))\\
 &=   \iint \left(
   \sqrt{t^+(x) \, u^+(x,y)}-\sqrt{t^-(x) \, u^-(x,y)} \right)^2 \,
 \ud\meas_x(x) \, \ud\meas_y(y)\\
& =\iint \left( \sqrt{t^+(x)} \left(    \sqrt{u^+(x,y)}-\sqrt{u^-(x,y)}
\right) + \left( \sqrt{t^+(x)} -\sqrt{t^-(x)} \right) \sqrt{u^{-}(x,y)}
\right)^2 \, \ud\meas_x(x) \, \ud\meas_y(y)\\
& = \iint \bigg( t^+(x) \left(    \sqrt{u^+(x,y)}-\sqrt{u^-(x,y)}
\right)^2 + \left( \sqrt{t^+(x)} -\sqrt{t^-(x)} \right)^2 u^{-}(x,y)\\
&\qquad+ 2 \sqrt{t^+(x)}  \left( \sqrt{t^+(x)} -\sqrt{t^-(x)}
\right) \sqrt{u^{-}(x,y)} \left(    \sqrt{u^+(x,y)}-\sqrt{u^-(x,y)}
\right)\bigg) 
\, \ud\meas_x(x) \, \ud\meas_y(y)\\
& = \int t^+(x) \, \d^2(u^-(x,y),u^+(x,y)) \, \ud\meas_x(x) + \d^2(t^-(x),t^+(x))
\sup_x \int
u^{-}(x,y) \, \ud\meas_y(y)\\
&\qquad + 2 \int \sqrt{t^+(x)}  \left( \sqrt{t^+(x)} -\sqrt{t^-(x)}
\right)  \int \sqrt{u^{-}(x,y)} \left(    \sqrt{u^+(x,y)}-\sqrt{u^-(x,y)}
\right) \, \ud\meas_y(y) \, \ud\meas_x(x)\\
& \leq \left( \sqrt{\int t^+(x) \, \ud\meas_x(x)} \, \sup_x d(u^-(x,y),u^+(x,y)) +
  d(t^-(x),t^+(x)) \, \sup_x \sqrt{\int
u^{-}(x,y) \, \ud\meas_y(y)}\right)^2. 
\end{align*}
Using 
 \begin{lemma}\label{lem:intbound} 
For any $\delta$-Hellinger bracket
   $\left[t^-,t^+\right]$, $\int t^- \ud\meas \leq 1$ and
 $ \int t^+ \ud\meas \leq \left(\delta+\sqrt{1+\delta^2}\right)^2.$
 \end{lemma}
we deduce using $\delta \leq \sqrt{2}/3$
\begin{align*}
 \d^2(t^-(x) \, u^-(x,y),t^+(x) \, u^+(x,y)) &\leq \left( \delta +\sqrt{1+\delta^2} + 1\right)^2 \delta^2
\\
&\leq \left( \sqrt{2}/3 + \sqrt{1+2/9} +1 \right)^2 \delta^2
\\&
\leq 9 \delta^2
  \end{align*}
\end{proof}

 \begin{proof}[Proof of \cref{lem:helprod}]
   \begin{align*}
  \d^2_y&\left(\int_x t^-(x) \, u^-(x,y) \, \ud\meas_x(x),\int_x t^+(x) \,
    u^+(x,y) \,
    \ud\meas_{x}(x)\right)\\
 & =
  \int_y \left( \sqrt{\int_x t^+(x) \, u^+(x,y) \, \ud\meas_x(x)}
-
\sqrt{\int_x t^-(x) \, u^-(x,y) \, \ud\meas_x(x)}
\right)^2
\ud\meas_y(y)\\
& = \int_y \int_x t^+(x) \, u^+(x,y) \, \ud\meas_x(x) \,\ud\meas_y(y)
+ \int_y \int_x t^-(x) \, u^-(x,y) \, \ud\meas_x(x) \, \ud\meas_y(y)\\
&\qquad - 2 \int_y \sqrt{\int_x t^+(x) \,u^+(x,y) \, \ud\meas_x(x)} \,
\sqrt{\int_x t^-(x) \, u^-(x,y) \, \ud\meas_x(x)}
\ud\meas_y(y) \\
& \leq \int_y \int_x t^+(x) \, u^+(x,y) \, \ud\meas_x(x) \, \ud\meas_y(y)
+ \int_y \int_x t^-(x) \, u^-(x,y) \, \ud\meas_x(x) \, \ud\meas_y(y)\\
&\qquad - 2 \int_y \int_x \sqrt{t^+(x) \, u^+(x,y)} \, \sqrt{t^-(x) \,
  u^-(x,y)} \, \ud\meas_x(x) \,
\ud\meas_y(y) \\
& \leq \d^2_{x,y}\left( t^-(x) \, u^-(x,y), t^+(x) \, u^+(x,y) \right)
\end{align*}
\end{proof}

\begin{proof}[Proof of \cref{lem:intbound}]
The first point is straightforward as $t^-$ is upper-bounded by a density.

For the second point,
  \begin{align*}
    \int t^+ \, \ud\meas & =  \int \left( t^+ -t^- \right) \, \ud\meas  +  \int
    t^- \, \ud\meas 
\leq \int \left(\sqrt{t^+} -\sqrt{t^-}\right)\left(\sqrt{t^+}
  +\sqrt{t^-}\right) \, \ud\meas  + 1
\\& 
\leq 2 \int \left(\sqrt{t^+} -\sqrt{t^-}\right) \sqrt{t^+} \, \ud\meas
+ 1
\leq 2 \left( \int \left(\sqrt{t^+} -\sqrt{t^-}\right)^2 \, \ud\meas
\right)^{1/2} \left( \int t^+ d\, \meas
\right)^{1/2} +1\\
\int t^+ \ \ud\meas  & \leq 2 \delta\left( \int t^+ \, \ud\meas
\right)^{1/2} +1
  \end{align*}
Solving the corresponding inequality  yields
\begin{align*}
   \int t^+ \ud\meas
 & \leq \left( \delta + \sqrt{1+\delta^2} \right)^2.
\end{align*}
\end{proof}

\subsection{Entropy of Gaussian families (Lemma)} 

\begin{proof}[Proof of \cref{lem:gridA}]
We first define  $\tilde{g}_i$ as the set of integers such that 
\begin{align*}
  \forall 1 \leq i < \dimSpace, 
  \lambdam(1+\delta_\Amod)^{\tilde{g}_i} \leq A_{i,i} <
  \lambdam(1+\delta_\Amod)^{\tilde{g}_i+1}.
\end{align*}
By construction $\tilde{g}_i\in\N$ and
$\lambdam(1+\delta_\Amod)^{\tilde{g}_i}\leq \lambdaM$. Now
as $A_{\dimSpace,\dimSpace}=\frac{1}{\prod_{i=1}^{\dimSpace-1}
  A_{i,i}}$,
\begin{align*}
\frac{1}{\prod_{i=1}^{\dimSpace-1}
\lambdam(1+\delta_\Amod)^{\tilde{g}_i+1}}
= \frac{(1+\delta_\Amod)^{-(\dimSpace-1)}}{\prod_{i=1}^{\dimSpace-1}
 \lambdam(1+\delta_\Amod)^{\tilde{g}_i}} 
 < A_{\dimSpace,\dimSpace} \leq \frac{1}{\prod_{i=1}^{\dimSpace-1}
 \lambdam(1+\delta_\Amod)^{\tilde{g}_i}}.
\end{align*}
There is thus an integer $d$  between $0$ and $\dimSpace-2$ such that
\begin{align*}
  \frac{(1+\delta_\Amod)^{-d-1}}{\prod_{i=1}^{\dimSpace-1}
 \lambdam(1+\delta_\Amod)^{\tilde{g}_i}} < A_{\dimSpace,\dimSpace} \leq \frac{(1+\delta_\Amod)^{-d}}{\prod_{i=1}^{\dimSpace-1}
 \lambdam(1+\delta_\Amod)^{\tilde{g}_i}}.
\end{align*}
Let $g_i=\tilde{g}_i+1$ if $i\leq d$ and $g_i=\tilde{g}_i$
otherwise, then
\begin{align*}
  \forall 1 \leq i < \dimSpace, 
  \lambdam(1+\delta_\Amod)^{g_i-1} \leq A_{i,i} <
  \lambdam(1+\delta_\Amod)^{g_i+1}
\end{align*}
which implies $\lambdam(1+\delta_\Amod)^{g_i} \leq
(1+\delta_\Amod)\lambdaM$. Now
\begin{align*}
\frac{1}{\prod_{i=1}^{\dimSpace-1}
 \lambdam(1+\delta_\Amod)^{g_i}} = \frac{(1+\delta_\Amod)^{-d}}{\prod_{i=1}^{\dimSpace-1}
 \lambdam(1+\delta_\Amod)^{\tilde{g}_i}}
\end{align*}
and thus
\begin{align*}
  \frac{(1+\delta_\Amod)^{-1}}{\prod_{i=1}^{\dimSpace-1}
 \lambdam(1+\delta_\Amod)^{g_i}} < A_{\dimSpace,\dimSpace} \leq \frac{1}{\prod_{i=1}^{\dimSpace-1}
 \lambdam(1+\delta_\Amod)^{g_i}}
\end{align*}
which implies
\begin{align*}
 \lambdam \leq \frac{1}{\prod_{i=1}^{\dimSpace-1}
 \lambdam(1+\delta_\Amod)^{g_i}} \leq (1+\delta_\Amod) \lambdaM.
\end{align*}
Thus the diagonal matrix $\tilde{A}$ defined by
\begin{align*}
  \forall 1 \leq 1 \leq \dimSpace - 1, \tilde{A}_{i,i}=\lambdam (1+\delta_\Amod)^{g_i}
\end{align*}
and $\tilde{A}_{\dimSpace,\dimSpace}=\frac{1}{\prod_{i=1}^{\dimSpace-1}
  \tilde{A}_{i,i}}$ belongs to
$\Grid_{\Amod}(\lambdam,\lambdaM,\dimSpace,\delta_\Amod)$. Furthermore, we
can write
for any $1 \leq i \leq \dimSpace -1$
\begin{align*}
  &\tilde{A}_{i,i} (1+\delta_\Amod)^{-1} \leq A_{i,i} <
  \tilde{A}_{i,i} (1+\delta_\Amod)\\
\intertext{which implies}
  &\tilde{A}_{i,i}^{-1} (1+\delta_\Amod)^{-1} < A_{i,i}^{-1} <
  \tilde{A}_{i,i}^{-1} (1+\delta_\Amod) \\
\intertext{and thus}
\left| A_{i,i}^{-1}-\tilde{A}_{i,i}^{-1}\right| &\leq \tilde{A}_{i,i}^{-1} \max\left(1
+\delta_\Amod - 1, 1-(1+\delta_\Amod)^{-1}\right) = \tilde{A}_{i,i}^{-1} \max\left(
\delta_\Amod, \frac{\delta_\Amod}{1+\delta_\Amod}\right)\\
&  \leq \lambdam^{-1}
\delta_\Amod.
\end{align*}
Along the same lines,
\begin{align*}
&  (1+\delta_\Amod)^{-1} \tilde{A}_{\dimSpace,\dimSpace} \leq
  A_{\dimSpace,\dimSpace} \leq \tilde{A}_{\dimSpace,\dimSpace} 
\\
\intertext{thus}
&   \tilde{A}_{\dimSpace,\dimSpace}^{-1} \leq
  A_{\dimSpace,\dimSpace}^{-1} \leq (1+\delta_\Amod) \tilde{A}_{\dimSpace,\dimSpace}^{-1} 
\intertext{and}
&\left|  \tilde{A}_{\dimSpace,\dimSpace}^{-1} -
  A_{\dimSpace,\dimSpace}^{-1} \right| \leq
\tilde{A}_{\dimSpace,\dimSpace}^{-1} \delta_\Amod \leq \lambdam^{-1} \delta_\Amod.
\end{align*}
\end{proof}

\begin{proof}[Proof of \cref{lem:gaussbracketnet}]

We first prove that $[t^-,t^+]$ is a $\delta/9$
Hellinger bracket. As $(1+\deltaVar)\tilde{\Sigma}^{-1}-(1+\deltaVar)^{-1} \tilde{\Sigma}^{-1} = \left( (1+\deltaVar) -
    (1+\deltaVar)^{-1}\right) \tilde{\Sigma}^{-1}$ is a positive definite matrix,
  one can apply 
\begin{lemma}
\label{lemma:gaussratio}
Let $\Gauss_{(\mu_1,\Sigma_1)}$ and $\Gauss_{(\mu_2,\Sigma_2)}$ be two
Gaussian densities with full rank covariance matrix in dimension
$\dimSpace$ such that $\Sigma_1^{-1}-\Sigma_2^{-1}$ is a positive
definite matrix, for any $x\in \R^{\dimSpace}$
\begin{align*}
  \frac{\Gauss_{(\mu_1,\Sigma_1)}(x)}{\Gauss_{(\mu_2,\Sigma_2)}(x)}
& \leq \sqrt{\frac{\left|\Sigma_2\right|}{\left|\Sigma_1\right|}} \exp
  \left( \frac{1}{2} \left( \mu_1- \mu_2 \right)'
    \left( \Sigma_2 - \Sigma_1\right)^{-1} \left( \mu_1- \mu_2 \right)
    \right).
\end{align*}
\end{lemma}
proved by \textcite{maugis09:_gauss}. This yields using eventually $\kappag\geq\frac{1}{2}$
  \begin{align*}
    \frac{t^-(x)}{t^+(x)} & =
    \frac{(1+\kappag \deltaVar)^{-\dimSpace}}{(1+\kappag \deltaVar)^{\dimSpace}}
    \frac{\Gauss_{\mu,(1+\deltaVar)^{-1}
        \tilde{\Sigma}}(x)}{\Gauss_{\mu,(1+\deltaVar) \tilde{\Sigma}}(x)}
\leq \frac{1}{(1+\kappag\deltaVar)^{2 \dimSpace}} \sqrt{%
  \frac{(1+\deltaVar)^{\dimSpace}}{(1+\deltaVar)^{-\dimSpace}}}
\leq \frac{(1+\deltaVar)^{\dimSpace}}{(1+\kappag\deltaVar)^{2\dimSpace}}
\\& 
\leq \left(\frac{1+\deltaVar}{(1+\kappag\deltaVar)^2}\right)^{\dimSpace}
\leq \left(\frac{1+\deltaVar}{1+ 2 \kappag \deltaVar + \kappag^2
    \deltaVar^2}\right)^{\dimSpace}
\leq 1
  \end{align*}

Concerning the Hellinger width,
  \begin{align*}
 d^2(t^-,t^+) & = \int t^-(x) \, \ud x  + \int t^+(x) \, \ud x - 2 \int
 \sqrt{t^-(x)} \sqrt{t^+(x)} \, \ud x\\
& = (1+\kappag\deltaVar)^{-\dimSpace}+(1+\kappag\deltaVar)^{\dimSpace}\\
&\qquad\qquad- 2
(1+\kappag\deltaVar)^{-\dimSpace/2} (1+\kappag\deltaVar)^{\dimSpace/2} \int
\sqrt{\Gauss_{\mu,(1+\deltaVar)^{-1} \tilde{\Sigma}}(x)}
\sqrt{\Gauss_{\mu,(1+\deltaVar) \tilde{\Sigma}}(x)} \, \ud x\\
& = (1+\kappag\deltaVar)^{-\dimSpace}+(1+\kappag\deltaVar)^{\dimSpace}-
\left(2-d^2\left(\Gauss_{\mu,(1+\deltaVar)^{-1}
      \tilde{\Sigma}}(x),\Gauss_{\mu,(1+\deltaVar)^{-1}
      \tilde{\Sigma}}(x)\right)\right).
\end{align*}
Using
\begin{lemma}
\label{lemma:gausshel}
Let $\Gauss_{(\mu_1,\Sigma_1)}$ and $\Gauss_{(\mu_2,\Sigma_2)}$ be two
Gaussian densities with full rank covariance matrix in dimension $\dimSpace$,
\begin{align*}
\mspace{-10mu}d^2\left(\Gauss_{(\mu_1,\Sigma_1)},\Gauss_{(\mu_2,\Sigma_2)}
\right) = 2 \left( 1 - 2^{\dimSpace/2} \left|\Sigma_1
    \Sigma_2\right|^{-1/4} \left| \Sigma_1^{-1} + \Sigma_2^{-1}
  \right|^{-1/2}\exp \left( - \frac{1}{4} \left( \mu_1- \mu_2 \right)'
    \left( \Sigma_1 + \Sigma_2\right)^{-1} \left( \mu_1- \mu_2 \right)
    \right)
\right).
\end{align*}
\end{lemma}
also proved in \cite{maugis09:_gauss}, we derive
\begin{align*}
 d^2(t^-,t^+) & = \int t^-(x) \, \ud x  + \int t^+(x) \, \ud x - 2 \int
 \sqrt{t^-(x)} \sqrt{t^+(x)} \, \ud x\\
&= (1+\kappag\deltaVar)^{-\dimSpace}+(1+\kappag\deltaVar)^{\dimSpace}-2\,
2^{\dimSpace/2} \left( (1+\deltaVar) + (1+\deltaVar)^{-1}
\right)^{-\dimSpace/2}\\
& = 2 - 2\,
2^{\dimSpace/2} \left( (1+\deltaVar) + (1+\deltaVar)^{-1}
\right)^{-\dimSpace/2} + (1+\kappag\deltaVar)^{-\dimSpace}+(1+\kappag\deltaVar)^{\dimSpace}-2
\end{align*}
Combining
\begin{lemma}
\label{lem:bounddelta}
For any $0<\delta\leq \sqrt{2}$ and any $\dimSpace\geq 1$, let $\kappag\geq
\frac{3}{4}$ and $\constcosh=\sqrt{\kappag^2 \cosh(\frac{\kappag}{6}) +
   \frac{1}{2}}$, if $\deltaVar \leq
\frac{1}{9\constcosh}\frac{\delta}{\dimSpace}$, then 
\begin{align*}
 \deltaVar \leq \frac{1}{\dimSpace} \frac{1}{6} \leq \frac{1}{6}.
\end{align*}
\end{lemma}
and
\begin{lemma}
\label{lemma:cosh}
For any $d \in \N$, for any $\deltaVar>0$,
\begin{align*}
  2-2\, 2^{d/2}\, \left( (1+\deltaVar)+(1+\deltaVar)^{-1}\right)^{-d/2}
\leq \frac{d^2 \deltaVar^2}{2}.
\end{align*}
Furthermore, if $d \deltaVar \leq c$, then
\begin{align*}
 (1+\kappag \deltaVar)^d+(1+\kappag \deltaVar)^{-d}-2 \leq \kappag^2 \cosh(\kappag c) d^2 \deltaVar^2.
\end{align*}
\end{lemma}
with $c=\frac{1}{6}$
yields
\begin{align*}
 d^2(t^-,t^+) & \leq \left(  \kappag^2 \cosh(\frac{\kappag}{6}) +
   \frac{1}{2} \right) \dimSpace^2 \deltaVar^2
\leq \left(\frac{\delta}{9}\right)^2
  \end{align*}
as  $\deltaVar \leq
\frac{1}{9\constcosh}\frac{\delta}{\dimSpace}$

We now focus on the proof of \(\displaystyle
t^-(x) \leq \Gauss_{\mu,\Sigma}(x) \leq t^+(x)
\). As
\begin{lemma}
\label{lem:eigenvaluesinverse}
Under Assumptions of \cref{lem:gaussbracketnet},
$(1+\deltaVar)\tilde{\Sigma}^{-1}-\Sigma^1$ and
$\Sigma^{-1}-(1+\deltaVar)\tilde{\Sigma}^{-1}$ are positive definite and satisfies
\begin{align*}
\forall x \in \R^{\dimSpace}, x' \left(
  (1+\deltaVar)\tilde{\Sigma}^{-1}-\Sigma^{-1} \right) x&\geq \frac{1}{4}
\tilde{L}^{-1} \frac{1}{\lambdaM} \deltaVar \|x\|^2 \\
\forall x \in \R^{\dimSpace}, x' \left(
 \Sigma^{-1}-(1+\deltaVar)^{-1} \tilde{\Sigma}^{-1}\right) x&\geq \frac{3}{4(1+\deltaVar)}
\tilde{L}^{-1} \frac{1}{\lambdaM} \deltaVar \|x\|^2
 \end{align*}
\end{lemma}
we can apply
\cref{lemma:gaussratio} on Gaussian density ratio to both 
\begin{align*}
  \frac{\Gauss_{\mu,\Sigma}(x)}{(1+\kappag \deltaVar)^{\dimSpace}
    \Gauss_{\tilde{\mu},(1+\deltaVar) \tilde{\Sigma}}(x)} 
    \quad\text{and}\quad
 \frac{(1+\kappag\deltaVar)^{-\dimSpace}
      \Gauss_{\tilde{\mu},(1+\deltaVar)^{-1}
        \tilde{\Sigma}}(x)}{\Gauss_{\mu,\Sigma}(x)}
 \end{align*}
to prove that they are smaller than $1$.

For the first one, using
\begin{align*}
  \frac{\Gauss_{\mu,\Sigma}(x)}{(1+\kappag \deltaVar)^{\dimSpace}
    \Gauss_{\tilde{\mu},(1+\deltaVar) \tilde{\Sigma}}(x)}
& \leq (1+\kappag \deltaVar)^{-\dimSpace} \left( \sqrt{\frac{|(1+\deltaVar)
      \tilde{\Sigma}|}{|\Sigma|}} \exp \left( 
 \frac{1}{2} (\mu-\tilde{\mu})' \left( (1+\deltaVar) \tilde{\Sigma} - \Sigma \right)^{-1} (\mu-\tilde{\mu})
 \right)
 \right)\\
& \leq \frac{(1+\deltaVar)^{\dimSpace/2}}{(1+\kappag \deltaVar)^{\dimSpace}} \left( \sqrt{\frac{|
      \tilde{\Sigma}|}{|\Sigma|}} \exp \left( 
 \frac{1}{2} (\mu-\tilde{\mu})' \left( (1+\deltaVar) \tilde{\Sigma} - \Sigma \right)^{-1} (\mu-\tilde{\mu})
 \right)
 \right).
\end{align*}
Now 
\begin{align*}
\left( (1+\deltaVar) \tilde{\Sigma} - \Sigma \right)^{-1}&= \left( (1+\deltaVar) \tilde{\Sigma} \left(
\Sigma^{-1}-(1+\deltaVar)^{-1} \tilde{\Sigma}^{-1} \right) \Sigma
\right)^{-1}\\
& = (1+\deltaVar)^{-1} \Sigma^{-1} \left(
\Sigma^{-1}-(1+\deltaVar)^{-1} \tilde{\Sigma}^{-1} \right)^{-1} \tilde{\Sigma}^{-1}
\end{align*}
and thus
\begin{align*}
(\mu-\tilde{\mu})' \left( (1+\deltaVar) \tilde{\Sigma} - \Sigma \right)^{-1}
(\mu-\tilde{\mu})
&\leq (1+\deltaVar)^{-1} \Lm^{-1} \lambdam^{-1} 
\frac{4(1+\deltaVar)}{3}
\tilde{L} \lambdaM \deltaVar^{-1}
\tilde{L}^{-1}
\lambdam^{-1} \|\mu-\tilde{\mu}\|^2 
\\& \leq  \frac{4(1+\deltaVar)}{3} 
(1+\deltaVar)^{-1} \deltaVar^{-1}\Lm^{-1} \lambdam^{-1} \frac{\lambdaM}{\lambdam}
 \dimSpace \gammakappa \Lm \lambdam
 \frac{\lambdam}{\lambdaM}\deltaVar^2
\\& 
\leq \frac{4}{3} \gammakappa \dimSpace \deltaVar
\end{align*}
Now as by construction,
\begin{align*}
\frac{|\tilde{\Sigma}|}{|\Sigma|} &\leq (1+\frac{1}{2}\deltaVar)^{\dimSpace}, 
\end{align*}
one obtains
\begin{align*}
  \frac{\Gauss_{\mu,\Sigma}}{(1+\kappag \delta)^{\dimSpace}
    \Gauss_{\tilde{\mu},(1+\delta) \tilde{\Sigma}}}
&
\leq \frac{(1+\deltaVar)^{\dimSpace/2}}{(1+\kappag \deltaVar)^{\dimSpace}} 
(1+\frac{1}{2}\deltaVar)^{\dimSpace/2} 
\exp \left(
\frac{1}{2} 
\frac{4}{3} \gammakappa \dimSpace \deltaVar
\right)\\
&\leq \Bigg( \frac{\sqrt{1+\deltaVar}\sqrt{1+\frac{1}{2}\deltaVar}}{1+\kappag\deltaVar}
  \exp \left( 
\frac{2}{3}
\gammakappa\deltaVar
\right) \Bigg)^{\dimSpace}.
\end{align*}
It is thus sufficient to prove that
\begin{align*}
  \frac{\sqrt{1+\deltaVar}\sqrt{1+\frac{1}{2}\deltaVar}}{1+\kappag\deltaVar}
  \exp \left( 
\frac{2}{3}
\gammakappa \deltaVar
\right) \leq 1
\end{align*}
or equivalently
\begin{align*}
  \frac{2}{3}
 \gamma_k  \deltaVar \leq \Log \left(
  \frac{1+\kappag\deltaVar}{\sqrt{1+\deltaVar}\sqrt{1+\frac{1}{2} \deltaVar}} \right).
\end{align*}
Now let 
\begin{align*}
f_1(\deltaVar) & = \Log \left(
  \frac{1+\kappag\deltaVar}{\sqrt{1+\deltaVar}\sqrt{1+\frac{1}{2} \deltaVar}}
\right)
= \Log ( 1+ \kappag \deltaVar) - \frac{1}{2}
  \Log(1+\deltaVar) - \frac{1}{2} \Log(1+\frac{1}{2} \deltaVar)\\
f_1'(\deltaVar) &= \frac{\kappag}{1+\kappag\deltaVar} -
\frac{\frac{1}{2}}{1+\deltaVar} -
\frac{\frac{1}{4}}{1+\frac{1}{2}\deltaVar}
= \frac{\frac{3}{4} ( \kappag - \frac{2}{3}) \deltaVar + \kappag -
  \frac{3}{4}}
{(1+\kappag\deltaVar)(1+\deltaVar)(1+\frac{1}{2} \deltaVar)}
 \intertext{and thus provided $\kappag>\frac{3}{4}$, as $\deltaVar
   \leq \frac{1}{6}$}
 f_1'(\deltaVar) &>
 \frac{\kappag-\frac{3}{4}}{(1+\frac{\kappag}{6})(1+\frac{1}{6})(1+\frac{1}{12})}\\
\intertext{Finally, as $f_1(0)=0$, one deduces}
f_1(\deltaVar) &>
\frac{\kappag-\frac{3}{4}}{(1+\frac{\kappag}{6})(1+\frac{1}{6})(1+\frac{1}{12})}
\deltaVar \geq
\frac{2}{3}
 \gamma_k  \deltaVar
\end{align*}
which implies thus   
\begin{align*}
\frac{\Gauss_{\mu,\Sigma}(x)}{(1+\kappag \deltaVar)^{\dimSpace}
    \Gauss_{\tilde{\mu},(1+\deltaVar) \tilde{\Sigma}}(x)} \leq 1
\end{align*} 
or $\Gauss_{\mu,\Sigma}(x) \leq t^+(x)$.

The second case is handled in the same way.
\begin{align*}
   &\frac{(1+\kappag\deltaVar)^{-\dimSpace}
      \Gauss_{\tilde{\mu},(1+\deltaVar)^{-1}
        \tilde{\Sigma}}(x)}{\Gauss_{\mu,\Sigma}(x)} \\
&\qquad \leq (1+\kappag\deltaVar)^{-\dimSpace}
\left( \sqrt{\frac{|
      \Sigma|}{|(1+\deltaVar)^{-1}\tilde{\Sigma}|}} 
\exp \left( 
 \frac{1}{2} (\mu-\tilde{\mu})' \left( \Sigma - (1+\deltaVar)^{-1} \tilde{\Sigma} \right)^{-1} (\mu-\tilde{\mu})
 \right)
 \right)\\
&\qquad\leq \frac{(1+\deltaVar)^{\dimSpace/2}}{(1+\kappag\deltaVar)^{\dimSpace}}
\exp \left( 
 \frac{1}{2} (\mu-\tilde{\mu})' \left( \Sigma - (1+\deltaVar)^{-1} \tilde{\Sigma} \right)^{-1} (\mu-\tilde{\mu})
 \right)
\end{align*}
Now as
\begin{align*}
  \left( \Sigma - (1+\deltaVar)^{-1} \tilde{\Sigma} \right)^{-1} &= \left(
    \Sigma \left( (1+\deltaVar) \tilde{\Sigma}^{-1} - \Sigma^{-1} \right) (1+\deltaVar)^{-1}
    \tilde{\Sigma} \right)^{-1}
\\& 
= (1+\deltaVar) \tilde{\Sigma}^{-1} \left( (1+\deltaVar) \tilde{\Sigma}^{-1} -
  \Sigma^{-1} \right)^{-1} \Sigma^{-1}
\end{align*}
and thus
\begin{align*}
(\mu-\tilde{\mu})'     \left( \Sigma - (1+\deltaVar)^{-1} \tilde{\Sigma}
\right)^{-1} (\mu-\tilde{\mu})
 &\leq  (1 + \deltaVar) \tilde{L}^{-1} \lambdam^{-1} 
4
\tilde{L} \lambdaM \deltaVar^{-1}
\Lm^{-1} \lambdam^{-1}
\|\mu-\tilde{\mu}\|^2
\\& 
\leq (1 + \deltaVar) \Lm^{-1} \lambdam^{-1} 
4 \frac{\lambdaM}{\lambdam} \deltaVar^{-1}
 \dimSpace \gammakappa \Lm \lambdam
 \frac{\lambdam}{\lambdaM}\deltaVar^2
\\& 
\leq   4 \dimSpace \gammakappa  (1 + \deltaVar) \deltaVar
\end{align*}
one deduces
\begin{align*}
\frac{(1+\kappag\deltaVar)^{-\dimSpace}
      \Gauss_{\tilde{\mu},(1+\deltaVar)^{-1}
        \tilde{\Sigma}}(x)}{\Gauss_{\mu,\Sigma}(x)}
 &\leq \frac{(1+\deltaVar)^{\dimSpace/2}}{(1+\kappag\deltaVar)^{\dimSpace}}
 \exp \left( 
  \frac{1}{2} 
 4 \dimSpace \gammakappa  (1 + \deltaVar) \deltaVar
  \right)\\
& \leq \left( 
\frac{\sqrt{1+\deltaVar}}{1+\kappag\deltaVar}
 \exp \left( 
2 \gammakappa  (1 + \deltaVar)
\deltaVar\right)
\right)^{\dimSpace}.
\end{align*}
All we need to prove is thus
\begin{align*}
  \frac{\sqrt{1+\deltaVar}}{1+\kappag\deltaVar}
 \exp \left( 
2 \gammakappa  (1 + \deltaVar)
\deltaVar\right) \leq 1
\end{align*}
or equivalently
\begin{align*}
2 \gammakappa  (1 + \deltaVar)
\deltaVar \leq \Log \left(
  \frac{1+\kappag\deltaVar}{\sqrt{1+\deltaVar}} \right).
\end{align*}
Let 
\begin{align*}
  f_2(\deltaVar) & = \Log \left(
  \frac{1+\kappag\deltaVar}{\sqrt{1+\deltaVar}} \right)
= \Log(1+\kappag\deltaVar)-\frac{1}{2} \Log(1+\deltaVar)\\
f_2'(\deltaVar) & = \frac{\kappag}{1+\kappag\deltaVar}-\frac{%
  \frac{1}{2}}{1+\deltaVar}
 = \frac{\frac{\kappag}{2} \deltaVar +
  \kappag-\frac{1}{2}}{(1+\kappag\deltaVar)(1+\deltaVar)}\\
\intertext{and thus provided $\kappag>\frac{3}{4}$, as $\deltaVar
   \leq \frac{1}{6}$}
f_2'(\deltaVar) & > \frac{\kappag-\frac{1}{2}}{(1+\frac{\kappag}{6})(1+\frac{1}{6})}
\intertext{Finally, as $f_2(0)=0$, one deduces}
f_2(\deltaVar) &>
\frac{\kappag-\frac{1}{2}}{(1+\frac{\kappag}{6})(1+\frac{1}{6})}
\deltaVar \geq 2 \gammakappa  (1 + \frac{1}{6})
\deltaVar \geq 2 \gammakappa  (1 + \deltaVar)
\deltaVar  
\end{align*}
which implies
\begin{align*}
   \frac{(1+\kappag\deltaVar)^{-\dimSpace}
      \Gauss_{\tilde{\mu},(1+\deltaVar)^{-1}
        \tilde{\Sigma}}(x)}{\Gauss_{\mu,\Sigma}(x)}  \leq 1  
\end{align*}
or equivalently $t^-(x)\leq \Gauss_{\mu,\Sigma}(x)$.
\end{proof}

\begin{proof}[Proof of \cref{lem:bounddelta}]
A straightforward computation yields
\begin{align*}
  \deltaVar & \leq \frac{1}{9\constcosh}\frac{\delta}{\dimSpace} 
\leq \frac{1}{\dimSpace} \frac{\sqrt{2}}{9\sqrt{\left(\frac{3}{4}\right)^2 +
    \frac{1}{2}}}  \leq \frac{1}{\dimSpace} \frac{1}{6} \leq \frac{1}{6}.
\end{align*}
\end{proof}

\begin{proof}[Proof of \cref{lemma:cosh}]
  \begin{align*}
    2-2\, 2^{d/2}\, \left( (1+\deltaVar)+(1+\deltaVar)^{-1}\right)^{-d/2} & =
    2\left( 1 - \left(
        \frac{e^{\Log(1+\deltaVar)}+e^{-\Log(1+\deltaVar)}}{2}\right)^{-d/2}
      \right)\\
& = 2 \left( 1 - \left( \cosh\left(\Log (1+\deltaVar)\right)\right)^{-d/2}
  \right)\\
&   = 2 f\left(\Log(1+\deltaVar)\right)\\
  \end{align*}
where $f(x)=1-\cosh(x)^{-d/2}$. Studying this function yields
\begin{align*}
  f'(x) & = \frac{d}{2} \sinh(x) \cosh(x)^{-d/2-1}\\
  f''(x) & = \frac{d}{2} \cosh(x)^{-d/2} - \frac{d}{2} \left(
    \frac{d}{2}+1 \right) \sinh(x)^2 \cosh(x)^{-d/2-2}\\
& = \frac{d}{2} \left( 1-\left( \frac{d}{2} +1 \right) \left(
    \frac{\sinh(x)}{\cosh(x)} \right)^2 \right) \cosh(x)^{-d/2}\\
\intertext{and, as $\cosh(x)\geq 1$,}
f''(x) & \leq \frac{d}{2}.
\intertext{%
Now as $f(0)=0$ and $f'(0)=0$, this implies for any $x\geq0$}
  f(x) & \leq \frac{d}{2} \frac{x^2}{2} \leq \frac{d^2}{2} \frac{x^2}{2}.
\end{align*}
We deduce thus that
\begin{align*}
    2-2\, 2^{d/2}\, \left( (1+\deltaVar)+(1+\deltaVar)^{-1}\right)^{-d/2}
   & \leq \frac{1}{2} d^2 \left(\Log(1+\deltaVar)\right)^2 \\
\intertext {and using $\Log(1+\deltaVar)\leq\deltaVar$}
  2-2\, 2^{d/2}\, \left( (1+\deltaVar)+(1+\deltaVar)^{-1}\right)^{-d/2} & \leq \frac{1}{2} d^2 \deltaVar^2.
\end{align*}

Now,
\begin{align*}
  \left(1+\kappag\deltaVar\right)^d +  \left(1+\kappag\deltaVar\right)^{-d} -2
& = 2\left( \cosh\left(d \Log(1+\kappag\deltaVar)\right)-1 \right)
= 2 g\left(d \Log(1+\kappag\deltaVar)\right)
\end{align*}
with $g(x)= \cosh(x)-1$. Studying this function yields
\begin{align*}
  g'(x)&=\sinh(x)\quad\text{and}\quad
  g''(x)  = \cosh(x)
\intertext{and thus, as $g(0)=0$ and $g'(0)=0$, for any $0\leq x \leq c$}
g(x) & \leq \cosh(c) \frac{x^2}{2}.
\end{align*}
As $\Log(1+\kappag\deltaVar) \leq \kappag\deltaVar$, $d\deltaVar\leq c$ implies
$d\Log(1+\kappag\deltaVar) \leq \kappag c$, we obtain thus
\begin{align*}
  \left(1+\kappag\deltaVar\right)^d +  \left(1+\kappag\deltaVar\right)^{-d} -2
&\leq \cosh(\kappag c) d^2   \left(\Log(1+\kappag\deltaVar)\right)^2
 \leq \kappag^2 \cosh(\kappag c) d^2 \deltaVar^2.
\end{align*}
\end{proof}

\begin{proof}[Proof of \cref{lem:eigenvaluesinverse}]
We deduce this result from a slightly more general:
\begin{lemma}
\label{lem:eigenvaluesinversespec}
Let $\deltaVar>0$.

Let $(L,A,D) \in [\Lm,\LM] \times 
\Diag(\lambdam,\lambdaM) \times SO(\dimSpace)$ and
$(\tilde{L},\tilde{A},\tilde{D})\in [\Lm,\LM] \times 
\Diag(\lambdam,+\infty) \times SO(\dimSpace)$ , define
$\Sigma=L D A D'$ and  $\tilde{\Sigma}=\tilde{L} \tilde{D}
\tilde{A} \tilde{D'}$.

If
\begin{align*}
\begin{cases}
(1+\delta_\Lmod)^{-1} \tilde{L} \leq L \leq \tilde{L}\\
\forall 1\leq i \leq \dimSpace,\quad
|A_{i,i}^{-1}-\tilde{A}_{i,i}^{-1}| \leq 
\delta_\Amod \lambdam^{-1}
\\
 \forall x \in \R^{\dimSpace},\quad \|Dx-\tilde{D}x\| \leq
\delta_\Dmod \|x\|
\end{cases}
\end{align*}
then $(1+\deltaVar)\tilde{\Sigma}^{-1}-\Sigma^{-1}$ and
$\Sigma^{-1}-(1+\deltaVar)^{-1}\tilde{\Sigma}^{-1}$  satisfies
\begin{align*}
\forall x \in \R^{\dimSpace}, x' \left(
  (1+\deltaVar)\tilde{\Sigma}^{-1}-\Sigma^{-1} \right) x&\geq 
\tilde{L}^{-1} \left( (\deltaVar-\delta_\Lmod) \lambdaM^{-1} -
  (1+\deltaVar)\lambdam^{-1} \left( 
2 \delta_\Dmod  + \delta_\Amod \right) \right)
\|x\|^2 \\
\forall x \in \R^{\dimSpace}, x' \left(
 \Sigma^{-1}-(1+\deltaVar)^{-1} \tilde{\Sigma}^{-1}\right) x&\geq 
\frac{\tilde{L}^{-1}}{1+\deltaVar} \left(
  \deltaVar \lambdaM^{-1} - \lambdam^{-1} \left( 
2 \delta_\Dmod  + \delta_\Amod \right) \right)
\|x\|^2
 \end{align*}
\end{lemma}

Indeed \cref{lem:bounddelta} ensures that $\deltaVar\leq\frac{1}{6}$.
Hence, if we let $\delta_\Lmod=\frac{1}{2} \deltaVar$ and
  $\delta_\Amod=\delta_\Dmod=\frac{1}{14}
  \frac{\lambdam}{\lambdaM} \deltaVar $, bounds of the previous
  Lemma become $\forall x \in \R^{\dimSpace},$
\begin{align*}
 x' \left(
  (1+\deltaVar)\tilde{\Sigma}^{-1}-\Sigma^{-1} \right) x&\geq 
\tilde{L}^{-1} \left( (\deltaVar-\delta_\Lmod) \lambdaM^{-1} -
  (1+\deltaVar)\lambdam^{-1} \left( 
2 \delta_\Dmod  + \delta_\Amod \right) \right)
\|x\|^2 \\
& \geq \tilde{L}^{-1} \left( \left( \deltaVar - \frac{1}{2} \deltaVar\right)  \lambdaM^{-1} -
  (1+\deltaVar)\lambdam^{-1} 3   \frac{1}{14}
  \frac{\lambdam}{\lambdaM} \deltaVar \right)
\|x\|^2\\
& \geq \frac{1}{4}
\tilde{L}^{-1} \frac{1}{\lambdaM} \deltaVar \|x\|^2 \\
\intertext{while $\forall x \in \R^{\dimSpace},$}
x' \left(
 \Sigma^{-1}-(1+\deltaVar)^{-1} \tilde{\Sigma}^{-1}\right) x&\geq 
\frac{\tilde{L}^{-1}}{1+\deltaVar} \left(
  \deltaVar \lambdaM^{-1} - \lambdam^{-1} \left( 
2 \delta_\Dmod  + 1\delta_\Amod \right) \right)
\|x\|^2\\
& \geq
\frac{\tilde{L}^{-1}}{1+\deltaVar} \left(
  \deltaVar \lambdaM^{-1} - \lambdam^{-1} 3 \frac{1}{14}
  \frac{\lambdam}{\lambdaM} \deltaVar\right)
\|x\|^2\\
&\geq \frac{3}{4(1+\deltaVar)}
\tilde{L}^{-1} \frac{1}{\lambdaM} \deltaVar \|x\|^2.
 \end{align*}
\end{proof}

\begin{proof}[Proof of \cref{lem:eigenvaluesinversespec}]
By definition,
\begin{align*}
  x' \left( (1+\deltaVar)\tilde{\Sigma}^{-1}-\Sigma^{-1} \right) x & =
  (1+\deltaVar) \tilde{L}^{-1} \sum_{i=1}^{\dimSpace} \tilde{A}_{i,i}^{-1}
  |\tilde{D}_{i}'x|^2
- L^{-1} \sum_{i=1}^{\dimSpace} A_{i,i}^{-1}
  |D_{i}'x|^2\\
& = (1+\deltaVar) \tilde{L}^{-1} \sum_{i=1}^{\dimSpace} \tilde{A}_{i,i}^{-1}
  |\tilde{D}_{i}'x|^2
- (1+\deltaVar) \tilde{L}^{-1} \sum_{i=1}^{\dimSpace} \tilde{A}_{i,i}^{-1}
  |D_{i}'x|^2
 \\
& \quad + (1+\deltaVar) \tilde{L}^{-1} \sum_{i=1}^{\dimSpace} \tilde{A}_{i,i}^{-1}
  |D_{i}'x|^2 - (1+\deltaVar) \tilde{L}^{-1} \sum_{i=1}^{\dimSpace} A_{i,i}^{-1}
  |D_{i}'x|^2\\
& \quad + (1+\deltaVar) \tilde{L}^{-1} \sum_{i=1}^{\dimSpace} A_{i,i}^{-1}
  |D_{i}'x|^2 - L^{-1} \sum_{i=1}^{\dimSpace} A_{i,i}^{-1}
  |D_{i}'x|^2
\end{align*}
Similarly,
\begin{align*}
  x' \left( \Sigma^{-1}-(1+\deltaVar)^{-1}\tilde{\Sigma}^{-1} \right) x & =
L^{-1} \sum_{i=1}^{\dimSpace} A_{i,i}^{-1}
  |D_{i}'x|^2 -   (1+\deltaVar)^{-1} \tilde{L}^{-1} \sum_{i=1}^{\dimSpace} \tilde{A}_{i,i}^{-1}
  |\tilde{D}_{i}'x|^2\\
& = L^{-1} \sum_{i=1}^{\dimSpace} A_{i,i}^{-1}
  |D_{i}'x|^2 -   (1+\deltaVar)^{-1} \tilde{L}^{-1} \sum_{i=1}^{\dimSpace} A_{i,i}^{-1}
  |D_{i}'x|^2 \\
& \quad + (1+\deltaVar)^{-1} \tilde{L}^{-1} \sum_{i=1}^{\dimSpace} A_{i,i}^{-1}
  |D_{i}'x|^2  - (1+\deltaVar)^{-1} \tilde{L}^{-1} \sum_{i=1}^{\dimSpace} \tilde{A}_{i,i}^{-1}
  |D_{i}'x|^2 \\
& \quad + (1+\deltaVar)^{-1} \tilde{L}^{-1} \sum_{i=1}^{\dimSpace} \tilde{A}_{i,i}^{-1}
  |D_{i}'x|^2 - (1+\deltaVar)^{-1} \tilde{L}^{-1} \sum_{i=1}^{\dimSpace} \tilde{A}_{i,i}^{-1}
  |D_{j_D,i}'x|^2
\end{align*}

Now
\begin{align*}
 \left|  \sum_{i=1}^{\dimSpace} \tilde{A}_{i,i}^{-1}
  |\tilde{D}_{i}'x|^2
-  \sum_{i=1}^{\dimSpace} \tilde{A}_{i,i}^{-1}
  |D_{i}'x|^2 \right| & \leq  \sum_{i=1}^{\dimSpace} \tilde{A}_{i,i}^{-1}
 \left| |\tilde{D}_{i}'x|^2 - |D_{i}'x|^2 \right|\\
& \leq  \lambdam^{-1} \sum_{i=1}^{\dimSpace}
\left| |\tilde{D}_{i}'x|^2 - |D_{i}'x|^2 \right|   \\
& \leq  \lambdam^{-1} \sum_{i=1}^{\dimSpace}
\left| |\tilde{D}_{i}'x| - |D_{i}'x| \right| \left| |\tilde{D}_{i}'x| + |D_{i}'x|
\right|\\
& \leq   \lambdam^{-1} \left( \sum_{i=1}^{\dimSpace}
\left| (\tilde{D}_{i}-D_{i})'x\right|^2 \right)^{1/2} 
\left( \sum_{i=1}^{\dimSpace}
\left| (\tilde{D}_{i}+D_{i})'x\right|^2 \right)^{1/2} \\
& \leq \lambdam^{-1} \delta_\Dmod \|x\| 2 \|x\| = \lambdam^{-1}
2 \delta_\Dmod \|x\|^2.
\end{align*}
Furthermore,
\begin{align*}
\left|  \sum_{i=1}^{\dimSpace} \tilde{A}_{i,i}^{-1}
  |D_{i}'x|^2 - \sum_{i=1}^{\dimSpace} A_{i,i}^{-1}
  |D_{i}'x|^2 \right| & \leq \sum_{i=1}^{\dimSpace} \left|
\tilde{A}_{i,i}^{-1} - A_{i,i}^{-1} \right|
  |D_{i}'x|^2 \\
& \leq \delta_\Amod \lambdam^{-1}\sum_{i=1}^{\dimSpace} |D_{i}'x|^2 = 
\delta_\Amod \lambdam^{-1} \|x\|^2.
\end{align*}

We then notice that
\begin{align*}
(1+\deltaVar) \tilde{L}^{-1} \sum_{i=1}^{\dimSpace} A_{i,i}^{-1}
  |D_{i}'x|^2 - L^{-1} \sum_{i=1}^{\dimSpace} A_{i,i}^{-1}
  |D_{i}'x|^2 & = \left( (1+\deltaVar) \tilde{L}^{-1} -L^{-1} \right) \sum_{i=1}^{\dimSpace} A_{i,i}^{-1}
  |D_{i}'x|^2\\
& \geq (\deltaVar-\delta_\Lmod) \tilde{L}^{-1} \lambdaM^{-1} \|x\|^2
\end{align*}
while
\begin{align*}
  L^{-1} \sum_{i=1}^{\dimSpace} A_{i,i}^{-1}
  |D_{i}'x|^2 -   (1+\deltaVar)^{-1} \tilde{L}^{-1} \sum_{i=1}^{\dimSpace} A_{i,i}^{-1}
  |D_{i}'x|^2 & = \left( L^{-1} - (1+\deltaVar)^{-1} \tilde{L}^{-1} \right)
  \sum_{i=1}^{\dimSpace} A_{i,i}^{-1}
  |D_{i}'x|^2\\
& \geq \left( 1-(1+\deltaVar)^{-1} \right) \tilde{L}^{-1}
\lambdaM^{-1} \|x\|^2\\
& \geq \frac{\deltaVar}{1+\deltaVar} \lambdaM^{-1} \|x\|^2
\end{align*}

We deduce thus that
\begin{align*}
  x' \left( (1+\deltaVar)\tilde{\Sigma}^{-1}-\Sigma^{-1} \right) x & \geq
  (\deltaVar-\delta_\Lmod) \tilde{L}^{-1} \lambdaM^{-1} \|x\|^2 - (1+\deltaVar)
\tilde{L}^{-1} \lambdam^{-1}
  \left( 
2 \delta_\Dmod  + 2 \delta_\Amod \right) \|x\|^2\\
& \geq  \tilde{L}^{-1} \left( (\deltaVar-\delta_\Lmod) \lambdaM^{-1} -
  (1+\deltaVar)\lambdam^{-1} \left( 
2 \delta_\Dmod  + \delta_\Amod \right) \right) \|x\|^2 
\intertext{and}
  x' \left( \Sigma^{-1}-(1+\deltaVar)^{-1}\tilde{\Sigma}^{-1} \right) x & \geq
\frac{\deltaVar}{1+\deltaVar} \lambdaM^{-1} \|x\|^2
- (1+\deltaVar)^{-1} \tilde{L}^{-1} \lambdam^{-1} \left( 
2 \delta_\Dmod  + \delta_\Amod \right) \|x\|^2\\
& \geq
\frac{\tilde{L}^{-1}}{1+\deltaVar} \left(
  \deltaVar \lambdaM^{-1} - \lambdam^{-1} \left( 
2 \delta_\Dmod  + \delta_\Amod \right) \right) \|x\|^2
\end{align*}
\end{proof}
}{}


\bibliographystyle{plainnat}
\bibliography{NotesBiblio}

\end{document}